\documentclass[11pt]{article}
\usepackage[margin=1in]{geometry}

  \usepackage{phaistos}

\usepackage{tipa}

\usepackage{listings}
\usepackage{graphicx}
\usepackage{amsmath,amsthm,amsfonts,amssymb,amscd}

\usepackage[colorlinks=true, pdfstartview=FitV, linkcolor=blue,citecolor=blue, urlcolor=blue]{hyperref}
\usepackage{times}

\usepackage{wrapfig}
\usepackage{mathtools}
\usepackage{mathabx}
\usepackage{float}
\usepackage{makeidx}
\usepackage{stmaryrd}
\usepackage{mathrsfs}
\usepackage{emptypage}
\usepackage{amsrefs}
\usepackage{xfrac}

\usepackage{tikzsymbols}

\usepackage{enumitem}

\usepackage{upgreek}
\allowdisplaybreaks
 
\newcommand{\XXX}{{\mathcal X}}
\newcommand{\FFF}{{\mathcal F}}
\newcommand{\LLL}{{\mathcal L}}
\renewcommand{\epsilon}{\varepsilon}

\newcommand{\tempred}[1]{#1}

\def\les{\lesssim}

\def\eps{\varepsilon}
\renewcommand*{\div}{\ensuremath{\mathrm{div\,}}}

\newcommand{\norm}[1]{\left \| #1 \right\|} 
\newcommand{\snorm}[1]{\bigl\| #1 \bigr\|} 
\newcommand{\abs}[1]{\left|#1\right|}
\newcommand{\sabs}[1]{\bigl|#1\bigr|}

\newcommand{\RR}{\mathbb R}

\newcommand{\OO}{\mathcal O}
\newcommand{\RSZ}{\mathcal R}
\renewcommand*{\tilde}{\widetilde}

\renewcommand*{\bar}{\overline}
\newcommand{\Ncal}{{\mathsf{N}}}
\newcommand{\Tcal}{{\mathsf{T}}}
\newcommand{\Vcal}{{\mathsf{v}}}
\newcommand{\Jcal}{{\mathsf{J}}}

\newcommand{\acal}{{\mathsf{a}}}
\newcommand{\bcal}{{\mathsf{b}}}

\newcommand{\xcal}{{\text{x}}}
\newcommand{\tcal}{{\text{t}}}
\newcommand{\scal}{{\text{k}}}

\newcommand{\EE}{\mathcal E}
\newcommand{\PP}{\mathcal P}

\newcommand{\ppp}{\psi}

 \newcommand*{\Id}{\ensuremath{\mathrm{Id\,}}}

\newtheorem{theorem}{Theorem}[section]
\newtheorem{lemma}[theorem]{Lemma}

\newtheorem{proposition}[theorem]{Proposition}
\newtheorem{corollary}[theorem]{Corollary}
\theoremstyle{definition}
\newtheorem{definition}[theorem]{Definition}

\newtheorem{remark}[theorem]{Remark}

\numberwithin{equation}{section}

\def\sound{S}

\def\p{\partial}

\def\hh{\mathcal{H}}
\def\pp{\mathcal{P}}
\def\uu{U}

\def\ckg{\check \gamma}
\def\modckg{{ \abs{\ckg}} }
\def\tt{\Tcal}
\def\nn{\Ncal}
\def\mrg{\mathring{\zeta}}
\def\mru{\mathring{u}}

\def\mrs{\mathring{\sigma}}
\def\gui{\p^\gamma U_i}

\def\comm#1#2{{\llbracket#1,#2\rrbracket}}

\def\tu{\tilde{u} }

\def\XX{{\scriptstyle \mathcal{X} }}
\def\pw{\Phi_{\scriptscriptstyle W}}
\def\pz{\Phi_{\scriptscriptstyle Z}}
\def\pa{\Phi_{\scriptscriptstyle U}}
\def\pu{\Phi_{\scriptscriptstyle U}}

\def\pzy{\pz^{\scriptscriptstyle y_0}}
\def\pay{\pa^{\scriptscriptstyle y_0}}

\def\tu{\tilde{u}}

\usepackage{fancyhdr}
\title{{\bf Shock formation and vorticity creation for 3d Euler}}

\author{
{Tristan Buckmaster}\thanks{\footnotesize Department of Mathematics, 
Princeton University, Princeton, NJ 08544,
 \href{buckmaster@math.princeton.edu}{buckmaster@math.princeton.edu} }
\and  
{Steve Shkoller}\thanks{\footnotesize Department of Mathematics, UC Davis, Davis, CA 95616, \href{shkoller@math.ucdavis.edu}{shkoller@math.ucdavis.edu}.}
\and 
{Vlad Vicol}\thanks{\footnotesize Courant Institute of Mathematical Sciences, New York University, New York, NY 10012, \href{vicol@cims.nyu.edu}{vicol@cims.nyu.edu}. }
}

\date{} %
\begin{document}

\maketitle

\begin{abstract}
\noindent
We analyze the shock formation process for the 3d non-isentropic Euler equations with the ideal gas law, in which sounds waves interact with entropy waves to produce vorticity.  Building on our theory for isentropic flows in \cite{BuShVi2019a,BuShVi2019b}, we give a constructive proof of shock formation from  smooth initial data.  Specifically, we prove that there exist smooth solutions to the non-isentropic Euler equations which form a  generic  stable shock with explicitly computable blowup time, location, and direction. This is achieved by establishing the asymptotic stability of a  generic shock profile in modulated self-similar variables,  controlling the interaction of wave families via: (i) pointwise bounds along Lagrangian trajectories, (ii) geometric vorticity structure, and (iii) high-order energy estimates in Sobolev spaces.
\end{abstract}

\renewcommand{\baselinestretch}{0.65}\normalsize
\setcounter{tocdepth}{1}

\tableofcontents
\renewcommand{\baselinestretch}{1.0}\normalsize

\section{Introduction}
The three-dimensional Euler equations of gas dynamics, introduced by Euler in~\cite{Euler1757}, are a hyperbolic system of five coupled equations, and can be written as
\begin{subequations}
\label{eq:Euler}
\begin{align}
\rho\left(\partial_\tcal u + (u \cdot \nabla_{\!\xcal} )u \right) + \nabla_{\!\xcal} p(\rho, \scal) &= 0 \,,  \label{eq:momentum} \\
\partial_\tcal \rho  +  (u \cdot \nabla_{\!\xcal}) \rho +  \rho\operatorname{div}_\xcal  u &=0 \,,  \label{eq:mass} \\
\partial_\tcal \scal  +  (u \cdot \nabla_{\!\xcal})\scal&=0 \,,  \label{eq:entropy}
\end{align}
\end{subequations}
for spacial variable $\xcal=(\xcal_1,\xcal_2,\xcal_3) \in \mathbb{R}^3  $, temporal  variable $\tcal \in \mathbb{R}  $, velocity
 $u :\mathbb{R}^3  \times \mathbb{R}  \to \mathbb{R}^3$, density
 $\rho: \mathbb{R}^3 \times \mathbb{R}  \to \mathbb{R}  _+$, and entropy
 $\scal: \mathbb{R}^3  \times \mathbb{R}$.   The pressure\footnote{The evolution equation for $\rho$ can be replaced with
 the equation for pressure given by $\partial_\tcal p  +  (u \cdot \nabla_{\!\xcal}) p + \upgamma p\operatorname{div}_\xcal  u=0$.}
  $p=p(\rho, \scal):\mathbb{R}^3 \times \mathbb{R}  \to \mathbb{R}  _+$  is a function of both density and entropy, with equation-of-state given by the ideal gas law
$$
 p(\rho,\scal) = \tfrac{1}{\upgamma} \rho^\upgamma e^\scal\,,
$$
where the adiabatic constant $\upgamma >1$.   If smooth initial conditions are prescribed at an initial time $\tcal_0$, then a classical solution to
\eqref{eq:Euler} exists up to a finite time $T_*$, the lifespan,  when a singularity or blowup develops \cite{Si1985}.   The mechanism of blowup 
for smooth solutions to \eqref{eq:Euler} as $t \to T_*$, including rate, direction, locus,  and profile is heretofore unknown.

Our primary aim  is  thus the detailed analysis of the formation  of the first shock or blowup  for smooth solutions to \eqref{eq:Euler}.   We shall prove that for an open
set of initial conditions, smooth solutions to \eqref{eq:Euler} evolve steepening wavefronts and form an asymptotically self-similar cusp-type first shock
with explicit rate, location, and direction.
The major difficulty in the analysis of the non-isentropic Euler dynamics stems from the interaction of sound waves, entropy waves, and vorticity waves. 
Non-isentropic flows
can have a misalignment of density and entropy gradients,  thus leading to dynamic vorticity creation, even from
irrotational initial data.  

To highlight the challenge created by the interaction of different wave families, we must examine the evolution of the vorticity vector  which we   shall now 
derive.  To do so, it is convenient to write the Euler equations using the sound speed.
We introduce the adiabatic exponent
$$
\alpha = \tfrac{ \upgamma -1}{2}  \,
$$
so that
the sound speed $c(\rho) = \sqrt{ \sfrac{\p p}{\p \rho} }$ can be written as
$c= e^{{\frac{\scal}{2}} } \rho^ \alpha$, and $p= \tfrac{1}{\gamma} \rho c^2$.
We define the scaled sound speed $\sigma$ by
\begin{equation}\label{sigma1}
\sigma =  \tfrac{1}{\alpha } c= \tfrac{1}{\alpha }  e^{{\frac{\scal}{2}} } \rho^ \alpha \,,
\end{equation} 
and write the  Euler equations \eqref{eq:Euler}  as a system for $(u, \sigma , \scal)$ as follows:
\begin{subequations}
\label{eq:Euler2}
\begin{align}
\p_\tcal u + (u \cdot \nabla_{\!\xcal}) u + \alpha \sigma  \nabla_{\!\xcal} \sigma  &=  \tfrac{\alpha }{2\gamma} \sigma^2  \nabla_{\!\xcal} \scal \,,  \label{eq:momentum2} \\
\partial_\tcal \sigma + (u \cdot \nabla_{\!\xcal}) \sigma  + \alpha \sigma \operatorname{div}_\xcal u&=0 \,,  \label{eq:mass2} \\
\partial_\tcal \scal  +  (u \cdot\nabla_{\!\xcal}) \scal &=0 \,.  \label{eq:entropy2}
\end{align}
\end{subequations}

We let $\omega= \operatorname{curl}_{\xcal} u$ denote the vorticity vector, and 
define the {\it specific vorticity vector}  by $ \zeta = \tfrac{\omega}{\rho} $.  A straightforward computation shows that $\zeta$ is a solution to
\begin{align} 
\p_\tcal \zeta  + (u \cdot \nabla_{\!\xcal }) \zeta - \left( \zeta  \cdot \nabla_{\!\xcal} \right) u  
=  \tfrac{\alpha }{\gamma} \tfrac{\sigma}{\rho} \nabla_{\!\xcal} \sigma \times  \nabla_{\!\xcal} \scal   \,.   \label{specific-vorticity}
\end{align} 
The term term $\tfrac{\alpha }{\gamma} \tfrac{\sigma}{\rho} \nabla_{\!\xcal} \sigma \times  \nabla_{\!\xcal} \scal  $ on the right side of \eqref{specific-vorticity}
can also be written as $\rho^{-3}  \nabla_{\!\xcal} \rho \times \nabla_{\!\xcal} p$ and is referred to as {\it baroclinic torque}.   Clearly, the potential vorticity, the
component of $\zeta$ in the direction of the density gradient,  can only be generated by vortex stretching, whereas baroclinic vorticity modes are produced from the interaction of acoustic waves and entropy waves.     This (baroclinic) vorticity production is the fundamental
mechanism for the excitation and stabilization of 
both the Rayleigh-Taylor and Richtmyer-Meshkov instabilities, and plays a fundamental role in atmospheric science as well as numerous flows of engineering significance.   

Of course, it  is possible to  simplify the Euler dynamics in a manner that still retains the steepening of sound waves, but removes complications associated
to the interaction of different wave families.  This can be achieved by  considering  the subclass of flows for  which the entropy is a constant;  such flows 
are called {\it isentropic}, and the pressure is
a function of density alone:  $p=\tfrac{1}{\upgamma} \rho^\upgamma$.     Note,  that for isentropic flow, baroclinic torque vanishes, and thus the
specific vorticity $\zeta$ is Lie advected as a vector field. Acoustic modes can no longer interact with entropy waves to create vorticity; rather, vorticity is merely advected.  As such, two further subclasses of flows exist: irrotational flow and flow with advected vorticity.  For irrotational flow, only sound waves
propagate, while for initial data with vorticity, there is an interaction between acoustic modes and vorticity modes that must be carefully analyzed, as 
controlling the growth of vorticity is essential to the study of shock formation.  For non-isentropic dynamics, the presence of baroclinic torque creates a
fundamentally new challenge in the estimation of the growth of vorticity.   Why? Because as the first shock forms, the magnitude of baroclinic torque becomes infinite!
Due to some  cancellations using geometric coordinates adapted to the steepening wave front, even though  baroclinic torque blows up, we 
shall prove that  the vorticity remains  bounded up to the time of shock formation;  furthermore,  {\it irrotational} 
initial data can be chosen with non-zero baroclinic torque such that vorticity is instantaneously produced and remains non-trivial throughout the shock formation process.

By a significant extension of the methodology we developed in  \cite{BuShVi2019a,BuShVi2019b}, we shall prove the following:
\begin{theorem}[Rough statement of  the main theorem]
For an open set of smooth initial data  with a maximally negative gradient of size $\OO({\sfrac{1}{\eps}} )$, for $\eps>0$ sufficiently small, there exist smooth solutions to the non-isentropic 3d Euler equations \eqref{eq:Euler} which form a shock singularity at time  $T_*=\OO(\eps)$. 
The first singularity occurs at a single point in space, whose location can be explicitly computed, along with the precise time at which it occurs.
The blowup profile is shown to be a cusp with $C^ {\sfrac{1}{3}} $ regularity, and the singularity is given by an asymptotically self-similar shock profile which is stable with
respect to the $H^m(\RR^3)$ topology for $m\ge 18$.   If an irrotational initial velocity is prescribed, vorticity is instantaneously produced, and remains 
bounded and non-trivial up to the blowup time $T_*$.  
\end{theorem} 
A precise statement of the main result will be given below as Theorem \ref{thm:main:S-S}.

\subsection{Prior results}

In one space dimension,  the theory of finite-time blowup of smooth solutions and shock formation to the Euler equations is well established.   The literature
is too vast to provide a review here. See, 
for example,  \cite{Ri1860}, \cite{Lax1964}, \cite{John1974}, \cite{Li1979}, \cite{Ma1984}, \cite{Da2010}.   In contrast, in multiple space dimensions,
only the
{\it isentropic} shock formation problem has been studied:  shock formation was established for  irrotational flows by \cite{Ch2007} and  \cite{ChMi2014}
(see also  \cite{Ch2019}), 
 for 2d  isentropic flows with vorticity by
\cite{LuSp2018} and \cite{BuShVi2019a}, and for 3d isentropic flows with vorticity by \cite{BuShVi2019b}.   
For {\it non-isentropic} flow in multiple space dimensions, prior to this paper, it 
was  {\em only}  known  that  $C^1$ solutions have a finite lifespan \cite{Si1985}. 
The main result of this paper is  the detailed analysis of the shock formation process.  

As we noted above, one of the major difficulties in the analysis of non-isentropic flows is due to the interaction of multiple wave families: 
sound waves, vorticity waves, and entropy waves.   Indeed, the analysis of quasilinear hyperbolic systems with multiple wave speeds is just emerging.
As stated in \cite{Sp2018}, prior to the results
in \cite{LuSp2018, Sp2018, BuShVi2019a,BuShVi2019b},  there have been no constructive proofs of shock formation for a quasilinear hyperbolic system in more than one spatial dimension, featuring multiple wave speeds. We note that the irrotational (isentropic) Euler equations  can be written as a scalar 
quasilinear wave equation with only one wave speed; formation of shocks for systems with a single wave speed have been studied by 
\cite{Al1999a,Al1999b,Ch2007,ChMi2014,Sp2016,MiYu2017,Mi2018}.

Finally, we mention that there are other possible blowup mechanisms for the Euler equations; for example,
a precise characterization of implosion for spherically symmetric isentropic flow has  recently been 
given  in \cite{MeRaRoSz2019a,MeRaRoSz2019b}.

 \subsection{Main ideas in the proof}

Because of the presence of multiple wave speeds, multiple wave families,  and their nonlinear interactions, the Euler dynamics offer a rich tapestry of dynamic
behavior, and yet when zooming-in on the formation of the first shock, the Euler solution  shares  fundamental features with the wave-steepening blowup of the 
3d Burgers solution.    For this reason,
our study of the mechanism of shock formation for smooth solutions of \eqref{eq:Euler2} as $t \to T_*$ makes use of a blowup profile $\bar W(y)$, 
one example of a stable stationary solution to the 3d self-similar Burgers equation
\begin{align}
 -\tfrac 12 \bar W + \left( \tfrac{3}{2} y_1+ \bar W \right) \partial_{y_1} \bar W+ \tfrac{1}{2} y_2  \partial_{y_2} \bar W   + \tfrac{1}{2} y_3 \partial_{y_3} \bar W     = 0 
\label{eq:Burgers:self:similar}
\end{align}
which has an explicit representation.
If we consider the 3d Burgers equation $\p_t v + v \cdot \nabla v=0$ in physical spacetime variables $(\xcal,t)$, then a smooth solution $v=(v_1,v_2,v_3)$ which forms a first shock at $t=T_*$ is given by\footnote{In fact,  as  established in Appendix \ref{sec:delay}, there are many closely related  stable self-similar solutions to the Burgers equations which allow for a slight modification of 
$v_1$.} 
\begin{align} 
v_1(\xcal_1,  \xcal_2,\xcal_3 ,t) =  (T_*- t)^ {\frac{1}{2}} \bar W\left( \tfrac{\xcal_1}{(T_*-t)^ {\frac{3}{2}} }, \tfrac{ \xcal_2}{(T_*-t)^ {\frac{1}{2}} } , 
\tfrac{ \xcal_3}{(T_*-t)^ {\frac{1}{2}} }\right)  \label{v1}
\end{align} 
with $v_2=0$ and $v_3=0$.  Explicit properties of the blowup profile $\bar W(y)$ together with the solution for 
$v_1(\xcal, \tcal)$ give precise information of the blowup mechanism as $t \to T_*$, including the blowup time $T_*$, the blowup location $\xcal=0$, and the
blowup direction $e_1$.     We note that we have made a particular choice of direction for our Burgers solution $v$; specifically, we have chosen to let the wave steepen
along the $e_1$ blowup direction, whereas we could have used the profile $\bar W$ to form a blowup in any direction.

Although the non-isentropic Euler system is significantly more complicated, we are nevertheless able to use the Burgers stationary solution $\bar W$ to
describe the blowup mechanism for smooth solutions of \eqref{eq:Euler2} as $t \to T_*$.   This requires a number coordinate and variable transformations
that are constructed upon two geometric principles: first, we build into our transformations a family of time-dependent modulation functions whose purpose is
to fight against the destabilizing action of the finite-dimensional symmetry groups of the Euler equations, and second, we design a coordinate system which both
follows and deforms with the steepening Euler solution. 

Let us now elaborate on these ideas.   The blowup profile $\bar W(y)$ has an explicit formula which shows that $y=0$ is a global minimum for 
$\p_{y_1}W(y)$, and with the following properties verified: $\bar W(0)=0$, $\p_{y_1} \bar W(0)
=-1$, $\p_{y_2}\bar W(0)=\p_{y_3}\bar W(0)=0$, $ \nabla_y ^2 \bar W(0)=0$, and
\begin{align} 
\nabla ^2 \p_{y_1}\bar W(0) >0 \,. \label{genericity}
\end{align} 
Positive-definiteness of the Hessian of  $\p_1 W$ at $y=0$ is a {\it genericity condition} for the blowup mechanism, and has been used
in the study of  blowup for quasilinear wave equations \cite{Al1999a} and discussed in  \cite{CaErHoLa1993,Ch2007} as an important selection criterion for stable shocks.

Returning now to the identity \eqref{v1}, if the initial time is fixed to be $t_0 = -\eps$ for $\eps>0$, we can set  $T_* =0$; the initial condition for $v_1$ is then given by
$v_1(\xcal, -\eps) =\eps^{ {\frac{1}{2}} } \bar W( \eps^ {-\frac{3}{2}} \xcal_1,  \eps^ {-\frac{1}{2}} \xcal_2,  \eps^ {-\frac{1}{2}} \xcal_3)$, hence with
$(y_1, y_2, y_3) = ((-t)^{-\frac{3}{2}} \xcal_1, (-t)^{-\frac{1}{2}} \xcal_2,  (-t)^{-\frac{1}{2}} \xcal_3)$, we see that the
properties of $\bar W(y)$ at $y= 0$ show that $v_1(0, -\eps) =0$, $\p_{\xcal_1} v_1(0, -\eps) = \tfrac{1}{\eps}\p_{y_1} \bar W(0)= -\tfrac{1}{\eps}$, 
$\p_{\xcal_2} v_1(0, -\eps) =0$, $\p_{\xcal_3} v_1(0, -\eps) =0$, $ \nabla _{\!\xcal}^2 v_1(0,-\eps) =0$ and the genericity condition \eqref{genericity} is also
satisfied so that $\nabla_{\!\xcal} ^2 \p_{\xcal_1}v_1(0,-\eps) >0$.   We see that for the 3d Burgers equation, if we start with a maximally negative slope equal
to $ -{\tfrac{1}{\eps}} $ at time $t=-\eps$ and $\xcal=0$, then the first shock occurs at time $T_*=0$ and $\xcal=0$, and by virtue of \eqref{v1}, the blowup
mechanism is self-similar 
\begin{align} 
\p_{\xcal_1} v_1(0,t) = \tfrac{1}{T_*-t} \p_1 \bar W(0) = - \tfrac{1}{T_*-t} \,.
\label{dv1}
\end{align} 
Of course, no such formula as \eqref{v1} exists for the Euler equations, but we can nevertheless use the properties
of $\bar W$ to develop a new type of stability theory for the Euler equations in self-similar variables.

Thus, the first step in our proof of shock formation for the non-isentropic Euler equations is the mapping of the physical  spacetime coordinates 
$(\xcal,\tcal)$ to {\it modulated} self-similar spacetime coordinates  $(y,s)$, together with a succession of transformations that map
the original variables $(u,\sigma,\scal)$ into  geometric Riemann-like variables $(W,Z,A,K)$, in which the  dynamically dominant variable $W(y,s)$ mimics 
the properties of $\bar W(y)$ near the blowup location $y=0$.  The use of modulation functions for the analysis of  self-similar dispersive equations
was pioneered in \cite{Merle96,MeZa97} and used more recently in \cite{MeRa05,MeRaSz2018}.
 The initial data is prescribed at self-similar time $s_0 = -\log\eps$, and we require
$\p^\gamma W(y,-\log\eps)$ to verify the same conditions as $\p^\gamma\bar W(y)$ at the point $y=0$ for all multi-indices $\abs{\gamma} \le 2$.  
 Just as we noted above, we are now making a choice of blowup direction; the initial data is chosen so that its maximal negative slope is in the $e_1$ direction,
 but unlike the Burgers solution, the rotational symmetry of the Euler dynamics does not preserve this direction.  In fact, the various symmetries
of the Euler equations prevent these conditions on $\p^\gamma W(0,s)$
to be maintained under the natural evolution, and for this reason, ten time-dependent modulation functions are used to ensure that 
$\p^\gamma W(0,s)=\p^\gamma \bar W(0)$ for  $\abs{\gamma} \le 2$ and
for all $s\ge -\log\eps$.     Of these ten modulation functions, seven of them are associated to symmetries of the Euler equations (see Section 1.3 in 
\cite{BuShVi2019b}), and three of the modulation functions are associated to a spatially quadratic time-dependent parameterization 
$f(t, \xcal_2,\xcal_3) = \phi_{22}(t) \xcal_2^2+ 2\phi_{23}(t) \xcal_2\xcal_3 + \phi_{33}(t) \xcal_3^2$ of the steepening front, where the matrix $\phi_{\mu\nu}(t)$ modulates the
curvature, and denotes the induced second-fundamental form.   Associated to this  parametric surface $f( \xcal_2,\xcal_3,t) $ is a time-dependent
orthonormal basis $(\Ncal, \Tcal^2,\Tcal^3)$ representing the normal and tangential directions.   The steepening front moves in the $\Ncal$ direction
and the dominant Riemann variable is defined as $w = u \cdot \Ncal + \sigma$.   With respect to coordinates $x$ which themselves depend on $f$, the
variable $w(x,t)$ is associated to the dominant self-similar variable $W(y,s)$ by a formula which is analogous to \eqref{v1}:
\begin{align*}
w(x_1,x_2,x_3,t) =  (\tau(t)- t)^ {\frac{1}{2}}  W\left( \tfrac{x_1}{(\tau(t)-t)^ {\frac{3}{2}} }, \tfrac{x_2}{(\tau(t)-t)^ {\frac{1}{2}} } , 
\tfrac{x_3}{(\tau(t)-t)^ {\frac{1}{2}} }, s \right), \  -s=\log(\tau(t)-t)\,,
\end{align*} 
where $\tau(t)$ modulates the blowup time and converges to $T_*$ as $t\to T_*$.   Differentiating $w$ in the direction $\Ncal$ of the steepening front,
it can be shown that
\begin{align}
\p_\Ncal w(\xi(t), t) = e^s \p_{y_1}W(0,s)  = -\tfrac{1}{\tau(t) -t}  \to - \infty  \qquad \text{ as } \qquad t\to T_*  \,,  \label{cool-cool}
\end{align}
where $\xi(t)$ modulates the blowup location.   The blowup \eqref{cool-cool} is the geometric analogue of \eqref{dv1}, and requires a well-defined
limit as $t \to T_*$ which, in turn, requires that $W(y,s)$ remains  well defined for all $-\log\eps \le s$.

It therefore  becomes clear that in order to establish  stable self-similar shock formation, we must prove global existence  of solutions to the  
Euler equations in self-similar coordinates $(y,s)$, and the majority of our work is devoted to this end.   The understanding of the
damping/anti-damping structure of the  Euler equations in 
self-similar coordinates $(y,s)$ along Lagrangian trajectories is key to our analysis;   the undifferentiated Euler equations have anti-damping terms, but upon 
spatial differentiation, damping emerges,
and the more derivatives that are applied,  the stronger the damping becomes.   A consequence of this observation is that pointwise bounds for lower-order 
derivatives cannot rely on either damping or  traditional Eulerian-type analysis, but rather on  sharp (lower) bounds on the motion of the three families of
trajectories associated to the three waves speeds present.   In self-similar coordinates, almost all of the trajectories in these
three wave families {\em escape to infinity}  and having sharp rates-of-escape for each family can be combined with spatial decay properties of
the Riemann-type function $W(y,s)$ to close a system of highly coupled bootstrap bounds for derivatives up to order two.

On the other hand, it is not possible to close estimates for the Euler equations using only pointwise bounds due to inherent derivative loss, and higher-order energy estimates must therefore be employed. Modified energy estimates are performed for a system of variables comprised of $U$, $ \sound  e^{-\frac{ K }{2 \upgamma}}$, and $e^\frac{ K }{2 \upgamma}$, where $U$, $S$, and $K$ are the  self-similar versions of $u$, $\sigma$, and $\scal$, respectively. The use of these variables removes the hyperbolic degeneracy associated to vanishing
density. Combined with  the  weighted pointwise bounds for lower-order derivatives, we prove global existence in a modified $\dot H^m$-norm, $m \ge 18$. 

While for the subclass of irrotational flows  the above two types of estimates suffice, 
for rotational flows it is essential to obtain uniform bounds for the vorticity all the way to the blowup time.   Even for isentropic dynamics, in which the specific vorticity is  Lie advected, 
analysis in self-similar coordinates appears top create logarithmic losses in temporal decay (see \cite{BuShVi2019b}).  Instead, the  specific vorticity $\zeta$ is estimated in physical coordinates using geometric components $(\zeta \cdot \Ncal, \zeta \cdot \Tcal^2, \zeta \cdot \Tcal^3)$, which yield a  cancellation at highest order. For the non-isentropic dynamics, an additional difficulty arises because the vorticity equation \eqref{specific-vorticity} is forced by the baroclinic torque $ \tfrac{\alpha }{\gamma} \tfrac{\sigma}{\rho} \nabla_{\!\xcal} \sigma \times  \nabla_{\!\xcal} \scal $, which blows up as $t \to T_*$. Indeed, from formula \eqref{11bottlesdownto7} below, and the bounds established in Sections~\ref{sec:vorticity:sound} and~\ref{sec:creation}, we may show that the tangential components of the baroclinic torque term satisfy 
$$
 \sabs{  ( \tfrac{\sigma}{\rho}  \Tcal^\cdot  \cdot  \nabla_{\!\xcal} \sigma \times  \nabla_{\!\xcal} \scal )(\xi(t),t)} 
 \gtrsim \tfrac{1}{T_* -t} \,.
$$
A main feature of our proof is to show that in spite of the fact that the Lie-advection for the specific vorticity is forced by a diverging term, $\zeta$ remains uniformly bounded up to $T_*$. This is achieved by noting that the divergence of the velocity gains a space derivative when integrated along trajectories with speed $u$, and by taking advantage of certain cancellations which arise due to our geometric framework.

Finally, we examine baroclinic vorticity production.   We prove that even if the initial velocity is irrotational, vorticity is instantaneously produced due to the
baroclinic torque, and our analysis shows that this created vorticity remains non-trivial  in an open neighborhood of the steepening front all the way up to the
first shock.  We thus provide a constructive proof of shock formation for Euler in the regime in which vorticity is  created, and not simply Lie advected.

\subsection{Outline} In Section \ref{sec:variables}, we introduce a succession of variable changes and Riemann-type variables which allow then allow us
to write the Euler equations in modulated self-similar coordinates.   A precise specification of the data and the statement of the main results is then given
in Section \ref{sec:results}.   In Section \ref{sec:bootstrap}, we introduce the bootstrap assumptions for the modulation functions as well as the primary variables solving the self-similar Euler equations; these bootstrap assumptions consist of carefully chosen weighted (in both space and time) bounds.   
A fundamental aspect of our proof requires a detailed estimates for the rates of escape of the trajectories corresponding to the different wave speeds, and 
Section \ref{sec:Lagrangian} is devoted to this analysis.    In Section
\ref{sec:vorticity:sound}, we establish pointwise bounds for the vorticity, and in Section \ref{sec:creation} we show that there exists irrotational initial velocity
fields from which  vorticity is  created and  remains non-trivial at the first shock.
Energy estimates in self-similar variables are established in Section \ref{sec:energy}, using the modified variables \eqref{eq:thunder:variables}.   In
Section \ref{sec:preliminary}, we establish weighted (pointwise) estimates for functions appearing in the forcing, damping, and transport of the differentiated
Euler system.  In turn, these weighted bounds allow us to close the bootstrap assumptions for $W$, $Z$, $A$, $K$ and their partial derivatives, and this
is achieved in Sections \ref{sec:Z:A}--\ref{sec:W}, while in Section \ref{sec:constraints}, we close the bootstrap bounds for the dynamic modulation functions.
Finally, in Section \ref{sec:conclusion-of-proof}, we explain how all of the obtained bounds are used to prove Theorem \ref{thm:main:S-S}; in particular,
we show that  
$\lim_{s\to \infty } W(y,s)  = \bar W_ \mathcal{A} (y) $ for any fixed $y \in \RR^3$, where $\bar W_ \mathcal{A} (y)$ is a stable stationary solution to the
self-similar 3d Burgers equations. A  family of such stationary solutions is constructed in  Appendix \ref{sec:toolshed}, which also contains an interpolation
inequality that is used throughout the paper, as well as some detailed computations leading to the evolution equations for the modulation functions.

\section{Transforming the Euler equations into geometric self-similar variables}
\label{sec:variables}

We now make a succession of variable transformations for both dependent and independent variables. 
We begin by rescaling time as
\begin{align}
\tcal \mapsto  \tfrac{1+ \alpha }{2} \tcal = t\,.
\label{eq:time:rescale}
\end{align}
We next introduce ten modulation variables 
which satisfy a coupled
system of ODEs that will be given in \eqref{eq:dot:phi:dot:n}--\eqref{eq:dot:kappa:dot:tau}.   For each time $t$, they are defined as follows:
\begin{subequations} 
\label{mod-def}
\begin{alignat}{2}
&R(t) \in \mathbb{S}^2 \qquad  &&\text{ rotation matrix from } e_1 \text{ to  the direction of steepening front } n(t) \,, \\
&\xi(t) \in \mathbb{R}^3 && \text{ translation vector used to fix the location of the developing shock }  \,,  \\
&\phi(t) \in  \mathbb{R}^3 && \text{ 2x2 symmetric matrix giving the curvature of the  developing shock front }  \,,  \\
& \tau(t) \in \mathbb{R}  && \text{ scalar used to track exact the blowup time } \,, \\
&\kappa(t) \in \mathbb{R}  && \text{ scalar used to fix the speed of the developing shock} \,.
\end{alignat} 
\end{subequations} 
The matrix $R(t)$ is defined in terms of   two time-dependent rotation angles $n_2(t)$ and $n_3(t)$ as follows.
We define $n(t) = (\sqrt{1- n_2^2(t)+n_3^2(t)}, n_2(t),n_3(t)) $
and a skew-symmetric matrix $\tilde R$ whose first row is the vector $(0,-n_2,-n_3)$, first column is $(0,n_2,n_3)$, and has $0$ entries otherwise. In terms of $\tilde R$, we define the rotation matrix
\begin{align}
  R(t)  
= \Id + \tilde R(t) + \frac{1 - e_1 \cdot n(t) }{\abs{e_1 \times n(t)}^2} \tilde R^2(t)   \,.
\label{eq:R:def}
\end{align}
It is the two angles $n_2(t)$ and $n_3(t)$ whose evolution is given in \eqref{eq:dot:phi:dot:n}.

Using these modulation functions, 
we next proceed to make a succession of transformations
of both the independent and dependent variables, finally arriving at a novel modulated self-similar form of the dynamics.

\subsection{Rotating the direction and translating the location of the steepening wavefront} 
\label{sec:time:dependent:coord}

We introduce the new independent  variable
\begin{align}
 \tilde x = R^T(t) (\xcal-\xi(t))
 \label{eq:tilde:x:def}
\end{align}
and  corresponding dependent variables as
\begin{align}
\tu ( \tilde x ,t) =   R^T(t)   u(\xcal ,t) \, ,
\qquad 
\tilde \sigma ( \tilde x ,t) =    \sigma(\xcal,t)\,, \qquad 
\tilde \scal ( \tilde x ,t) =    \scal(\xcal,t)  \, .
 \label{eq:tilde:u:def}
\end{align}
It follows that \eqref{eq:Euler2} is transformed to
\begin{subequations}
\label{eq:new:Euler}
\begin{align}
 \tfrac{1+ \alpha }{2} \p_t \tilde u - \dot Q\tilde u + \Big( ( \tilde v  + \tilde u)\cdot \nabla_{\! \tilde x} \Big) \tilde u 
 + \alpha \tilde \sigma   \nabla_{\! \tilde x}  \tilde \sigma  &= \tfrac{\alpha }{2\gamma} \tilde \sigma^2  \nabla_{\!\tilde x} \tilde\scal \\
  \tfrac{1+ \alpha }{2}\partial_t \tilde \sigma + \Big(( \tilde v  + \tilde u) \cdot \nabla_{\! \tilde x}\Big)\tilde \sigma +  \alpha   \tilde \sigma  \div_{\! \tilde x} \tilde u&=0\\
 \tfrac{1+ \alpha }{2}\partial_t \tilde \scal + \Big(( \tilde v  + \tilde u) \cdot \nabla_{\! \tilde x}\Big)\tilde \scal &=0 
 \label{eq:entropy3}
\end{align}
\end{subequations}
where 
\begin{align} 
\dot Q = \dot R^T R\ \ \text{ and } \ \
\tilde v(\tilde x,t) := \dot Q  \tilde x  - R^T \dot \xi \,. \label{tildev}
\end{align}

The  density and pressure in this rotated and translated frame are given by
\begin{align}
\tilde \rho ( \tilde x ,t) =    \rho(\xcal,t)\,, \qquad 
\tilde p ( \tilde x ,t) =    p(\xcal,t)
 \label{eq:tilde:p:def}
\end{align}
satisfy
\begin{subequations} 
\begin{align} 
 \tfrac{1+ \alpha }{2}\partial_t \tilde \rho + \Big(( \tilde v  + \tilde u) \cdot \nabla_{\! \tilde x}\Big)\tilde \rho +     \tilde \rho  \div_{\! \tilde x} \tilde u&=0\,,
 \label{tdensity} \\
 \tfrac{1+ \alpha }{2}\partial_t \tilde p + \Big(( \tilde v  + \tilde u) \cdot \nabla_{\! \tilde x}\Big)\tilde p +  \upgamma   \tilde p  \div_{\! \tilde x} \tilde u&=0\,,
 \label{tpressure}
\end{align} 
\end{subequations} 
and we also have the alternative form of the  momentum equation
\begin{align} 
 \tfrac{1+ \alpha }{2} \p_t \tilde u - \dot Q\tilde u + \Big( ( \tilde v  + \tilde u)\cdot \nabla_{\! \tilde x} \Big) \tilde u 
    + ( \alpha \tilde \sigma)^ {-\frac{1}{ \alpha }} e^ {\frac{\tilde\scal}{2\alpha }}\nabla_{\!\tilde x} \tilde p =0 \,.
\label{eq:new:momentum:alt}
\end{align} 
This follows from the form of the momentum equation given by
$\p_\tcal u + (u \cdot \nabla_{\!\xcal}) u +  ( \alpha \sigma)^ {-\frac{1}{ \alpha }} e^ {\frac{\scal}{2\alpha }}  \nabla_{\!\xcal}p  =  0$
where,  from \eqref{sigma1}, we have used that $\rho ^{-1} =
( \alpha \sigma)^ {-\frac{1}{ \alpha }} e^ {\frac{\scal}{2\alpha }} $.

Similarly, defining the transformed {\em specific vorticity} vector  $\tilde \zeta$ by
\begin{equation}\label{vort00}
\tilde \zeta ( \tilde x ,t) =   R^T(t)   \zeta(\xcal ,t) 
 \,,
\end{equation} 
we have that $\tilde \zeta$ solves 
\begin{align}
 \tfrac{1+ \alpha }{2} \p_t \tilde \zeta - \dot Q \tilde \zeta + \Big( ( \tilde v + \tilde u)\cdot \nabla_{\! \tilde x}\Big)\tilde \zeta 
 - \Big(\tilde \zeta \cdot \nabla_{\! \tilde x}\Big) \tilde u &=  \tfrac{\alpha }{\gamma}\tfrac{ \tilde \sigma}{ \tilde \rho} \nabla_{\!\tilde x} \tilde \sigma \times  \nabla_{\! \tilde x}\tilde \scal   \,. \label{tvorticity}
\end{align}
Deriving \eqref{tvorticity} from \eqref{specific-vorticity} fundamentally uses that  $\dot Q$ is skew-symmetric, and the fact that the cross product  is invariant to rotation.

\subsection{Coordinates adapted to the shape of the steepening wavefront}
\label{sec:Riemann:sheep}

We next define a quadratic surface over the $\tilde x_2$-$\tilde x_3$ plane given by the graph
\begin{equation}\label{surface}
\left( f(\tilde x_2, \tilde x_3, t), \tilde x_2, \tilde x_3\right)  \,,
\end{equation} 
which approximates the steepening shock, and where
\begin{align}
 f(\check{\tilde x},t) = \tfrac{1}{2} \phi_{\nu\gamma}(t) \tilde x_\nu\tilde x_\gamma  \,. \label{def_f}
\end{align}

Associated to the parameterized surface \eqref{surface}, we define the unit-length normal and tangent  vectors\footnote{ As we noted in
\cite{BuShVi2019b}, 
 $(\nn, \tt^2,\tt^3)$ defines an orthonormal basis and  $ \Tcal^2 \times  \Tcal^3=\Ncal$,  $\Ncal \times \Tcal^2= \Tcal^3$ 
and  $\Ncal \times \Tcal^3= -\Tcal^2$.}
\begin{align} 
\Ncal = \Jcal^{-1} (1, -f,_2, -f,_3)\,, \ 
\Tcal^2 =  \left( \tfrac{f,_2}{\Jcal}, 1 - \tfrac{(f,_2)^2}{\Jcal(\Jcal+1)} \,,  \tfrac{- f,_2f,_3}{\Jcal(\Jcal+1)}\right) \, , \
\Tcal^3 =   \left( \tfrac{f,_3}{\Jcal},  \tfrac{- f,_2f,_3}{\Jcal(\Jcal+1)} \,,  1 - \tfrac{(f,_3)^2}{\Jcal(\Jcal+1)} \right) \,,  
\label{tangent}
\end{align} 
where
$\Jcal = ( 1+  |f,_2|^2+ |f,_3|^2)^ {\frac{1}{2}}$.\footnote{Here and throughout the paper we are using the notation $\varphi_{,\mu}=\p_{x_\mu}\varphi$, and $\p_\mu \varphi = \p_{y_\mu} \varphi$.}

In order to `flatten'   the developing shock front, we make one further transformation of the  independent  space variables\footnote{ Note that  only
 the $\tilde x_1$ coordinate is modified.} 
\begin{equation}\label{x-sheep}
x_1 = \tilde x_1 - f(\tilde x_2,\tilde x_3,t)\,, \qquad x_2 = \tilde x_2\,, \qquad x_3 = \tilde x_3 \,,
\end{equation}
and define the transformed dependent variables by
\begin{subequations} 
\label{usigma-sheep}
\begin{align} 
\mru(x ,t) & =\tilde u (\tilde x, t) = \tilde u (x_1+ f(x_2,x_3,t),x_2,x_3,t)\,, \\
\mrs(x ,t) & =\tilde \sigma (\tilde x, t)  = \tilde \sigma  (x_1+ f(x_2,x_3,t),x_2,x_3,t)\,,  \label{sigma-sheep} \\ 
\mathring \rho(x ,t) & =\tilde \rho (\tilde x, t)  = \tilde \rho  (x_1+ f(x_2,x_3,t),x_2,x_3,t) \,, \\
\mathring \scal(x ,t) & =\tilde \scal (\tilde x, t)  = \tilde \scal  (x_1+ f(x_2,x_3,t),x_2,x_3,t) \,, \\
\mathring p(x ,t) & =\tilde p (\tilde x, t)  = \tilde p  (x_1+ f(x_2,x_3,t),x_2,x_3,t) \,.
\end{align} 
\end{subequations} 
We shall also make use of the $ \alpha $-dependent parameters
\begin{align} 
\beta_1 = \beta_1(\alpha)= \tfrac{1}{1+\alpha}, \qquad \beta_2 = \beta_2(\alpha)=\tfrac{1-\alpha}{1+\alpha}, \qquad \beta_3 = \beta_3(\alpha)= \tfrac{\alpha}{1+\alpha}, \qquad \beta_4 = \beta_4(\alpha)= \tfrac{\beta_3(\alpha)}{1+2\alpha}\,,
\label{eq:various:beta}
\end{align} 
where  
$0 \le \beta_i = \beta_i(\alpha) < 1$.

Using the time rescaling from \eqref{eq:time:rescale}, the system \eqref{eq:new:Euler} can be written as
\eqref{usigma-sheep} as
\begin{subequations}
\label{Euler_sheep}
\begin{align}
& \p_t \mru  - 2\beta_1\dot Q \mru + 2\beta_1  ( -\tfrac{\dot f }{2 \beta_1 } +\Jcal  v \cdot \nn + \Jcal \mru \cdot \nn  ) \p_1\mru  
 +2\beta_1(v_\nu + \mru_\nu)\p_\nu \mru    + 2\beta_3 \mrs( \Jcal \Ncal   \p_1 \mrs +   \delta^{\cdot \nu}    \p_\nu \mrs ) \notag \\
& \qquad \qquad =  \beta_4 \mathring \sigma^2 (\Jcal \Ncal \p_1 \mathring \scal + \delta ^{ \cdot \nu} \p_\nu \mathring \scal) \,,  \\
&
 \partial_t \mrs + 2 \beta_1 ( -\tfrac{\dot f }{2 \beta_1 }+ \Jcal v \cdot \nn + \Jcal \mru \cdot \nn  )\p_1\mrs   + 2\beta_1(v_\nu + \mru_\nu)\p_\nu \mrs
 + 2\beta_3 \mrs\left( \p_1 \mru \cdot \nn \Jcal + \p_\nu \mru_\nu \right) =0 \,,\\
 &
 \partial_t \mathring\scal + 2 \beta_1 (  -\tfrac{\dot f }{2 \beta_1 }+ \Jcal v \cdot \nn + \Jcal \mru \cdot \nn  )\p_1\mathring \scal  + 2\beta_1(v_\nu + \mru_\nu)\p_\nu \mathring \scal =0 \,,
\end{align}
\end{subequations}
where in analogy to \eqref{usigma-sheep}, we have denoted
\begin{align}
v(x,t) =  \tilde v (\tilde x, t ) = \tilde v (x_1+ f(x_2,x_3,t),x_2,x_3,t)\,.
\label{eq:v:x:t:def}
\end{align}
Note in particular the identity $v_i(x,t) = \dot Q_{i1}(x_1 + f(\check x,t)) + \dot Q_{i\nu} x_\nu - R_{ji} \dot \xi_j$. 
 The density equation \eqref{tdensity}  becomes
\begin{align} 
 \partial_t \mathring \rho + 2 \beta_1 ( -\tfrac{\dot f }{2 \beta_1 } + \Jcal v \cdot \nn + \Jcal \mru \cdot \nn  )\p_1\mathring \rho   + 2\beta_1(v_\nu + \mru_\nu)\p_\nu\mathring \rho
 + 2\beta_1 \mathring\rho\left( \p_1 \mru \cdot \nn \Jcal + \p_\nu \mru_\nu \right) =0 \,,\label{density-sheep}
\end{align} 
 the pressure equation \eqref{tpressure} is transformed to
\begin{align} 
 \partial_t \mathring p + 2 \beta_1 ( -\tfrac{\dot f }{2 \beta_1 } + \Jcal v \cdot \nn + \Jcal \mru \cdot \nn  )\p_1\mathring p   + 2\beta_1(v_\nu + \mru_\nu)\p_\nu \mathring p
 + 2\beta_1\upgamma \mathring p \left( \p_1 \mru \cdot \nn \Jcal + \p_\nu \mru_\nu \right) =0 \,, \label{pressure-sheep}
\end{align} 
and the alternative form of the momentum equation \eqref{eq:new:momentum:alt} is written as
\begin{align} 
& \p_t \mru  - 2\beta_1\dot Q \mru + 2\beta_1  ( -\tfrac{\dot f }{2 \beta_1 }+\Jcal  v \cdot \nn + \Jcal \mru \cdot \nn  ) \p_1\mru  
 +2\beta_1(v_\nu + \mru_\nu)\p_\nu \mru 
  \notag   \\
  & \qquad\qquad\qquad\qquad\qquad
  +2 \beta_1 ( \alpha \mathring \sigma)^ {-\frac{1}{ \alpha }} e^ {\frac{\mathring\scal}{2\alpha }}( \Jcal \Ncal   \p_1 \mathring p +   \delta^{\cdot \nu}    \p_\nu \mathring p ) =0 \,.
 \label{momentum-sheep-alt}
\end{align} 
Similarly, the transformed specific vorticity vector is
\begin{equation}\label{sv-sheep}
\mrg(x,t) = \tilde \zeta(\tilde x,t) = \tilde \zeta (x_1+ f(x_2,x_3,t),x_2,x_3,t)\,, 
\end{equation} 
so that the equation \eqref{tvorticity} becomes
\begin{align} 
&\p_t \mrg - 2\beta_1\dot Q \mrg 
+ 2 \beta_1 ( -\tfrac{\dot f }{2 \beta_1 } + \Jcal  v \cdot \Ncal + \Jcal \mru \cdot \Ncal) \p_1 \mrg + 2 \beta_1 ( v_\nu + \mru_\nu) \p_\nu \mrg  -2\beta_1 \Jcal \Ncal \cdot \mrg \p_1 \mru  - 2 \beta_1 \mrg_\nu \p_\nu \mru  \notag \\
 & \qquad \qquad =  \tfrac{\alpha }{\gamma}\tfrac{ \mathring \sigma}{ \mathring \rho} \nabla_{\!\tilde x} \mathring \sigma \times  \nabla_{\! \tilde x}\mathring \scal \,.
\label{svorticity_sheep}
\end{align} 
Note that the gradient appearing on the right side is with respect to $\tilde x$.
We record for later use that
\begin{align} 
 \nabla_{\! \tilde x} \mathring \sigma\! \times\!  \nabla_{\!  \tilde x}\mathring \scal 
 = \left(\p_{\Tcal^2} \mathring \sigma \p_{\Tcal^3} \mathring \scal - \p_{\Tcal^3} \mathring \sigma \p_{\Tcal^2} \mathring \scal\right)\Ncal
 +
  \left(\p_{\Tcal^3} \mathring \sigma \p_{\Ncal} \mathring \scal - \p_{\Ncal} \mathring \sigma \p_{\Tcal^3} \mathring \scal\right)\Tcal^2
  +  \left(\p_{\Ncal} \mathring \sigma \p_{\Tcal^2} \mathring \scal - \p_{\Tcal^2} \mathring \sigma \p_{\Ncal} \mathring \scal\right)\Tcal^3 \,,
 \label{11bottlesdownto7}
\end{align} 
where
$$
\p_{\Ncal}  = \Ncal \cdot  \nabla_{\! \tilde x} \ \text{ and }  \  \p_{\Tcal^\nu}  = \Tcal^\nu \cdot  \nabla_{\! \tilde x} \,.
$$

\subsection{Riemann variables adapted to the shock geometry}
Just as for the isentropic Euler equations that we analyzed in \cite{BuShVi2019b}, the non-isentropic
Euler system \eqref{Euler_sheep} has a rad geometric structure arising from the use of Riemann-type variables,  defined by
\begin{align} 
w = \mru \cdot \Ncal + \mrs  \,, \qquad  z =  \mru   \cdot \Ncal - \mrs   \,, \qquad  a_\nu = \mru \cdot \Tcal^\nu
\label{tildeu-dot-T}
\end{align} 
so that
\begin{align} 
\mru \cdot \Ncal  = \tfrac{1}{2} ( w+ z) \,, \qquad  \mrs = \tfrac{1}{2} (  w-   z) \,.  \label{tildeu-dot-N}
\end{align} 
The Euler sytem \eqref{Euler_sheep} can be written in terms of  $(  w,   z,  a_2,   a_3, \scal)$ as\footnote{The time
rescaling \eqref{eq:time:rescale} sets the coefficient of $w \p_1 w$ in \eqref{alrightalrightalright} to $1$, which provides a convenient
framework to study the $w$ equation as a perturbation of Burgers-type evolution.}
\begin{subequations}
\label{eq:Euler-riemann-sheep}
\begin{align}
&\p_t  w +  \left(2 \beta_1 (-\tfrac{\dot f }{2 \beta_1 } +\Jcal v \cdot \nn) + \Jcal w + \beta_2 \Jcal  z  \right) \p_1  w
+ \left(2 \beta_1 v_\mu + w \Ncal _\mu - \beta_2 z \Ncal _\mu  + 2\beta_1 a_\nu \Tcal^\nu_\mu\right) \p_\mu w \notag \\
&
\quad = - 2\beta_3 \mrs \Tcal^\nu_\mu \p_\mu a_\nu +  2 \beta_1 a_\nu \Tcal^\nu_i   \dot{\Ncal_i} + 2 \beta_1  \dot Q_{ij}  a_\nu \Tcal^\nu_j \Ncal _i  
+2 \beta_1  \left( v_\mu + \mru\cdot \Ncal  \Ncal _\mu +  a_\nu \Tcal^\nu_\mu\right)  a_\gamma \Tcal^\gamma_i 
\Ncal_{i,\mu}  \notag\\
&\quad \qquad 
- 2\beta_3 \mrs  (a_\nu \Tcal^\nu_{\mu,\mu}
+  \mru\cdot \Ncal  \Ncal_{\mu,\mu} ) 
+  \beta_4 \mathring \sigma^2 (\Jcal \p_1 \mathring \scal + \Ncal_\mu \p_\mu  \mathring \scal) 
 \,, \label{alrightalrightalright}\\
&\p_t  z +  \left(2\beta_1 (-\tfrac{\dot f }{2 \beta_1 }  + \Jcal v \cdot \nn)   +\beta_2 \Jcal  w + \Jcal z \right) \p_1  z
+ \left(2 \beta_1  v_\mu + \beta_2  w  \Ncal_\mu
+z \Ncal _\mu + 2 \beta_1  a_\nu \Tcal^\nu_\mu\right) \p_\mu z + 
 \notag \\
&
\quad = 2\beta_3 \mrs \Tcal^\nu_\mu \p_\mu a_\nu
+2 \beta_1   a_\nu \Tcal^\nu_i   \dot{\Ncal_i} +2 \beta_1  \dot Q_{ij}  a_\nu \Tcal^\nu_j \Ncal _i  
+2 \beta_1  \left( v_\mu +  \mru\cdot \Ncal  \Ncal _\mu +  a_\nu \Tcal^\nu_\mu\right)  a_\gamma \Tcal^\gamma_i \Ncal_{i,\mu}  \notag\\
&\quad  \qquad 
+ 2\beta_3 \mrs  (a_\nu \Tcal^\nu_{\mu,\mu}
+  \mru\cdot \Ncal  \Ncal_{\mu,\mu} )
+  \beta_4 \mathring \sigma^2 (\Jcal \p_1 \mathring \scal + \Ncal_\mu \p_\mu \mathring \scal) 
\,, 
\\
&\p_t  a_\nu +  \left(2 \beta_1(-\tfrac{\dot f }{2 \beta_1 } +\Jcal v\cdot \nn) +\beta_1 \Jcal w + \beta_1 \Jcal z   \right) \p_1  a_\nu 
+ 2 \beta_1  \left(v_\mu + {\tfrac{1}{2}} ( w+ z) \Ncal _\mu + a_\gamma \Tcal^\gamma_\mu \right)  \p_\mu a_\nu \notag\\
&\quad  = -2\beta_3\mrs \Tcal^\nu_\mu  \p_\mu \mrs
+   2\beta_1  \left(   \mru\cdot \Ncal  \Ncal_i +  a_\gamma \Tcal^\gamma_i\right) \dot\Tcal^\nu_i  
 +   2\beta_1  \dot Q_{ij} \left( (\mru\cdot \Ncal  \Ncal _j  +  a_\gamma  \Tcal^\gamma_j \right) \Tcal^\nu _i   \notag \\
 &\quad  \qquad 
  +  \beta_1 \left( v_\mu + \mru\cdot \Ncal \Ncal _\mu 
 +  2a_\gamma \Tcal^\gamma_\mu\right)  \left(  \mru\cdot \Ncal \Ncal_i +  a_\gamma \Tcal^\gamma_i\right) \Tcal^\nu_{i,\mu}
 + \beta_4 \mathring \sigma^2 \Tcal^\nu_\mu \p_\mu \mathring \scal
 \,,\\
 &
 \partial_t \mathring\scal + 2 \beta_1 ( -\tfrac{\dot f }{2 \beta_1 }  + \Jcal v \cdot \nn + \Jcal \mru \cdot \nn  )\p_1\mathring \scal  + 2\beta_1(v_\nu + \mru_\nu)\p_\nu \mathring \scal =0\,.
 \end{align}
\end{subequations}

\subsection{Euler equations in modulated self-similar Riemann-type variables}
\label{sec:s:s:variables}
Finally, to facilitate the analysis of shock formation, we  introduce the (modulated) 
self-similar variables:
\begin{subequations}
\label{eq:y:s:def}
\begin{align}
 s &= s(t)=-\log(\tau(t)-t) \, , \label{s(t)}\\
  y_1 &=   y_1(  x_1,t)=\frac{  x_1}{(\tau(t)-t)^{\frac32}} =   x_1 e^{\frac{3s}{2}} \,, \\   
  y_j &=   y_j(  x_j ,t)=\frac{  x_j}{(\tau(t)-t)^{\frac12}} =   x_j e^{\frac s2}\,, \qquad \mbox{for} \qquad j\in \{2,3\}  
 \,.
\end{align}
\end{subequations}

Using the self-similar variables $y$ and $s$, we rewrite the functions $w$, $z$, $a_\nu$,  $\mathring \scal$, and $v$,  defined in \eqref{tildeu-dot-T} and
 \eqref{eq:v:x:t:def},  as 
\begin{subequations}
\label{eq:ss:ansatz}
\begin{align}
w(  x,t)&=e^{-\frac s2}  W(  y,s)+\kappa (t) \,, \label{w_ansatz}  \\
z(  x,t)&=  Z(  y,s) \,,  \label{z_ansatz}   \\
a_\nu(  x ,t)&= A_\nu (  y,s) \, , \label{a_ansatz} \\
\mathring \scal(  x ,t)&= K(  y,s) \,, \label{s_ansatz} \\
v(x,t) &= V(y,s)\,, \label{v_ansatz}
\end{align}
\end{subequations}
 so that 
\begin{equation}\label{def_V}
V_i(y,s) =  \dot Q_{i1} \left( e^{-\frac{3s}{2}}y_1 +  {\tfrac 12 e^{-s} \phi_{\nu\mu} y_\nu y_\mu} \right)+e^{-\frac s2} \dot Q_{i\nu}y_\nu  - R_{ji} \dot \xi_j \, .
\end{equation} 
Introducing the parameter
$$\beta_\tau = \beta_\tau(t) = \tfrac{1}{1-\dot \tau(t)}\,,$$
the Euler system~\eqref{eq:Euler-riemann-sheep} is written in self-similar coordinates as
\begin{subequations} 
\label{euler-ss}
\begin{align} 
( \p_s- \tfrac{1}{2} )  W 
+  \left( g_{ W}  + \tfrac{3}{2}   y_1  \right)  \p_{1}  W
+  \left(  h_{W}^\mu + \tfrac{1}{2} y_\mu \right) \p_{\mu} W  
&=   {F}_{ W} - e^{-\frac s2} \beta_\tau \dot \kappa  \label{eq:euler:ss:a} \\
\p_s  Z + \left( g_{ Z}+\tfrac{3}{2}  y_1   \right) \p_{1}  Z 
+  \left(  h_{Z}^\mu + \tfrac{1}{2} y_\mu \right) \p_{\mu} Z  &= {F} _{ Z} \label{eq:euler:ss:b}\\
\p_s  A_\nu + \left( g_{U}+ \tfrac{3}{2} y_1 \right) \p_{1}   A_\nu
+  \left(  h_{U}^\mu + \tfrac{1}{2} y_\mu \right) \p_{\mu} A_\nu &= {F}_{A \nu}  \label{eq:euler:ss:c} \\
 \p_s K + (g_U + \tfrac{3}{2} y_1) \p_{1} K+ (h^\nu_U + \tfrac{1}{2} y_\nu) \p_{\nu}K&=0\label{eq:euler:ss:d} \,,
\end{align} 
\end{subequations} 
where the $y_1$ transport functions are defined by
\begin{subequations}
\label{eq:g:def}
\begin{align}
g_{ W}&=  \beta_\tau \Jcal W + \beta_\tau e^{\frac s2} \left( - \dot f +  \Jcal \left(\kappa  + \beta_2  Z + 2\beta_1    V \cdot \nn \right)  \right) = \beta_\tau  \Jcal W +  {G}_{ W} \label{eq:gW}  \\
g_{ Z}&=   \beta_2 \beta_\tau \Jcal  W + \beta_\tau e^{\frac s2} \left( - \dot f +  \Jcal \left(\beta_2 \kappa  +   Z  + 2\beta_1  V \cdot \nn \right)       \right)
=\beta_2 \beta_\tau \Jcal W + G_{ Z} \label{eq:gZ}\\
g_{U}&=  \beta_1 \beta_\tau \Jcal W + \beta_\tau e^{\frac s2} \left( - \dot f  + \Jcal   \left(\beta_1 \kappa  +  \beta_1 Z + 2  \beta_1 V \cdot \nn  \right)   \right)      
= \beta_1 \beta_\tau \Jcal  W  + G_{U}  \label{eq:gA}
\end{align}
\end{subequations}
 the $y_\nu$ transport functions are given as
\begin{subequations}
\label{eq:h:def}
\begin{align}
 h_W^\mu&= 
  \beta_\tau e^{-s} \Ncal_\mu W + \beta_\tau  e^{-\frac s2}  \left( 2    \beta_1  V_\mu + \Ncal_\mu  \kappa  - \beta_2 \Ncal_\mu Z +2 \beta_1 A_\gamma \Tcal^\gamma_\mu \right)  
\,  \label{eq:hW} \\
 h_Z^\mu&= 
 \beta_\tau \beta_2 e^{-s} \Ncal_\mu W + \beta_\tau  e^{-\frac s2}  \left(2    \beta_1  V_\mu + \beta_2 \Ncal_\mu  \kappa + \Ncal_\mu Z + 2 \beta_1 A_\gamma \Tcal^\gamma_\mu \right)  
\,  \label{eq:hZ} \\
 h_{U}^\mu&=   \beta_\tau \beta_1 e^{-s} \Ncal_\mu W +
\beta_\tau  e^{-\frac s2}  \left(2    \beta_1  V_\mu + \beta_1 \Ncal_\mu  \kappa + \beta_1 \Ncal_\mu Z + 2 \beta_1 A_\gamma \Tcal^\gamma_\mu \right)  
\,  \label{eq:hA} 
\end{align}
\end{subequations}
and the forcing functions are
\begin{subequations}
\label{eq:F:def}
\begin{align}
 {F}_{ W}&= - 2 \beta_3 \beta_\tau \sound \Tcal^\nu_\mu \partial_{\mu} A_\nu
 +  2 \beta_1 \beta_\tau  e^{-\frac s2}   A_\nu \Tcal^\nu_i \dot{\Ncal}_i
 +  2 \beta_1 \beta_\tau  e^{-\frac s2} \dot Q_{ij} A_\nu \Tcal^\nu_j \Ncal_i \notag \\
 & \quad
 + 2 \beta_1\beta_\tau e^{-\frac s2}\left(V_\mu +\Ncal_\mu  U \cdot \nn +  A_\nu \Tcal^\nu_\mu \right) A_\gamma \Tcal^\gamma_i \Ncal_{i,\mu}
 - 2 \beta_3 \beta_\tau e^{-\frac s2} \sound \left( A_\nu \Tcal^\nu_{\mu,\mu} + U\cdot \Ncal \Ncal_{\mu,\mu} \right) \notag \\
 & \quad
 + \beta _4 \beta_\tau \sound^2( \Jcal e^s \p_1 K + \Ncal_\mu \p_\mu K)
 \label{eq:FW:def}\\
 {F}_{ Z}&=  2   \beta_3 \beta_\tau e^{-\frac s2} \sound \Tcal^\nu_\mu \partial_{\mu} A_\nu
  +  2 \beta_1 \beta_\tau  e^{-s}   A_\nu \Tcal^\nu_i \dot{\Ncal}_i 
 +  2 \beta_1 \beta_\tau  e^{-s} \dot Q_{ij} A_\nu \Tcal^\nu_j \Ncal_i \notag \\
 & \quad
 + 2\beta_1 \beta_\tau e^{-s}\left(V_\mu +\Ncal_\mu  U \cdot \nn +  A_\nu \Tcal^\nu_\mu \right) A_\gamma \Tcal^\gamma_i \Ncal_{i,\mu} + 2 \beta_3 \beta_\tau e^{-s} \sound \left( A_\nu \Tcal^\nu_{\mu,\mu} + U\cdot \Ncal \Ncal_{\mu,\mu} \right) \notag  \\
 & \quad
 +  \beta _4 \beta_\tau \sound^2( \Jcal e^{\frac{s}{2}}  \p_1 K + \Ncal_\mu  e^{-\frac{s}{2}} \p_\mu K)
 \label{eq:FZ:def}\\
 {F}_{ A\nu}&= - 2  \beta_3 \beta_\tau e^{-\frac s2}  \sound T^\nu_\mu \p_{\mu} \sound
 + 2 \beta_1 \beta_\tau e^{-s} \left( U\cdot\Ncal \Ncal_i + A_\gamma \Tcal^\gamma_i\right) \dot{\Tcal}^\nu_i 
 + 2\beta_1 \beta_\tau e^{-s} \dot{Q}_{ij} (  U\cdot \Ncal \Ncal_j + A_\gamma \Tcal^\gamma_j ) \Tcal^\nu_i \notag\\
 &\quad 
 + 2\beta_1 \beta_\tau e^{-s} \left(V_\mu + U\cdot \Ncal \Ncal_\mu + A_\gamma \Tcal^\gamma_\mu\right) \left(U\cdot \Ncal \Ncal_i + A_\gamma \Tcal^\gamma_i \right) \Tcal^\nu_{i,\mu}  +   \beta_4  \beta_\tau e^{-\frac{s}{2}}  \sound^2 \Tcal^\nu_\mu \p_\mu K
 \label{eq:A:def}\,.
\end{align}
\end{subequations}
In \eqref{eq:F:def} we have also used the self-similar variants of $\mru$, $\mrs$, and $\mathring \scal$ which, together with the self-similar variant of
$\mathring p$, are given by
\begin{subequations} 
\label{trannies}
\begin{align}
\mru( x,t)&= U(  y,s) \,, \label{U-trammy}   \\
\mathring \rho(x,t) & = R(y,s)  \,, \label{R-trammy} \\
\mrs( x,t)&=  \sound(  y,s) \,, \label{Sigma-trammy}  \\
\mathring p( x,t)&=  P(  y,s) \,, \label{P-trammy}
\end{align}
\end{subequations} 
so that 
\begin{align}
U \cdot \Ncal = \tfrac 12 \left( \kappa + e^{-\frac s2} W + Z\right) \qquad \mbox{and}\qquad  \sound = \tfrac 12 \left( \kappa + e^{-\frac s2} W - Z\right) \,.
\label{eq:UdotN:Sigma}
\end{align}
The system \eqref{euler-ss} may be written as
\begin{subequations} 
\begin{align*} 
 \p_s W- \tfrac{1}{2}  W  + (\mathcal{V} _W \cdot \nabla) W  &= {F} _W  \,, \\
\p_s Z + (\mathcal{V}_Z \cdot \nabla) Z &={F} _Z \,,   \\
\p_s A_\nu +  (\mathcal{V} _U \cdot \nabla) A_\nu &={F}_{A\nu}  \,, \\
\p_s K +  (\mathcal{V} _U \cdot \nabla) K & =0 \,,
\end{align*} 
\end{subequations} 
where the  transport velocities are abbreviated as
\begin{subequations} 
\label{eq:transport_velocities}
\begin{align} 
\mathcal{V} _W &= \left(  g_W  + \tfrac{3}{2} y_1 \,,     h_W^2 + \tfrac{1}{2} y_2\,,     h_W^3 + \tfrac{1}{2} y_3 \right) \,, \\
\mathcal{V} _Z & = \left(  g_Z  + \tfrac{3}{2} y_1  \,,     h_Z^2 + \tfrac{1}{2} y_2\,,     h_Z^3 + \tfrac{1}{2} y_3 \right) \,, \\
\mathcal{V} _U & = \left(  g_U  + \tfrac{3}{2} y_1\,,     h_U^2 + \tfrac{1}{2} y_2\,,     h_U^3 + \tfrac{1}{2} y_3 \right) \label{V_U} \, .
\end{align} 
\end{subequations} 

\subsection{Self-similar Euler equations in terms of velocity, pressure, and entropy}

From \eqref{Euler_sheep}, \eqref{pressure-sheep}, \eqref{momentum-sheep-alt},  \eqref{eq:y:s:def}, \eqref{U-trammy}, \eqref{Sigma-trammy} we deduce that $(U,P,K)$ are solutions of
\begin{subequations} 
\label{USigma-euler-ss}
\begin{align} 
 &\p_s U_i - 2 \beta_1 \beta_\tau e^{-s} \dot Q_{ij} U_j + (\mathcal{V}_U \cdot  \nabla )U_i
+ 2\beta_\tau \beta_1 ( \alpha \sound)^ {-\frac{1}{ \alpha }} e^ {\frac{K}{2\alpha }} (  \Jcal \Ncal_i e^{\frac s2}  \p_{1} P
+ \delta^{i\nu} e^{-\frac s2} \p_{\nu} P)= 0 \, ,
\\
 &\p_s P +(\mathcal{V}_U \cdot  \nabla ) P
 +2\beta_\tau\beta_1 \gamma e^{\frac s2}  P \p_{1}U \cdot \nn \Jcal + 2\beta_\tau\beta_1 \gamma e^{-\frac s2} P \p_{\nu} U_\nu =0 \,,  \label{Plong-SS}\\
 &\p_s K + (\mathcal{V}_U \cdot  \nabla )K=0 \,. 
\end{align} 
\end{subequations}

For the purpose of performing high-order energy estimates, it is convenient to introduce
\begin{subequations}
\label{eq:thunder:variables}
\begin{align} 
\uu & =  U\, \label{Unew} \\
\pp & =   \sound  e^{-\frac{ K }{2 \upgamma}}  = \tfrac{1}{\alpha} (\upgamma P)^{\frac{\alpha}{\gamma}} \,, \label{Pnew} \\
\hh & =   e^\frac{ K }{2 \upgamma}   \,, \label{Snew} 
\end{align} 
\end{subequations}
and re-express the system of equations \eqref{USigma-euler-ss} as the following $(\uu, \pp, \hh)$-system:
\begin{subequations} 
\label{UPS-new}
\begin{align}
&\p_s \uu_i  + (\mathcal{V} _U \cdot \nabla) \uu_i
+  2 \beta_\tau \beta_3 \hh^2 \pp  \left( \Jcal \Ncal_i e^{\frac s2}  \p_{1} \pp
+\delta^{i\nu} e^{-\frac s2} \p_{\nu} \pp \right) 
 = 2\beta_\tau \beta_1 e^{-s} \dot Q_{ij} \uu_j 
  \, ,\\
& \p_s \pp + (\mathcal{V} _U \cdot \nabla) \pp 
+ 2 \beta_\tau\beta_3 \pp  \left( e^{\frac s2} \Jcal \Ncal \cdot  \p_{ 1} \uu   +  e^{-\frac s2}   \p_{\nu} \uu_\nu\right) =0 \,, \label{Pnew-ss}  \\
&\p_s \hh +  (\mathcal{V} _U \cdot \nabla) \hh =0\,. \label{Snew-ss}
\end{align} 
\end{subequations}

Finally, we define the self-similar variant of the specific vorticity via  
\begin{align} 
\mrg( x,t)&= \Omega(  y,s) \,.  \label{svort-trammy}
\end{align}

\subsection{Evolution of higher order derivatives}
\subsubsection{Higher-order derivatives for the $(W,Z,A,K)$-system}

We shall also need the differentiated form of the system  \eqref{euler-ss}, which we record here for convenience.  For a multi-index $\gamma \in {\mathbb N}_0^3$, we use the notation $\gamma = (\gamma_1, \check \gamma) = (\gamma_1, \gamma_2,\gamma_3)$.  We have that 
\begin{subequations} 
\label{euler_for_Linfinity}
\begin{align}
\left( \p_s + \tfrac{3\gamma_1 + \gamma_2 + \gamma_3-1}{2} + \beta_\tau \left(1 + \gamma_1  {\bf 1}_{\gamma_1\geq 2}   \right) \Jcal \p_1 W \right)\p^\gamma W   +  \left( \mathcal{V}_W \cdot \nabla\right) \p^\gamma W &= F^{(\gamma)}_W \,,
\label{euler_for_Linfinity:a}
\\
\left( \p_s + \tfrac{3\gamma_1 + \gamma_2 + \gamma_3}{2} + \beta_2 \beta_\tau  \gamma_1  \Jcal \p_1 W \right)\p^\gamma Z   +  \left( \mathcal{V}_Z \cdot \nabla\right) \p^\gamma Z &= F^{(\gamma)}_Z  \,,
 \label{euler_for_Linfinity:b}
\\
\left( \p_s + \tfrac{3\gamma_1 + \gamma_2 + \gamma_3}{2} + \beta_1 \beta_\tau  \gamma_1 \Jcal \p_1 W \right)\p^\gamma A_\nu   +  \left( \mathcal{V}_U \cdot \nabla\right) \p^\gamma A_\nu &= F^{(\gamma)}_{A\nu} \,,
 \label{euler_for_Linfinity:c}
\\
\left( \p_s + \tfrac{3\gamma_1 + \gamma_2 + \gamma_3}{2} + \beta_1 \beta_\tau  \gamma_1 \Jcal \p_1 W \right)\p^\gamma K   +  \left( \mathcal{V}_U \cdot \nabla\right) \p^\gamma K &= F^{(\gamma)}_{K} \,,
 \label{euler_for_Linfinity:d}
\end{align}
\end{subequations}
where $|\gamma|\geq 1$ and the forcing terms are  
\begin{align}
\label{eq:F:W:def}
F^{(\gamma)}_W 
&= \p^\gamma F_W 
- \sum_{0\leq \beta < \gamma} {\gamma \choose \beta}  \left(\p^{\gamma-\beta}G_W \p_1 \p^\beta W + \p^{\gamma-\beta} h_{W}^\mu \p_\mu \p^\beta W\right) \notag\\
&\quad - \beta_\tau {\bf 1}_{|\gamma|\geq 3}\sum_{\substack{1\leq |\beta| \leq |\gamma|-2 \\ \beta\le\gamma}} {\gamma \choose \beta}  \p^{\gamma-\beta} (\Jcal W)   \p_1\p^\beta W  
- \beta_\tau {\bf 1}_{|\gamma|\geq 2}\sum_{\substack{ |\beta| = |\gamma|-1 \\ \beta\le\gamma, \beta_1 = \gamma_1}} {\gamma \choose \beta}  \p^{\gamma-\beta} (\Jcal W)    \p_1\p^\beta W 
\end{align}
for the $\p^\gamma W$ evolution, and  
\begin{subequations}
\label{eq:F:ZA:def}
\begin{align}
F^{(\gamma)}_Z
&= \p^{\gamma} F_Z 
- \sum_{0\leq \beta < \gamma} {\gamma \choose \beta}  \left(\p^{\gamma-\beta}G_Z \p_1 \p^\beta Z + \p^{\gamma-\beta}h_Z^\mu \p_\mu \p^\beta Z\right) \notag\\
&\quad - \beta_2 \beta_\tau {\bf 1}_{|\gamma|\geq 2}\sum_{\substack{0\leq |\beta| \leq |\gamma|-2 \\ \beta\le\gamma}} {\gamma \choose \beta}  \p^{\gamma-\beta} (\Jcal W)   \p_1\p^\beta Z  
- \beta_2 \beta_\tau  \sum_{\substack{ |\beta| = |\gamma|-1 \\ \beta\le\gamma, \beta_1 = \gamma_1}} {\gamma \choose \beta}  \p^{\gamma-\beta} (\Jcal W)    \p_1\p^\beta Z 
\\
F^{(\gamma)}_{A\nu}
&= \p^\gamma F_{A\nu} 
- \sum_{0\leq \beta < \gamma} { \gamma \choose \beta}   \left(\p^{\gamma-\beta}G_{U} \p_1 \p^\beta A_\nu + \p^{\gamma-\beta} h_{U}^\mu \p_\mu \p^\beta A_\nu\right) \notag\\
&\quad - \beta_1 \beta_\tau  {\bf 1}_{|\gamma|\geq 2} \sum_{\substack{0\leq |\beta| \leq |\gamma|-2 \\ \beta\le\gamma}} {\gamma \choose \beta}  \p^{\gamma-\beta} (\Jcal W)   \p_1\p^\beta A_\nu
- \beta_1 \beta_\tau  \sum_{\substack{ |\beta| = |\gamma|-1 \\ \beta\le\gamma, \beta_1 = \gamma_1}} {\gamma \choose \beta}  \p^{\gamma-\beta} (\Jcal W)    \p_1\p^\beta A_\nu
\\
F^{(\gamma)}_{K}
&=  
- \sum_{0\leq \beta < \gamma} { \gamma \choose \beta}   \left(\p^{\gamma-\beta}G_{U} \p_1 \p^\beta K + \p^{\gamma-\beta} h_{U}^\mu \p_\mu \p^\beta K\right) \notag\\
&\quad - \beta_1 \beta_\tau  {\bf 1}_{|\gamma|\geq 2} \sum_{\substack{0\leq |\beta| \leq |\gamma|-2 \\ \beta\le\gamma}} {\gamma \choose \beta}  \p^{\gamma-\beta} (\Jcal W)   \p_1\p^\beta K
- \beta_1 \beta_\tau  \sum_{\substack{ |\beta| = |\gamma|-1 \\ \beta\le\gamma, \beta_1 = \gamma_1}} {\gamma \choose \beta}  \p^{\gamma-\beta} (\Jcal W)    \p_1\p^\beta K
\end{align}
\end{subequations}
for the $\p^\gamma Z$, $\p^\gamma A_\nu$, and $\p^\gamma K$ evolutions.

\subsubsection{Higher-order derivatives for $\tilde W$}
We let $\bar W(y)$ denote a particular self-similar, stable, stationary solution of the 3d Burgers equation, given by
 \begin{align}
\bar W(y) = \langle \check y\rangle W_{\rm 1d} \left(\frac{y_1}{\langle \check y\rangle^3}\right)   \label{eq:barW:def}
 \end{align} 
where $\langle \check y\rangle=  1+ y_2^2 + y_3^2 $ is the Japanese bracket, and where $W_{\rm 1d}(y_1) $ 
is the stable globally self-similar solution of the 1d Burgers equation, i.e.,   $W_{\rm 1d}(y_1) $  is a solution to $W_{\rm 1d} + W_{\rm 1d}^3=-y_1$.
We refer the reader to \cite{CaSmWa1996}, \cite{CoGhMa2018}, and Section 2.7 of \cite{BuShVi2019b} for the explicit form
of $W_{\rm 1d}(y_1) $  and for properties of $\bar W(y) $.   We note that $\bar W$ is one example from the ten-dimensional family $\bar W_ \mathcal{A} $ of stable 
stationary solutions to the self-similar 3d Burgers equation which are given by Proposition \ref{prop-stationary-burgers} in Appendix \ref{sec:delay}.   
The symmetric $3$-tensor $ \mathcal{A} $ represents $\p^\gamma \bar W_ \mathcal{A}(0) $ for $\abs{\gamma}=3$.    The function $\bar W$ is in fact equal
to $\bar W_ \mathcal{A} $ for the case that $\mathcal{A} _{111}=6$, $\mathcal{A} _{122}=\mathcal{A} _{133}=2$, and all other components vanish.

Of paramount importance to our analysis, is the evolution of the perturbation
\begin{align}
\label{eq:tilde:W:def} 
\tilde W(y,s) = W(y,s) - \bar W(y)  
\end{align}
which satisfies 
\begin{align}
  &\p_s \tilde W+(\beta_\tau \Jcal \p_1 \bar W- \tfrac{1}{2})  \tilde W  + (\mathcal{V} _W \cdot \nabla) \tilde W\notag \\
 &\qquad\qquad= F _W - e^{-\frac s2} \beta_\tau \dot \kappa +   ((\beta_\tau \Jcal -1)\bar W-G_W)\p_1 \bar W- h_W^\mu \p_\mu \bar W
 =: \tilde{F}_W \, .\label{eq:tilde:W:evo} 
\end{align}
Applying $\p^\gamma$ to \eqref{eq:tilde:W:evo}, we obtain that  $\p^\gamma \tilde W$ obeys
\begin{align}\label{eq:p:gamma:tilde:W:evo}
\left( \p_s + \tfrac{3\gamma_1 + \gamma_2 + \gamma_3-1}{2} + \beta_\tau \Jcal  \left(\p_1 \bar W + \gamma_1  \p_1 W \right) \right)\p^\gamma \tilde W   +  \left( \mathcal{V}_W \cdot \nabla\right) \p^\gamma \tilde W &= \tilde F^{(\gamma)}_W
\end{align}
for $\abs{\gamma} \geq 1$, 
where the forcing term $\tilde F^{(\gamma)}_W$ is given by
\begin{align}
\tilde F^{(\gamma)}_W 
&= \p^\gamma \tilde F_W 
- \sum_{0\leq \beta < \gamma} {\gamma \choose \beta}  \left(\p^{\gamma-\beta}G_W \p_1 \p^\beta \tilde W + \p^{\gamma-\beta} h_{W}^\mu \p_\mu \p^\beta \tilde W+\beta_{\tau}\p^{\gamma-\beta} (\Jcal \partial_1\bar W)   \p^\beta \tilde W  \right) \notag\\
&\quad - \beta_\tau {\bf 1}_{|\gamma|\geq 2}\sum_{\substack{1\leq |\beta| \leq |\gamma|-2 \\ \beta\le\gamma}} {\gamma \choose \beta}  \p^{\gamma-\beta} (\Jcal W)   \p_1\p^\beta \tilde W  
- \beta_\tau \sum_{\substack{ |\beta| = |\gamma|-1 \\ \beta\le\gamma, \beta_1 = \gamma_1}} {\gamma \choose \beta}  \p^{\gamma-\beta} (\Jcal W)    \p_1\p^\beta \tilde W \label{eq:p:gamma:tilde:F}\,.
\end{align}

\subsection{Constraints and the evolution of dynamic modulation variables}
\label{sec:constraints:initial}
The use of modulated self-similar variables allows us to ensure that the evolution of $W$ in \eqref{eq:euler:ss:a} maintains the constraints
\begin{align} 
W(0,s)=0 \,, \quad  \p_{1} W(0,s)=-1\,, \quad \check \nabla W(0,s)=0\,,  \quad
\nabla^2 W(0,s)=0 \,,
\label{eq:constraints}
\end{align} 
for all $s\geq -\log \eps$. This is achieved by choosing our $10$ time-dependent dynamic modulation  parameters $\{n_\nu\}_{\nu=2}^3, \{\xi_i\}_{i=1}^3, \kappa, \tau, \{\phi_{\nu \mu}\}_{\nu,\mu =2}^3$ to satisfy a $10$-by-$10$ coupled system of ODEs, which we describe next.

At time $t=-\eps$ the modulation parameters are defined as
\begin{align}
\kappa(-\eps) = \kappa_0, \qquad \tau(-\eps) = \xi(-\eps) = n_\mu (-\eps) =  0, \qquad  \phi_{\nu\mu}(-\eps) = \phi_{0,\nu \mu}
\label{eq:modulation:IC}
\,,
\end{align}
where $\kappa_0$ is defined in \eqref{vomey1} and $\phi_0$ is defined by \eqref{eq:phi:0:def}.
In order to determine the time derivatives of our $10$ modulation parameters, we use the explicit form of the evolution equations for $W$, $\nabla W$ and $\nabla^2 W$ (cf.~\eqref{eq:euler:ss:a} and \eqref{euler_for_Linfinity:a}), which are evaluated at $y=0$ and take into account the  constraints in \eqref{eq:constraints}. Note that in this subsection we only collect the equations which {\em implicitly} define the evolution of the modulation parameters; only in Section~\ref{sec:constraints} do we untangle the coupled nature of these implicitly defined ODEs to actually define the evolution of the constraints (cf.~\eqref{eq:dot:phi:dot:n} and \eqref{eq:dot:kappa:dot:tau}), and prove that the resulting ODEs are globally well-posed. 

Throughout the paper, for a function $\varphi(y,s)$,  we shall denote $\varphi(0,s)$ by $\varphi^0(s)$.
We make a preliminary observation regarding the value at $y=0$ for the forcing terms $F_W^{(\gamma)}$ which appear in the evolution \eqref{euler_for_Linfinity:a} for $\p^\gamma W$. Using \eqref{eq:constraints} it is not hard to check that for  any $\gamma \in {\mathbb N}_0^3$ with $\abs{\gamma} =1$ or $\abs{\gamma} = 2$ we have that
\begin{align}
F_W^{(\gamma),0} = \p^\gamma F_W^0 + \p^\gamma G_W^0 \, .
\label{eq:Gulash} 
\end{align}
Therefore, it is sufficient to know the derivatives up to order $2$ of $F_W$ and $G_W$ at $y=0$; these derivatives may be computed explicitly, and for convenience of the reader we have listed them in Appendix~\eqref{sec:explicit:crap}, see equations \eqref{Birds_are_rad} and \eqref{Bears_are_not_rad}.
Next, we turn to the evolution equations for the modulation parameters. 

First, we evaluate the equation for $W$ in~\eqref{eq:euler:ss:a} at $y=0$ to obtain a definition for $\dot \kappa$. Using \eqref{eq:euler:ss:a} and \eqref{eq:constraints} we obtain that 
\begin{align}
 - G_W^0 = F_W^0 - e^{-\frac s2} \beta_\tau \dot \kappa \qquad \Rightarrow \qquad \dot \kappa = \tfrac{1}{\beta_\tau} e^{\frac s2} \left( F_W^0 + G_W^0 \right) \, .
  \label{eq:dot:kappa:1}
\end{align}

Second, we evaluate the equation for $\p_1 W$ at $y=0$ and obtain a formula for $\dot \tau$.
Indeed, using that $-1+\beta_\tau = \frac{\dot\tau}{1-\dot\tau} = \dot \tau \beta_\tau$, we obtain from \eqref{euler_for_Linfinity:a} with $\gamma = e_1$ that 
\begin{align}
 - (1-\beta_\tau) = \p_1 F_W^0 + \p_1 G_W^0  \qquad \Rightarrow \qquad \dot \tau = \tfrac{1}{\beta_\tau} \left(\p_1 F_W^0 + \p_1 G_W^0 \right) \, .
  \label{eq:dot:tau:1}
\end{align}

Third, we turn to the evolution equation for $\check \nabla W$ at $y=0$, which allows us to compute $\dot Q_{1j}$. Evaluating \eqref{euler_for_Linfinity:a} with $\gamma = e_\nu$ at $y=0$ and using \eqref{eq:Gulash} we obtain for $\nu \in \{2,3\}$ that
\begin{align}
F_W^{0,(0,1,0)} = F_W^{0,(0,0,1)} = 0 \qquad \Rightarrow \qquad \p_\nu F_W^0 + \p_\nu G_W^0 = 0 \, .
\label{eq:check_derivative_GW}
\end{align}
It is not immediately apparent that \eqref{eq:check_derivative_GW} determines $\dot Q_{1j}$. In order to see this one has to inspect the explicit formula for $\p_\nu G_W^0$ in \eqref{eq:GW:0:c}, and to note that $\p_\nu G_W^0 = 2\beta_1 \dot Q_{1\nu} + $ terms which are all small (bounded by $\eps$ to a positive power). This is explained in \eqref{eq:dot:Q:1} below. Note that once $\dot Q_{1j}$ is known, we can determine $\dot{\check n}$ thorough an algebraic computation; this will be achieved in \eqref{eq:dot:n:def} below. 

Fourth, we analyze the evolution of $\p_1 \nabla W$ at $y=0$. This constraint allows us to compute $G_W^0$ and $h_W^{\mu,0}$, which will in turn allow us to express $\dot \xi_i$; we initially focus on computing $G_W^0$ and $h_W^{\mu,0}$. Evaluating \eqref{euler_for_Linfinity:a} with $\gamma = e_1 + e_i$ at $y=0$ for $i\in \{1,2,3\}$, and using \eqref{eq:Gulash},  we obtain
\begin{align}
G_W^0 \p_{1i1}W^0 + h_W^{\mu,0} \p_{1i\mu} W^0 &= \p_{1i} F_W^0 + \p_{1i} G_W^0 \, .
\label{eq:Hessian:emerges}
\end{align}
On the left side of the above identity we recognize the matrix  
\begin{align}
{\mathcal H}^0(s) := (\p_1\nabla^2 W)^0(s) 
\label{eq:d1W:Hessian}
\end{align}
acting on the vector with components $G_W^0$, $h_{W}^{2,0}$, and $h_{W}^{3,0}$. We will show (see~\eqref{eq:inverse:Hessian} below) that the matrix ${\mathcal H}^0$  remains very close to the matrix ${\rm diag}(6,2,2)$, for all $s\geq -\log \eps$, and thus it is invertible .
Therefore, we can express
\begin{subequations}
\label{eq:Hamburgler}
\begin{align}
G_W^0 &= ({\mathcal H}^0)^{-1}_{1i}  (\p_{1i} F_W^0 + \p_{1i} G_W^0) \label{eq:GW:def:1} \\
h_W^{\mu,0} &= ({\mathcal H}^0)^{-1}_{\mu i}  (\p_{1i} F_W^0 + \p_{1i} G_W^0) \label{eq:hj:def:1}\, .
\end{align}
\end{subequations}
Once \eqref{eq:Hamburgler} is obtained, we may derive the evolution for $\dot \xi_i$. Indeed, from \eqref{eq:hW}, \eqref{def_f} evaluated at $\tilde x =0$, the definition of $V$ in \eqref{def_V}, 
the constraints in \eqref{eq:constraints} and the identities  $\Ncal_\mu^0 = 0$, $\Tcal_\mu^{\gamma,0} = \delta_{\gamma\mu}$ we have that 
\begin{align}
\tfrac{1}{\beta_\tau} h_W^{\mu,0} &= 2 \beta_1  e^{-\frac s2} \left(  A_\mu^0  -  R_{j\mu} \dot \xi_j \right) \,, \label{eq:hj:0:a} 
\end{align} 
Similarly, from the definition of $G_W$ in \eqref{eq:gW}, \eqref{def_f}, and the constraints in \eqref{eq:constraints}, we deduce that
\begin{align}
\tfrac{1}{\beta_\tau} G_W^0 &=  e^{\frac s2} \left( \kappa + \beta_2 Z^0 - 2\beta_1 R_{j1} \dot \xi_j \right) \,.\label{eq:GW:0:a}
\end{align}
Since the matrix $R$ is orthogonal (hence invertible), it is clear that \eqref{eq:Hamburgler}, \eqref{eq:hj:0:a}, and \eqref{eq:GW:0:a} determine  $\dot \xi_j$.

Lastly, we use the evolution of $\check \nabla^2 W$ at $y=0$ in order to determine $\dot{\phi}_{\nu\gamma}$. Evaluating  \eqref{euler_for_Linfinity:a} with $\gamma = e_\nu + e_\gamma$ at $y=0$ and using  \eqref{eq:Gulash}, we obtain 
\begin{align}
\label{eq:vanilla:candle}
G_W^0 \p_{1\nu\gamma} W^0  + h_W^{\mu,0} \p_{\mu \nu \gamma}W^0  = \p_{\nu\gamma} F_W^0 + \p_{\nu\gamma} G_W^0
\end{align}
for $\nu ,\gamma \in \{2,3\}$.
Using \eqref{eq:GW:def:1} and \eqref{eq:hj:def:1} we rewrite the above identity as
\begin{align}
\p_{\nu\gamma} G_W^0 = ({\mathcal H}^0)^{-1}_{1i}  (\p_{1i} F_W^0 + \p_{1i} G_W^0)  \p_{1\nu\gamma} W^0  + ({\mathcal H}^0)^{-1}_{\mu i}  (\p_{1i} F_W^0 + \p_{1i} G_W^0)  \p_{\mu \nu \gamma}W^0  - \p_{\nu\gamma} F_W^0 \, .
\label{eq:D2:GW:def:1}
\end{align}
As with \eqref{eq:check_derivative_GW} earlier, it is not immediately clear that \eqref{eq:D2:GW:def:1} determines the evolution of $\dot{\phi}_{\nu\gamma}$. In order to see this, we need to inspect the precise definition of $\p_{\nu\gamma} G_W^0$ (cf.~\eqref{eq:GW:0:f} below), which yields that  $\dot \phi_{\nu\gamma} = - e^{\frac s2} \tfrac{1}{\beta_\tau} \p_{\nu\gamma} G_W^0 + $ terms which are smaller (by a positive power of $\epsilon$). Details are given in \eqref{eq:dot:phi:def:1} below. 

The computations in this subsection derive implicit definitions for the time derivatives of our ten modulation parameters.  In Section~\ref{sec:constraints} we will show that the resulting system of ODEs for the modulation parameters is in fact solvable globally in time.

 \section{Main results}
\label{sec:results}
\subsection{Data in physical variables $(\xcal, t)$}
\label{sec:data:real}
It is convenient to set  $t_0 = -\eps$. This corresponds to $\tcal_0 = - \frac{2}{1+\alpha} \eps$. We define  initial conditions for the modulation variables defined in \eqref{mod-def} as follows:
\begin{align} 
\kappa_0 := \kappa(-\eps)  \,, \ \ \ \ \tau_0:=\tau(-\eps)=0\,,  \ \ \ \  \xi_0:=\xi(-\eps)=0\,, \ \ \ \ \check n_0:= \check n(-\eps) =0\,, \ \ \ \  \phi_0:= \phi(-\eps)  \,,
\label{mod-ic}
\end{align} 
where 
\begin{align} 
\kappa_0 >1\,, \quad \abs{\phi_0} \le \eps \,.
\label{ic-kappa0-phi0}
\end{align} 
Next, we define the initial value for the parameterization $f$ of the front by
$$
f_0(\check \xcal)= \tfrac 12 {\phi_0}_{\nu\mu} \xcal_\nu \xcal_\mu \, ,
$$
and according to \eqref{tangent}, we define the orthonormal basis $(\Ncal_0,\Tcal_0^2,\Tcal_0^3)$ by
\begin{subequations}
\label{geom000}
\begin{align}
\Ncal_0 &=  \Jcal_0^{-1} (1, -f_{0,_2}, -f_{0,_3}), \qquad \mbox{where} \qquad \Jcal_0 = ( 1+  |f_{0,_2}|^2+ |f_{0,_3}|^2)^ {\frac{1}{2}}, 
\label{geom00}\\
\Tcal^2_0 &=  \left( \tfrac{f_{0,_2}}{\Jcal_0}, 1 -  \tfrac{(f_{0,_2})^2}{\Jcal_0(\Jcal_0+1)} \,,  \tfrac{- f_{0,_2}f_{0,_3}}{\Jcal_0(\Jcal_0+1)}\right) \, , \qquad \mbox{and} \qquad
\Tcal^3_0 =  \left( \tfrac{f_{0,_3}}{\Jcal_0},  \tfrac{- f_{0,_2}f_{0,_3}}{\Jcal_0(\Jcal_0+1)} \,,  1  -  \tfrac{(f_{0,_3})^2}{\Jcal_0(\Jcal_0+1)} \right).
\label{geom0}
\end{align} 
\end{subequations}
From \eqref{ic-kappa0-phi0} and  \eqref{geom000} we deduce  
\begin{align} 
\abs{ \Ncal_0 - e_1} \le \eps\,, \qquad \abs{ \Tcal_0^\nu - e_\nu} \le \eps \,.   \label{gerd}
\end{align} 
At $t=-\eps$, the   variable $x$ is given by 
\begin{align} 
x_1 = \xcal_1-  f_0(\check \xcal)\,,\qquad x_2 = \xcal_2\,, \qquad  x_3 = \xcal_3 \,,   \label{ic-cov}
\end{align} 
which is a consequence of \eqref{mod-ic}, \eqref{eq:tilde:x:def}, and \eqref{x-sheep}.

The remaining initial  conditions  are for the velocity field, density, and entropy which then provides us with  the rescaled sound speed:
$$
u_0(\xcal) := u(\xcal,-\eps), \qquad \rho_0(\xcal) := \rho(\xcal, -\eps) \,,  \qquad \scal_0(\xcal) := \scal(\xcal, -\eps)\,, \qquad 
\sigma_0(\xcal) := \tfrac{\rho_0^ \alpha }{\alpha } e^{{\frac{\scal_0}{2}} } \,.
$$
Following \eqref{usigma-sheep} and \eqref{tildeu-dot-T}, we introduce the Riemann-type variables at initial time $t=-\eps$ as
\begin{align} \label{eq:tilde:wza:0}
\tilde w_0(\xcal)& :=   u_0(\xcal) \cdot \Ncal_0(\check \xcal) + \sigma_0(\xcal)  \,, \qquad
\tilde z_0(\xcal) :=   u_0(\xcal) \cdot \Ncal_0(\check \xcal) - \sigma_0(\xcal)  \,, \qquad
\tilde a_{0\nu}(\xcal) :=  u_0(\xcal) \cdot \Tcal^\nu(\check \xcal) \,.
\end{align} 
Using \eqref{ic-cov} and the fact that $\tilde w_0(\xcal)=w(x,-\eps) $ and that $\check \nabla f_0 ( 0) = 0$, it follows that
\begin{align} 
\p_{x_\nu} \p_{x_\mu} w_0(0) = \p_{\xcal_\nu} \p_{\xcal_\mu} \tilde w_0(0) + \p_{\xcal_1} w_0(0) \phi_{0 \nu \mu}\,.  \label{phi0-cond}
\end{align} 
As we will explain below, we will require that 
$\partial_{\xcal_1} \tilde w_0(0) = -{\tfrac{1}{\eps}}$, $  \check\nabla_{\xcal} \tilde w_0(0) =0$ 
 $ \check\nabla^2_{x} w_0(0) =0$,  
and that
$ \abs{\check\nabla_{ \! \xcal}^2 \tilde w_0(0)} \le 1 $, and thus from \eqref{phi0-cond}, we find that
\begin{align}
\phi_{0\nu\mu} =    \eps \p_{\xcal_\nu}  \p_{\xcal_\mu} \tilde w_0(0) \,,
\label{eq:phi:0:def}
\end{align}
which shows that \eqref{ic-kappa0-phi0} holds.

In order to establish the formation of a stable self-similar shock, we shall stipulate conditions on the initial data.  It is convenient to first explain
these conditions in self-similar variables, and we now proceed to do so.

\subsection{Data in self-similar variables $(y,s)$} 
At $s= -\log \eps$ we have that $\tau_0=0$, and thus the self-similar variables $y$ are given by
\begin{align}
\label{eq:ss:y:initial}
y_1 = \eps^{-\frac 32} x_1 = \eps^{-\frac 32}\left(\xcal_1 - f_0(\check \xcal)\right)\, ,
\qquad \mbox{and} \qquad
\check y = \eps^{-\frac 12} \check x = \eps^{-\frac 12} \check \xcal\,.
\end{align}
Second, we use \eqref{eq:ss:ansatz},  \eqref{mod-ic}, and \eqref{eq:tilde:wza:0}, to define 
\begin{alignat*}{2}
W(y,-\log \eps) &= \eps^{-\frac 12} \left( \tilde w_0(\xcal) - \kappa_0 \right)\,, \qquad
&&Z(y,-\log \eps) = \tilde z_0 (\xcal)\, , \\
A_\nu(y,-\log \eps) &= \tilde a_{0\nu} (\xcal) \,, 
\qquad
&&K(y,-\log \eps) = \tilde \scal_0 (\xcal)
\, .
\end{alignat*}
This initial data is supported in the set $\XXX_0$, given by 
\begin{align}
\XXX_0 =\left\{\abs{y_1}\leq   \eps ^{-1} ,\abs{\check y}\leq  \eps^{-\frac 13} \right\}\,.
\label{eq:support-init}
\end{align}

At $y=0$, we shall mimic the behavior of $\bar W(0)$ and assume that at initial time $s=-\log\eps$, 
\begin{align}
W(0,-\log \eps) = 0\, , 
\quad 
\p_1 W(0,-\log \eps) = -1\,,
\quad 
\check \nabla W(0,-\log \eps) = 0 \,,
\quad
\nabla^2 W(0,-\log \eps) = 0\,.
\label{eq:fat:cat}
\end{align}

We define a sufficiently large parameter $M = M(\alpha,\kappa_0)\geq 1$ (which is in particular independent of $\eps$), a small length scale $\ell$, and a large length scale $\LLL$  by 
\begin{subequations}
\begin{align}
\ell &= (\log M)^{-5}\,,
\label{eq:ell:choice}
\\
\LLL &=\eps^{-\frac{1}{10}}  \, .
\label{eq:pounds}
\end{align}
\end{subequations}
For $\abs{y}\leq \ell$ we shall prove that $W$ is well approximated by its series expansion at $y=0$, while for
$\ell \leq \abs{y} \leq \LLL$ we show that $W$ and $\nabla W$ track $\bar W$ and $\nabla \bar W$, respectively. 

For the initial datum of $\tilde W = W - \bar W$ given by
$$
\tilde W (y,-\log\eps)= W(y,-\log \eps) - \bar W(y) \, ,
$$
we suppose  that for $\abs{y} \leq \LLL$,
\begin{subequations}
\label{ss-ic-1}
\begin{alignat}{2}
\eta^{-\frac 16}(y) \abs{\tilde W(y,-\log \eps)} &\leq \eps^{\frac{1}{10}} 
\label{eq:tilde:W:zero:derivative}
\\
\eta^{\frac 13}(y) \abs{\p_1 \tilde W(y,-\log \eps)} &\leq \eps^{\frac{1}{11}}
\label{eq:tilde:W:p1=1}
\\
\abs{\check \nabla \tilde W(y,-\log \eps)} &\leq \eps^{\frac{1}{12}}
\label{eq:tilde:W:check=1}
\,,
\end{alignat}
\end{subequations}
where $\eta(y) = 1 + y_1^2 + \abs{\check y}^6$.
In the smaller region $\abs{y} \leq \ell$, we assume that
\begin{align}
 \abs{\p^\gamma \tilde W(y,-\log \eps)} &\leq \eps^{\frac 18} \qquad  \mbox{ for } \abs{\gamma}=4
\label{eq:tilde:W:4:derivative}
\, ,
\end{align}
and at $y=0$, we have that
\begin{align}
\abs{\p^\gamma \tilde W(0,-\log \eps)} \leq \eps^{\frac 12- \frac{4}{2m-7}}
\qquad \mbox{ for }  \abs{\gamma}=3 \, .
\label{eq:tilde:W:3:derivative:0}
\end{align}
For $y$ in the region  $\{ \abs{y} \geq \LLL \} \cap \XXX_0$, we suppose that
\begin{subequations}
\begin{align}
 \eta^{-\frac{1}{6}}(y)\abs{W(y,-\log \eps)} &\leq  1+ \eps^ {\frac{1}{11}} 
 \label{eq:rio:de:caca:1}
\\
  \eta^{\frac 13}\left(y\right) \abs{\p_1 W(y,-\log \eps)} &\leq  1 +  \eps^{\frac{1}{12}} 
   \label{eq:rio:de:caca:2}
\\
\abs{\check \nabla W(y,-\log \eps)} &\leq   \tfrac 34 
 \label{eq:rio:de:caca:3}
\end{align}
\end{subequations}
while for the second derivatives of $W$, globally for all $y \in \XXX_0$ we shall assume that 
\begin{subequations}
\label{eq:W:gamma=2:total}
\begin{align}
\eta^{\frac 13}(y) \abs{\p^\gamma W(y,-\log \eps)} &\leq 1 \qquad  \mbox{ for }   \gamma_1 =1  \mbox{ and } \abs{\check\gamma}=1 
\label{eq:W:gamma=2:p1}
\\
\eta^{\frac 13}(y) \ppp^ {-\frac{1}{4}}(y,-\log \eps)  \abs{\p^\gamma W(y,-\log \eps)} &\leq 1 \qquad  \mbox{ for }   \gamma =(2,0,0) 
\label{eq:W:gamma=2:p2}
\\
\eta^{\frac 16}(y) \abs{\check\nabla^2 W(y,-\log \eps)} &\leq 1
\label{eq:W:gamma=2}
\,,
 \end{align}
 \end{subequations}
where $\ppp(y,-\log \eps) = \eta^{-1}(y) + \eps^3 \eta(y)$.

For the initial conditions of $Z$, $A$, and $K$, we    require that
\begin{align}
 \abs{\partial^{\gamma} Z(y,-\log \eps)} &
\leq  \begin{cases}
  \eps^{\frac 32},  &\mbox{if } \gamma_1\geq 1\mbox{ and } \abs{  \gamma}=1,2\\
  \eps , & \mbox{if } \gamma_1=0\mbox{ and } \abs{\check \gamma} =0,1,2
\end{cases} \,, \label{eq:Z_bootstrap:IC}\\
 \abs{\partial^{\gamma} A(y,-\log \eps)}&
\leq  \begin{cases}
  \eps^{\frac 32},  &\mbox{if } \gamma_1= 1\mbox{ and } \abs{\check \gamma}=0\\
  \eps , & \mbox{if } \gamma_1=0\mbox{ and } \abs{\check \gamma}=0,1,2 
\end{cases} \,, \label{eq:A_bootstrap:IC}\\
 \abs{\partial^{\gamma} K(y,-\log \eps)}&
\leq  \begin{cases}
  \eps^{2},  &\mbox{if } \gamma_1= 1 \mbox{ and } \abs{\check \gamma}=0,1\\
  \eps^{\frac{9}{4}}\eta^{- \frac{1}{15} }(y)    &\mbox{if } \gamma_1= 2\mbox{ and } \abs{\check \gamma}=0\\
   \eps , & \mbox{if } \gamma_1=0\mbox{ and }  \abs{\check \gamma}=0,1,2 
\end{cases} \,. \label{eq:S_bootstrap:IC}
\end{align}
Consequently, 
the initial specific vorticity in self-similar variables satisfies
\begin{align} 
\snorm{\Omega( \cdot ,-\log\eps) \cdot \Ncal_0}_{L^ \infty } \le \eps^ {\frac{1}{4}}   \ \ \text{ and } \ \ 
\snorm{ \Omega( \cdot ,-\log\eps) \cdot \Tcal_0^\nu}_{L^ \infty } \le 1\,,  \label{eq:svort:IC:SS}
\end{align}
and the initial scaled sound speed  satisfies
\begin{align} 
\snorm{S(y,-\log\eps) - {\tfrac{\kappa_0}{2}} }_{L^ \infty } \le \eps^{\frac{1}{7}}  \,.   \label{sigma-IC}
\end{align}

Lastly, for the Sobolev norm of the initial condition,  we suppose that for all $m \ge 18$,
\begin{align}
\eps \snorm{W( \cdot , -\log \eps)}_{\dot H^m}^2 + \snorm{Z( \cdot , -\log \eps)}_{\dot H^m}^2 + \snorm{A( \cdot , -\log \eps)}_{\dot H^m}^2
+ \snorm{K( \cdot , -\log \eps)}_{\dot H^m}^2  \leq  \eps  \,.
\label{eq:data:Hk}
\end{align}

\begin{lemma}[Initial datum suitable for vorticity creation]
\label{lem:IC:for:creation}
There exists initial datum $W(y,-\log\eps)$ with support in the set $\XXX_0$ defined in \eqref{eq:support-init}, which satisfies the bounds \eqref{ss-ic-1}--\eqref{eq:W:gamma=2:total}, and which additionally can be chosen to satisfy
\begin{align} 
 -  \tfrac12 \abs{y_1}^{-\frac 23}  \le  \p_1 W(y,-\log \eps) \le  - \tfrac 14 \abs{y_1}^{-\frac 23} 
\quad \text{ for } \quad \left\{ \eps^{-\frac{1}{10}}  \le \abs{y_1}  \le  2 \kappa_0 \eps^{-\frac{1}{2}}, \ 
\abs{\check y} \le \eps^ \frac{1}{3} \right\} .
\label{sunny-days}
\end{align} 
Moreover, associated to this choice of $W(y,-\log \eps)$,  letting $Z(y,-\log \eps) =0$ and $\phi_0=0$, there exists an $A(y, -\log\eps)$, such that 
\begin{align} 
\tilde u( \tilde x, -\eps)=U(y,-\log \eps) = \left( \tfrac{1}{2}(\eps^ {\frac{1}{2}}  W(y,-\log\eps) + \kappa_0), A_2(y,-\log\eps), A_3(y,-\log\eps) \right)  \label{vomey-vomey}
\end{align} 
is irrotational with respect to the physical space variable $\tilde x$.
\end{lemma} 
\begin{proof}[Proof of Lemma~\ref{lem:IC:for:creation}] 
The proof of \eqref{sunny-days} is based on the introduction of a cutoff functions in both the $y_1$ direction and in the $\check y$ directions, and 
the multiplication of the globally self-similar profile $\bar W$ by these cutoffs. The only non-trivial part of this argument is to choose the dependence of the aforementioned cutoffs on $\eps^{-1}$.

We start by defining a cutoff function with two parameters. For $b\geq 2a > 0$ we let $\upeta[a,b](r)$ be a smooth non-increasing function which is identically equal to $1$ for $r\in [0,a]$, and vanishes identically for $r \in [a+b,\infty)$. For the purposes of this lemma we may take the piecewise linear cutoff function and mollify it  with a compactly supported mollifier with characteristic length which is $\eps$-dependent. For example, we may mollify with a mollifier of compact support at scale $\eps^{\frac{1}{10}}$ the function which equals $1$ for $r \leq a+ \eps^{\frac{1}{10}}$, equals $0$ for $\ge a+b - \eps^{\frac{1}{10}}$, and is given by $1 - (r-a - \eps^{\frac{1}{10}})( b- 2 \eps^{\frac{1}{10}})^{-1}$ for $a< r< a+b$.
In particular, we may ensure that up to a constant factor of $\eps^{\frac{1}{10}}$ the derivative of $\upeta[a,b](r)$ is given by $- b^{-1}$ on the region $r\in (a,a+b)$, and vanishes outside of this region. Similarly, the second derivative of this cutoff function is bounded by a constant multiple of $b^{-1} \eps^{-\frac{1}{10}}$ on the region where it does not vanish.

Finally, we define the initial datum $W(y,-\log \eps)$ to be a cut-off version of $\bar W$, according to
\begin{align}  
\label{eq:IC:CHOICE}
W(y,-\log\eps) 
=   \bar W(y) \,  
\upeta\left[\eps^ {-\frac{1}{2}-\frac{1}{16}} ,  \eps^ {- \frac{3}{4}}\right](\abs{y_1})\,
\upeta\left[\eps^ {-\frac{1}{4}}, 100  \eps^ {-\frac{1}{4}}\right](\abs{\check y}) \,.
\end{align} 

A lengthy but routine computation which uses   properties of the explicit function $\bar W$ (see e.g.~\cite[Equation (2.48) and Remark~3.3]{BuShVi2019b}), shows that the function $W(y,-\log\eps)$ satisfies the conditions \eqref{ss-ic-1}--\eqref{eq:W:gamma=2:total}. We omit these details, but give the proof of condition~\ref{sunny-days} which is essential for the vorticity creation argument. We   note that for $  \abs{y_1} \le  2  \kappa_0 \eps^{-\frac{1}{2}}$ we have that  $\upeta [\eps^ {-\frac{1}{2}-\frac{1}{16}} ,  \eps^ {- \frac{3}{4}} ](\abs{y_1}) = 1$, and for $\abs{\check y}\leq  \eps^{\frac{1}{3}}$ we   have $\upeta [\eps^ {-\frac{1}{4}}, 100  \eps^ {-\frac{1}{4}} ](\abs{\check y})=1$. Thus, in the region relevant for \eqref{sunny-days}, by using \eqref{eq:barW:def} we have 
\begin{align}
\p_1 W(y,-\log \eps) =  \p_1 \bar W(y) = \tfrac{1}{1 + \abs{\check y}^{2}} W_{\rm 1D}'\left( \tfrac{y_1}{(1 + \abs{\check y}^{2})^{\frac 32} }\right)
\,.
\label{eq:underrated:computation}
\end{align}
The function $W_{\rm 1D}$ is explicit, and the Taylor series of its derivative around infinity is given by $W_{\rm 1D}'(r) = - \tfrac{1}{3}  r^{-\frac 23} 
- \tfrac{1}{9} r^{-\frac 43} + \OO(r^{-\frac 83})$. Using that we are interested in a region where $\abs{\check y} \leq  \eps^{\frac{1}{3}}$, and $ \eps^{-\frac{1}{10}}  \le \abs{y_1}  \le 2  \kappa_0 \eps^{-\frac{1}{2}}$, upon choosing $\eps$ sufficiently small (so that the Taylor series expansion around infinity is the relevant one), we immediately deduce that from \eqref{eq:underrated:computation} that 
\[
- \tfrac{ 1+ \eps^{\frac{1}{10}} }{3}  \abs{y_1}^{-\frac 23} 
\leq  
\tfrac{1}{1 + \abs{\check y}^{2}} W_{\rm 1D}'\left( \tfrac{y_1}{(1 + \abs{\check y}^{2})^{\frac 32} }\right)
\leq 
- \tfrac{ 1- \eps^{\frac{1}{10}} }{3} \abs{y_1}^{-\frac 23} 
\]
in the region of relevance to \eqref{sunny-days}. This establishes the existence of $W$ satisfying \eqref{sunny-days} as well as the bounds 
\eqref{ss-ic-1}--\eqref{eq:W:gamma=2:total}.

Next, for $W(y,-\log \eps)$ given by \eqref{eq:IC:CHOICE} and  with $Z(y, -\log\eps)=0$, we shall now prove the existence of an irrotational initial velocity field 
$\tilde u(\tilde x, \eps)$ satisfying \eqref{vomey-vomey}.

We first set $\phi_0=0$ so that $\Ncal_0=e_1$,  $\Tcal_0^\nu=e_\nu$, and $\Jcal_0=1$,
and $(\tilde x_1, \tilde x_\nu)   =(\epsilon^{-\frac{3}{2}} y_1, \epsilon^{-\frac{1}{2}} y_\nu)$.
We have that   $\tilde w_0 (\tilde x, -\eps) = \eps^ {\frac{1}{2}}  W(y,-\log\eps) + \kappa_0$, and
from \eqref{vomey-vomey}, we see that
$$
\tilde u (\tilde x ,\eps)_1= \tfrac{1}{2} \tilde w_0 \,.
$$
In order to ensure that $\tilde u_1 = \p_{\tilde x_1} \Psi$,
we  define
$$\Psi(\tilde x) = \tfrac{1}{2} \int_0^{\tilde x_1} \tilde w_0(\tilde x_1', \check{\tilde x}) d\tilde x'_1  -\tfrac{1}{2} \int_0^{\infty } \tilde w_0(\tilde x_1', \check{\tilde x}) d\tilde x'_1    $$
for $\tilde x_1 >0$ and then extend $\Psi(\tilde x) $ as an even function in $\tilde x_1$.
We now define
\begin{align} 
\tilde a_\nu(\tilde x, -\eps) = \p_{\tilde x_\nu} \Psi (\tilde x)\,,  \label{i-dont-know}
\end{align} 
so that  $ \tilde u (\tilde x, -\eps) = \nabla_{\! \tilde x} \Psi(\tilde x)$, which implies that $ \operatorname{curl} _{ \tilde x} \tilde u(x, -\eps) =0$.
We write \eqref{i-dont-know} in self-similar coordinates as 
$$
A_\nu(y, -\log \eps) = -\tfrac{1}{2} \eps^ {\frac{3}{2}}  \int_{y_1}^{\infty }   \p_{\nu} W(y_1',\check y, -\log \eps) dy_1' \,.
$$
Using the definition of $W(y,-\log\eps)$ given in \eqref{eq:IC:CHOICE}, a lengthy computation shows that 
$A(y, -\log\eps)$ satisfies the 
bounds \eqref{eq:A_bootstrap:IC} and \eqref{eq:data:Hk}.
\end{proof}

\subsection{Statement of the main theorem in self-similar variables and asymptotic stability}

\begin{theorem}[Stability and shock formation via self-similar variables] 
\label{thm:main:S-S}
For $ \alpha = {\tfrac{\gamma-1}{2}}$ and $\gamma>1$, 
let $\kappa_0 = \kappa_0(\alpha)>1$ be chosen sufficiently large. 
Suppose that at initial time $s=-\log \eps$, the initial data $(W_0,Z_0,A_0,K_0) = (W,Z,A,K)|_{s=-\log\eps}$ are supported in the set $\XXX_0$ from \eqref{eq:support-init}, and obey conditions \eqref{eq:fat:cat}--\eqref{eq:data:Hk}.  Assume that the modulation functions have initial conditions
compatible with \eqref{mod-ic}--\eqref{ic-kappa0-phi0}.

There exist  $M=M(\alpha, \kappa_0)\ge 1$  sufficiently large, $\eps = \eps(\alpha,\kappa_0,M) \in (0,1)$  sufficiently small, and unique global-in-time solutions $(W,Z,A,K)$ to \eqref{euler-ss} with the following properties. $(W,Z,A,K)$ are supported in the time-dependent cylinder $\XXX(s)$  defined in \eqref{eq:support},
$$(W,Z,A,K) \in C([-\log\eps, + \infty ); H^m)\cap  C^1([-\log\eps, + \infty ); H^{m-1}) \ \text{  for  } m\ge 18\,,$$
and 
\begin{align*}
 \snorm{W( \cdot , s)}_{\dot H^m}^2\!+e^s\snorm{Z( \cdot , s)}_{\dot H^m}^2 \!+ e^s\snorm{A( \cdot , s)}_{\dot H^m}^2 \! + e^s\snorm{K( \cdot , s)}_{\dot H^m}^2
 &\!\leq\!  16 \kappa_0^2 \lambda ^{-m}e^{-s-\log \eps}\! +\! (1\! - \! e^{-s}\eps^{-1} ) M^{4m}  
\end{align*}
for a constant $\lambda=\lambda(m) \in (0,1)$.
The Riemann function $W(y,s)$ remains   close to the generic and stable self-similar blowup profile $\bar W$; upon defining the weight function
$\eta(y) = 1 + y_1^2 + \abs{\check y}^6$, we have that 
the perturbation $\tilde W = W - \bar W$ satisfies
$$ \abs{ \tilde W(y,s)} \leq \eps^{\frac{1}{11}} \eta^{\frac 16}(y) \,, \ \ 
\abs{\p_1 \tilde W(y,s)} \leq \eps^{\frac{1}{12}} \eta^{-\frac 13}(y)\,, \ \ 
\abs{\check \nabla \tilde W(y,s)} \leq \eps^{\frac{1}{13}} \,, $$
for all $\abs{y} \le \eps^{-\frac{1}{10}}$ and $s \ge -\log\eps$. Furthermore, $\p^\gamma \tilde W(0,s) =0$ for all $\abs{\gamma} \le 2$, and the bounds \eqref{eq:tildeW_decay2} and \eqref{eq:bootstrap:Wtilde3:at:0} hold. Additionally,   $W(y,s)$ satisfies the bounds given in \eqref{eq:W_decay} and \eqref{eq:W:GN}.

As $s\to \infty$, $W(y,s)$ converges to an asymptotic profile $\bar W_{\mathcal A}(y)$ which satisfies:
\begin{itemize}[itemsep=2pt,parsep=2pt,leftmargin=.15in]
\item $\bar W_{\mathcal A}$ is a  $C^\infty$ smooth solution to the self-similar 3D Burgers equation~\eqref{eq:Burgers:self:similar}.
\item $\bar W_{\mathcal A}(y)$ obeys the  genericity condition  \eqref{genericity}.
\item $\bar W_{\mathcal A}$ is 
uniquely determined by the $10$ parameters ${\mathcal A}_\alpha = \lim_{s\to\infty} \p^\alpha W(0,s)$ for $\abs{\alpha}=3$.
\end{itemize}

The amplitude of the functions $Z$, $A$, and $K$ remains $\OO(\eps)$ for all $s \ge -\log\eps$, while for each $ \abs{\gamma}\le m$, 
$\p^\gamma Z( \cdot , s) \to 0$,   $\p^\gamma A( \cdot , s) \to 0$, and $\p^\gamma K( \cdot , s) \to 0$ as $s \to + \infty $, and $Z$ and $A$ satisfy the 
bounds \eqref{eq:Z_bootstrap},  \eqref{eq:A_bootstrap},  \eqref{eq:S_bootstrap}.

The scaled sound speed $S(y,s)$ satisfies
$$
\snorm{S( \cdot , s) - \tfrac{\kappa_0}{2} }_{ L^ \infty } \le \eps^ {\frac{1}{8}}  \ \ \text{ for all } \ s \ge -\log\eps  \,.
$$
The specific vorticity $\Omega(y,s) = \mathring \zeta(x,t)$ satisfies for all $s\ge -\log\eps$,
$$
\snorm{\Omega \circ \pay (\cdot ,s) -\Omega(\cdot ,-\log\eps) }_{L^\infty}  \le \eps^ {\frac{1}{20}} 
$$
where $\pay$ is defined in \eqref{phi-flow}.  Furthermore, there exists irrotational initial data from which vorticity is instantaneously created and
remains nonzero in a neighborhood of the shock location $(0,T_*)$: see Theorem \ref{thm:vorticity:creation} for details.
\end{theorem}

For concision, the initial data was assumed to have the support property \eqref{eq:support-init} 
and satisfy the conditions \eqref{eq:fat:cat}.  By using the symmetries of the Euler equations, we can generalize these conditions to allow for data in
 a non-trivial  open set in the
$H^m$ topology.    
\begin{theorem}[Open set of initial conditions]
\label{thm:open:set:IC} 
Let $\tilde{\mathcal F}$ denote the set of initial data satisfying the hypothesis of Theorem~\ref{thm:main:S-S}.
There exists an open neighborhood of $\tilde{\mathcal F}$ in the $H^m$ topology, denoted by ${\mathcal F}$, such that for any initial data to the Euler equations taken from ${\mathcal F}$, the  conclusions of Theorem \ref{thm:main:S-S} hold. 
\end{theorem}

\subsection{Shock formation in physical variables $(\xcal, t)$}
We shall now interpret the assumptions and results of Theorem \ref{thm:main:S-S} in the  context of physical variables $(\xcal, t)$.
The function $\tilde w_0(\xcal) = w(x, -\eps) = \eps^ {\frac{1}{2}} W(y,-\log \eps)+\kappa_0$ is chosen such that 
the minimum (negative) slope of $\tilde w_0$ occurs in the $e_1$ direction, and 
$\partial_{\xcal_1} \tilde w_0$ attains its global minimum at $ \xcal=0$, and from \eqref{eq:fat:cat}, satisfies
\begin{align} 
\tilde w_0(0) = \kappa_0\,,\qquad \partial_{\xcal_1} \tilde w_0(0) = -{\tfrac{1}{\eps}} \,, \qquad  \check\nabla_{\xcal} \tilde w_0(0) =0 \,, \qquad
\nabla_{\!\xcal} \p_{\xcal_1} \tilde w_0 (0)=0\,. \label{vomey1}
\end{align} 
Of course, there are a number of additional conditions on $\tilde w_0(\xcal)$ and its partial derivatives which exactly correspond to conditions 
\eqref{ss-ic-1}--\eqref{eq:W:gamma=2:total} by the change of variables \eqref{eq:y:s:def}, but the conditions \eqref{vomey1} are fundamental to the
stable self-similar point shock formation process.

We shall assume that
the support of the initial data  $(\tilde w_0 - \kappa_0, \tilde z_0, \tilde a_0)$, is  contained in the  set  
$\XX_0= \{\abs{\xcal_1}\leq  \tfrac{1}{2}  \eps^ {\frac{1}{2}}  ,\abs{\check \xcal}\leq  \eps^ {\frac{1}{6}}  \}$,
which in turn shows that
 $u_0 \cdot \Ncal_0 - \frac{\kappa_0}{2}$, $\sigma_0 - \frac{\kappa_0}{2}$, and $u_0 \cdot \Tcal^\nu$ are compactly 
supported in $\XX_0$. In view of the coordinate transformation \eqref{ic-cov} and the bound \eqref{ic-kappa0-phi0}, the functions of $x$ defined in \eqref{eq:tilde:wza:0}, 
namely $(  w_0,   z_0,   a_0, \scal_0)$, have spatial support contained in the set  $\{\abs{x_1}\leq   \tfrac{1}{2} \eps^ {\frac{1}{2}}  
+ \eps  ,\abs{\check x}\leq  \eps^ {\frac{1}{6}} \} \subset \{ \abs{x_1}\leq   \eps^ {\frac{1}{2}}   ,\abs{\check x}\leq  \eps^ {\frac{1}{6}} \}$.  This larger set 
corresponds to the support condition \eqref{eq:support-init}
under the transformation \eqref{eq:y:s:def}.

For the initial conditions of $\tilde z_0$, $\tilde a_0$, and $\scal_0$, from \eqref{eq:Z_bootstrap:IC}--\eqref{eq:S_bootstrap:IC}, we have that\footnote{The bound for  $\p_{\xcal_1} a_0$ can be replaced by a bound that depends on $\kappa_0$, thus permitting arbitrarily large initial
vorticity.} 
\begin{alignat*}{2}
&\abs{\tilde z_0(\xcal)} \le \eps\,, \qquad  \abs{\p_{\xcal_1}\tilde z_0(\xcal)} \le 1 \,, \qquad &&\abs{ \check\nabla_{\!\xcal } \tilde z_0(\xcal)} \le   \eps^ {\frac{1}{2}} \,, \ \\
&\abs{\tilde a_0(\xcal)} \le \eps\,, \qquad  \abs{\p_{\xcal_1}\tilde a_0(\xcal)} \le 1 \,, \qquad &&\abs{ \check\nabla_{\!\xcal } \tilde a_0(\xcal)} \le  \eps^ {\frac{1}{2}}
    \,,   \\
&\abs{\tilde \scal_0(\xcal)} \le \eps\,, \qquad  \abs{\p_{\xcal_1}\tilde \scal_0(\xcal)} \le \eps^ {\frac{1}{2}} \,, \qquad &&\abs{ \check\nabla_{\!\xcal } \tilde \scal_0(\xcal)} \le  \eps^ {\frac{1}{2}}
    \,,
\end{alignat*} 
together with conditions on higher-order derivatives\footnote{We deduce from \eqref{eq:data:Hk} that at $t=-\eps$, the Sobolev norm
must satisfy
$\sum_{\abs{\gamma}=m} 
\eps^2 \snorm{ \p_x^\gamma   w_0}_{L^2}^2 + \snorm{ \p_x^\gamma   z_0}_{L^2}^2 + \snorm{ \p_x^\gamma   a_0}_{L^2}^2  
 + \snorm{ \p_x^\gamma   \scal_0}_{L^2}^2   \le 
\eps^{ {\frac{7}{2}}-(3\gamma_1+\abs{\check\gamma} )}  $.
See (3.21)--(3.22) in \cite{BuShVi2019b} for details.} that follow \eqref{eq:Z_bootstrap:IC}--\eqref{eq:S_bootstrap:IC} and \eqref{eq:data:Hk}.

The initial specific vorticity $\tilde \zeta(\tilde x,-\eps) = \mathring \zeta(x,-\eps) = \Omega(y,-\log\eps)$ satisfies condition \eqref{eq:svort:IC:SS},
and the initial scale sound speed $\tilde\sigma(\tilde x, -\eps) = \mathring \sigma(x,-\eps) = S(y,-\log\eps)$ satisfies \eqref{sigma-IC}.

We now summarize the statement of Theorem \ref{thm:main:S-S} in the physical variables.  Suppose that the initial data 
$\tilde w_0$, $\tilde z_0$, $\tilde a_0$, and $\scal_0$ satisfy the conditions stated above and that
 $ \alpha = {\tfrac{\gamma-1}{2}} >0$ is fixed.   There exist a sufficiently large $\kappa_0 = \kappa_0(\alpha) > 1$,  and a sufficiently small 
$\eps = \eps(\alpha,\kappa_0) \in (0,1)$ such that 
there exists a time $T_* = \OO(\eps^2)$ and a unique solution $(u,\rho,\scal) \in C([-\eps,T_*);H^m)\cap C^1([-\eps,T_*);H^{m-1})$ to  \eqref{eq:Euler} which blows up in an
 asymptotically self-similar fashion  at time $T_*$, at a single point  $\xi_* \in \mathbb{R}^3  $. 
In particular, the following results hold:
\begin{enumerate}[itemsep=2pt,parsep=2pt,leftmargin=.3in]
\renewcommand{\theenumi}{(\roman{enumi})}
\renewcommand{\labelenumi}{\theenumi}
\item The blowup time $T_* = \OO(\eps^2)$ and the blowup location  $\xi_*  = \OO( \epsilon )$ are explicitly computable, with $T_*$ defined by the condition
$\int_{-\eps}^{T_*} (1-\dot \tau(t)) dt = \eps $ and with the blowup location given by
$\xi_*  = \lim_{t\to T_*} \xi(t)$.  The amplitude modulation function satisfies $\abs{\kappa_*-\kappa_0} = \OO(\eps^{\frac 32})$ where $\kappa_*  = \lim_{t\to T_*} \kappa(t)$.

\item For each $t \in [-\eps, T_*)$, we have $
\sabs{\Ncal(\check {\tilde x},t) - \Ncal_0(\check \xcal)} +\sabs{\Tcal^\nu(\check {\tilde x},t) - \Tcal_0^\nu(\check \xcal)}  = \OO(\eps) \,.$

\item We have $\sup_{t\in[-\eps, T_*)} \left( \norm{\tilde u \cdot \Ncal - {\tfrac{1}{2}} \kappa_0 }_{ L^\infty }+\norm{\tilde u \cdot \Tcal^\nu}_{ L^\infty }+
\norm{\tilde \sigma - \tfrac{1}{2} \kappa_0}_{ L^\infty } +\| \zeta\|_{ L^\infty } \right) \les 1$.
\item There holds $\lim_{t \to T_*} \Ncal \cdot \nabla_{\! \tilde x} \tilde w(\xi(t),t) = -\infty $ and  $\frac{1}{2(T_*-t)} \leq  \norm{\Ncal \cdot \nabla_{\! \tilde x} \tilde w(\cdot,t)}_{L^\infty} \leq \frac{2}{T_*-t}$ as  $t \to T_*$.
\item At the time of blowup, $\tilde w( \cdot , T_*)$ has a cusp-type singularity with  $C^ {\sfrac{1}{3}} $ H\"{o}lder regularity.
\item Only the $\p_\Ncal$ derivative of $\tilde u \cdot \Ncal$ and $\tilde \rho$ blowup, while the other first order derivatives remain  bounded:
\begin{subequations} 
\label{lagavulin16}
\begin{align} 
&\quad \lim_{t\to T_*} \Ncal \cdot \nabla_{\! \tilde x} (\tilde u \cdot \Ncal) (\xi(t), t) = \lim_{t\to T_*} \Ncal \cdot \nabla_{\! \tilde x} \tilde \rho ( \xi(t), t) = -\infty \,, \\
&  \sup_{t\in[-\eps,T_*)}  \| \Tcal^\nu \cdot \nabla_{\! \tilde x}\tilde \rho(\cdot , t)\|_{ L^ \infty } +  \| \Tcal^\nu \cdot \nabla_{\! \tilde x} \tilde u(\cdot , t)\|_{ L^ \infty } + 
\|\Ncal \cdot \nabla_{\! \tilde x} (\tilde u \cdot \Tcal^\nu)(\cdot , t)\|_{ L^ \infty}  \les 1 \,.
\label{damn-thats-nice}
\end{align} 
\end{subequations} 
\item Both $\tilde \scal$ and $ \nabla _{\! \tilde x} \tilde \scal$ remain bounded:
\begin{align}
\sup_{t\in[-\eps,T_*)} \snorm{\tilde \scal(\cdot,t)}_{L^\infty} + \snorm{ \nabla _{\! \tilde x} \tilde \scal (\cdot,t)}_{L^\infty} \les \eps^{\frac 18}
\,.
\label{eq:scal:Linfinity}
\end{align}
\item Let $\p_t X(\xcal, t) = u(X(\xcal, t), t)$ with $X(\xcal, -\eps)=\xcal$ so that $X(\xcal, t)$ is the Lagrangian flow. Then there exists constants $c_1$, $c_2$ such that
$ c_1 \le \abs{ \nabla _{\! \xcal} X(\xcal, t) } \le c_2$ for all $t\in [-\eps, T_*)$.
\item The scaled sound $\tilde \sigma$ remains uniformly bounded from below and satisfies 
$$
\snorm{ \tilde \sigma ( \cdot , t) -   \tfrac{ \kappa_0 }{2}}_{ L^ \infty } \le \eps^ {\sfrac{1}{8}}  \ \ \text{ for all } \ t\in[-\eps, T_*]  \,.
$$
\item The vorticity satisfies $\snorm{\omega(\cdot, t) }_{L^\infty}  \le C_0 \snorm{\omega(\cdot,-\eps)}_{L^\infty}$   for all 
$t\in[-\eps, T_*] $ for a universal constant $C_0$, and if 
$\abs{\omega(\cdot, -\eps)} \geq c_0 >0$ on the set  $B(0,2 \eps^{\sfrac 34})$ then at  the blowup location $\xi_*$ there is  nontrivial vorticity, and moreover
\begin{align*}
\abs{\omega(\cdot, T_*)} \geq \tfrac{c_0}{C_0} \qquad \mbox{on the set} \qquad    B(0, \eps^{\sfrac 34})  \,.
\end{align*}
\end{enumerate}

\section{Bootstrap assumptions}
\label{sec:bootstrap}

As discussed above, the proof of Theorem~\ref{thm:main:S-S} consists of a bootstrap argument, which we make precise in this section.  For $M$ sufficiently 
large, depending on $\kappa_0$ and on $\alpha$, and for $\eps$ sufficiently small, depending on $M$, $\kappa_0$, and $\alpha$, we postulate that the 
modulation functions are bounded as in \eqref{mod-boot}, that $(W,Z,A,K)$ are supported in the set given by \eqref{eq:support}, that $W$ satisfies  
\eqref{eq:W_decay},  $\tilde W$ obeys \eqref{eq:tildeW_decay}--\eqref{eq:bootstrap:Wtilde3:at:0}, and that  $Z$, $A$, and $K$ are bounded as in \eqref{eq:Z_bootstrap}--\eqref{eq:S_bootstrap}. All these bounds have explicit constants in them. 
In the subsequent sections of the paper, we prove  that the these estimates in fact hold with strictly better pre-factors, which in view of a continuation argument yields the proof of Theorem~\ref{thm:main:S-S}.

\subsection{Dynamic variables}
For the dynamic modulation variables,  we assume that
\begin{subequations}
\label{mod-boot}
\begin{align}
&\tfrac 12 \kappa_0 \leq  \kappa(t)  \leq  2 \kappa_0, &&\abs{\tau(t)} \leq M \eps^2, &&&\abs{\xi(t)} \leq  M^{\frac 14}  \eps   , &&&& \abs{\check n(t)} \leq   M^2\eps^{\frac 32},    &&&&& \abs{\phi(t)} \leq M^2 \eps, 
\label{eq:speed:bound} \\
&\abs{\dot \kappa(t)} \leq  e^{-\frac{3s}{10}} , &&\abs{\dot \tau(t)} \leq  M e^{-s}, &&&|\dot \xi(t)| \leq  M^{\frac 14} , &&&& |\dot{\check{n}}(t)| \leq M^2 \eps^{\frac 12}, &&&&& |\dot{\phi}(t)| \leq M^2 ,
\label{eq:acceleration:bound}
\end{align}
\end{subequations}
for all $-\eps \leq t < T_*$.

From \eqref{tildev} and (A.4)--(A.5) in \cite{BuShVi2019b}, and the bootstrap assumptions \eqref{mod-boot}, we   obtain that
\begin{align}
|\dot Q(t)|  \leq 2 M^2 \eps^{\frac 12} \,.
\label{eq:dot:Q}
\end{align}
Also, from the $\dot \tau$ estimate in \eqref{eq:acceleration:bound},  we obtain 
\begin{align}
\abs{1-\beta_\tau} = \frac{\abs{\dot \tau}}{1-\dot\tau} \leq 2 M  e^{-s} \leq 2 M  \eps
\label{eq:beta:tau} 
\end{align}
upon taking $\eps$  sufficiently small.

\subsection{Spatial support bootstrap}
\label{subsec:support}
We shall assume that $(W,Z,A)$ have support in the set
\begin{align}
\XXX(s):=\left\{\abs{y_1}\leq  2\eps^{\frac12}  e^{\frac 32s},\abs{\check y}\leq  2 \eps^ {\frac{1}{6}} e^{\frac{s}{2}}  \right\} \qquad \text{ for all } \qquad s \ge -\log\eps \,.
\label{eq:support}
\end{align}
We introduce the weights 
\begin{align*}
\eta(y) = 1 + y_1^2 + \abs{\check y}^6 \qquad \mbox{and} \qquad \tilde \eta(y) = \eta(y) +  \abs{\check y}^2  \, ,  \label{eq:eta:def}
\end{align*}
as well as the $s$-dependent weight function
\begin{align*} 
\ppp(y,s) =  \tfrac{1}{\eta(y)}  + e^{-3s} \eta(y)    \,.
\end{align*} 
For $y \in \XXX(s)$, we note that  
\begin{equation}\label{e:space_time_conv}
\eta(y) \leq 40 \eps e^{3s} \qquad \Leftrightarrow \qquad \eta^{\frac{1}{3}}(y) \leq 4 \eps^{\frac 13}  e^{s}
 \end{equation}
for all $y \in \RR^3$. 
Since $\eta \ppp  = 1+ e^{-3s} \eta^2$, we have $e^{-3s} \eta^2 \le \eta\ppp$, and thus 
\begin{equation}\label{eq:phi:interp}
e^{-s}\les \ppp^{\frac{q}3}\eta^{-\frac13(2-q)} 
\end{equation}
holds for $1<q\le 2$.

\subsection{$W$ bootstrap}
The bootstrap assumptions on $W$ and its derivatives are
\begin{align}
 \abs{\partial^{\gamma} W(y,s)}&
\leq   \begin{cases}
(1+ \eps^{\frac{1}{20}}) \eta^{\frac 16}, & \mbox{if } \abs{\gamma}  = 0\,,  \\
\tilde \eta^{-\frac 13}\left(\tfrac y2\right) {\bf 1}_{\abs{y}\leq \LLL} + 2 \eta^{-\frac 13}(y) {\bf 1}_{\abs{y}\geq \LLL}, & \mbox{if }  \gamma_1 = 1 \mbox{ and } \abs{\check \gamma}=0\,, \\
1, & \mbox{if } \gamma_1=0 \mbox{ and }  \abs{\check \gamma} = 1\,,\\
 M^{\frac{2}{3}} \eta^{-\frac 13}, & \mbox{if }  \gamma_1 = 1 \mbox{ and } \abs{\check\gamma}=1\,, \\
M^{\frac{1}{3}}  \eta^{-\frac 13} \ppp^{\frac{1}{4}}  , & \mbox{if }  \gamma_1 = 2 \mbox{ and } \abs{\check\gamma}=0\,, \\
M \eta^{-\frac{1}{6}}, & \mbox{if } \gamma_1 = 0 \mbox{ and } \abs{\check \gamma} = 2 \,.
\end{cases}\label{eq:W_decay}
\end{align}
Next, for $\abs{y}\leq \LLL$, we assume that\footnote{While the first three bounds stated in \eqref{eq:W_decay} follow directly from the   properties of 
$\bar W$ stated in (2.48) of  \cite{BuShVi2019b}, and those of $\tilde W$ in \eqref{eq:tildeW_decay},  the estimate for 
$\p_1 W$ makes use of the fact that $ \tilde \eta^{-\frac 13}(y) + \eps^{\frac{1}{12}} \eta^{-\frac 13}(y) \leq \tilde \eta^{-1/3}(y/2)$. }
\begin{subequations}
\label{eq:tildeW_decay}
\begin{align}
\abs{ \tilde W(y,s)} &\leq \eps^{\frac{1}{11}} \eta^{\frac 16}(y) \,,
\label{eq:bootstrap:Wtilde} \\
\abs{\p_1 \tilde W(y,s)} &\leq \eps^{\frac{1}{12}} \eta^{-\frac 13}(y) \,,
\label{eq:bootstrap:Wtilde1} \\
\abs{\check \nabla \tilde W(y,s)} &\leq \eps^{\frac{1}{13}}  \,,
\label{eq:bootstrap:Wtilde2} 
\end{align}
\end{subequations}
where $\LLL$ is defined as  in \eqref{eq:pounds}. 
Furthermore, for $\abs{y} \leq \ell$ (as defined in \eqref{eq:ell:choice})  we assume that
\begin{subequations}
\label{eq:tildeW_decay2}
\begin{align}
\abs{\p^\gamma \tilde W(y,s)} &\leq (\log M)^4 \eps^{\frac {1}{10}}\abs{y}^{4-\abs{\gamma}}+M\eps^{\frac 14} \abs{y}^{3-\abs{\gamma}} \leq 2 (\log M)^4 \eps^{\frac{1}{10}}\ell^{4-\abs{\gamma}}\,, && \mbox{for all }   \abs{\gamma} \leq 3\,,
\label{eq:bootstrap:Wtilde:near:0} \\
\abs{\p^\gamma \tilde W(y,s)} &\leq \eps^{\frac{1}{10}} (\log M)^{\abs{\check \gamma}}\,, && \mbox{for all }    \abs{\gamma} = 4\,,
\label{eq:bootstrap:Wtilde4}
\end{align}
\end{subequations}
while at $y=0$,  we assume that
\begin{align}
\abs{\p^\gamma \tilde W(0,s)} &\leq \eps^{\frac14} \,, && \mbox{for all} \qquad \abs{\gamma} = 3\,,\label{eq:bootstrap:Wtilde3:at:0}
\end{align} 
for all $s\geq -\log \eps$.

\begin{lemma}[Lower bound for $\Jcal \p_1 W$]\label{lem:Jp1W}
\begin{align} 
\Jcal \p_1 W(y,s) \geq -1 \quad \mbox{and} \quad \Jcal \p_1 \bar W(y,s) \geq -1 \ \ \text{ for all } y \in \mathbb{R}^3 \,, s \ge -\log \eps    \,. \label{Jp1W}
\end{align} 
\end{lemma} 
The proof of this lemma is given in the proof of Lemma 4.2 in \cite{BuShVi2019b}.

\subsection{$Z$ and $A$ bootstrap}
The bootstrap assumptions on $Z$, $A$, $K$, and their derivatives are:
\begin{align}
 \abs{\partial^{\gamma} Z(y,s)} &
\leq  \begin{cases}
 M^{\frac{1+\abs{\check\gamma}}{2}}  e^{-\frac 32 s },  &\mbox{if } \gamma_1\geq 1\mbox{ and } \abs{  \gamma}=1,2\\
M \eps^{ \frac{2-\abs{\check \gamma}}2 } e^{-\frac {\abs{\check \gamma}}2 s}, & \mbox{if } \gamma_1=0\mbox{ and } \abs{\check \gamma} =0,1,2 \, ,
\end{cases}\label{eq:Z_bootstrap}\\
 \abs{\partial^{\gamma} A(y,s)}&
\leq  \begin{cases}
M e^{-\frac 32 s },  &\mbox{if } \gamma_1= 1\mbox{ and } \abs{\check \gamma}=0\\
M \eps^{ \frac{2-\abs{\check \gamma}}2 } e^{-\frac {\abs{\check \gamma}}2 s}, & \mbox{if } \gamma_1=0\mbox{ and } \abs{\check \gamma}=0,1,2 \, ,
\end{cases}\label{eq:A_bootstrap} \\
 \abs{\partial^{\gamma} K(y,s)}&
\leq  \begin{cases} \eps^ {\frac{1}{4}}  e^{-\frac 32 s },  &\mbox{if } \gamma_1= 1\mbox{ and } \abs{\check \gamma}=0 \\
\eps^ \frac{1}{8}   e^{-\frac{13}{8} s } ,  &\mbox{if } \gamma_1= 1\mbox{ and } \abs{\check \gamma}=1 \\
\eps^ \frac{1}{8}  e^{-2 s } \eta^{- \frac{1}{15} }(y) ,  &\mbox{if }  \gamma_1= 2\mbox{ and } \abs{\check \gamma}=0\\
\eps^{ \frac{1}8 } e^{-\frac {\abs{\check \gamma}}2 s }, & \mbox{if } \gamma_1=0\mbox{ and } \abs{\check \gamma}= 1,2 \,.
\end{cases}\label{eq:S_bootstrap}
\end{align}

\begin{remark}
Since $K$ satisfies a transport equation, the pointwise bound 
\begin{align}
\label{eq:S_bootstrap:*}
\abs{K(y,s)} \leq \eps 
\end{align}
follows directly from the initial datum assumption \eqref{eq:S_bootstrap:IC}.
\end{remark}

\subsection{Further consequences of the bootstrap assumptions}
\label{subsection:energy}
The bootstrap bounds  \eqref{mod-boot}, \eqref{e:space_time_conv}, \eqref{eq:W_decay}--\eqref{eq:bootstrap:Wtilde3:at:0}, \eqref{eq:Z_bootstrap}, and \eqref{eq:A_bootstrap} have a number of consequences, which we collect here for future reference. The first is a global-in-time $L^2$-based Sobolev estimate:
\begin{proposition}[$\dot H^m$ estimate for $W$, $Z$, and $A$]\label{cor:L2}
For  integers  $m \geq 18$ and for a constant $\lambda = \lambda(m)$, 
\begin{subequations} 
\begin{align}
\snorm{Z( \cdot , s)}_{\dot H^m}^2 + \snorm{A( \cdot , s)}_{\dot H^m}^2  + \snorm{K( \cdot , s)}_{\dot H^m}^2
 &\leq  16 \kappa_0^2 \lambda ^{-m}\eps^{-1} e^{-2s} +e^{-s} (1 -  e^{-s}\eps^{-1} ) M^{4m}   \,, \label{eq:AZ-L2} \\
 \snorm{W( \cdot , s)}_{\dot H^m}^2
 &\leq 16 \kappa_0^2 \lambda^{-m} \eps^{-1} e^{-s} +(1 -  e^{-s}\eps^{-1}) M^{4m} \,, \label{eq:W-L2} 
\end{align}
\end{subequations} 
for all $s\ge  -\log \eps$.
\end{proposition}
The proof of Proposition~\ref{cor:L2}, which will be given  at the end of Section \ref{sec:energy},  relies only upon the initial data assumption \eqref{eq:data:Hk}, on the support bound \eqref{e:space_time_conv},  on  $L^ \infty $ estimates for $\p^\gamma W$, $\p^\gamma Z$, and $\p^\gamma K$ when 
$\abs{\gamma} \le 2$, on $\p^\gamma A$ pointwise bounds for $\abs{\gamma}\leq 1$, and on $\check \nabla^2 A$ bounds. That is, Proposition~\ref{cor:L2} follows directly from \eqref{eq:data:Hk} and the bootstrap assumptions \eqref{mod-boot}, \eqref{e:space_time_conv},  \eqref{eq:W_decay}, \eqref{eq:Z_bootstrap},  and \eqref{eq:A_bootstrap}.  

The reason we state Proposition~\ref{cor:L2} at this stage of the analysis  is that the $\dot{H}^m$ estimates and linear interpolation yield useful information for higher order 
derivatives of $(W,Z,A,K)$, which are  needed in order to close the bootstrap assumptions for high order derivatives. These bounds are summarized as:
\begin{lemma}\label{eq:higher:order:ests}
For integers $m \ge 18$, we have that
\begin{align}\label{eq:A:higher:order}
 \abs{\partial^{\gamma} A(y,s)}&
\les  \begin{cases}
 e^{-( \frac 32  - \frac{2\abs{\gamma}-1}{2m-5})s} ,  &\mbox{if } \gamma_1\geq 1\mbox{ and } \abs{ \gamma}=2,3\\
  e^{- (1 - \frac{\abs{\gamma}-1}{2m-7})s}    , & \mbox{if } \abs{\gamma}=3,4,5 \, ,
\end{cases} 
\end{align}

\begin{align}\label{eq:Z:higher:order}
 \abs{\partial^{\gamma} Z(y,s)}&
\les  \begin{cases}
 e^{- (\frac32 -\frac{3}{2m-7})s} ,  &\mbox{if } \gamma_1\geq 1\mbox{ and } \abs{ \gamma}=3\\
 e^{- (1 - \frac{\abs{\gamma}-1}{2m-7})s}    , & \mbox{if } \abs{\gamma}=3,4,5 \, ,
\end{cases} 
\end{align}

\begin{align}
 \abs{\partial^{\gamma} W(y,s)}&
\les  \begin{cases}
 e^{\frac{2s}{2m-7}} \eta^{-\frac 13} ,  &\mbox{if } \gamma_1 =1 \mbox{ and } \abs{\check \gamma}=2\\
 e^{\frac{s}{2m-7}} \eta^{-\frac 16} ,  &\mbox{if } \gamma_1=0\mbox{ and } \abs{\check \gamma}=3  \\
e^{\frac{3s}{2m-7}}  \eta^{-\frac 13}\ppp^ {\frac{1}{4}}   ,  &\mbox{if } \gamma_1\ge 2\mbox{ and } \abs{ \gamma}=3  \,,
\end{cases} \label{eq:W:GN}
\end{align} 

\begin{align}
\label{eq:K:higher:order}
 \abs{\partial^{\gamma} K(y,s)}&
\les  \begin{cases}
e^{-(\frac{13}{8}  - \frac{9}{4(2 m-7 )})s}  ,  &\mbox{if } \gamma_1= 1\mbox{ and } \abs{\check \gamma}=2\\
 e^{-(2-{\frac{4}{2m-7}} )s}  \eta^{ -\frac{1}{15} },  &\mbox{if }   \gamma_1\ge 2\mbox{ and } \abs{ \gamma}=3\\
e^{- (1 - \frac{\abs{\gamma}-2}{2m-7})s}    , & \mbox{if } \abs{\gamma}=3,4,5 \, ,
\end{cases} 
\end{align}
\end{lemma}
\begin{proof}[Proof of Lemma~\ref{eq:higher:order:ests}]
The bounds for \eqref{eq:A:higher:order} and \eqref{eq:Z:higher:order}, as well as the first two estimates in \eqref{eq:W:GN}
are proven in Lemma 4.4 in \cite{BuShVi2019b}.

We then consider the third estimate in \eqref{eq:W:GN} and hence  estimate $\partial^{\gamma} W(y,s)$ for the case $\gamma_1\geq 2 $ and $\abs{\gamma}=3$.  We write
\begin{align*}
\eta^{\frac{1}{3}} \ppp^{-{\frac{1}{4}} } \nabla \p_{11}W
=\underbrace{\nabla  \left( \eta^{\frac{1}{3}} \ppp^{-{\frac{1}{4}} }  \p_{11}W\right)}_{=:I}
-\underbrace{\nabla(  \eta^{\frac{1}{3}} \ppp^{-{\frac{1}{4}} } ) \; \p_{11}W}_{=:II} \,.
\end{align*}
Since $\abs{\nabla  (  \eta^{\frac{1}{3}} \ppp^{-{\frac{1}{4}} }) } \les  \eta^{  \frac 13}$, it follows from  \eqref{eq:W_decay} that
\begin{align*}
\abs{II} \les M^ {\frac{1}{3}}  \ppp^{\frac{1}{4}}  \les M\,.
\end{align*}

Now we apply Lemma~\ref{lem:GN} to the function $\eta^{\frac{1}{3}} \ppp^{-\frac 14}  \p_{11}W$, appeal to the estimate \eqref{eq:W_decay}, and to the Leibniz rule to obtain that
\begin{align*}
\abs{I} 
\les \norm{\eta^{\frac{1}{3}} \ppp^{-{\frac{1}{4}} }  \p_{11}W}_{\dot H^{m-2}}^{\frac{2}{2m-7}}\norm{\eta^{\frac{1}{3}} \ppp^{-{\frac{1}{4}} }  \p_{11}W}_{L^{\infty}}^{\frac{2m-9}{2m-7}} 
&\les M \norm{\eta^{\frac{1}{3}} \ppp^{-{\frac{1}{4}} }  \p_{11}W}_{\dot H^{m-2}}^{\frac{2}{2m-7}}\, ,
\end{align*}
where we have used that $m\ge 18$ for the last inequality as is required by Proposition~\ref{cor:L2}.   We next estimate the $\dot H^{m-2}$ norm of $\eta^{\frac{1}{3}} \ppp^{-{\frac{1}{4}} }  \p_{11}W$.   To do so, we shall use the fact that
$W(\cdot,s)$ has support in the set $\XXX(s)$ defined in \eqref{eq:support}. We find that
\begin{align}
\norm{ \eta^{\frac{1}{3}} \ppp^{-{\frac{1}{4}} }  \p_{11}W}_{\dot H^{k-2}} 
&\les \sum_{m'=0}^{m-2}\norm{D^{m-m'-2}\left(\eta^{\frac{1}{3}} \ppp^{-{\frac{1}{4}} } \right)D^{m'}\p^{\gamma''}W}_{L^2} \notag\\
&\les \sum_{m'=0}^{m-2}\norm{D^{m-m'-2}\left(\eta^{\frac{1}{3}} \ppp^{-{\frac{1}{4}} } \right)}_{L^{\frac{2(m-1)}{m-2-m'}}(\XXX(s))} \norm{D^{m'}\p^{\gamma''}W}_{L^\frac{2(m-1)}{m'+1}} \notag\\
&\les \sum_{m'=0}^{m-2}\norm{D^{m-m'-2}\left(\eta^{\frac{1}{3}} \ppp^{-{\frac{1}{4}} } \right)}_{L^{\frac{2(m-1)}{m-2-m'}}(\XXX(s))} \norm{\nabla W}_{L^\infty}^{1-\frac{m'+1}{m-1}} \norm{W}_{\dot{H}^m}^{\frac{m'+1}{m-1}} \, ,\label{e:Fauci}
\end{align}
Using \eqref{eq:W_decay} and Proposition~\ref{cor:L2}, the $W$ terms are bounded as 
$$
\norm{\nabla W}_{L^\infty}^{1-\frac{m'+1}{m-1}} \norm{W}_{\dot{H}^m}^{\frac{m'+1}{m-1}}  \les  M^{2m}
$$
for all $m \in \{0,\dots,m-2\}$.
Moreover, using that $\abs{D^{m-m'-2}(\eta^{\frac{1}{3}} \ppp^{-{\frac{1}{4}} } )} \les \eta^{\frac{1}{3}}$ together with \eqref{e:space_time_conv},
we have that
\begin{align}
\norm{D^{m-m'-2}(\eta^{\frac{1}{3}} \ppp^{-{\frac{1}{4}} } )}_{L^{\frac{2(m-1)}{m-m'-2}}(\XXX(s))}\les \eps^{\frac13} e^{s}  \,,   \label{W:ng:rad2}
\end{align} 
with the usual abuse of notation $L^{\frac{2(m-1)}{m-m'-2}}=L^{\infty}$ for $m'=m-2$. Combining the above estimates, we obtain the inequality
\begin{align}
\abs{I} &\les M^{2m} \left( \eps^{\frac13} e^{{\frac{11}{8}}s}\right)^{\frac{2}{2m-7}} \les  e^{\frac{2s}{2(2m-7)}}  \label{W:ng:rad3}
\end{align}
for $\eps$  sufficiently small. From the above estimate,  we obtain the third inequality in \eqref{eq:W:GN}.

We next consider the bounds \eqref{eq:K:higher:order}, and we begin with the case that $\gamma_1 \geq 1$ and $\abs{\check\gamma} =2$.
Applying Lemma~\ref{lem:GN}  to the function $\p_1\check\nabla K$,  and using \eqref{eq:S_bootstrap} and Proposition~\ref{cor:L2}, we have that
\begin{align*}
\norm{\partial^{\gamma} K}_{L^\infty} 
&\les \norm{K}_{\dot H^{m}}^{\frac{2}{2m-7}} \norm{\partial_1\check\nabla K}_{L^\infty}^{\frac{2m -9}{2m-7}}
\les  \left(M^{2m} e^{-\frac s2} \right)^{\frac{2}{2m-7}}  \left(\eps^{\frac18} e^{-\frac{13}{8} s}\right)^{\frac{2m -9}{2m-7}}
\les  e^{- \frac{109-26m}{8(2 m-7) }s}  \,.
\end{align*}

We next consider the second inequality in \eqref{eq:K:higher:order} and proceed to estimate 
$\abs{\eta^\frac{1}{15}  \nabla \p_{11}K}$.   We write
\begin{align*}
\eta^\frac{1}{15}   \nabla \p_{11}K
=\underbrace{\nabla  \left( \eta^\frac{1}{15}   \p_{11}K\right)}_{=:I}
-\underbrace{\nabla \eta^\frac{1}{15} \; \p_{11}K}_{=:II} \,.
\end{align*}
Since $\sabs{\nabla   \eta^\frac{1}{15}  } \le 1$, it  follows from  \eqref{eq:S_bootstrap} that
\begin{align*}
\abs{II} \les e^{-2s}\,.
\end{align*}
By Lemma~\ref{lem:GN} and \eqref{eq:S_bootstrap}, 
\begin{align*}
\abs{I} 
\les \norm{ \eta^\frac{1}{15}   \p_{11}K}_{\dot H^{m-2}}^{\frac{2}{2m-7}}\norm{ \eta^\frac{1}{15}  \p_{11}K}_{L^{\infty}}^{\frac{2m-9}{2m-7}} 
&\les M e^{-(2-{\frac{4}{2m-7}} )s} \norm{ \eta^\frac{1}{15}   \p_{11}K}_{\dot H^{m-2}}^{\frac{2}{2m-7}}\, ,
\end{align*}
Following the calculation \eqref{e:Fauci},  we have that
\begin{align*}
\norm{\eta^{\frac{1}{15}}  \p_{11}K}_{\dot H^{m-2}} 
&\les \sum_{m'=0}^{m-2}\norm{D^{m-m'-2}\eta^\frac{1}{15}  }_{L^{\frac{2(m-1)}{m-2-m'}}(\XXX(s))}
 \norm{\nabla K}_{L^\infty}^{1-\frac{m'+1}{m-1}} \norm{K}_{\dot{H}^m}^{\frac{m'+1}{m-1}} \, .
\end{align*}
Applying   \eqref{e:space_time_conv},  we obtain that
\begin{align*} 
\norm{D^{m-m'-2} \eta^\frac{1}{15}  }_{L^{\frac{2(m-1)}{m-m'-2}}(\XXX(s))}\les \eps^\frac{1}{15}   e^{\frac{1}{5}  s}  \,.
\end{align*} 
From \eqref{eq:S_bootstrap} and Proposition \ref{cor:L2}, 
\begin{align*} 
 \norm{\nabla K}_{L^\infty}^{1-\frac{m'+1}{m-1}} \norm{K}_{\dot{H}^m}^{\frac{m'+1}{m-1}} 
 \le e^{-{\frac{s}{2}} }   \,.
\end{align*} 
From the above estimates, together with \eqref{eq:S_bootstrap}, we determine that
\begin{align*}
\abs{I} 
\les M e^{-(2-{\frac{4}{2m-7}} )s}   \left(  \eps^\frac{1}{15}  e^{-{\frac{3}{10}}  } \right)^{\frac{2}{2m-7}} \les e^{-(2-{\frac{4}{2m-7}} )s}  \,.
\end{align*}
This estimate establishes the second bound in \eqref{eq:K:higher:order}.
For $\abs{\gamma} \in \{3,4,5\}$ we apply Lemma~\ref{lem:GN} to $\nabla^2 K$, and together with \eqref{eq:S_bootstrap} and Proposition~\ref{cor:L2}, 
we find that
\begin{align}
\norm{\partial^{\gamma}  K }_{L^\infty}
\les \norm{K}_{\dot H^{m}}^{\frac{2\abs{\gamma}-4}{2m-7}} \norm{\nabla^2 K}_{L^\infty}^{\frac{2m -3 -2\abs{\gamma}}{2m-7}}
\les \left(M^{2m} e^{-\frac s2} \right)^{\frac{2\abs{\gamma}-4}{2m-7}}   \left( \eps^{\frac18}e^{-s}\right)^{\frac{2m -3 -2\abs{\gamma}}{2m-7}}
\les   e^{- (1 - \frac{\abs{\gamma}-2}{2m-7})s} \, .  
\notag
\end{align}
where we have assumed that $\eps$ is taken sufficiently small.
\end{proof}

\subsection{Bounds for $U \cdot \Ncal$ and $S$}
\label{sec:spiny-anteater0}
Finally, we note that as a consequence of the definitions \eqref{eq:UdotN:Sigma}, we have the following estimates on $U\cdot \Ncal $ and $\sound$.

\begin{lemma}\label{lem:US_est}
For $y \in \XXX(s)$ we have
\begin{equation}
\abs{\p^\gamma U\cdot\Ncal}+\abs{\p^\gamma \sound}
\les  \begin{cases}
 M^{\frac 14} , & \mbox{if } \abs{\gamma}  = 0  \\
 M^{\frac{1}{3}} e^{-\frac s2}\eta^{-\frac13}, & \mbox{if }  \gamma_1=1 \mbox{ and } \abs{\check \gamma}=0 \\
e^{-\frac s2}, & \mbox{if } \gamma_1=0 \mbox{ and }  \abs{\check \gamma} = 1\\
M^{\frac{2}{3}} e^{-\frac s2}\eta^{-\frac13}, & \mbox{if }  \gamma_1 =1 \mbox{ and } \abs{ \check\gamma}=1 \\
M^{\frac{2}{3}} e^{-\frac s2}\eta^{-\frac13}\ppp^ {\frac{1}{4}} , & \mbox{if }  \gamma_1 =2 \mbox{ and } \abs{ \check\gamma}=0 \\
 M  e^{-\frac s2}\eta^{-\frac16}, & \mbox{if } \gamma_1 = 0 \mbox{ and } \abs{\check \gamma} = 2\\
e^{\left(-\frac12+\frac{3}{2m-7}\right)s} \eta^{-\frac13},  &\mbox{if } \gamma_1=1\mbox{ and } \abs{\check \gamma}=2\\
e^{\left(-\frac12+\frac{1}{2m-7}\right)s}\eta^{-\frac16} ,  &\mbox{if } \gamma_1=0\mbox{ and } \abs{ \check\gamma}=3 \\
e^{\left(-\frac12+\frac{3}{2m-7}\right)s} \eta^{-\frac 13}\ppp^{\frac{1}{4}}   ,  &\mbox{if } \gamma_1\ge 2\mbox{ and } \abs{ \gamma}=3 
\end{cases}
\label{eq:US_est}
\,.
\end{equation}
Additionally, for 
$\abs{y} \leq \ell$ and $\abs{\gamma} = 4$ we have the bound 
\[
 \abs{\p^\gamma U\cdot\Ncal}+\abs{\p^\gamma \sound} \les  e^{- \frac{s}{2}}\,. 
\]
\end{lemma}
\begin{proof}[Proof of Lemma~\ref{lem:US_est}]
We  shall only establish the bounds for  $\p^\gamma U\cdot\Ncal$ as the  estimates for $\p^\gamma \sound$ are obtained in the identical fashion.
Since$\abs{\kappa} \leq M^{\frac 14}$,  it follows from \eqref{eq:UdotN:Sigma}  that 
$\abs{\p^\gamma U\cdot\Ncal}\les \abs{\kappa} {\bf 1}_{\abs{\gamma}=0}+e^{-\frac s 2}\abs{\p^\gamma W}+\abs{\p^\gamma Z}$.  The desired bounds
are obtained by an application of 
\eqref{eq:W_decay}, \eqref{eq:bootstrap:Wtilde4}, \eqref{eq:Z_bootstrap}, Lemma \ref{eq:higher:order:ests} and \eqref{e:space_time_conv}.
\end{proof}

\begin{proposition}[ $L ^ \infty $ bound for the sound speed]\label{prop:sound}
We have that
\begin{align} \label{Sigma-bound}
\snorm{\sound( \cdot , s) - \tfrac{\kappa_0}{2} }_{ L^ \infty } \le \eps^ {\frac{1}{8}}  \ \ \text{ for all } \ s \ge -\log\eps  \,.
\end{align} 
\end{proposition} 
\begin{proof}[Proof of Proposition \ref{prop:sound}]
By \eqref{eq:UdotN:Sigma}, we have that
$$
\sound( \cdot , s) - \tfrac{\kappa_0}{2}= \tfrac{\kappa-\kappa_0}{2} + {\tfrac{1}{2}} (e^ {-\frac{s}{2}} W-Z) \,.
$$
By \eqref{mod-boot}, \eqref{e:space_time_conv},  \eqref{eq:W_decay},  and \eqref{eq:Z_bootstrap}, and the triangle inequality,
$$\snorm{\sound( \cdot , s) - \tfrac{\kappa_0}{2} }_{ L^ \infty } \les \eps^ {\frac{1}{6}} $$
which concludes the proof.
\end{proof}

\subsection{The blowup time and location}\label{sec:spiny-anteater1}
The blowup time $T_*$ is defined uniquely by the condition $\tau(T_*) = T_*$ which by \eqref{eq:modulation:IC} is equivalent to
\begin{equation}\label{good-vlad}
\int_{-\eps}^{T_*} (1-\dot \tau(t)) dt = \eps \, .
\end{equation} 
The estimate for $\dot \tau$  in \eqref{eq:acceleration:bound} shows that  for $\eps$ taken sufficiently small,
\begin{align} 
\abs{T_*} \leq 2M^2 \eps^{2} \,. \label{T*-bound}
\end{align} 

We also note here that the bootstrap assumption  \eqref{eq:acceleration:bound} and the definition of $T_*$ ensures that  $\tau(t) > t$ for all $t\in [-\eps, T_*)$.
Indeed, when $t = - \eps$, we have  that $\tau(-\eps) = 0 > -\eps$, and the function $t\mapsto \int_{-\eps}^t (1-\dot \tau) dt' - \eps = t -\tau(t)$ is strictly increasing.

The blowup location is determined by $\xi_* = \xi(T_*)$, which by \eqref{eq:modulation:IC} is the same as 
$$
\xi_* = \int_{-\eps}^{T_*} \dot \xi(t) dt \, .
$$
In view of \eqref{eq:acceleration:bound}, for $\eps$ small enough,
find  that 
\begin{align} 
\abs{\xi_*} \leq M \eps \,,\label{xi*-bound}
\end{align} 
so that the blowup location is $\OO(\eps)$ close to the origin. 

\subsection{H\"{o}lder bound for $w$}\label{sec:spiny-anteater2}
As we proved in \cite{BuShVi2019b}, the self-similar scaling \eqref{eq:y:s:def} and decay rate  \eqref{eq:W_decay}  for $W(y,s)$  show that 
$$w \in L^\infty([-\eps,T_*);C^{\sfrac 13}) \,, $$
and the $C^\alpha$ H\"older norms of $ w$, with $\alpha>\sfrac 13$,  blowup as $t\to T_*$ with a rate proportional to $(T_*-t)^{\sfrac{(1-3\alpha)}{2}}$.

\section{Bounds on Lagrangian trajectories}
\label{sec:Lagrangian}
 
\subsection{The Lagrangian flows in self-similar variables}
In self-similar variables $(y,s)$, we define   Lagrangian flows associated to the transport velocities in \eqref{eq:transport_velocities} by
\begin{subequations} 
\label{flows}
\begin{alignat}{2} 
\p_s \pw (y,s) &=  \mathcal{V}_W(\pw(y,s), s) \,, \qquad && \pw(y,s_0)=y\,,\\
\p_s \pz (y,s)  &= \mathcal{V}_Z(\pz(y,s), s) \,, \qquad && \pz(y,s_0)=y\,, \\
\p_s \pa (y,s) &=  \mathcal{V}_U(\pa(y,s), s) \,, \qquad && \pa(y,s_0)=y \,,
\end{alignat}
\end{subequations} 
for  $s_0 \geq -\log\eps$.  With $\Phi$ denoting either $\pw$, $\pz$, or $\pu$, we shall denote trajectories emanating from a point $y_0$ at time $s_0$ by
\begin{align} 
\Phi^{y_0}(s)=\Phi (y_0,s) \qquad \text{ with } \qquad \Phi (y_0,s_0)=y_0 \,. \label{traj}
\end{align}

\subsubsection{Esimates for the support and a lower bound for $\pw$}
Since the bounds for $\abs{G_W}$, $\abs{h_W}$, and $\abs{W}$ are the same as in \cite{BuShVi2019b}, the proofs of the following two lemmas
are the same as Lemma 8.1 and 8.2 in \cite{BuShVi2019b}.

The bootstrap assumption \eqref{eq:support} on the size of the support is closed via the following
\begin{lemma}[Estimates on the support]\label{lem:support}
 Let  $\Phi$ denote either $\pw^{y_0}$, $\pz^{y_0}$ or $\pa^{y_0}$.  For any
$y_0 \in \mathcal{X} _0$ defined in \eqref{eq:support-init}, we have that
\begin{subequations} 
\label{eq:upper:bound:traj}
\begin{align} 
  \abs{\Phi_1(s)} & \le  \tfrac{3}{2}  \eps^{\frac12}  e^{\frac 32s} \,,  \label{trajectory1}\\
 \abs{\check \Phi(s)} & \le  \tfrac{3}{2}   \eps^{\frac{1}{6}}  e^{\frac{s}{2}} \,,    \label{trajectory2}
\end{align} 
\end{subequations} 
for all  $s \ge -\log\eps$.
\end{lemma} 

We shall also make use of the lower bound given by 
\begin{lemma} \label{lem:escape}
Let $y_0 \in \RR^3$ be such that $\abs{y_0} \geq \ell$. Let $s_0 \geq - \log \eps$. 
Then, the trajectory $\pw^{y_0}$ moves away from the origin at an exponential rate, and we have the lower bound
\begin{equation}
\abs{\pw^{y_0}(s)}\geq \abs{y_0}e^{\frac{s-s_0}{5}}\label{eq:escape_from_LA} 
\end{equation}
for all  $s \geq s_0$.
\end{lemma} 

\begin{lemma} \label{lem:escape2}
Given  $s_0 \geq - \log \eps$ and $s > s_0$, 
let $y_0 \in \RR^3$ be such that $\abs{y_0} \geq \mathcal{L} $ and  $\abs{\check\Phi_W^{y_0}(s)} \le M \eps^{\frac 12}$.
Then, we have that 
\begin{equation}
\abs{(\pw^{y_0})_1 (s')}\geq \tfrac 34 \abs{(y_0)_1}e^{\frac{3(s'-s_0)}{2}} 
\qquad \mbox{and} \qquad
\abs{\check\Phi_W^{y_0}(s')}\leq M \eps^{\frac 12} 
\label{eq:escape_from_SF} 
\end{equation}
for all  $s_0 \le s' \le s$.  
\end{lemma} 
\begin{proof}[Proof of Lemma~\ref{lem:escape2}]
Fix $(y_0,s_0)$ and let us denote $(\pw^{y_0})_1(s) = \Phi_1(s)$ and 
$\check\Phi_W^{y_0}(s) = \check \Phi(s)$.  

According to \eqref{flows} and \eqref{eq:transport_velocities}, we have that 
$
\p_s \Phi_\nu = \tfrac 12 \Phi_\nu + h_W^\nu \circ \Phi.
$ 
Solving this ODE on the interval $[s',s]$, with arbitrary $s' \in [s_0,s)$, we obtain that 
\begin{align*}
\Phi_\nu(s') = \Phi_\nu(s) e^{-\frac{s-s'}{2}} - \int_{s'}^s e^{-\frac{s''-s'}{2}} h_W^\nu\circ \Phi(s'') ds'' \,. 
\end{align*}
Using that by \eqref{e:h_estimates} we have $\abs{h_W (\cdot,s)} \leq M^{\frac 12}  e^{-\frac s2}$, and appealing to the assumption $\abs{\Phi_\nu(s)} \le M \eps^{\frac 12}$, we obtain that 
\begin{align*}
\abs{\Phi_\nu(s')} 
\leq \abs{\Phi_\nu(s)} e^{-\frac{s-s'}{2}} +M^{\frac 12}  \int_{s'}^s e^{-\frac{s''-s'}{2}} e^{-\frac{s''}{2}} ds'' 
\leq M \eps^{\frac 12} e^{-\frac{s-s'}{2}} + M^{\frac 12} e^{-\frac{s'}{2}} (1- e^{-(s-s')}) 
\leq M \eps^{\frac 12}
\,. 
\end{align*}
where in the last inequality we have used that $s' \geq s_0 \geq -\log \eps$, so that $e^{-\frac{s'}{2}} \leq \eps^{\frac 12} e^{-\frac{(s'-s_0)}{2}}$. This proves the second claim in \eqref{eq:escape_from_SF}.

In order to prove the first claim in \eqref{eq:escape_from_SF}, we again recall \eqref{flows} and \eqref{eq:transport_velocities}, which gives 
$
\p_s \Phi_1 = \tfrac 32 \Phi_1 + \beta_\tau W\circ \Phi +G_W\circ \Phi \,.
$
In view of the bound established for $\check \Phi$ and of the information we have from Lemma~\ref{lem:escape}, we already know that $\abs{y_0}\geq \LLL$ implies that $\abs{\Phi_1(s')}\geq \frac{\LLL}{2} e^{(s'-s_0)/5}$ for all $s' \in [s_0,s]$, so that $\Phi_1(s')$ is much larger than $1$. Thus, from  \eqref{eq:beta:tau} and the first bound in \eqref{eq:W_decay}, we have 
\begin{align*}
\beta_\tau \sabs{ W\circ \Phi(s')} \leq (1+2M \eps) (1+\eps^{\frac{1}{20}})  \left(1 + \sabs{\Phi_1(s')}^2 + (M\eps^{\frac 12})^6\right)^{\frac 16} \leq 2 \sabs{\Phi_1(s')}^{\frac 13} \,.
\end{align*}
Similarly, the first estimate in Lemma \ref{lemma_g}, in which we use an extra factor of $M$ to absorb the implicit constant in the $\les$ symbol, and the previously established bound \eqref{trajectory1} imply that
\begin{align*}
\sabs{G_W\circ \Phi(s')} \leq M^2 e^{-\frac{s'}{2}} + M^{\frac 32} e^{-s'} \sabs{\Phi_1(s')} + M^2 \eps^{\frac 56} \leq M^{2} e^{-s'} \sabs{\Phi_1(s')} \leq 2 M^{2} \eps^{\frac 13}  \sabs{\Phi_1(s')}^{\frac 13}\,.
\end{align*}
Combining the above two estimates with the ODE satisfied by $\Phi_1$ we derive that
$$
\tfrac{1}{2} \tfrac{d}{ds} \abs{\Phi_1(s')}^2 \ge \tfrac{3}{2}  \abs{\Phi_1(s')}^2 - 3 \abs{\Phi_1(s')}^{ \frac 43}  \,.
$$
By explicitly integrating the above ODE, and using our earlier observation that $\abs{(y_1)_0} \geq \frac 12 \eps^{-\frac{1}{10}}$ for all $s' \in [s_0,s]$, we derive that
$$
\abs{\Phi_1(s')} \geq \left(\abs{(y_0)_1}^{\frac 23}  - 2 \right)^{\frac 32} e^{\frac{3(s'-s_0)}{2}} \geq \tfrac 34  \abs{(y_0)_1} e^{\frac{3(s'-s_0)}{2}}
$$
which completes the proof. 
\end{proof}

\subsubsection{Lower bounds for $\pz$ and $\pu$}
We now establish important lower-bounds for $\pz^{y_0}(s)$ or $\pa^{y_0}(s)=\pu^{y_0}(s)$.  
\begin{lemma}\label{lem:phiZ} 
Let $\Phi(s)$ denote either $\pz^{y_0}(s)$ or  $\ \pu^{y_0}(s)$.  If
\begin{align} 
\kappa_0\geq \frac{3}{1-\max(\beta_1,\beta_2)}\,, \label{k0-lb}
\end{align} 
then for any $y_0 \in \XXX_0$ defined in \eqref{eq:support-init}, there exists an $s_*\geq -\log \eps$ such that
\begin{equation}\label{phi-lowerbound}
\abs{\Phi_1(s)}\geq \min\left(\abs{e^{\frac s2}-e^{\frac {s_*}2}}, e^{\frac s2}\right)\,.
\end{equation}
In particular, we have the following inequalities:
\begin{equation}\label{phi-lowerbound_conseq}
\int_{-\log \eps }^{\infty} e^{\sigma_1 s'}(1+\abs{\Phi_1(s')})^{-\sigma_2}\,ds'\leq C\,,
\end{equation}
for $0\leq \sigma_1 \le \sfrac12$ and $2\sigma_1 < \sigma_2$, where the constant $C$ depends only on the choice of $\sigma_1$ and $\sigma_2$.
\end{lemma}
This is a slight generalization of Lemma 8.3 in \cite{BuShVi2019b}, where we now allow the value $\sigma_1= {\sfrac{1}{2}} $.   The only addition
to the proof requires an estimate for the integral $ \mathcal{I} $ in the proof of Lemma 8.3 in \cite{BuShVi2019b}.   In particular, 
 for $\sigma_1= {\sfrac{1}{2}} $, we see that
\begin{align*}
\mathcal I & =2 \int_{\eps^{-\frac 12}}^{\infty}\left(1+\sabs{r-e^{\frac{s_*}{2}}}\right)^{-\sigma_2}\,dr \les 1\,.
\end{align*} 
The implicit constant only depends on $\sigma_1$ and $\sigma_2$.

\subsubsection{The time integral of $\abs{\p_1 W }$ along $\pz^{y_0}$}
An immediate consequence of \eqref{phi-lowerbound_conseq} is the following
\begin{corollary}\label{cor:p1Wz} 
For all $s \ge -\log \eps$,
\begin{align} 
\sup_{y_0 \in \XXX_0} \int_{ -\log\eps}^s \abs{\p_1 W } \circ \pz^{y_0}(s') ds' \les 1   \label{eq:p1W:PhiZ} \,.
\end{align} 
\end{corollary} 
\begin{proof}[Proof of Corollary \ref{cor:p1Wz}]
The bound \eqref{eq:p1W:PhiZ} follows using the second estimate in \eqref{eq:W_decay} together with \eqref{phi-lowerbound_conseq} with
$\sigma_1=0$ and $\sigma_2= {\frac{2}{3}} $.
\end{proof}

\subsection{The Lagrangian flow $\varphi(x,t)$}
With respect to the independent variables $(x,t)$, the transport velocity for $\mathring u$ in \eqref{momentum-sheep-alt} is given by
 \begin{align}
\Vcal = (\Vcal_1,\Vcal_2,\Vcal_3) = 2\beta_1 \left(-\tfrac{\dot f }{2 \beta_1 } + \Jcal  v \cdot \Ncal + \Jcal \mru \cdot \Ncal ,  v_2 + \mru_2, v_3 + \mru_3\right) \,.
\label{eq:fucked:up:v}
\end{align}
 We let $\varphi(x,t)$ denote the flow of $\Vcal$ so that
\begin{subequations} 
\label{phi-flow}
\begin{align} 
\p_t \varphi(x,t) & = \Vcal(\varphi(x,t),t)\,, \ \ \  t> -\eps\,,  \\
 \varphi(x,-\eps)&=x\,, 
\end{align} 
\end{subequations} 
and we denote by $\varphi_{x_0}(t)$ the trajectory emanating from $x_0$.

\subsubsection{Asymptotic non-positivity for $\p_1W$}
\begin{lemma}
\label{lem:non:positivity}
For all $y\in \RR^3$ and $s\geq -\log \eps$, we have 
\begin{align}
\max\left\{ \p_1W(y,s),0\right\} \leq 4 e^{-\frac{s}{15}}  \,.
\label{eq:Celtics:suck}
\end{align}
\end{lemma}
\begin{proof}[Proof of Lemma~\ref{lem:non:positivity}]
We start with the region $\abs{y}\leq \LLL = \eps^{-\frac{1}{10}}$. Here, due to the bootstrap \eqref{eq:bootstrap:Wtilde1} for $\p_1 \tilde W$ and the fact that $\p_1 \bar W  \leq - \tilde \eta^{-\frac 13}$ (see (2.48) in \cite{BuShVi2019b}), we deduce that 
\begin{align} 
\p_1 W(y,s) = \p_1 \bar W(y) + \p_1 \tilde W(y,s) \leq -  \tilde \eta^{-\frac 13}(y) + \eps^{\frac{1}{12}} \eta^{-\frac 13}(y) < 0  \qquad \text{for}\qquad \abs{y} \le \eps^{-\frac{1}{10}} \,,
\label{perezoso-gordo}
\end{align} 
upon taking $\eps$ sufficiently small, and using that  $\tilde \eta(y) \leq 2 \eta(y)$. Thus, for $\abs{y}\leq \LLL$ the bound \eqref{eq:Celtics:suck} holds. 

Next, let us consider the region $\abs{y} \geq e^{\frac{s}{10}}$. Here we have that $\eta(y) \geq \frac 12 e^{\frac{s}{5}}$. Combining this bound with the second line of \eqref{eq:W_decay}, we arrive at
\[
\abs{\p_1 W(y,s)} \leq 2 \eta^{-\frac 13}(y) \leq 4 e^{-\frac{s}{15}} \,.
\]
Thus, \eqref{eq:Celtics:suck} also holds in the region $\abs{y} \geq e^{\frac{s}{10}}$.

It remains to consider the region $\LLL < \abs{y} < e^{\frac{s}{10}}$. Notice that by the definition of $\LLL=e^{\frac{-\log \eps}{10}}$, in this case we have that $s>-\log \eps$. For such a fixed $(y,s)$
we trace the particle trajectory of  the flow $\mathcal{V}_W$ backwards in time, and write $\pw^{y_0}(s) = y$, where the initial datum $\pw^{y_0}(s_0) = y_0$ is  given  by the property that $\abs{y_0} = \LLL$ if $s_0>-\log \eps$, and $\abs{y_0} > \LLL$ if $s_0 = -\log \eps$. We claim that the second option is not possible, so that we must have $s_0>-\log \eps$ and $\abs{y_0} = \LLL$. To see this, we appeal to Lemma~\ref{lem:escape}, which is applicable since $\abs{y_0} \geq \LLL \geq \ell$, and which gives the bound $\abs{\pw^{y_0}(s)}\geq \abs{y_0} e^{\frac{s-s_0}{5}}$. Thus, in the case that $s_0 = -\log \eps$ and $\abs{y_0} > \LLL$, this bound implies
\[
e^{\frac{s}{10}} > \abs{y} = \abs{\pw^{y_0}(s)} \geq \abs{y_0} e^{\frac{s-s_0}{5}} > \LLL e^{\frac{s+\log \eps}{5}} = \eps^{-\frac{1}{10}} e^{\frac{s+\log \eps}{5}} = e^{\frac{s}{10}}e^{\frac{s+\log \eps}{10}} > e^{\frac{s}{10}}
\]
since $s> -\log \eps$. This yields the desired contradiction, guaranteeing that $\abs{y_0} = \LLL$ and $s_0>-\log \eps$. At this stage we appeal to the evolution of $\p_1 W$ given in \eqref{euler_for_Linfinity:a} with $\gamma = (1,0,0)$,
and deduce that $e^{\frac{s}{2}} \p_1 W$ satisfies the equation
\[
\p_s  (e^{\frac{s}{2}} \p_1 W)  + \left(   \tfrac{1}{2} + \beta_\tau   \Jcal \p_1 W \right) (e^{\frac{s}{2}} \p_1 W)   +  \left( \mathcal{V}_W \cdot \nabla\right) (e^{\frac{s}{2}}  \p_1 W)  = e^{\frac{s}{2}}  F^{(1,0,0)}_W
\,.
\]
Composing with $\pw^{y_0}$ and appealing to Gr\"onwall's inequality on the interval $[s_0,s]$,  we obtain that
\begin{align}
e^{\frac{s}{2}} \p_1 W(y,s) 
&= e^{\frac{s_0}{2}} \p_1 W(y_0,s_0) \exp\left(- \int_{s_0}^s \tfrac{1}{2} + \beta_\tau (\Jcal \p_1 W)\circ\pw^{y_0}(s') ds' \right)\notag\\
&\quad + \int_{s_0}^s e^{\frac{s'}{2}} F^{(1,0,0)}_W\circ\pw^{y_0}(s')  \exp\left(- \int_{s'}^s \tfrac{1}{2} + \beta_\tau (\Jcal \p_1 W)\circ\pw^{y_0}(s'') ds'' \right) ds' \,.
\label{eq:Celtics:suck:a:lot}
\end{align}
We now use the information that $\abs{y_0} = \LLL$, and thus, as established earlier, $\p_1 W(y_0,s_0) < 0$. Hence, the first term on the right side of \eqref{eq:Celtics:suck:a:lot} is strictly negative (as the exponential is positive), so that it does not contribute to the positive part of $\p_1 W$. We deduce, by also appealing to the $F^{(1,0,0)}_W$ estimate in \eqref{eq:forcing_W} and the $\p_1 W$ bootstrap in \eqref{eq:W_decay}, that 
\begin{align*}
e^{\frac{s}{2}} \max\left\{\p_1 W(y,s) ,0\right\}
&\leq \int_{s_0}^s e^{\frac{s'}{2}} \abs{F^{(1,0,0)}_W\circ\pw^{y_0}(s')}  \exp\left(- \int_{s'}^s \tfrac{1}{2} + \beta_\tau (\Jcal \p_1 W)\circ\pw^{y_0}(s'') ds'' \right) ds' 
\notag\\
&\les M \int_{s_0}^s   \eta^{-\frac 13}\circ\pw^{y_0}(s')   \exp\left(4  \int_{s'}^s  \eta^{-\frac 13}\circ\pw^{y_0}(s'')  ds'' \right) ds' \,.
\end{align*}
The proof is completed by appealing to the  bound established in \eqref{eq:need:a:label}, namely $ \int_{s_0}^s \eta^{-\frac 13} \circ\pw^{y_0}(s') ds' \leq \eps^{\frac{1}{16}}$, which holds for $\abs{y_0} \geq \LLL$, and which implies 
$
e^{\frac{s}{2}} \max\left\{\p_1 W(y,s) ,0\right\} \les M \eps^{\frac{1}{16}} \exp(4 \eps^{\frac{1}{16}}) \leq 1.
$
\end{proof}

From Lemma~\ref{lem:non:positivity},  we immediately deduce the following
\begin{corollary}
\label{cor:out:of:labels}
For any $t\in[-\eps,T_*)$ we have 
\begin{align}
\int_{-\eps}^t \max\{\p_{x_1}  \mru \cdot \Ncal,0\} dt' \leq \eps^{\frac{1}{16}}
\end{align}
uniformly pointwise in space.
\end{corollary}
\begin{proof}[Proof of Corollary~\ref{cor:out:of:labels}]
Recall that cf.~\eqref{tildeu-dot-N} and \eqref{w_ansatz}--\eqref{z_ansatz} that 
\[
\p_{x_1} \mru \cdot \Ncal = \tfrac 12 ( \p_{x_1} w + \p_{x_1} z ) = \tfrac 12 e^s \p_1 W + \tfrac 12 e^{\frac{3s}{2}} \p_1 Z \,.
\]
From \eqref{eq:Z_bootstrap} we know that $e^{\frac{3s}{2}} \abs{\p_1 Z } \leq M^{\frac 12}$, and since the function $\max\{ \cdot,0\}$ is convex and in fact sub-additive, we deduce from Lemma~\ref{lem:non:positivity} that
\[
\max\{\p_{x_1}  \mru \cdot \Ncal,0\} \leq \tfrac 12 e^s \max\{\p_1 W,0\}  + \tfrac12  e^{\frac{3s}{2}} \max\{\p_1 Z,0\}  \leq 2 e^{\frac{14 s}{15}} + \tfrac{1}{2} M^{\frac 12}\,.
\]
Writing $dt' = \beta_\tau e^{-s'} ds'$, the desired bound follows from
\[
\int_{-\log \eps}^\infty \left(2 e^{\frac{14}{15}s'} + \tfrac{1}{2} M^{\frac 12}\right)\beta_\tau e^{-s'} ds' 
\leq 60 \eps^{\frac{1}{15}} + M^{\frac 12} \eps
\] 
concluding the proof.
\end{proof}

\subsubsection{The time integral of $\abs{\p_1 W }$ along $\pu^{y_0}$}
We next establish the following:
\begin{lemma}\label{lem:p1Wu} 
For all $s \ge -\log \eps$,
\begin{align} 
\sup_{y_0 \in \XXX_0} \int_{ -\log\eps}^s \abs{\p_1 W } \circ \pu^{y_0}(s') ds' \les \eps^ {\frac{1}{18}}    \label{eq:p1W:PhiU} \,.
\end{align} 
\end{lemma} 
\begin{proof}[Proof of Lemma \ref{lem:p1Wu}] 
From the definition of the transport velocity $\Vcal$ in  \eqref{eq:fucked:up:v}, observe that 
\begin{align} 
\operatorname{div}_{\!  x}\Vcal =   \operatorname{div}_{\! \tilde x}\mathring u = 2\beta_1(\p_{x_1} \mru \cdot \nn \Jcal + \p_{x_\nu} \mru_\nu)
\label{eq:remarkable:argument}
\end{align} 
where we have used the fact that 
$$\operatorname{div}_{\!  x}v = \p_{x_j} v_j =  \p_{x_1} \Jcal \Ncal \cdot v + \p_{x_\mu} v_\mu = \operatorname{div}_{\!  \tilde x}v    $$
and that from \eqref{eq:v:x:t:def},
$\operatorname{div}_{\! \tilde x}v= \dot Q_{ii} =0$, and that 
$\operatorname{div}_{\!   x} (-\dot f, 0  , 0)=0$.  Hence, the conservation of mass equation
 \eqref{density-sheep} can be written as 
\begin{align} 
 \partial_t \mathring\rho +  \Vcal \cdot \nabla_{\! x}\mathring\rho
 + \mathring \rho \operatorname{div}_{\! x}\Vcal =0  \,,
\label{density-good1}
\end{align} 
and composing \eqref{density-good1} with the flow $\varphi$ given by \eqref{phi-flow}, we see that
\begin{align} 
 \partial_t (\mathring\rho \circ \varphi)
 = (\mathring \rho \circ \varphi)  (\operatorname{div}_{\! x}\Vcal) \circ \varphi  \,.
\label{density-good2}
\end{align} 
Since 
\begin{align} 
 \partial_t (\det \nabla_{\! x} \varphi)
 = \det \nabla_{\! x} \varphi (\operatorname{div}_{\! x}\Vcal) \circ \varphi  \,,
\label{Jac-good}
\end{align} 
and $\det \nabla_{\! x} \varphi (x, -\eps) =1$, 
it follows that $$\mathring\rho \circ \varphi = (\det \nabla_{\! x} \varphi) ^{-1} \mathring\rho_0\,.$$

Note that using \eqref{sigma1}, \eqref{Sigma-bound} and \eqref{eq:S_bootstrap:*} yields
\begin{align}
\abs{\rho-(\tfrac {\alpha\kappa_0} 2)^{\frac 1 \alpha}}&= \abs{( \alpha e^{-{\frac{\scal}{2}} } \sigma)^ {\frac{1}{\alpha }}  -(\tfrac {\alpha\kappa_0} 2)^{\frac 1 \alpha}} \notag\\
&\les  \abs{( \alpha e^{-{\frac{\scal}{2}} } \sigma)^ {\frac{1}{\alpha }} -( \tfrac{\alpha e^{-{\frac{\scal}{2}} } \kappa_0}{2})^ {\frac{1}{\alpha }}  }+ \abs{( \tfrac{\alpha e^{-{\frac{\scal}{2}} } \kappa_0}{2})^ {\frac{1}{\alpha }} -( \tfrac{\alpha \kappa_0}{2})^ {\frac{1}{\alpha }}  } \notag\\
&\les \eps^{\frac18} ( \tfrac{\alpha \kappa_0}{2})^ {\frac{1}{\alpha }-1}  
\les \eps^{\frac19} \,.\label{density-final0}
\end{align}
Therefore, by  \eqref{density-final0} and  \eqref{density-final0}, we have that
\begin{align}
\abs{\det ( \nabla_{\!x} \varphi(x, t) )-1}&\leq \abs{\tfrac{\mathring \rho_0}{\mathring \rho}-1}\notag\\
&\leq \abs{\tfrac{\mathring \rho_0}{\mathring \rho}-\tfrac{(\tfrac {\alpha\kappa_0} 2)^{\frac 1 \alpha}}{\mathring \rho}}+\abs{\tfrac{(\tfrac {\alpha\kappa_0} 2)^{\frac 1 \alpha}}{\mathring \rho}- 1}\les \eps^{\frac19}  \,.\label{light-buzz}
\end{align}

From \eqref{density-good2} and \eqref{Jac-good}, we have that
\begin{align*} 
\frac{d}{dt}\det \nabla_{\!x} \varphi =  \det \nabla_{\!x} \varphi  (\div_{x} \Vcal ) \circ \varphi =  \det \nabla_{\!x} \varphi  (\div_{\tilde x} \mathring u) \circ \varphi
\end{align*} 
leads to
\begin{align} 
 \det \nabla_{\!x} \varphi(x, t) =\exp \int_{-\eps}^t (\operatorname{div}_{\tilde x}  \mathring u \circ \varphi)(x, t') dt' \,. \label{nice-buzz}
\end{align} 
Hence,
\begin{align}
 - \eps^ {\frac{1}{9}}   \les \int_{-\eps}^t (\operatorname{div}_{\tilde x}  \mathring u \circ \varphi)(x, t') dt'  \les  \eps^ {\frac{1}{9}}    \ \ \text{ for all } \ x \in \mathbb{R}^3   \,.  \label{light-buzz2}
\end{align} 
From \eqref{a_ansatz},  \eqref{U-trammy}, \eqref{eq:A_bootstrap}, and \eqref{eq:US_est}
\begin{align} 
\snorm{  \p_{ x_\nu} \mathring u_\nu (\cdot ,t) }_{ L^ \infty } \les 1 \,. \label{feeling-good}
\end{align} 
It follows from \eqref{eq:speed:bound} and  \eqref{feeling-good} that
\begin{align} 
\int_{-\eps}^{t} \snorm{  \p_{ x_\nu} \mathring u_\nu (\cdot ,t) }_{ L^ \infty } dt' \les {\tau(t) + \eps} \les {M\eps^2 + \eps} \le \eps^ {\frac{1}{2}}  \,.   \label{even-better-mofo}
\end{align} 
Thus, with  \eqref{e:bounds_on_garbage}, \eqref{eq:remarkable:argument}, \eqref{light-buzz},  and \eqref{even-better-mofo}, we have that
\begin{align}
\abs{ \int_{-\eps}^t   \p_{ x_1} \mathring u \cdot  \Ncal \circ \varphi  dt'}\leq \abs{ \int_{-\eps}^t\left(\tfrac{1}{\Jcal}\p_{x_\nu} \mru_\nu\right)\circ  \varphi  dt'}  +\tfrac{1}{2\beta_1} \abs{\int \left( \operatorname{div}_{\! \tilde x}\mathring u\right)\circ  \varphi  dt'} \les \eps^{\frac 19}   \,. \label{buzzed-good-bro}
\end{align}
By Corollary~\ref{cor:out:of:labels}, the integral of the positive part of $\p_{ x_1} \mathring u \cdot  \Ncal $ is small. Therefore, the above estimate gives a bound on the negative part of $\p_{ x_1} \mathring u \cdot  \Ncal $ as well. In summary, by \eqref{buzzed-good-bro}  and Corollary~\ref{cor:out:of:labels}, we then have that
\begin{align} 
 \int_{-\eps}^t  \abs{ \p_{ x_1} \mathring u \cdot  \Ncal \circ \varphi } dt' \le \eps^ {\frac{1}{18}}   \,. \label{buzzed-good-bro2}
\end{align} 
Then, from  \eqref{tildeu-dot-N} and the bootstrap assumptions \eqref{eq:speed:bound} and  \eqref{eq:Z_bootstrap}, we see that
$ \int_{-\eps}^t  \abs{ \p_{ x_1} w \circ \varphi } dt' \le  \eps^ {\frac{1}{19}} $, and in particular, for any $x_0 \in \XX_0$, we have that
\begin{align} 
\sup_{x_0 \in \XX_0} \int_{-\eps}^t  \abs{ \p_{ x_1} w \circ \varphi_{x_0} } dt' \le  \eps^ {\frac{1}{19}}   \,. \label{buzzed-good-bro3}
\end{align} 
Since the flow $\Phi(y,s)$ is related to the flow $\varphi(x,t)$ via
$$
\Phi_1 (y,s) =  e^ {\frac{3}{2}s}  \varphi_1(x,t) \,, \ \ \Phi_\nu (y,s) = e^ {\frac{s}{2}} \varphi_\nu( x, t) \,,
$$
and since $\p_{x_1} w = e^s \p_1W$, using \eqref{s(t)}, the estimate \eqref{eq:p1W:PhiU} follows.
\end{proof} 
 
 \subsubsection{The Lagrangian flow $X(\tilde x, t)$}
 We next introduce the Lagrangian flow $X$ associated to the transport velocity in \eqref{tvorticity}, namely $2\beta_1(\tilde v + \tilde u)$, as the solution to 
\begin{subequations} 
\label{X-flow}
\begin{align} 
 \p_tX(\tilde x, t) & = 2\beta_1(\tilde v +\tilde u)(X(\tilde x, t),t)\,, \qquad  t \in [-\eps, T_*] \,,  \\
 X(\tilde x, -\eps) &= \tilde x \,.
\end{align} 
\end{subequations} 
Note that the flow $X(\tilde x, t)$ is related to the flow $\varphi(x,t)$ given in \eqref{phi-flow}, via the transformation
\begin{align} 
\varphi_1(x,t) = X_1(\tilde x, t) - f(\check X (\check {\tilde x}, t),t) \,, \ \ \varphi_\nu(x,t) = X_\nu(\tilde x,t) \,,  \label{X-to-phi}
\end{align} 
and that $X(\tilde x, t)$ is related to the flow $\Phi(y,s):=\pu(y,s)$  by
\begin{align} 
\Phi_1 (y,s) =  e^ {\frac{3}{2}s} (X_1(\tilde x, t) - f(\check X (\check {\tilde x}, t),t)) \,, \ \ \Phi_\nu (y,s) = e^ {\frac{s}{2}} X_\nu(\tilde x, t)  \,. \label{good-vlad-260}
\end{align}

In this subsection we obtain three results, which play an important role in the proof of vorticity creation: the first is an estimate on $\abs{\nabla_{\! \tilde x} X(\cdot,t) - \Id}$, cf.~\eqref{deformation-bounds}; the second is a precise bound on the label $\tilde x_0$ such that $X(\tilde x_0,t) \to 0$ as $t\to T_*$ (recall that $0$ is the location at which the first singularity occurs), cf.~Lemma~\ref{lem:initial:particle:location}; the third result is a precise lower bound on $-\int_{-\eps}^{T_*} \p_{\tilde x_1} \tilde w \circ X$, cf.~Lemma~\ref{lem:dont:read:this}. 

First, we estimate the deformation rate of the flow $X$ on the time interval $[-\eps,T_*]$. The evolution of $\nabla_{\! \tilde x} X$ is given by
\begin{align} 
\frac{d}{dt} \p_{\tilde x _j} X_i  & = 2 \beta_1 \left(\p_{\tilde x _k} (\tilde v_i+ \tilde u_i) \circ X\right) \p_{\tilde x_j}X_k \label{eq:DX0}  \,.
\end{align} 
We note that using the bounds \eqref{eq:V:bnd}, the argument given in \eqref{light-buzz}--\eqref{buzzed-good-bro}, together with the identical argument given in Section 13 of \cite{BuShVi2019b}, we may show that there exists a universal constant  $C \geq 1 $ (in particular, $\eps$-independent) such that
\begin{align} 
 \tfrac{1}{C}  \le \abs{ \nabla_{\!\tilde x} X}  \le C  \,. \label{muy-culo-gordo}
\end{align} 
The bound \eqref{muy-culo-gordo} can however be made sharper, and we show (cf.~\eqref{deformation-bounds} below) that $\abs{\nabla_{\! \tilde x} X - \Id} \leq  \eps^{\frac{1}{20}}$ uniformly on $[-\eps, T_*)$. In order to prove this, we appeal to \eqref{eq:DX0}, from which we subtract $\Id_{ij}$ and then we contract with $\partial_{ \tilde x_j} X_i - \Id_{ij}$, to obtain that 
\begin{align} 
\tfrac{d}{2dt} \abs{\nabla_{\!\tilde x}X-\Id}^2 =
(\partial_{ \tilde x_j} X_i - \Id_{ij}) {\mathcal S}_{ik} (\partial_{ \tilde x_j} X_k - \Id_{kj}) + {\mathcal S}_{ij} (\partial_{ \tilde x_j} X_i - \Id_{ij})
\label{eq:matrix:ODE}
\end{align}  
we have introduced the notation
\begin{align*}
 {\mathcal S}_{ik} = 2 \beta_1 \left(\p_{\tilde x _k} (\tilde v_i+ \tilde u_i) \circ X\right)
\end{align*}
and for a matrix $A_{ij}$ we denote the Euclidean norm as $\abs{A}^2 = A_{ij} A_{ij}$.  Because of \eqref{usigma-sheep}, which implies that for a vector field $b$ we have $b\cdot \nabla_{\!\tilde x} \tilde u_j = \check b \cdot \check \nabla_{\! x} \mru_j + \Jcal b\cdot \Ncal \p_{x_1} \mru_j$, using the relation \eqref{X-to-phi} between the $\tilde x$ and $x$ Lagrangian trajectories $X$ and respectively $\varphi$, and appealing to \eqref{tildeu-dot-T}-\eqref{tildeu-dot-N}, we note that the following identities hold
\begin{subequations}
\label{eq:tilde:geometric:deformation}
\begin{align}
\p_{\Ncal} \tilde u \cdot \Ncal \circ X &= 
\Jcal \p_{x_1} \mru \cdot \Ncal \circ \varphi - \tfrac 12 \Ncal_\gamma \p_{x_\gamma} (w+z) \circ \varphi + \Ncal_\gamma \mru\cdot \Ncal_{,\gamma} \circ \varphi
\\
\p_{\Ncal} \tilde u \cdot \Tcal^\nu \circ X &= 
\Jcal \p_{x_1} a_\nu \circ \varphi -  \Ncal_\gamma \p_{x_\gamma} a_\nu \circ \varphi + \Ncal_\gamma \mru\cdot \Tcal^\nu_{,\gamma} \circ \varphi
\\
\p_{\Tcal^\mu} \tilde u \cdot \Ncal \circ X &= 
\tfrac 12 \Tcal^\mu_\gamma \p_{x_\gamma} (w+z)  \circ \varphi - \Tcal^\mu_\gamma \mru\cdot \Ncal_{,\gamma} \circ \varphi
\\
\p_{\Tcal^\mu} \tilde u \cdot \Tcal^\nu \circ X &= 
\Tcal^\mu_\gamma \p_{x_\gamma} a_\nu  \circ \varphi - \Tcal^\mu_\gamma \mru\cdot \Tcal^\nu_{,\gamma} \circ \varphi
\end{align}
\end{subequations}
The first term on the right side of the first line of the above list has the worst estimate when time integrated, cf.~\eqref{buzzed-good-bro2}. Indeed for all the other terms in the above list, by appealing to the bootstrap assumptions \eqref{eq:support}--\eqref{eq:A_bootstrap} and the estimate \eqref{f-bounds}, we may deduce that their time integrals are $\OO(\eps)$. Combining these estimates we deduce that 
\begin{align}
\int_{-\eps}^{T_*} \abs{(\nabla_{\!\tilde x} \tilde u) \circ X} dt' \les \eps^{\frac{1}{18}} 
\label{eq:nabla:tilde:u}
\end{align}
Similarly, using the relations~\eqref{eq:v:x:t:def}, \eqref{v_ansatz}, \eqref{X-to-phi}, and the estimate \eqref{eq:V:bnd} we obtain that the time integral of $\abs{(\nabla_{\!\tilde x} \tilde v) \circ X}$ is $\OO(\eps)$. Summarizing, we have that the matrix appearing on the right side of \eqref{eq:matrix:ODE} satisfies
\begin{align}
\label{eq:matrix:ODE:stretch}
\int_{-\eps}^{T_*} \abs{\mathcal S} dt' \les \eps^{\frac{1}{18}} \,.
\end{align}
Using that  $\nabla_{\!\tilde x} X|_{t=-\eps}  = \Id$, from \eqref{eq:matrix:ODE}, \eqref{eq:matrix:ODE:stretch} and ODE type bounds we deduce that
\begin{align}
\label{deformation-bounds}
\sup_{t\in[-\eps,T_*]} \abs{ \nabla_{\!\tilde x} X(t) - \Id} \les e^{\int_{-\eps}^{t} \abs{\mathcal S} dt'} - 1  \les \eps^{\frac{1}{18}} e^{\eps^{\frac{1}{18}}} \leq \eps^{\frac{1}{20}}.
\end{align}
The above bound is merely a quantitative version of \eqref{muy-culo-gordo}; it will be used in the proof of Theorem~\ref{thm:vorticity:creation}.

\section{$L^ \infty $ bounds for specific vorticity}
\label{sec:vorticity:sound}
We now establish bounds to solutions $ \mrg $ of the specific vorticity equation \eqref{sv-ss}

From \eqref{svorticity_sheep} and \eqref{11bottlesdownto7}, we deduce that the normal and tangential components of the vorticity  satisfy 
\begin{subequations}
\label{sv-ss}
\begin{align}
\p_t (\mrg \cdot \Tcal^2)  +    \Vcal \cdot \nabla_{\! x} (\mrg \cdot \Tcal^2)  &= 
\FFF_{21} (\mrg \cdot \Ncal) + \FFF_{2\mu}(\mrg \cdot \Tcal^\mu) + \mathcal{G} _2 \\
\p_t (\mrg \cdot \Tcal^3)   +    \Vcal \cdot \nabla_{\! x} (\mrg \cdot \Tcal^3)  &=
\FFF_{31} (\mrg \cdot \Ncal) + \FFF_{3\mu}(\mrg \cdot \Tcal^\mu) + \mathcal{G} _3
\end{align}
\end{subequations}
where the transport velocity $\Vcal$ is defined by \eqref{eq:fucked:up:v}, and 
\begin{subequations}
\label{ssvort-force}
\begin{align}
\FFF_{21} &=  \Ncal \cdot \p_t \Tcal^2      + 2\beta_1 \dot{Q}_{ij} \Tcal^2_i \Ncal_j     + \Vcal_\nu ( \Ncal\cdot \Tcal^2_{,\nu})
  + 2\beta_1  \Ncal_{\nu} \p_{x_\nu} a_2
- 2\beta_1  \Ncal_\nu \mru \cdot  \Tcal^2_{,\nu}   
\\
\FFF_{22} &= 2\beta_1  \Tcal^2_{\nu} \p_{x_\nu} a_2   - 2\beta_1  \Tcal^2_\nu \mru \cdot  \Tcal^2_{,\nu} 
\\
\FFF_{23} &=   \Tcal^3 \cdot \p_t \Tcal^2   + 2\beta_1 \dot{Q}_{ij} \Tcal^2_i \Tcal_j^3    
 \Vcal_\nu (  \Tcal^3\cdot \Tcal^2_{,\nu})
  + 2\beta_1\Tcal^3_{\nu} \p_{x_\nu} a_2
  - 2\beta_1  \Tcal^3_\nu \mru \cdot  \Tcal^2_{,\nu} 
\\
\FFF_{31} &=  \Ncal \cdot \p_t \Tcal^3       + 2\beta_1 \dot{Q}_{ij} \Tcal^3_i \Ncal_j  + \Vcal_\nu ( \Ncal\cdot \Tcal^3_{,\nu}) 
  + 2\beta_1 \Ncal_{\nu} \p_{x_\nu} a_3
   - 2\beta_1  \Ncal_\nu \mru \cdot  \Tcal^3_{,\nu}  
   \\
\FFF_{32} &=    \Tcal^2 \cdot \p_t \Tcal^3   + 2\beta_1 \dot{Q}_{ij} \Tcal^3_i \Tcal_j^2   
 + \Vcal_\nu ( \Tcal^2\cdot \Tcal^3_{,\nu})
  + 2\beta_1 \Tcal^2_{\nu} \p_{x_\nu} a_3
   - 2\beta_1  \Tcal^2_\nu \mru \cdot  \Tcal^3_{,\nu}  
\\
\FFF_{33} &=  2\beta_1  \Tcal^3_{\nu} \p_{x_\nu} a_3 - 2\beta_1 \Tcal^3_\nu \mru \cdot  \Tcal^3_{,\nu}  \,,
 \end{align} 
\end{subequations}
and
\begin{subequations}
\label{ssvort-Gforce}
\begin{align}
\mathcal{G} _2 & = \tfrac{\alpha }{\gamma}\tfrac{ \mathring \sigma}{ \mathring \rho} 
(\p_{\Tcal^3} \mathring \sigma \p_{\Ncal} \mathring \scal - \p_{\Ncal} \mathring \sigma \p_{\Tcal^3} \mathring \scal)
+  \tfrac{\alpha }{\gamma}\tfrac{ \mathring \sigma}{ \mathring \rho} 
 \Tcal^2_1 f,_\nu ( \nabla_{\! x} \mathring \sigma\! \times\!  \nabla_{\!  x}\mathring \scal )_\nu \\
\mathcal{G} _3 & =  \tfrac{\alpha }{\gamma}\tfrac{ \mathring \sigma}{ \mathring \rho} 
(\p_{\Ncal} \mathring \sigma \p_{\Tcal^2} \mathring \scal - \p_{\Tcal^2} \mathring \sigma \p_{\Ncal} \mathring \scal)
+  \tfrac{\alpha }{\gamma}\tfrac{ \mathring \sigma}{ \mathring \rho} 
\Tcal^3_1 f,_\nu ( \nabla_{\! \tilde x} \mathring \sigma\! \times\!  \nabla_{\!  \tilde x}\mathring \scal )_\nu \,,
\end{align} 
\end{subequations}
and from \eqref{tangent},  $\Tcal^\mu_1 = {\frac{f,_\mu}{\Jcal}} $.
 
\begin{proposition}[Bounds on specific vorticity]\label{prop:vorticity}
For $-\eps \le t < \tau(T_*)$,
\begin{align} \label{svort-bound}
\snorm{\mathring \zeta \circ \varphi (\cdot ,t) -\mathring \zeta(\cdot ,-\eps) }_{L^\infty}  \le \eps^ {\frac{1}{20}} 
\, .
\end{align} 
\end{proposition} 
\begin{proof}[Proof of Proposition \ref{prop:vorticity}]
By the transformations \eqref{tildeu-dot-T}, \eqref{a_ansatz}, and \eqref{U-trammy} together with the
bootstrap bounds   \eqref{eq:A_bootstrap}, \eqref{eq:US_est}, Lemma \ref{lem:BBS}, we have that
\begin{align} 
\norm{\mru}_{L^ \infty } \les M^ {\frac{1}{4}} \,, \ \   \norm{\p_{x_\nu}(\mru \cdot \Ncal)}_{L^ \infty } \les 1 \,, \ \ 
\norm{\p_{x_\nu} a}_{L^ \infty }  \le  M \eps^ {\frac{1}{2}}\, , \ \ \norm{\Vcal}_{L^\infty} \les M^{\frac 14}\,.\label{boomical}
\end{align} 
Hence, these bounds, together with \eqref{eq:dot:Q} and Lemma \ref{lem:BBS} yields the following bounds on the forcing functions defined in \eqref{ssvort-force} 
\begin{align} 
 \snorm{ \FFF _{ij}}_{L^\infty} \les 1  \ \ \text{ for } \ i,j \in \{1,2,3\} \,. \label{force-bound}
\end{align} 
where we have used powers of $\eps$ to absorb powers of $M$.

Now, from the definitions \eqref{usigma-sheep},
\eqref{sv-sheep},  we have that
\begin{align} 
\mathring \rho(x,t) \mathring \sigma( x, t))^ {\sfrac{1}{\alpha }}  \mathring \zeta(x, t) = \tilde \rho( \tilde x, t) \tilde \zeta( \tilde x, t) = \tilde \omega ( \tilde x, t) 
= \operatorname{curl} _{ \tilde x}  \tilde u ( \tilde x, t)  = \operatorname{curl} _{ \tilde x} \mathring u (x, t) \,,
\notag
\end{align} 
and
\begin{align} 
\operatorname{curl}_{\tilde x} \mru \cdot \Ncal & = \Tcal^2_j  \p_{\tilde x_j} \mru     \cdot \Tcal^3 -    \Tcal^3_j \p_{\tilde x_j} \mru  \cdot \Tcal ^2
\notag \\
&   = \Tcal^2_\nu  \p_{ x_\nu} \mru  \cdot \Tcal^3 -    \Tcal^3_\nu \p_{ x_\nu} \mru \cdot \Tcal ^2 
\notag  \\
&   = \Tcal^2_\nu  \p_{ x_\nu} a_3 - \Tcal^2_\nu  \mru      \cdot  \Tcal^3_{,\nu}
-  \Tcal^3_\nu  \p_{ x_\nu} a_2 + \Tcal^3_\nu  \mru     \cdot  \Tcal^2_{,\nu} \,. \label{boom3}
\end{align} 
from which it follows that
\begin{align} 
 \mathring \zeta \cdot \Ncal = 
  \frac{ \Tcal^2_\nu  \p_{ x_\nu} a_3 - \Tcal^2_\nu  \mru     \cdot  \Tcal^3_{,\nu}
-  \Tcal^3_\nu  \p_{ x_\nu} a_2 + \Tcal^3_\nu  \mru     \cdot  \Tcal^2_{,\nu}}{ \mathring \rho}
   \,. \label{boom2}
\end{align} 

It follows from \eqref{gerd}, Lemma \ref{lem:BBS}, \eqref{density-final0}, \eqref{boomical}, and \eqref{boom2}, we have that
\begin{align} 
\sabs{ \mathring \zeta \cdot \Ncal} \les M^{\frac14}\eps +M\eps^ {\frac{1}{2}} \les \eps^{\frac13} \,,  \label{boom4}
\end{align} 
assuming $\eps$ is taken sufficiently small.

We define
$$
\overline {\mathcal{F}} _{ij} = \mathcal{F} _{ij} \circ \varphi_{x_0}\,, \ \  \overline {\mathcal{G}}_\mu = \mathcal{G}_\mu \circ \varphi_{x_0}\,, \ \ 
\mathcal{Q} _1 = (  \mrg \cdot \Ncal) \circ  \varphi_{x_0}\,, \ \ 
\mathcal{Q} _2 = (  \mrg \cdot \Tcal^2) \circ  \varphi_{x_0}\,, \ \ 
\mathcal{Q} _3 = (  \mrg \cdot \Tcal^3) \circ   \varphi_{x_0}\,,
$$
Then,  \eqref{sv-ss} is written as the following system of ODEs:
 \begin{align*} 
  \p_t \mathcal{Q} _2   =   \overline {\mathcal{F}} _{2 j} \mathcal{Q} _j  + \overline{ \mathcal{G} }_2 \,, \ \
   \p_t \mathcal{Q} _3     = \overline {\mathcal{F}} _{3 j} \mathcal{Q} _j + \overline{ \mathcal{G} }_3 \,.
 \end{align*} 
Hence,
\begin{align} 
\tfrac{1}{2} \tfrac{d}{dt} \left(  \mathcal{Q}_2^2 +  \mathcal{Q}_3^2 \right) 
& =     \overline {\mathcal{F}} _{\nu \mu} \mathcal{Q}_\nu  \mathcal{Q}_\mu +  \overline {\mathcal{F}} _{\mu 1}  \mathcal{Q}_\mu \mathcal{Q}_1
+ \mathcal{Q}_\mu \overline{ \mathcal{G} }_\mu \,.
 \label{D-bound1}
\end{align} 
Now, we set $ \mathcal{Y} = ( \mathcal{Q} _2^2+ \mathcal{Q} _3^2)^ {\frac{1}{2}} $.  
Using \eqref{force-bound} and \eqref{boom4}, we see from \eqref{D-bound1} that
\begin{align*} 
  \tfrac{d}{dt} \mathcal{Y} 
& \les    \mathcal{Y}  +  \eps^{\frac{1}{3}} 
+   \abs{\overline{ \mathcal{G} }_2} + \abs{\overline{ \mathcal{G} }_3}   \,,
\end{align*} 
and hence by Gronwall's inequality, 
\begin{align} 
\abs{\mathcal{Y}(t)-\mathcal{Y}(-\eps)}& \les (e^ {\int_{-\eps}^tC\,dt'} -1)   \mathcal{Y} (- \eps) + e^ {\int_{-\eps}^tC\,dt'}\int_{-\eps}^t  (\eps^{\frac{1}{3}} 
+   \abs{\overline{ \mathcal{G} }_2} + \abs{\overline{ \mathcal{G} }_3}  ) dr \notag\\
& \les \eps   \mathcal{Y} (- \eps) + \int_{-\eps}^t (\eps^{\frac{1}{3}} 
+   \abs{\overline{ \mathcal{G} }_2} + \abs{\overline{ \mathcal{G} }_3}  ) dr\,.
 \label{D-bound2}
\end{align} 
where we used the bound $t-\eps\leq \tau(T_*) \leq 2\eps$ from \eqref{eq:speed:bound}.

We now prove that $\int_{-\eps}^t \overline{ \mathcal{G} }_\mu(r) dr$ is bounded for all $t \ge -\eps$ such that  $t <  \tau(t) $. First note that by \eqref{s_ansatz} and \eqref{eq:S_bootstrap}, we see that
\begin{align} 
\snorm{  \nabla_{\! x} \mathring \scal( \cdot , t)}_{ L^ \infty } \les \eps^ {\frac{1}{8}}  \,,
  \label{even-better-mofo-2}
\end{align} 
so it remains for us to bound $ \exp \int_{-\eps}^{t} \abs{  \p_{\Tcal^\mu} \mathring \sigma \circ \varphi } dt'$ and 
$ \exp \int_{-\eps}^{t} \abs{  \p_{\Ncal} \mathring \sigma \circ \varphi} dt'$.   
Using the identities
\begin{align*}
(\Ncal \cdot \nabla_{\!\tilde x})\mathring \sigma  = \p_{x_1} \mathring\sigma \Jcal + \Ncal_\mu\p_{x_\mu} \mathring \sigma \text{ and }
(\Tcal^\nu \cdot \nabla_{\!\tilde x})\mathring \sigma    =  \Tcal^\nu_\mu\p_{x_\mu} \mathring \sigma \,,
\end{align*} 
and \eqref{tildeu-dot-T}, we see that
\begin{align*} 
 \p_{\Ncal} \mathring \sigma &= \p_{x_1} \mathring u \cdot \Ncal  \Jcal  - \p_{x_1} z  \Jcal
 + \Ncal_\mu\p_{x_\mu}( \mathring u \cdot \Ncal) - \Ncal_\mu\p_{x_\mu} z  \,,\\
  \p_{\Tcal^\nu} \mathring \sigma &=  \Tcal^\nu_\mu\p_{x_\mu} (\mathring u \cdot \Ncal) - \Tcal^\nu_\mu\p_{x_\mu}z \,.
\end{align*} 
From \eqref{z_ansatz}, \eqref{U-trammy}, \eqref{eq:Z_bootstrap}, and \eqref{eq:US_est}, we find that
\begin{align} 
\snorm{  \p_{\Tcal^\nu} \mathring \sigma }_{ L^ \infty }  \les 1 \,, \label{snorky1}
\end{align} 
and additionally with \eqref{buzzed-good-bro2}, we see that
\begin{align} 
 \int_{-\eps}^t \abs{  \p_{\Ncal} \mathring \sigma \circ \varphi } dt' \les  \eps^ {\frac{1}{18}}  \,. \label{snorky2}
\end{align} 
The estimates \eqref{even-better-mofo-2}, \eqref{snorky1}, and \eqref{snorky2}  together with \eqref{Sigma-bound} and \eqref{density-final0}
show that
\begin{align} 
\int_{-\eps}^t \abs{\overline{ \mathcal{G} }_\mu(s)} ds\les  \eps^ {\frac{1}{18}}  \,.
\label{snorky3}
\end{align} 

From \eqref{D-bound2} and \eqref{snorky3}, we have that
\begin{align*}
\abs{\mathcal{Q}_2(t)-\mathcal{Q}_2(-\eps)} +  \abs{\mathcal{Q}_3(t)-\mathcal{Q}_3(-\eps)}  \les\eps ( \abs{\mathcal{Q}_2(-\eps)} +  \abs{\mathcal{Q}_3(-\eps)})+\eps^ {\frac{1}{18}} 
\end{align*}
uniformly for all labels $x_0$. Since $\Ncal, \Tcal^2,\Tcal^3$ form an orthonormal basis, the above estimate and \eqref{boom4}, 
implies that \eqref{svort-bound} holds. 
\end{proof}

\section{Vorticity creation}
\label{sec:creation}

We analyze vorticity creation (see~Theorem~\ref{thm:vorticity:creation}) through the evolution of the specific vorticity vector $\tilde \zeta$ in $\tilde x$ variables, given in equation \eqref{tvorticity}. 
For this purpose, we recall that the Lagrangian flow $X$ associated to the transport velocity in \eqref{tvorticity}, was defined in \eqref{X-flow} above. Before turning to Theorem~\ref{thm:vorticity:creation}, we establish two preliminary results associated to the flow $X$, which play an important role in the proof of vorticity creation: the first  is a precise bound on the label $\tilde x_0$ with the property that $X(\tilde x_0,t) \to 0$ as $t\to T_*$, cf.~Lemma~\ref{lem:initial:particle:location}; the second is a precise lower bound on the amplification factor 
$-\int_{-\eps}^{T_*} \p_{\tilde x_1} \tilde w (X(\tilde x_0, t),t) dt$, cf.~Lemma~\ref{lem:dont:read:this}. 

\subsection{The blowup trajectory and a bound on the amplification factor}
We obtain an estimate for the position of the particle $\tilde x_0$, which is carried by the flow $X(\cdot,t)$ to the blowup location $\tilde x=0$ as $t\to T_*$. 
\begin{lemma}[Initial location of particle trajectory leading to blowup]
\label{lem:initial:particle:location}
With the flow $X$ defined by \eqref{X-flow}, let  $X_{\tilde x_0}(t)$ denote the trajectory  which 
emanates from the point $\tilde x_0$. If $\lim_{t\to T_*}  X_{\tilde x_0}(t) =0$, then 
\begin{align}
 \abs{ (\tilde x_0)_1 - \beta_3 \kappa_0 \eps} \leq 5 \eps^{\frac 76}  \,, \qquad \sabs{\check {\tilde x}_0} \le 5 \eps^{\frac 76}  \,.   \label{miller-highlife}
\end{align} 
\end{lemma}
\begin{proof}[Proof of Lemma~\ref{lem:initial:particle:location}]
We consider the trajectory $X_{\tilde x_0}(t)$ for which $X_{\tilde x_0}(T_*)=0$ and for notational simplicity, we drop the subscript $\tilde x_0$ and use
$X(t)$ to denote this trajectory. The main idea is that the initial position of the particle $X(t)$, i.e. $\tilde x_0$, may be computed by passing $t\to T_*$ in the identity $X(t) - \tilde x_0 = \int_{-  \eps}^t \partial_t X(t') dt'$, leading to 
\begin{align}
\tilde x_0 = - \int_{-\eps}^{T_*} \p_t X(t') dt' \,. 
\label{eq:annoying:0}
\end{align}
By revisiting the right side of \eqref{X-flow}, we obtain a sharp estimate for the right side of the above identity. 

For convenience, in analogy to \eqref{tildeu-dot-T} we define
\begin{align}  
\tilde w =  \tilde u \cdot \Ncal + \tilde \sigma  \,, \qquad \tilde z =  \tilde u \cdot \Ncal - \tilde \sigma  \,, \qquad  \tilde a_\nu = \tilde u \cdot \Tcal^\nu \,. \label{tildewz}
\end{align} 
We note that $\p_{\tilde x_1} \tilde w (\tilde x, t) = \p_{x_1} w(x,t)$.
Furthermore, using \eqref{tildev} 
we have that
\begin{align} 
\p_t X & = 2\beta_1( \tilde v + \tilde u \cdot \Ncal \, \Ncal +  \tilde u \cdot \Tcal^\nu \Tcal ^\nu) \circ X \notag \\
&= 2\beta_1 \dot Q X - 2\beta_1 R^T \dot \xi + \beta_1(  \tilde w \Ncal  +   \tilde z \Ncal + 2 \tilde a_\nu \Tcal ^\nu) \circ X\,.
\label{eq:annoying:1}
\end{align} 
First we note that using that $\dot Q$ is skew symmetric, that $X(T_*) = 0$,  appealing to the bounds \eqref{eq:acceleration:bound}, \eqref{eq:A_bootstrap},  \eqref{eq:US_est},  together with \eqref{T*-bound}, from the Gr\"onwall inequality on $[t,T_*]$ we obtain that
\begin{align} 
\abs{X(t)} \les  M^\frac 14  \eps \,.  \label{coors-light-sucks}
\end{align} 
This estimate his however not sharp enough; to do better, we need to carefully bound the term $2\beta_1 R^T \dot \xi$ on the right side of \eqref{eq:annoying:1}. Note cf.~\eqref{def_V} we have that $(R^T \dot \xi)_i = - V_i(0,s)$. Then, evaluating \eqref{eq:gW} and \eqref{eq:hW} at $y=0$, using definition of the function $f$ and our constraints \eqref{eq:constraints}, we deduce 
\begin{align*}
2\beta_1  (R^T \dot \xi)_1  = \kappa + \beta_2 Z^0 - \tfrac{1}{\beta_\tau} e^{-\frac s2} G_W^0 
\qquad \mbox{and} \qquad 
2\beta_1  (R^T \dot \xi)_\mu = 2\beta_1 A_\mu^0 - \tfrac{1}{\beta_\tau} e^{\frac s2} h_W^{\mu,0}
\end{align*}
in analogy to \eqref{eq:hj:0:a} and \eqref{eq:GW:0:a}. Using the $\dot{\kappa}$ estimate in \eqref{eq:acceleration:bound}, the $Z$ and $A$ estimates in \eqref{eq:Z_bootstrap} and \eqref{eq:A_bootstrap}, and the bound \eqref{eq:GW:hW:0} for $G_W^0$ and $h_W^{\mu,0}$, which is a consequence of the bootstrap assumptions, we deduce that 
\begin{align}
\abs{ 2\beta_1  (R^T \dot \xi)_1 - \kappa_0} \les  M \eps
\qquad \mbox{and} \qquad 
\abs{ 2\beta_1  (R^T \dot \xi)_\mu} \les M \eps^{\frac 45}
\label{eq:annoying:2}
\end{align}
since $1 - \frac{5}{2m-7} > \frac 45$ for $m\geq 18$. Returning to \eqref{eq:annoying:1}, from \eqref{eq:dot:Q}, \eqref{eq:Z_bootstrap}, \eqref{eq:A_bootstrap} and \eqref{coors-light-sucks}, we have that
\begin{align}
\abs{2\beta_1 \dot Q X +  \beta_1( \tilde z \Ncal + 2 \tilde a_\nu \Tcal ^\nu) \circ X} \les M^{\frac 94} \eps^{\frac 32} + M \eps \les M \eps \,.
\label{eq:annoying:3}
\end{align}
Lastly, by \eqref{X-to-phi} we have $\tilde w \circ X = w\circ \varphi$, and by \eqref{w_ansatz} we have $w = \kappa + e^{-\frac s2} W$. Thus, by also appealing to \eqref{eq:acceleration:bound}, \eqref{e:space_time_conv}, \eqref{eq:W_decay}, \eqref{e:bounds_on_garbage}, and the fact that by $\abs{\phi_{\nu\mu}(-\eps)} \leq \eps$ we have that $\abs{\Ncal(-\eps) - e_1}\les \eps$,  we obtain
\begin{align} 
\abs{\tilde w \Ncal \circ X - \kappa_0 e_1} \leq \abs{\kappa \Ncal -\kappa_0 e_1}  + e^{-\frac s2} \norm{W}_{L^\infty(\XXX(s))} \leq 3 \eps^{\frac 16} \,.
\label{eq:annoying:4}
\end{align}
By inserting the estimates \eqref{eq:annoying:2}--\eqref{eq:annoying:4} into the right side of \eqref{eq:annoying:1} we obtain that 
\begin{align}
\abs{\p_t X_1 + \beta_3 \kappa_0} \leq  4 \eps^{\frac 16}
\qquad \mbox{and} \qquad
\abs{\p_t X_\nu} \leq  4 \eps^{\frac 16}
\label{eq:annoying:5}
\end{align}
upon taking $\eps$ to be sufficiently small in terms of $M$, and recalling that $\beta_1 - 1 = - \beta_3$.
To conclude the proof of the lemma we simply combine \eqref{eq:annoying:0} with \eqref{eq:annoying:5} and the estimate $\abs{T_*} \leq  \eps^{\frac 32}$, as given by \eqref{T*-bound}.
\end{proof}

\begin{remark}
For the particle trajectory from Lemma~\ref{lem:initial:particle:location}, integrating \eqref{eq:annoying:5} from on $[t,T_*]$, as opposed to $[-\eps,T_*]$ as was done in \eqref{eq:annoying:0}, we obtain that 
\begin{align}
\abs{X_1(t)  - \beta_3 \kappa_0 e^{-s}} \leq 5 \eps^{\frac 16} e^{-s} 
\qquad \mbox{and} \qquad 
\abs{X_\nu(t)} \leq 5 \eps^{\frac 16} e^{-s}  \,.
\label{corona-premier-is-piss}
\end{align}
Here we have again used that using \eqref{eq:acceleration:bound}, \eqref{eq:beta:tau}, and  \eqref{good-vlad} we have that  $\abs{e^s(T_*-t) - 1} \leq 2 M \eps$.
\end{remark}

The second preliminary estimate in this subsection is a lower bound on $-\int_{-\eps}^{T_*} \p_{\tilde x_1} \tilde w \circ X$, as this quantity plays a key role in our proof of vorticity creation (cf.~the estimate for the term $I_1$ in Theorem~\ref{thm:vorticity:creation}).
\begin{lemma} 
\label{lem:dont:read:this}
With the flow $X$ defined by \eqref{X-flow}, let  $X_{\tilde x_0}(t)$ denote the trajectory which 
emanates from the point $\tilde x_0$.
If  $X_{\tilde x_0}(T_*) =0$ and the initial condition $W(y, - \log \eps)$ satisfies \eqref{sunny-days},  then 
\begin{align}
-\int_{-\eps}^{T_*} \p_{\tilde x_1} \tilde w(X_{\tilde x_0}(t),t)dt  \ge  \tfrac{1}{9}  \kappa_0^{-\frac 23} \eps^{\frac 13}  \,.   \label{el-burro-es-feliz}
\end{align} 
\end{lemma} 
\begin{proof}[Proof of Lemma~\ref{lem:dont:read:this}]
The proof of the lemma is based on two ideas: first, the time integral in \eqref{el-burro-es-feliz} is dominated by values of $t$ which are very close to $-\eps$, where we can relate $\p_{\tilde x_1} \tilde w$ to its initial datum; second, the flow $X(t)$ is related to the self-similar flow $\Phi_U$ via the relation \eqref{good-vlad-260}, which allows us to appeal to sharp bounds for $\p_1 W$ in estimating the contribution to \eqref{el-burro-es-feliz} for $t \gg - \eps$. We implement these ideas as follows.

We consider the trajectory $X_{\tilde x_0}(t)$ for which $X_{\tilde x_0}(T_*)=0$ and for notational simplicity, we drop the subscript $\tilde x_0$ and use
$X(t)$ to denote this trajectory. The associated self-similar initial datum variable $y_0$ is given via \eqref{x-sheep} and \eqref{eq:y:s:def} as
\begin{align}
y_0 = (\eps^{-\frac 32} ((\tilde x_0)_1 - f (\check{\tilde{x}}_0)), \eps^{-\frac 12} \check{\tilde{x}}_0)
\,.
\end{align}
Due to Lemma~\ref{lem:initial:particle:location} we know that $\tilde x_0$ satisfies \eqref{miller-highlife}, and since $\abs{\phi_{\mu\nu}(-\eps)} \leq \eps$, we deduce that
\begin{align}
\abs{(y_0)_1 - \beta_3 \kappa_0 \eps^{-\frac 12}}\leq 6 \eps^{-\frac 13} 
\qquad \mbox{and} \qquad 
\abs{(y_0)_\nu}\leq 5 \eps^{ \frac 23}
\,.
\label{eq:f-ing:annoying:0} 
\end{align}
Note that these bounds are set up precisely to account for the region specified in \eqref{sunny-days}.  In view of the precise estimates on the trajectory $X_{\tilde x_0}(t)$, we directly obtain sharp bounds on the self-similar Lagrangian flow $\pu^{y_0}(s)$ emanating from $y_0$. Indeed,  by the $\phi$ bound in \eqref{eq:speed:bound}, the relation between $\pu$ and $X$ in \eqref{good-vlad-260}, and the bounds \eqref{corona-premier-is-piss}, we have that 
\begin{align} 
(\beta_3 \kappa_0 - \eps^{\frac 17}) e^{\frac s2} \leq (\pu^{y_0})_1(s)  \le  (\beta_3 \kappa_0 + \eps^{\frac 17}) e^ {\frac{s}{2}} \,,
\qquad \mbox{and} \qquad
\abs{(\pu^{y_0})_\nu(s)} \le \eps^{\frac 17}  e^ {-\frac{s}{2}} \,.     \label{good-steve}
\end{align} 

Next, due to  \eqref{good-vlad-260} and \eqref{tildewz} we have that
\begin{align}
 \p_{\tilde x_1} \tilde w \circ X_{\tilde x_0}(t) = e^s \p_1 W \circ \pu^{y_0}(s)
 \label{eq:f-ing:annoying:0.5} 
\end{align}
with the usual relation between $t$ and $s$ from \eqref{eq:y:s:def}. Since $dt = \beta_\tau e^{-s} ds$, we thus have that the integral we need to estimate in \eqref{el-burro-es-feliz} may be rewritten as
\begin{align}
-\int_{-\eps}^{T_*} \p_{\tilde x_1} \tilde w(X_{\tilde x_0}(t),t)dt  
= - \int_{-\log \eps}^\infty \beta_\tau  \p_1 W \circ \pu^{y_0}(s) ds \,.  
\label{eq:f-ing:annoying:1} 
\end{align}
Recall cf.~\eqref{eq:beta:tau} that $1 - 2Me^{-s} \leq \beta_\tau \leq 1+2 M e^{-s}$, so that we just need to bound from below the integral of $- \p_1 W \circ \pu^{y_0}$.
The remainder of the argument mimics the proof of Lemma~\ref{lem:non:positivity}.

Fix  $y_0$ as in \eqref{eq:f-ing:annoying:0}, $s\in [-\log \eps, \infty)$, and thus fix a value of $\pu^{y_0}(s)$.
We trace the particle trajectory of  the flow $\mathcal{V}_W$ (not $\mathcal{V}_U$!) backwards in time, and write $\pw^{y_0'}(s) = \pu^{y_0}(s)$, where the initial datum $\pw^{y_0'}(s_0) = y_0'$ is  given  by the property that $\abs{y_0'} = \LLL$ if $s_0\ge -\log \eps$, and $\abs{y_0'} > \LLL$ if $s_0 = -\log \eps$. 
We then appeal to Lemma~\ref{lem:escape2} with $y_0'$ replacing $y_0$. The lemma is applicable on the interval $[s_0,s]$ since $\abs{y_0'} \geq \LLL$ and by \eqref{good-steve} we have  $\sabs{\check\Phi_W^{y_0'}(s)} = \sabs{\check \Phi_U^{y_0}(s)} \leq \eps^{\frac 17} e^{-\frac s2} \leq \eps^{\frac 12}$. By \eqref{eq:escape_from_SF}, we thus obtain that for any $s' \in [-\log \eps,s]$ we have the estimates 
\begin{align}
\abs{(\pw^{y_0'})_1 (s')}\geq \tfrac 34 \abs{(y_0)_1}e^{\frac{3(s'-s_0)}{2}}
\qquad \mbox{and} \qquad
\abs{\check\Phi_W^{y_0'}(s')}\leq M \eps^{\frac 12}
\,.
\label{eq:f-ing:annoying:2} 
\end{align}

Let us first consider the case that $\abs{y_0'} > \LLL$ and $s_0 = -\log \eps$. 
Based on \eqref{eq:f-ing:annoying:2} we now claim that $\abs{(y_0')_1} \leq 2 \kappa_0 \eps^{-\frac 12}$. If not, then by appealing to the first estimate in \eqref{good-steve}, we thus deduce that 
\begin{align*}
\tfrac 32 \beta_3 \kappa_0 e^{\frac s2} 
\geq  \abs{(\pu^{y_0})_1(s)} 
= \sabs{(\pw^{y_0'})_1 (s)}
\geq \tfrac 34 \abs{(y_0)_1} e^{\frac{3(s-s_0)}{2}} 
> \tfrac 32 \kappa_0 \eps^{-\frac 12} e^{s-s_0}  e^{\frac s2} \eps^{\frac 12} \geq  \tfrac 32 \kappa_0  e^{\frac s2}
\,,
\end{align*}
which is a contradiction, since $\beta_3 = \frac{\alpha}{1+\alpha} < 1$. Therefore, from the above argument  and the second bound in 
\eqref{good-steve}  evaluated at $s' = s_0$, we have that $\LLL = \eps^{-\frac{1}{10}} < \abs{(y_0')_1} \leq 2 \kappa_0 \eps^{-\frac 12}$, and $\abs{(y_0')_\nu} \leq M \eps^{\frac 12} \leq \eps^{\frac 13}$. Therefore, the point $y_0'$ exactly lies in the region stipulated in \eqref{sunny-days}, and so by Lemma~\ref{lem:IC:for:creation} in this case we have that 
\begin{align}
\p_1 W\left(\pw^{y_0'}(s_0),s_0\right) = \p_1 W(y_0',-\log \eps) \in \left[ -\tfrac 12 \abs{(y_0')_1}^{-\frac 23},  -\tfrac 14 \abs{(y_0')_1}^{-\frac 23} \right] 
 \label{eq:f-ing:annoying:3} 
\end{align}

Next, let us first consider the case that $\abs{y_0'} = \LLL$ and $s_0 > -\log \eps$.
In this case, instead of appealing to \eqref{sunny-days} we use the bootstrap \eqref{eq:bootstrap:Wtilde1} and as shown earlier in \eqref{perezoso-gordo} we deduce 
\begin{align} 
\p_1 W\left(\pw^{y_0'}(s_0),s_0\right) = \p_1 W(y_0',s_0)  \leq -\tfrac 12  \tilde \eta^{-\frac 13}(y) \leq -\tfrac 14 \abs{(y_0')_1}^{-\frac 23}  \,,
 \label{eq:f-ing:annoying:4} 
\end{align} 
where we used \eqref{eq:f-ing:annoying:2} with $s'=s_0$ in the last inequality. 

Having established \eqref{eq:f-ing:annoying:3} and \eqref{eq:f-ing:annoying:4}, we use the $\p_1 W$ evolution given in \eqref{euler_for_Linfinity:a} with $\gamma = (1,0,0)$,
and deduce that  
\[
\p_s  ( \p_1 W \circ \pw^{y_0'})  + \left( 1 + \beta_\tau   \Jcal \p_1 W \circ \pw^{y_0'} \right) ( \p_1 W \circ \pw^{y_0'})    =  F^{(1,0,0)}_W \circ \pw^{y_0'}
\,.
\]
Integrating this expression on $[s_0,s]$,   recalling that by definition we have $\pw^{y_0'}(s) = \pu^{y_0}(s)$, using that by \eqref{eq:f-ing:annoying:3} and \eqref{eq:f-ing:annoying:4} we have that $-\p_1W(y_0',s_0)>0$, by appealing to the $F^{(1,0,0)}_W$ estimate in \eqref{eq:forcing_W} and to the $\p_1 W$ bootstrap in \eqref{eq:W_decay}, we deduce   
\begin{align}
-\p_1 W(\pu^{y_0}(s),s) 
&= - \p_1 W(y_0',s_0) \exp\left(- \int_{s_0}^s 1 + \beta_\tau (\Jcal \p_1 W)\circ\pw^{y_0'}(s') ds' \right)\notag\\
&\quad - \int_{s_0}^s  F^{(1,0,0)}_W\circ\pw^{y_0'}(s')  \exp\left(- \int_{s'}^s 1 + \beta_\tau (\Jcal \p_1 W)\circ\pw^{y_0'}(s'') ds'' \right) ds' 
\notag\\
&\geq \tfrac 14 \abs{(y_0')_1}^{-\frac 23} e^{-(s-s_0)} \exp\left(- 3 \int_{s_0}^s \eta^{-\frac 13} \circ\pw^{y_0'}(s') ds' \right)\notag\\
&\quad - \int_{s_0}^s e^{-\frac{s'}{5}} \eta^{-\frac 13} \circ \pw^{y_0'}(s') e^{-(s-s')} \exp\left(3 \int_{s'}^s  \eta^{-\frac 13} \circ\pw^{y_0}(s'') ds'' \right) ds' 
 \label{eq:f-ing:annoying:5} 
\end{align}
Since $\abs{y_0'} \geq \LLL$ , by \eqref{eq:f-ing:annoying:2} we have
$$
3 \int_{s_0}^s \eta^{-\frac 13} \circ\pw^{y_0'}(s') ds'  
\leq 4 \abs{(y_0')_1}^{-\frac 23} \int_{s_0}^s e^{-(s'-s_0)} ds' \leq  \eps^{\frac{1}{16}} \,,
$$
and 
$$
\int_{s_0}^se^{-\frac{s'}{5}} \eta^{-\frac 13} \circ\pw^{y_0'}(s') ds'  e^{-(s-s')}
\leq 2 e^{-s}  \abs{(y_0')_1}^{-\frac 23} \int_{s_0}^s e^{\frac{4s'}{5}} e^{-(s'-s_0)} ds' \leq  10 \eps^{\frac 15} e^{-(s-s_0)}  \abs{(y_0')_1}^{-\frac 23} \,.
$$
Inserting these estimates into \eqref{eq:f-ing:annoying:5}, we deduce
\begin{align}
-\p_1 W(\pu^{y_0}(s),s) 
&\geq \tfrac 15 \abs{(y_0')_1}^{-\frac 23} e^{-(s-s_0)}
\,.
 \label{eq:f-ing:annoying:6} 
\end{align}
The bound \eqref{eq:f-ing:annoying:6} holds both in the case that $s_0>-\log \eps$ and $\abs{y_0'} = \LLL$, and also in the case that $s_0 = -\log \eps$ and $\abs{y_0'} > \LLL$ and $\abs{(y_0')_1} \leq 2 \kappa_0 \eps^{-\frac 12}$. The last observation is that in either case, the bound \eqref{eq:f-ing:annoying:6} implies
\begin{align}
-\p_1 W(\pu^{y_0}(s),s) 
\geq \tfrac 15 (2 \kappa_0 \eps^{-\frac 12})^{-\frac 23} e^{-(s-s_0)}
\geq \tfrac 18  \kappa_0^{-\frac 23} \eps^{\frac 13}  e^{-(s+\log \eps)}
\,.
 \label{eq:f-ing:annoying:7} 
\end{align}

Lastly, using \eqref{eq:f-ing:annoying:7} we bound from below the right side of \eqref{eq:f-ing:annoying:1}  and obtain
\begin{align*}
- \int_{-\log \eps}^\infty \beta_\tau  \p_1 W \circ \pu^{y_0}(s) ds
\geq  
\tfrac{1-2 M \eps}{8}  \kappa_0^{-\frac 23} \eps^{\frac 13} \int_{-\log \eps}^{\infty} e^{-(s+\log \eps)} ds 
\geq \tfrac{1}{9}  \kappa_0^{-\frac 23} \eps^{\frac 13} 
\end{align*}
which completes the proof. 
\end{proof}

\subsection{Vorticity creation from irrotational data}

We now return to the specific vorticity equation \eqref{tvorticity} which we shall now write as
\begin{align}
  \label{specific-vorticity2}
 \p_t \tilde \zeta - 2\beta_1\dot Q \tilde \zeta + 2 \beta_1(\tilde v + \tilde u) \cdot \nabla_{\! \tilde x}\tilde \zeta 
 =
 2\beta_1 \operatorname{Def} _{\! \tilde x} \tilde u \cdot \tilde\zeta + \tilde b  
 \qquad \text{ for } \quad t \in [-\eps, T_*)  \, 
\end{align}
where we use $\tilde b$ to denote the baroclinic term in $(\tilde x, t)$ variables:
\begin{align}
\tilde b=2\beta_1\tfrac{\alpha }{\gamma} \tfrac{\tilde \sigma}{\tilde \rho} \nabla_{\!\tilde x} \tilde\sigma \times  \nabla_{\!\tilde x} \tilde \scal  \,,
\label{eq:baroclinic:clinic}
\end{align}
and the (rate of) deformation tensor is defined by
$$
 \operatorname{Def} _{\! \tilde x} \tilde u  = \tfrac{1}{2} (\nabla_{\! \tilde x} \tilde u +\nabla_{\! \tilde x} \tilde u^T)
$$
which is  the symmetric part of the velocity gradient.  
In components, $( \operatorname{Def} _{\! \tilde x} \tilde u  \cdot  \tilde \zeta)_i =
\tfrac{1}{2}( \p_{\tilde x _j} \tilde u_i + \p_{\tilde x _i} \tilde u_j) \tilde \zeta_j$.

By definition of the $X_{\tilde x_0}(t)$ flow in \eqref{X-flow}, so that $X_{\tilde x_0}(-\eps) = \tilde x_0$ upon composing \eqref{specific-vorticity2} with $X_{\tilde x_0}(t)$ and denoting
\begin{align} 
\upzeta(\tilde x_0,t) = \tilde \zeta \circ X_{\tilde x_0}(t) \,, \qquad \mathsf{D}(\tilde x_0,t)  = 2\beta_1 \operatorname{Def}_{\tilde x} \tilde u \circ X_{\tilde x_0}(t) \,, \qquad \bcal(\tilde x_0,t) = \tilde b \circ X_{\tilde x_0}(t) \,, 
\label{lagrangian-variables}
\end{align} 
we have
\begin{align}
\tfrac{d}{dt} \upzeta =(2\beta_1 \dot Q + \mathsf{D}) \cdot \upzeta + \bcal \,.
\label{eq:BlackJack:0}
\end{align}
At this stage two observations are in order. First, due to \eqref{X-to-phi} we have that $\upzeta = \tilde \zeta \circ X = \mathring{\zeta} \circ \varphi$, so that the bound \eqref{svort-bound} translates into
\begin{align}
\abs{\upzeta(\tilde x_0,t) - \tilde \zeta (\tilde x_0,-\eps) } \leq \eps^{\frac{1}{21}} 
\label{eq:svort-bound}\,.
\end{align}
Second, we note that by \eqref{eq:tilde:geometric:deformation}, \eqref{eq:nabla:tilde:u}, \eqref{f-bounds}, and \eqref{ic-kappa0-phi0}, for any $(i,j) \neq (1,1)$ we have
\begin{align}
\int_{-\eps}^{T_*} \abs{\mathsf{D}_{ij}(t')} dt' \les M \eps \,,
\label{eq:BlackJack:1}
\end{align}
while for $(i,j) = (1,1)$ we have 
\begin{align}
\int_{-\eps}^{T_*} \abs{\mathsf{D}_{11}(t')} dt' \les \eps^{\frac{1}{18}} \,.
\label{eq:BlackJack:2}
\end{align}
We omit the detailed proofs of \eqref{eq:BlackJack:1} and \eqref{eq:BlackJack:2} but note that as already discussed in the paragraph below \eqref{eq:tilde:geometric:deformation}, only the time integral of $\abs{\p_{\Ncal} \tilde u \cdot \Ncal \circ X}$ is not $\OO(\eps)$; and since $\abs{\Ncal - e_1} \les \eps$, this corresponds to only the $(1,1)$ component of the $\mathsf{D}$ matrix as having a time integral which may be larger than $\OO(\eps)$. Taking into account also the $\dot{Q}$ estimate in \eqref{eq:dot:Q} we rewrite 
\begin{align}
2\beta_1 \dot Q + \mathsf{D} =: {\rm diag}(\mathsf{D}_{11},0,0) +  \mathsf{D}_{\rm small} =:  \mathsf{D}_{\rm main} +  \mathsf{D}_{\rm small}
\label{eq:BlackJack:3}
\end{align}
with 
\begin{align}
\int_{-\eps}^{T_*} \abs{\mathsf{D}_{\rm small}(t')} dt' \les M \eps \,.
\label{eq:BlackJack:4}
\end{align}

With this information, since $\mathsf{D}_{\rm main}$ is a diagonal matrix, we may write the solution of ODE \eqref{eq:BlackJack:0}  pointwise in $\tilde x_0$ as
\begin{align} 
\upzeta(\cdot, t) =e^{\int_{-\eps}^t \! \mathsf{D}_{\rm main}(\cdot,t')dt' } \tilde \zeta (\cdot,-\eps)  
+ \int_{-\eps}^t e^{\int_{t'}^t \! \mathsf{D}_{\rm main}(\cdot,t'')dt'' }\left( \bcal(\cdot, t')
+ \mathsf{D}_{\rm small}(\cdot,t')\cdot \upzeta(\cdot,t') \right) dt' 
\,,
\label{vort-solution}
\end{align} 
where in view of \eqref{eq:BlackJack:2}
\begin{align}
\abs{e^{\int_{t'}^t \! \mathsf{D}_{\rm main}(\cdot,t'')dt'' } - \Id} = \abs{ {\rm diag}\left( e^{\int_{t'}^t \! \mathsf{D}_{11}(\cdot,t'')dt'' }, 1, 1\right)-\Id}  
\les  \eps^{\frac{1}{18} }
\,.
\label{eq:BlackJack:5}
\end{align}
The solution formula \eqref{vort-solution}, along with the bounds \eqref{eq:svort-bound}, \eqref{eq:BlackJack:4}, and \eqref{eq:BlackJack:5} show that vorticity creation is essentially implied by (lower) bounds on $\int_{-\eps}^t \bcal(\cdot,t') dt'$. This is indeed the main idea in the proof of vorticity creation, which we establish next.

In the following theorem, we show that when the initial vorticity is zero, the Euler dynamics instantaneously creates vorticity, and that for appropriately chosen initial data, the vorticity remains non-trivial at the formation of the shock.

\begin{theorem}[Vorticity creation]
\label{thm:vorticity:creation}
Consider  $\tilde x_0$ such that the flow $X_{\tilde x_0}(t)$ converges to the blowup point $0$ as $t\to T_*$. More generally, consider any $\tilde x_0$ satisfying \eqref{miller-highlife}.
Suppose that the initial datum verifies \eqref{sunny-days}, and that the initial baroclinic torque at this point, $\tilde b(\tilde x_0,-\eps)$, is {\em non-trivial}. For example, this may be ensured by choosing 
\begin{align}
\p_{\tilde x_1}\tilde \scal(\tilde x_0,-\eps)= 0\, \quad \p_{\tilde x_3}\tilde \scal(\tilde x_0,-\eps)= 0 \,, \quad 
\p_{\tilde x_2} \tilde \scal(\tilde x_0,-\eps) < 0 \,.
\label{eq:initial:torque}
\end{align}
If the initial datum is irrotational, i.e. 
$\tilde \zeta (\tilde x, -\eps) = 0$ for all $\tilde x \in \mathbb{R}^3$, then vorticity is instantaneously
created, and remains non-vanishing in the neighborhood of the shock location $(\tilde x,t) = (0,T_*)$.
Quantitatively, with the choice \eqref{eq:initial:torque} we have that
\begin{align}
\sabs{\tilde \zeta(\tilde x,t)} \geq c_\alpha \kappa_0^{\frac 13 - \frac{1}{\alpha}} \eps^{\frac 13}  \sabs{\p_{\tilde x_2} \tilde \scal_0(\tilde x_0)} 
\label{eq:shock:vorticity}
\end{align}
for all $(\tilde x,t)$ in a small neighborhood of the shock location $(0,T_*)$, where $c_\alpha>0$ is a constant that only depends on $\alpha$.
\end{theorem}  
\begin{proof}[Proof of Theorem~\ref{thm:vorticity:creation}] 
As alluded to in the discussion preceding the Theorem, the proof is based on following the Lagrangian flow $X_{\tilde x_0}(t)$ which arrives at the shock location as $t\to T_*$, and study the vorticity production caused by the baroclinic torque term $\bcal$. We note that \eqref{eq:shock:vorticity} is proven by establishing this bound at $\tilde x = X_{\tilde x_0}(t)$ with $t\to T_*$, for one component of the vorticity vector, and arguing by continuity, the fact that the vorticity remains continuous all the way up to the blowup time ensures that the lower bound holds for $(\tilde x,t)$ in a neighborhood of $(0,T_*)$.

For simplicity of the presentation we provide a lower bound on the third component of the vorticity; this is why in assumption \eqref{eq:initial:torque}  we have chosen very specific gradient components for $\tilde \scal$ and $\tilde \sigma$. 
Recall the notation \eqref{lagrangian-variables}. Using that the initial datum is irrotational, from the solution formula \eqref{vort-solution}, the bounds \eqref{eq:svort-bound}, \eqref{eq:BlackJack:4},  \eqref{eq:BlackJack:5} and the fact that the matrix $\mathsf{D}_{\rm main}$ only has a nontrivial $(1,1)$ entry, we obtain that
\begin{align} 
\abs{\upzeta_3(\tilde x_0, t) - \int_{-\eps}^t \bcal_3(\tilde x_0, t') dt'} \les (1+ \eps^{\frac{1}{18}}) \eps^{\frac{1}{21}} M \eps \les \eps 
\,.
\label{eq:BlackJack:6}
\end{align} 
The remainder of the proof consists of analyzing the time integral of $\bcal_3 (\tilde x_0,t) = \tilde b(X_{\tilde x_0}(t),t)$.

Let us denote the cofactor matrix associated to $\nabla_{\!\tilde x}  X$ and its Jacobian determinant, respectively, by
$$
\mathcal{B} (\tilde x, t) = \operatorname{Cof} (\nabla_{\!\tilde x} X)\,, \qquad J(\tilde x,t) = \det ( \nabla_{\!\tilde x} X) \,,
$$
so that 
$$
( \nabla_{\!\tilde x} X)^{-1} = J^{-1} \mathcal{B} \,.
$$
Two components of the cofactor matrix that we shall make use of  are given by
\begin{align*} 
\mathcal{B} ^2_2 & = \p_{\tilde x_2} X_2( \p_{\tilde x_1} X_1  \p_{\tilde x_3} X_3 - \p_{\tilde x_1} X_3 \p_{\tilde x_3} X_1   )  \,, \\
\mathcal{B} ^2_1 & = \p_{\tilde x_1} X_2( \p_{\tilde x_3} X_1  \p_{\tilde x_2} X_3 - \p_{\tilde x_2} X_1 \p_{\tilde x_3} X_3   ) \,.
\end{align*} 
From \eqref{deformation-bounds}, we see that
\begin{align} 
\abs{J - 1} \les \eps^{\frac{1}{20}} \,, 
\qquad
\abs{\mathcal{B} ^2_2 - 1} \les \eps^{\frac{1}{20}} \,, 
\qquad \mbox{and} \qquad 
\abs{\mathcal{B} ^2_1} \les  \eps^{\frac{1}{10}}  \,. \label{cowbell}
\end{align} 
Then,  transport equation
 \eqref{eq:entropy3} shows that
\begin{align} 
\tilde\scal \circ X_{\tilde x}(t) = \tilde\scal (\tilde x,-\eps) =: \tilde\scal_0 (\tilde x)
\label{culo-gordo-tonto}
\end{align} 
so that
\begin{align} 
\p_{\tilde x_j} \tilde \scal \circ X_{\tilde x}(t) = J^{-1}(\tilde x,t)  \p_{\tilde x_\ell} \tilde\scal_0(\tilde x) \, \mathcal{B}^\ell_j (\tilde x,t) \,.\label{nariz-gorda}
\end{align}
The point of the first two assumptions in \eqref{eq:initial:torque} is to single out one of the three elements in the sum over $\ell$ in \eqref{nariz-gorda}, which now reduces to 
\begin{align} 
\p_{\tilde x_j}\tilde \scal \circ X_{\tilde x}(t)  =J^{-1}(\tilde x,t)  \p_{\tilde x_2}\tilde \scal_0(\tilde x) \mathcal{B}^2_j  (\tilde x,t) \,.\label{fat-nose}
\end{align}

For the remainder of the proof, we fix $X$ to denote the trajectory which collides with the blowup at time $t=T_*$ so that $X(T_*)=0$. Using \eqref{fat-nose} and recalling \eqref{tildewz} we return to \eqref{eq:baroclinic:clinic} and obtain that 
\begin{align} 
 \bcal_3 = \tilde b \circ X 
 & = 
 2\beta_1\tfrac{\alpha }{\gamma} \tfrac{\tilde \sigma}{\tilde \rho}\circ X (\p_{\tilde x_1}  \tilde\sigma  \circ X  \p_{\tilde x_2} \tilde \scal \circ X   -\p_{\tilde x_2}  \tilde\sigma   \circ X \p_{\tilde x_1} \tilde \scal \circ X ) \notag \\
  & = 
 2\beta_1\tfrac{\alpha }{\gamma} \tfrac{\tilde \sigma}{\tilde \rho}\circ X J ^{-1} \p_{\tilde x_2} \tilde \scal_0 (\mathcal{B} ^2_2 \p_{\tilde x_1}  \tilde\sigma   \circ X
 -   \mathcal{B} ^2_1  \p_{\tilde x_2}  \tilde\sigma \circ X ) \notag \\
  & = 
\beta_1\tfrac{\alpha }{\gamma} \tfrac{\tilde \sigma}{\tilde \rho}\circ X J ^{-1} \p_{\tilde x_2} \tilde \scal_0 
 (\mathcal{B} ^2_2   \p_{\tilde x_1}\tilde w \circ X
 -  \mathcal{B} ^2_2  \p_{\tilde x_1}\tilde z \circ X
 - 2 \mathcal{B} ^2_1 \p_{\tilde x_2}  \tilde\sigma  \circ X ) \notag\\
 &=: \bcal_{3}^{(1)} - \bcal_{3}^{(2)} - \bcal_{3}^{(3)} \,.  \label{eq:baroclinic:clinic:1}
\end{align} 
We first note that by the relation of $\sigma$ and $\rho$, in view of \eqref{eq:entropy3} we have
\begin{align}
\tfrac{\beta_1 \alpha }{\gamma} \tfrac{\tilde \sigma}{\tilde \rho}\circ X
= \tfrac{ \beta_1}{\gamma} e^{\frac{\tilde \scal_0}{2}} (\tilde \rho \circ X)^{\alpha -1}
\end{align}
so that by \eqref{density-final0} and the initial $L^\infty$ assumption on $\scal(\cdot,-\eps)$ we have
\begin{align}
\abs{ \tfrac{\beta_1 \alpha }{\gamma} \tfrac{\tilde \sigma}{\tilde \rho}\circ X 
-  \tfrac{ \beta_1}{\gamma} (\tfrac {\alpha\kappa_0} 2)^{\frac{\alpha-1}{\alpha}}} \les \eps^{\frac{1}{10}} \,.
\label{eq:useless:beyond:belief}
\end{align}
Combined with \eqref{cowbell}, our bootstrap assumptions derivatives of  $Z$ in \eqref{eq:Z_bootstrap} and on $\uu \cdot \Ncal$ and $\sound$ in \eqref{eq:US_est}, similarly to \eqref{eq:BlackJack:4} we obtain that the last two terms in \eqref{eq:baroclinic:clinic:1} have time integrals bounded as 
\begin{align}
\int_{-\eps}^{T_*} \abs{\bcal_{3}^{(2)}(\tilde x_0, t)} + \abs{\bcal_{3}^{(2)}(\tilde x_0,t)} dt \les M \eps
\, .
\label{eq:baroclinic:clinic:2} 
\end{align}

In order to conclude the proof, we need to estimate the time integral of the first term in \eqref{eq:baroclinic:clinic:1}, namely $\bcal_{3}^{(1)}$. This is precisely the reason that Lemma~\ref{lem:dont:read:this} was created. First, we note that by \eqref{eq:f-ing:annoying:0.5} and \eqref{eq:f-ing:annoying:6} we have that $\p_{\tilde x_1} \tilde w \circ X(t) < 0$ for all $t \in [-\eps,T_*)$, that is, this term is signed. Taking into account \eqref{cowbell}, \eqref{eq:useless:beyond:belief}, and the third assumption in \eqref{eq:initial:torque} we obtain the pointwise in time bound
\begin{align}
\bcal_{3}^{(1)}(\tilde x_0,t) \geq  \tfrac{ \beta_1}{2 \gamma} (\tfrac {\alpha\kappa_0} 2)^{\frac{\alpha-1}{\alpha}} \p_{\tilde x_2} \tilde \scal_0(\tilde x_0) \, \p_{\tilde x_1} \tilde w \circ X_{\tilde x_0}(t) \,.
\label{eq:baroclinic:clinic:3} 
\end{align}
To conclude the proof we combine \eqref{eq:baroclinic:clinic:3} with \eqref{el-burro-es-feliz} and the assumption $\p_{\tilde x_2} \tilde \scal_0(\tilde x_0) < 0$ to deduce
\begin{align}
\int_{-\eps}^{T_*} \bcal_{3}^{(1)}(\tilde x_0,t) \geq  \tfrac{ \beta_1}{2 \gamma} (\tfrac {\alpha\kappa_0} 2)^{\frac{\alpha-1}{\alpha}} \sabs{\p_{\tilde x_2} \tilde \scal_0(\tilde x_0)}\tfrac{1}{9}  \kappa_0^{-\frac 23} \eps^{\frac 13} = 2 c_\alpha \kappa_0^{\frac 13 - \frac{1}{\alpha}} \eps^{\frac 13}  \sabs{\p_{\tilde x_2} \tilde \scal_0(\tilde x_0)} \,,
\label{eq:baroclinic:clinic:4} 
\end{align}
where $c_\alpha>0$ is a constant that depends only on $\alpha$. The point here is that the lower bound is $\OO(\eps^{\frac 13})$, while the error terms in both \eqref{eq:BlackJack:6} and \eqref{eq:baroclinic:clinic:2} are $\OO(\eps)$. Combining these estimates we deduce that  
\begin{align}
\upzeta_3(\tilde x_0, t) \geq  \int_{-\eps}^t \bcal_{3}^{(1)} (\tilde x_0, t') dt'  - M^2 \eps 
\geq \tfrac 32 c_\alpha \kappa_0^{\frac 13 - \frac{1}{\alpha}} \eps^{\frac 13}  \sabs{\p_{\tilde x_2} \tilde \scal_0(\tilde x_0)}
\end{align}
upon taking $\eps$ to be sufficiently small.
\end{proof}

\section{$\dot{H}^m$ bounds}\label{sec:energy}

\begin{definition}[Modified $\dot H^m$-norm]
For $m \geq 18$ we introduce the semi-norm
\begin{align}
E_m^2(s) = E_m^2[\uu,\pp,\hh](s) :=
\sum_{|\gamma|=m} \lambda^\modckg  \left( \norm{ \p^\gamma \uu(\cdot,s)}_{L^2}^2 +\norm{ \hh  \p^\gamma \pp(\cdot,s)}_{L^2}^2 
+ \kappa_0^2 \norm{\p^\gamma \hh(\cdot,s)}_{L^2}^2 \right)
\label{eq:Ek:def}
\end{align}
where $ \lambda = \lambda(m) \in (0,1)$ is to be made precise below (cf.~Lemma~\ref{lem:forcing:1}). 
\end{definition} 
Clearly, $E_m^2$ is equivalent to the homogenous Sobolev norm $\dot{H}^m$ for $\uu$, $\pp$, and $\hh$, and since $\kappa_0\geq 2$, we have the quantitative inequalities
\begin{align}
\tfrac{\lambda^m}{2}  \left( \snorm{\uu}_{\dot H^m}^2 + \snorm{ \pp}_{\dot H^m}^2 + \snorm{ \hh}_{\dot H^m}^2\right) \leq E_m^2 \leq \kappa_0^2 \left( \snorm{\uu}_{\dot H^m}^2 
 + \snorm{ \pp}_{\dot H^m}^2+ \snorm{ \hh}_{\dot H^m}^2 \right) \,.
 \label{norm_compare}
 \end{align} 
The bound \eqref{norm_compare} follows from 
\begin{align}
\abs{\hh(y,s) -  1} \leq \tfrac{2\eps}{\upgamma} 
\label{eq:useful:crap}
\end{align}
and the triangle inequality, upon taking $\eps$ sufficiently small.  In turn, \eqref{eq:useful:crap} is a consequence of the definition \eqref{Snew}, and of the bootstrap \eqref{eq:S_bootstrap}. 

Additionally, in order to apply the interpolation inequalities from Appendix~\ref{sec:interpolation}, we need to establish a quantitative equivalence between the $E_m$ semi-norm defined in \eqref{eq:Ek:def} and the classical homogenous $\dot{H}^m$ norm of the quantities $U$, $\sound$, and $K$ (recall that these are related to $\uu$, $\pp$, and $\hh$ via the nonlinear transformation given in \eqref{eq:thunder:variables}). In this direction we have

\begin{lemma}[Asymptotic equivalence of norms]
\label{lem:norm:equivalence}
For $\kappa_0 \geq 1$ sufficiently large in terms of $\upgamma$, and for $\eps$ sufficiently small in terms of $\kappa_0$, $M$, and $m$, we have the estimate
\begin{align}
 \lambda^m \left( \norm{U}_{\dot {H}^m}^2 + \norm{\sound}_{\dot{H}^m}^2 +  \norm{K}_{\dot {H}^m}^2 - e^{-2s}\right) 
 \leq E_m^2 
 \leq  \kappa_0^2 \left( \norm{U}_{\dot {H}^m}^2 + \norm{\sound}_{\dot{H}^m}^2  +  \norm{K}_{\dot {H}^m}^2 + e^{-2s} \right)
 \label{eq:ghastly}
\end{align}
for all $s\geq -\log \eps$. 
As a consequence, we also have the estimate
\begin{align}
\kappa_0^{-2} E_m^2 - e^{-2s} 
\leq  e^{-s} \norm{W}_{\dot{H}^m}^2 +  \norm{Z}_{\dot {H}^m}^2 + \norm{A}_{\dot{H}^m}^2 +  \norm{K}_{\dot {H}^m}^2  
 \leq 4 \lambda^{-m}  E_m^2 + 4 e^{-2s}  \,.
 \label{eq:ghastly:0}
\end{align}
\end{lemma}
\begin{proof}[Proof of Lemma~\ref{lem:norm:equivalence}]
We directly have
\begin{align}
 \lambda^m \norm{U}_{\dot{H}^m}^2 \leq \sum_{|\gamma|=m} \lambda^\modckg   \norm{ \p^\gamma \uu }_{L^2}^2
\leq \norm{U}_{\dot{H}^m}^2  
\label{eq:ghastly:1}
\end{align}
which gives a direct comparison between the $\dot{H}^m$ norm of $U$ and the $\uu$-part of $E_m$. 

Next, we turn to the $\hh$-part of $E_m$. The chain rule yields $\hh^{-1} \nabla \hh = \frac{1}{2 \upgamma} \nabla K$. Applying $m-1$ more derivatives, by the Fa\`a di Bruno formula, we have that there exists a constant $C_m$ which only depends on $m$, such that pointwise we have the bound
\begin{align}
\abs{\hh^{-1} \p^\gamma \hh - \tfrac{1}{2 \upgamma} \p^\gamma K } \leq C_m \sum_{(i_1,\ldots,i_{m-1})\in I_m} \prod_{j=1}^{m-1} \abs{D^j K}^{i_j}
\label{eq:shalala:1}
\end{align}
where the index set $I_m$ is given by $I_m = \{ (i_1,\ldots,i_{m-1}) \colon i_j\geq 0, \sum_{j=1}^{m-1} j i_j  = m \}$. In particular, note that whenever $(i_1,\ldots,i_{m-1}) \in I_m$, we must have $\sum_{j=1}^{m-1} i_j \geq 2$. This fact is crucial for the argument below, and has to do with the fact that we have already accounted on the left side for the term with the highest order of derivatives. In \eqref{eq:shalala:1} as usual we have written $D^j K$ to denote $D^{\beta} K$ for some multi-index $\beta$ with $\abs{\beta} = j$. Using the interpolation inequality \eqref{eq:special1}, for all $1\leq j \leq m-1$ we next estimate 
\begin{align}
\norm{|D^j K|^{i_j}}_{L^{\frac{2m}{j i_j}}}= \norm{D^j K}_{L^{\frac{2m}{j}}}^{i_j}  \les \norm{K}_{L^\infty}^{i_j(1- \frac{j}{m})}\norm{K}_{\dot{H}^m}^{\frac{j i_j}{m}}.
\label{eq:shalala:2}
\end{align}
Moreover, note that for $(i_1,\ldots,i_{m-1}) \in I_m$ we have that $\sum_{j=1}^{m-1} \frac{j i_j}{2m} = \frac{1}{2}$, so that these are H\"older conjugate exponents corresponding to an $L^2$ norm. Thus, applying the $L^2$ norm to \eqref{eq:shalala:1}, using the H\"older inequality, and the interpolation bound \eqref{eq:shalala:2}, we obtain
\begin{align}
 \norm{\hh^{-1} \p^\gamma \hh - \tfrac{1}{2\upgamma} \p^\gamma K }_{L^2}
 &\leq C_m \sum_{(i_1,\ldots,i_{m-1})\in I_m} \prod_{j=1}^{m-1} \norm{K}_{L^\infty}^{i_j(1- \frac{j}{m})}\norm{K}_{\dot{H}^m}^{\frac{j i_j}{m}}\notag\\
  &\leq C_m \sum_{(i_1,\ldots,i_{m-1})\in I_m} \norm{K}_{L^\infty}^{-1 + \sum_{j=1}^{m-1}i_j}\norm{K}_{\dot{H}^m}
  \,,
\end{align}
for some $m$-dependent constant $C_m$ (which may increase from line to line), whenever $\abs{\gamma} = m$.
At this point we use that $(i_1,\ldots,i_{m-1}) \in I_m$, we must have $\sum_{j=1}^{m-1} i_j \geq 2$, which is combined with the bootstrap \eqref{eq:S_bootstrap} to conclude 
\begin{align}
 \norm{\hh^{-1} \p^\gamma \hh - \tfrac{1}{2\upgamma} \p^\gamma K }_{L^2} 
  &\leq C_m \eps \norm{K}_{\dot{H}^m}\,.
\label{eq:shalala:3}
\end{align}
We next appeal to the pointwise estimate on $\hh$ in \eqref{eq:useful:crap}, and since $\kappa_0 \geq 1$, we deduce that 
\begin{align}
\tfrac{\kappa_0^2}{2\upgamma^2} \lambda^m \norm{K}_{\dot{H}^m}^2 \leq \sum_{|\gamma|=m} \lambda^\modckg   \kappa_0^2\norm{ \p^\gamma \hh}_{L^2}^2
\leq  \tfrac{\kappa_0^2}{\upgamma^2} \norm{K}_{\dot{H}^m}^2  
\label{eq:ghastly:2}
\end{align}
where we have used that $\lambda \in (0,1)$, and that $\eps$ is sufficiently small to absorb the $C_m$ constant in \eqref{eq:shalala:3}.

Lastly, we turn to the $\pp$-part of $E_m$. From \eqref{Pnew} and \eqref{Snew} we obtain $ \sound  = \pp \hh$, and thus, by the binomial formula and the Moser estimate \eqref{eq:Moser:inequality}, we have  
\begin{align*}
\norm{ \p^\gamma \sound  - \hh \p^\gamma \pp - \pp \p^\gamma \hh}_{L^2} \leq C_m \left( \norm{\nabla \hh}_{L^\infty} \norm{\pp}_{\dot H^{m-1}} + \norm{\nabla \pp}_{L^\infty} \norm{\hh}_{\dot{H}^{m-1}} \right) \,.
\end{align*}
Furthermore, using the interpolation bound \eqref{eq:special3} applied to $\nabla \pp$ and $\nabla \hh$, and the $\eps$-Young inequality, we obtain that for any $\delta \in (0,1)$ we have
\begin{align}
&\norm{  \p^\gamma \sound - \hh \p^\gamma \pp - \pp \p^\gamma \hh}_{L^2}  \notag\\
&\qquad \leq C_m \left( \norm{\nabla \hh}_{L^\infty} \norm{\nabla \pp}_{L^\infty}^{\frac{2}{2m-5}} \norm{\pp}_{\dot H^{m}}^{1-\frac{2}{2m-5}} + \norm{\nabla \pp}_{L^\infty}  \norm{\nabla \hh}_{L^\infty}^{\frac{2}{2m-5}} \norm{\hh}_{\dot H^{m}}^{1-\frac{2}{2m-5}} \right) \notag\\
&\qquad \leq \delta \norm{\pp}_{\dot H^{m}} + \delta \norm{\hh}_{\dot H^{m}} + C_m \delta^{-\frac{2m-7}{2}} \left( \norm{\nabla \hh}_{L^\infty}^{\frac{2m-5}{2}} \norm{\nabla \pp}_{L^\infty} + \norm{\nabla \pp}_{L^\infty}^{\frac{2m-5}{2}} \norm{\nabla \hh}_{L^\infty}  \right)
\label{eq:shalala:4}
\end{align}
where the $m$-dependent constant $C_m$ may change from line to line. From the definitions \eqref{Pnew}--\eqref{Snew}, the $K$ estimates in \eqref{eq:S_bootstrap}, the $W$ and $Z$ bounds in \eqref{eq:W_decay} and \eqref{eq:Z_bootstrap},  the relations $\hh \nabla \pp = \nabla \sound - \sound \hh^{-1} \nabla \hh$, and $2 \nabla \sound = e^{-\frac s2} \nabla W - \nabla Z$, we deduce 
\begin{align}
\norm{\nabla \hh}_{L^\infty} \leq \eps^{\frac 13} e^{-\frac s2} \qquad \mbox{and} \qquad 
\norm{\nabla \pp}_{L^\infty} \leq \left( \tfrac 12 + \eps^{\frac 14}\right) e^{-\frac s2}  \,.
\label{eq:shalala:5}
\end{align}
Taking $\eps$ to be sufficiently small to absorb the $m$ and $M$ dependent constants, we obtain from \eqref{eq:shalala:4} and \eqref{eq:shalala:5} that 
\begin{align}
\norm{ \p^\gamma \sound  - \hh \p^\gamma \pp - \pp \p^\gamma \hh}_{L^2}   
\leq \delta \norm{\pp}_{\dot H^{m}} + \delta \norm{\hh}_{\dot H^{m}}  +  \delta^{-\frac{2m-7}{2}} e^{-\frac{2m-3}{4} s}
\label{eq:shalala:6}
\end{align}
for any constant $\delta \in (0,1)$.  Using that $\abs{ \sound -  \kappa_0/2} \leq 5 \eps^{\frac 16}$ (which follows from the bootstrap assumptions on $\dot \kappa$, $W$, and $Z$), and appealing to \eqref{eq:useful:crap}, we obtain
\begin{align}
\abs{\pp(y,s)  - \tfrac{\kappa_0}{2} }  \leq  6 \eps^{\frac 16}  
\label{eq:shalala:7}
\end{align}
upon taking $\eps$ to be sufficiently small in terms of $M$ and $\kappa_0$. 
 At last, we combine  \eqref{eq:shalala:6}--\eqref{eq:shalala:7}, use the $\pp$ and $\hh$ part of the comparison \eqref{norm_compare},   choose $\delta$ sufficiently small depending on $\kappa_0$ and $\lambda$, and then $\eps$ sufficiently small in terms of $\kappa_0,  \lambda, \delta$ and $m$, to deduce that 
\begin{align}
 \lambda^{m} \norm{\sound}_{\dot{H}^m}^2 
\leq   \sum_{\abs{\gamma}=m} \lambda^{\abs{\check \gamma}}  \left( \norm{\hh \p^\gamma \pp}_{L^2}^2 + \kappa_0^2 \norm{\p^\gamma \hh}_{L^2}^2 \right) + e^{-2s}
\label{eq:ghastly:3}\,,
\end{align}
and taking $\kappa_0 \geq 2$, we also have
\begin{align}
 \norm{\sound}_{\dot{H}^m}^2  \geq  \tfrac 18  \sum_{\abs{\gamma}=m} \lambda^{\abs{\check \gamma}}  \left( \norm{\hh\p^\gamma \pp}_{L^2}^2 + \kappa_0^2 \norm{\p^\gamma \hh}_{L^2}^2 \right) -  e^{-2s} \,.
\label{eq:ghastly:4}
\end{align}
Combining \eqref{eq:ghastly:1}, \eqref{eq:ghastly:2}, \eqref{eq:ghastly:3}, and \eqref{eq:ghastly:4}, we arrive at the proof of \eqref{eq:ghastly}.

The proof of \eqref{eq:ghastly:0} follows once we recall the identities $ W = e^{\frac{s}{2}}  ( U \cdot \Ncal +\sound - \kappa)$,   $ Z =   U \cdot \Ncal -\sound  $, which follow from \eqref{eq:UdotN:Sigma}, and the definition
$ A_\nu  =   U \cdot \Tcal^\nu $. 
Therefore, by  \eqref{e:bounds_on_garbage}, \eqref{eq:special1}, using the  Poincar\'e inequality in the $\check y$ direction, and the fact that the diameter of $\XXX(s)$ in the $\check e$ directions is  $4 \eps^{\frac 16} e^{\frac{s}{2}}$, for any $\gamma$ with $\abs{\gamma} = m$, we obtain 
\begin{align*}
&\norm{e^{-\frac s2} \p^\gamma W -  \Ncal \cdot \p^\gamma U - \p^\gamma S}_{L^2}
+ \norm{\p^\gamma Z -  \Ncal \cdot \p^\gamma U + \p^\gamma S}_{L^2}
+\norm{\p^\gamma A_\nu -  \Tcal^\nu \cdot \p^\gamma U }_{L^2}
\notag\\
&\leq 2 \norm{\comm{\p^\gamma}{\Ncal}\cdot \uu}_{L^2} + \norm{\comm{\p^\gamma}{\Tcal^\nu}\cdot \uu}_{L^2} 
\notag\\
&\les  \sum_{j=1}^{m} \left(\norm{D^j \Ncal}_{L^\infty} + \norm{D^j \Tcal^\nu}_{L^\infty}\right) \norm{D^{m-j}\uu}_{L^2(\XXX(s))}
\notag\\
&\les \eps \sum_{j=1}^{m} e^{-\frac{j s}{2}} (4\eps^{\frac 16} e^{\frac s2})^{j} \norm{\uu}_{\dot{H}^m}
\notag\\
&\les \eps \norm{\uu}_{\dot{H}^m} \,.
\end{align*}
Summing over all $\gamma$ with $\abs{\gamma} =m$, and appealing to \eqref{eq:ghastly}, the estimate \eqref{eq:ghastly:0} follows.
\end{proof}

 \subsection{Higher-order derivatives for the $(\uu,\pp,\hh)$-system}
In order to estimate $E_m(s)$ we need the differentiated form of the $(\uu,\pp,\hh)$-system \eqref{UPS-new}. For this purpose, fix $\gamma \in {\mathbb N}_0^3$ with $\abs{\gamma} = m$, and apply $\p^\gamma$ to \eqref{UPS-new}, to obtain
\begin{subequations} 
\label{UPS-L2}
\begin{align}
&\p_s (\p^\gamma \uu_i)
+ ( \mathcal{V} _U \cdot \nabla ) (\p^\gamma \uu_i) 
+ \mathcal{D}_\gamma (\p^\gamma \uu_i)
- 2 \beta_1 \beta_\tau e^{-s} \dot Q_{ij} (\p^\gamma \uu_j)
+2 \beta_\tau \beta_3 \hh^2 (\p^\gamma \pp) \Jcal \Ncal_i e^{\frac s2} \p_1 \pp
\notag \\
& \quad
+ 2 \gamma_1 \beta_\tau \beta_3 \hh^2 e^{\frac s2} \p_1 \pp \Jcal \Ncal_i (\p^\gamma \pp)
+ 2 \beta_\tau \beta_3 \hh^2 \pp \left( \Jcal \Ncal_i e^{\frac{s}{2}}   \p_1 (\p^\gamma \pp) + e^{-\frac{s}{2}} \delta^{i\nu}  \p_\nu ( \p^\gamma \pp) \right)
= \mathcal{F} _{U_i}^{(\gamma)} ,
 \label{UPS-L2-U}
 \\
& \p_s (\p^\gamma \pp)  
+ \left(\mathcal{V} _U \cdot \nabla \right) (\p^\gamma \pp)
+ \mathcal{D}_\gamma (\p^\gamma \pp )   
+2 \beta_\tau \beta_3 e^{\frac s2} \Jcal \p_1(\uu \cdot \Ncal) (\p^\gamma \pp) 
\notag \\
& \quad
+2 \gamma_1  \beta_\tau \beta_3 e^{\frac s2} \p_1 \pp  \Jcal \Ncal_j (\p^\gamma \uu_j)
+2 \beta_\tau\beta_3 \pp  \left( e^{\frac{s}{2}} \Jcal \Ncal_j \p_1(\p^\gamma \uu_j) + e^{-\frac{s}{2}} \p_\nu( \p^\gamma \uu_\nu)\right) 
=  \mathcal{F} _{\pp}^{(\gamma)} , 
\label{UPS-L2-P} 
\\
&\p_s (\p^\gamma \hh)  
+ ( \mathcal{V} _U \cdot \nabla ) (\p^\gamma \hh)
+ \mathcal{D}_\gamma (\p^\gamma \hh)
 = \mathcal{F} _{\hh}^{(\gamma)} 
, \label{UPS-L2-K} 
\end{align} 
\end{subequations} 
where the damping function $ \mathcal{D} _\gamma$ is defined as 
\begin{align}
\label{Dgamma}
\mathcal{D} _\gamma =  \gamma_1 ( 1 + \p_1 g_U) 
 + \tfrac{1}{2} \abs{ \gamma}   \,,
\end{align} 
the transport velocity $ \mathcal{V} _U$ is given in \eqref{V_U}, and since $\abs{\gamma} \ge 3$ the forcing functions in \eqref{UPS-L2} are given by 
\begin{subequations} 
\label{UPS-forcing-abstract}
\begin{align} 
\mathcal{F} _{U_i}^{(\gamma)}
&=  D_\gamma (\p^\gamma \uu_i)  - \comm{\p^\gamma}{\mathcal{V}_U \cdot \nabla}\uu_i  
- 2 \beta_\tau \beta_3 e^{-\frac s2} \delta^{i\nu} \comm{\p^\gamma}{\hh^2\pp} \p_\nu \pp
\notag\\
&\qquad 
+ 2 \beta_\tau \beta_3 e^{\frac s2} \left((\p^\gamma \pp) \hh^2 \Jcal \Ncal_i \p_1 \pp 
+ \gamma_1 \hh^2 \p_1 \pp \Jcal \Ncal_i (\p^\gamma \pp) - \comm{\p^\gamma}{\hh^2\pp \Jcal \Ncal_i} \p_1 \pp \right)
\,, \label{FU-abstract}\\
\mathcal{F} _{\pp}^{(\gamma)}
&=  D_\gamma (\p^\gamma \pp)  - \comm{\p^\gamma}{\mathcal{V}_U \cdot \nabla}\pp  
- 2 \beta_\tau \beta_3 e^{-\frac s2}  \comm{\p^\gamma}{\pp} \p_\nu \uu_\nu 
\notag\\
&\qquad 
+ 2 \beta_\tau \beta_3 e^{\frac s2} \left( (\p^\gamma \pp) \Jcal \Ncal_j \p_1 \uu_j + \gamma_1 \p_1 \pp \Jcal \Ncal_i (\p^\gamma \uu_i ) - \comm{\p^\gamma}{\pp \Jcal \Ncal_i} \p_1 \uu_i \right)
\,, \label{FP-abstract}\\
\mathcal{F} _{\hh}^{(\gamma)}
&=  D_\gamma (\p^\gamma \hh) - \comm{\p^\gamma}{\mathcal{V}_U \cdot \nabla}\hh  
\,. \label{FS-abstract}
\end{align} 
\end{subequations} 
In \eqref{UPS-forcing-abstract} we have used the notation $\comm{a}{b}$ to denote the commutator $a b- b a$. 
Note that two additional forcing terms are singled out on the left side of \eqref{UPS-L2-P}; this is because these terms will turn out to contribute the main contribution that has to be absorbed in the damping term $D_\gamma$. 

The $E_m$ energy estimate is obtained by testing \eqref{UPS-L2-U} with $ \p^\gamma \uu_i$, \eqref{UPS-L2-P} with $\hh^2 \p^\gamma \pp$, and \eqref{UPS-L2-K} with $\kappa_0^2 \p^\gamma \hh$. Adding the resulting differential equations produces the cancelation of all terms involving $m+1$ derivatives, which upon integrating by parts allows us to close the energy estimate. This computation is detailed in Subsection~\ref{sec:Ek:energy:estimate} below. Prior to this, in the next subsection we give estimates for the forcing terms defined in \eqref{UPS-forcing-abstract}. 

\subsection{Forcing estimates}
In order to analyze \eqref{UPS-L2} we first estimate the forcing terms defined in \eqref{UPS-forcing-abstract}.
This is achieved next:

\begin{lemma}\label{lem:forcing:1} 
Consider the forcing functions  defined in
\eqref{UPS-forcing-abstract}.
Let $m \ge 18$, fix $0 < \delta\le \tfrac{1}{32}$, and define the parameter $\lambda = \lambda(\delta,m)$ from \eqref{eq:Ek:def} to equal $\lambda= \tfrac{\delta ^2}{16 m^2} $. Then, we have that
\begin{subequations} 
\begin{align} 
2 \sum_{\abs{\gamma}=m}
 \lambda^\modckg \int_{ \mathbb{R}^3  } \abs{ \mathcal{F} _{ U^i}^{(\gamma)} \,  \p^\gamma \uu_i } & 
\le   (5 + 9\delta )  E_m^2  + e^{-s} M^{4m-1}\,,   \label{eq:Hk:est:FU}\\
2 \sum_{\abs{\gamma}=m}
 \lambda^\modckg \int_{ \mathbb{R}^3  }\abs{ \mathcal{F} _{\pp}^{(\gamma)}\, \hh^2 \p^\gamma \pp} &
 \le  (2+ 8\delta )  E_m^2  +   e^{-s} M^{4m-1} \,,  \label{eq:Hk:est:FP}\\
 2 \sum_{\abs{\gamma}=m}
 \lambda^\modckg \kappa_0^2 \int_{ \mathbb{R}^3  }\abs{ \mathcal{F} _{\hh}^{(\gamma)}\, \p^\gamma \hh} &
 \le  (2+ 4\delta )  E_m^2  +   e^{-s} M^{4m-1} \,,   \label{eq:Hk:est:FE} 
\end{align} 
\end{subequations} 
for $ \eps $ taken sufficiently small in terms of $m$, $\delta$, $\lambda$, $M$, and $\kappa_0$.
\end{lemma} 
\begin{proof}[Proof of Lemma~\ref{lem:forcing:1}]
Throughout this proof, when there is no need to keep track of the binomial coefficients from the product rule we  denote a partial derivative $\p^\gamma$ with $\abs{\gamma}=m$ simply as $D^m$.

Upon expanding the commutator terms in \eqref{UPS-forcing-abstract}, the forcing functions defined here may be written as
\begin{subequations}
\begin{align}
\mathcal{F} _{ U_i}^{(\gamma)}  &= \mathcal{F} _{ U_i}^{(m)} + \mathcal{F} _{ U_i}^{(< m)}
\label{eq:F:U:sum}
\\
\mathcal{F} _{\pp}^{(\gamma)}  &= \mathcal{F} _{\pp}^{(m)} + \mathcal{F} _{\pp}^{(< m)}
\label{eq:F:P:sum}
\\
\mathcal{F} _{\hh}^{(\gamma)}  &= \mathcal{F} _{\hh}^{(m)} + \mathcal{F} _{\hh}^{(< m)}
\label{eq:F:E:sum}
\end{align}
\end{subequations}
where the upper index $(m)$ indicates that terms with exactly $m$ derivatives are present, while the upper index $(<m)$ indicates that all terms have at most $m-1$ derivatives on them. These terms are defined by
\begin{subequations}
\begin{align}
 \mathcal{F} _{ U_i}^{(m)}
 &=
- \left(  \gamma_\mu \p_\mu g_U \p_1 \p^{\gamma- e_\mu} \uu_i + \gamma_j \p_j h_U^\nu \p_\nu \p^{\gamma - e_j} \uu_i + \p^\gamma g_U \p_1 \uu_i + \p^\gamma h_U^\nu \p_\nu \uu_i \right)
\notag\\
&\ -2 \beta_\tau \beta_3\left(\gamma_\mu  e^{\frac s2}  \p_\mu (\hh^2 \pp \Jcal \Ncal_i) \p^{\gamma - e_\mu}\p_1 \pp  +2 \gamma_1 \pp  \Jcal \Ncal_i \hh  e^{\frac s2}  \p_1\hh \p^\gamma \pp + e^{-\frac s2} \delta^{i\nu} \gamma_j \p_j (\hh^2 \pp) \p^{\gamma - e_j} \p_\nu \pp \right)
\notag\\
&\ -2 \beta_\tau \beta_3 \left( e^{-\frac s2} \delta^{i\nu} 
\p_\nu \pp \p^\gamma (\hh^2 \pp) +  e^{\frac s2} \pp  \p_1 \pp  \p^{\gamma} (\hh^2 \Jcal \Ncal_i)   \right)
\notag\\
&=: \mathcal{F} _{ U_i,(1)}^{(m)} + \mathcal{F} _{ U_i,(2)}^{(m)} + \mathcal{F} _{ U_i,(3)}^{(m)}
\label{eq:F:U:k}
\\
\mathcal{F} _{ U_i}^{(<m)}
&= - \sum_{j=1}^{m-2} \sum_{\abs{\beta}=j, \beta \leq \gamma} {\gamma \choose \beta}
\left( \p^{\gamma -\beta} g_U \p^\beta \p_1 \uu_i + \p^{\gamma - \beta } h_U^\nu \p^\beta \p_\nu \uu_i \right)
\notag\\
&-2 \beta_\tau \beta_3  \sum_{j=1}^{m-2} \sum_{\abs{\beta}=j, \beta \leq \gamma} {\gamma \choose \beta}
\left( e^{\frac s2} \p^{\gamma-\beta} (\hh^2 \pp \Jcal \Ncal_i) \p^\beta\p_1 \pp + e^{-\frac s2} \delta^{i\nu}\p^{\gamma-\beta} (\hh^2 \pp) \p^\beta\p_\nu \pp\right)
\notag\\
&-2 \beta_\tau \beta_3 e^{\frac s2} \p_1 \pp \sum_{j=1}^{m-1} \sum_{\abs{\beta}=j,\beta\leq \gamma} {\gamma \choose \beta} \p^{\gamma-\beta } (\hh^2 \Jcal \Ncal_i)  \p^\beta \pp 
\notag\\
&=: \mathcal{F} _{ U_i,(1)}^{(<m)} + \mathcal{F} _{ U_i,(2)}^{(<m)} + \mathcal{F} _{ U_i,(3)}^{(<m)}
\label{eq:F:U:less:k}
\end{align}
\end{subequations}
for the $\p^\gamma \uu$ evolution, by
\begin{align}
\mathcal{F} _{\pp}^{(m)} 
&=
- \left(  \gamma_\mu \p_\mu g_U \p_1 \p^{\gamma- e_\mu} \pp + \gamma_j \p_j h_U^\nu \p_\nu \p^{\gamma - e_j} \pp + \p^\gamma g_U  \p_1 \pp +  \p^\gamma h_U^\nu \p_\nu \pp   \right)
\notag\\
&\ - 2 \beta_\tau \beta_3 \left(\gamma_\mu e^{\frac s2}   \p_\mu (\pp \Jcal \Ncal_i) \p^{\gamma - e_\mu}\p_1 \uu_i +  e^{-\frac s2} \p_\nu \uu_\nu \p^{\gamma} \pp    +  e^{-\frac s2} \gamma_j \p_j \pp \p^{\gamma - e_j} \p_\nu \uu_\nu  \right)\notag\\
&=: \mathcal{F} _{\pp,(1)}^{(m)} + \mathcal{F} _{\pp,(2)}^{(m)}
\label{eq:F:P:k}
\\
\mathcal{F} _{\pp}^{(<m)} 
&=- \sum_{j=1}^{m-2} \sum_{\abs{\beta}=j, \beta \leq \gamma} {\gamma \choose \beta}
\left( \p^{\gamma -\beta} g_U \p^\beta \p_1 \pp + \p^{\gamma - \beta } h_U^\nu \p^\beta \p_\nu \pp \right)
\notag\\
&\ -2 \beta_\tau \beta_3  \sum_{j=1}^{m-2} \sum_{\abs{\beta}=j, \beta \leq \gamma} {\gamma \choose \beta}
\left( e^{\frac s2} \p^{\gamma-\beta} (\pp \Jcal \Ncal_i) \p^\beta\p_1 \uu_i + e^{-\frac s2} \p^{\gamma-\beta} \pp \p^\beta\p_\nu \uu_\nu \right)
\notag\\
&\ -2 \beta_\tau \beta_3  \p_1 \uu_i \gamma_\mu (\Jcal \Ncal_i)_{,\mu} \p^{\gamma - e_\mu} \pp
\notag\\
&=: \mathcal{F} _{\pp,(1)}^{(<m)} + \mathcal{F} _{\pp,(2)}^{(<m)} + \mathcal{F} _{\pp,(3)}^{(<m)}
\label{eq:F:P:less:k}
\end{align}
for the $\p^\gamma \pp$ equation,   
and by
\begin{subequations}
\begin{align}
\mathcal{F} _{\hh}^{(m)} 
&=
- \left(  \gamma_\mu \p_\mu g_U \p_1 \p^{\gamma- e_\mu} \hh + \gamma_j \p_j h_U^\nu \p_\nu \p^{\gamma - e_j} \hh + \p^\gamma g_U \p_1 \hh + \p^\gamma h_U^\nu \p_\nu \hh \right)
\label{eq:F:E:k}
\\
\mathcal{F} _{\hh}^{(<m)} 
&=- \sum_{j=1}^{m-2} \sum_{\abs{\beta}=j, \beta \leq \gamma} {\gamma \choose \beta}
\left( \p^{\gamma -\beta} g_U \p^\beta \p_1 \hh + \p^{\gamma - \beta } h_U^\nu \p^\beta \p_\nu \hh \right)
\label{eq:F:E:less:k}
\end{align}
\end{subequations} 
for the $\p^\gamma \hh$ equation.

{\bf Proof of \eqref{eq:Hk:est:FU}.} 
We shall first prove \eqref{eq:Hk:est:FU}, and to do so, we estimate separately the terms  in the  sum \eqref{eq:F:U:sum}.  Let us treat the term which contains the highest-order derivatives, namely $\mathcal{F}_{\uu_i}^{(m)}$. This term is decomposed in three pieces cf.~\eqref{eq:F:U:k}, and we estimate each piece separately. 

Recall that $g_U$ and $h_U^\nu$ are defined in \eqref{eq:gA} and \eqref{eq:hA} and that 
\begin{align} 
U_i = U\cdot \Ncal \Ncal_i +  A_\nu \Tcal^\nu_i = \tfrac{1}{2} ( e^{-\frac{s}{2}} W + \kappa+ Z )\Ncal_i + A_\nu \Tcal^\nu_i \,.  \label{Udef00}
\end{align}
Also, note that $f$ and $V$ are quadratic functions of $\check y$, whereas $\Jcal \Ncal$ is an affine function of $\check y$; therefore $\p^\gamma$ annihilates these terms and we have\footnote{Note that \eqref{eq:Scottie} holds whenever $\abs{\gamma}\geq 4$. This is because $g_U = 2\beta_1\beta_\tau e^{\frac s2} ( \uu\cdot\Ncal \Jcal + V \cdot \Ncal \Jcal - \dot f)$, with $V \cdot \Ncal \Jcal$ being a cubic polynomial in $y$, and $\dot f$ a quadratic polynomial in $\check y$.}
\begin{align}
\tfrac{1}{\beta_1 \beta_\tau}\p^\gamma g_U &= 2 e^{\frac s2} \p^\gamma  (\Jcal \Ncal \cdot \uu) 
=  2 e^{\frac s2} \Jcal \Ncal \cdot \p^\gamma \uu + 2 \gamma_\mu  (\Jcal \Ncal_\rho)_{,\iota} \p^{\gamma - e_\iota} \uu_\rho
\label{eq:Scottie}\\
\tfrac{1}{\beta_1 \beta_\tau}\p^\gamma h_U^\mu &= 2 e^{-\frac s2} \p^\gamma \uu_\mu
\label{eq:Dennis} 
\end{align}
In view of these definitions,   using that $\lambda \leq 1$, that $\beta_\tau \beta_1 \leq 1$, and that $\p_1 \p^{\gamma- e_\mu} \uu_i$ produces a favorable imbalance of $\lambda^{\frac 12}$, for the first term in \eqref{eq:F:U:k}  we have that 
\begin{align}   
&2 \sum_{\abs{\gamma}=m}
 \lambda^\modckg \int_{ \mathbb{R}^3  }\abs{ F _{U_i,(1)}^{(m)}  \,  \p^\gamma \uu_i }
 \notag\\
&\qquad
\leq
2 E_m^2 \left( m \lambda^{\frac 12} \norm{\check\nabla  g_U}_{L^\infty} + m \norm{\nabla h_U}_{L^\infty}   + 2 e^{\frac s2} \abs{\Jcal} \norm{\p_1 \uu}_{L^\infty}  + 2 e^{-\frac s2} \norm{\check \nabla U}_{L^\infty} \right)
 \notag\\
&\qquad \qquad+ 4 m E_m \norm{(\Jcal \Ncal)_{,\iota}}_{L^\infty} \norm{\p_1 U}_{L^\infty} \norm{\uu}_{\dot H^{m-1}} \,.
\label{FU0.5}
\end{align}
Estimate \eqref{FU0.5} is the perfect example of the usage of the parameter $\lambda$ appearing in the definition of the energy $E_m$: it yields a factor of $\lambda^{\frac 12}$ next to the term $m \norm{\check\nabla  g_U}_{L^\infty} \approx m$ in the first term of \eqref{FU0.5}. Without this factor, the resulting coefficient of $E_m^2$ appearing on the right side of \eqref{eq:Hk:est:FU} would be larger than $2m$, which would not allow us to close the energy estimate. But by choosing $\lambda = \frac{\delta^2}{12m^2}$, we have that $2m \lambda^{\frac 12} < \delta$. Using the definitions of $g_U$, $h_U^\nu$, and $U$, the bounds  \eqref{eq:beta:tau}, \eqref{e:space_time_conv}, \eqref{eq:W_decay},  \eqref{Jp1W}, \eqref{eq:Z_bootstrap},  \eqref{eq:A_bootstrap},  \eqref{e:bounds_on_garbage}, \eqref{eq:Celtics:suck}, the norm equivalence \eqref{norm_compare}, and the interpolation inequality \eqref{eq:special3} applied to $\nabla U$, we estimate 
\begin{align*}
\norm{\check \nabla g_U}_{L^\infty} &\leq \norm{\check \nabla (\Jcal W)}_{L^\infty} + \norm{\check \nabla G_U}_{L^\infty} \leq 1 + \eps^{\frac 14} \\
\norm{\nabla h_U}_{L^\infty} &\leq \eps^{\frac 14} \\
\norm{\Jcal \p_1 U}_{L^\infty} &\leq \tfrac 12 e^{-\frac s2} \norm{\Jcal \p_1 W}_{L^\infty} + \norm{\p_1 Z}_{L^\infty} + 2 \norm{\p_1 A}_{L^\infty}  \leq \tfrac{1}{2} (1+ \eps^{\frac 14}) e^{-\frac s2}\\
 \norm{\check \nabla U}_{L^\infty} &\leq  \eps^{\frac 14}\\
 \norm{(\Jcal \Ncal)_{,\iota}}_{L^\infty} &\leq \eps^{\frac 14}\\
 \norm{\p_1 U}_{L^\infty} \norm{U}_{\dot{H}^{m-1}} &\leq C_m \norm{\p_1 U}_{L^\infty}^{\frac{2m-3}{2m-5}} \norm{U}_{\dot{H}^m}^{\frac{2m-7}{2m-5}} \leq \norm{U}_{\dot{H}^m} + C_m' \norm{\p_1 U}_{L^\infty}^{\frac{2m-3}{2}} \leq 2 \lambda^{-\frac m2} E_m + e^{-s}
\end{align*}
for an arbitrary $\delta \in (0,1)$, 
upon choosing $\eps$ to be sufficiently small to absorb the stray powers of $M$ and all implicit, $\delta$-dependent and $m$-dependent constants. Combining the above estimates with \eqref{FU0.5}, we obtain
\begin{align}   
&2 \sum_{\abs{\gamma}=m}
 \lambda^\modckg \int_{ \mathbb{R}^3  }\abs{ F _{U_i,(1)}^{(m)}  \,  \p^\gamma \uu_i }
 \notag\\
&\qquad
\leq
2 E_m^2 \left( \tfrac{\delta}{4} (1 + \eps^{\frac 14}) + m \eps^{\frac 14}  + 1 + \eps^{\frac 14} +2  \eps^{\frac 34}  \right) + 4 m E_m \eps^{\frac 14} \left( 2 \lambda^{-\frac m2} E_m + e^{-s} \right)
\notag\\
&\qquad 
\leq (2 + \delta ) E_m^2 + e^{-2s} \,.
\label{FU1}
\end{align}
 Quite similarly, using that $\lambda \leq 1$, that $\beta_\tau \beta_3 \leq 1$, and that $\p_1 \p^{\gamma- e_\mu} \pp$ produces a favorable imbalance of $\lambda^{\frac 12}$, for the second term in \eqref{eq:F:U:k}, we have
\begin{align}   
&2 \sum_{\abs{\gamma}=m}
 \lambda^\modckg \int_{ \mathbb{R}^3  }\abs{ F _{U_i,(2)}^{(m)}  \p^\gamma \uu_i  }
 \notag\\
&\  \leq 4 E_m^2  \left( m\lambda^{\frac 12} e^{\frac s2} \norm{\hh^{-1} \check \nabla(\hh^2 \pp \Jcal \Ncal)}_{L^\infty} + 2m e^{\frac s2} \norm{\pp \Jcal \Ncal \p_1 \hh}_{L^\infty} + m e^{-\frac s2} \norm{\hh^{-1} \nabla(\hh^2 \pp)}_{L^\infty} \right)
\, .
\label{FU1.5}
\end{align} 
Using the estimates \eqref{e:bounds_on_garbage}, \eqref{eq:useful:crap}, \eqref{eq:shalala:5}, and \eqref{eq:shalala:7}
we obtain that
\begin{align*}
\norm{\hh^{-1} \check \nabla(\hh^2 \pp \Jcal \Ncal)}_{L^\infty}
&\leq \left(\tfrac 12 + \eps^{\frac 16} \right) e^{-\frac s2}  \\
\norm{\pp\Jcal \Ncal \p_1 \hh}_{L^\infty} 
&\leq \eps^{\frac 14} e^{-\frac s2} \\
\norm{\hh^{-1} \nabla(\hh^2 \pp)}_{L^\infty}
&\leq \left(\tfrac 12 + \eps^{\frac 16} \right) e^{-\frac s2} \,. 
\end{align*}
Using the above estimates, and recalling our choice of $\lambda = \frac{\delta^2}{16 m^2}$, the bound \eqref{FU1.5} becomes
\begin{align}   
2 \sum_{\abs{\gamma}=m}
 \lambda^\modckg \int_{ \mathbb{R}^3  }\abs{ F _{U_i,(2)}^{(m)}  \p^\gamma \uu_i  }
 \leq 4 E_m^2 \left( \tfrac{\delta}{4} (\tfrac 12 + \eps^{\frac 16}) + 2m \eps^{\frac 14} + m \eps (\tfrac 12 + \eps^{\frac 16})  \right)
 \leq \delta E_m^2
\, ,
\label{FU2}
\end{align} 
upon taking $\eps$ to be sufficiently small. Lastly, for the third   term in \eqref{eq:F:U:k}, we similarly have
\begin{align}   
&2 \sum_{\abs{\gamma}=m}
 \lambda^\modckg \int_{ \mathbb{R}^3  }\abs{ F _{U_i,(3)}^{(m)}  \p^\gamma \uu_i  }
 \notag\\
&\  \leq 4 E_m  e^{-\frac s2} \norm{\check \nabla \pp}_{L^\infty} \norm{\hh^2 \pp}_{\dot{H}^m} + 4 e^{\frac s2}  \norm{\pp}_{L^\infty} \norm{\p_1 \pp}_{L^\infty} \sum_{\abs{\gamma}=m}
 \lambda^\modckg \int_{ \mathbb{R}^3  } \abs{ \p^\gamma (\hh^2 \Jcal \Ncal) \cdot \p^\gamma \uu}  
\, .
\label{FU2.5}
\end{align} 
For the second term in \eqref{FU2.5} we recall that $\Jcal \Ncal$ is an affine function, and thus $D^2 (\Jcal \Ncal) = 0$. From the Leibniz rule,  the Moser inequality \eqref{eq:Moser:inequality},   the estimates \eqref{eq:speed:bound},  \eqref{e:bounds_on_garbage}, \eqref{eq:useful:crap}, \eqref{eq:shalala:5},   the interpolation bound \eqref{eq:special3}, and the norm comparison~\eqref{norm_compare}, we moreover have that
\begin{align}
\norm{\p^\gamma (\hh^2 \Jcal \Ncal) - 2 \hh \Jcal \Ncal \p^\gamma \hh}_{L^2} 
&\leq e^{-\frac s2} \norm{(\Jcal \Ncal)_{,\mu}}_{L^\infty} \gamma_\mu \norm{\p^{\gamma - e_\mu}(\hh^2)}_{L^2} + C_m \norm{\Jcal \Ncal}_{L^\infty} \sum_{j=1}^{m-1} \norm{D^{j} \hh \, D^{m-j} \hh}_{L^2}
\notag\\
&\leq C_m \eps^{\frac 12} e^{-\frac s2} \norm{\hh}_{L^\infty} \norm{\hh}_{\dot{H}^{m-1}} + C_m  \norm{\nabla \hh}_{L^\infty} \norm{\hh}_{\dot{H}^{m-1}}
\notag\\
&\leq C_m \eps^{\frac 12} e^{-\frac s2} \norm{\hh}_{L^\infty} \norm{\nabla \hh}_{L^\infty}^{\frac{2}{2m-5}} \norm{\hh}_{\dot{H}^{m}}^{\frac{2m-7}{2m-5}} + C_m  \norm{\nabla \hh}_{L^\infty}^{\frac{2m-3}{2m-5}} \norm{\hh}_{\dot{H}^{m}}^{\frac{2m-7}{2m-5}}
\notag\\
&\leq  C_m  (\eps^{\frac 13} e^{-\frac s2})^{\frac{2m-3}{2m-5}} (\lambda^{-\frac m2} E_m)^{\frac{2m-7}{2m-5}}
\notag\\
&\leq \eps^{\frac 13} e^{-\frac s2} E_m + \eps^{\frac 12} e^{-s}
\label{eq:quantitative:diarrhea:1} 
\end{align}
by taking $\eps$ to be sufficiently small in terms of $m$ and $\lambda$. From \eqref{e:bounds_on_garbage},  \eqref{eq:useful:crap}, \eqref{eq:quantitative:diarrhea:1}, the definition of the $E_m$ norm in \eqref{eq:Ek:def}, and the Cauchy–Bunyakovsky inequality we deduce that 
\begin{align}
\sum_{\abs{\gamma}=m}
 \lambda^\modckg \int_{ \mathbb{R}^3  } \abs{ \p^\gamma (\hh^2 \Jcal \Ncal) \cdot \p^\gamma \uu}  
 &\leq 2 \eps^{\frac 13} e^{-\frac s2} E_m^2 + \eps e^{-s} + 3 \kappa_0^{-1} E_m^2\, .
 \label{eq:quantitative:diarrhea:2} 
\end{align}
The above estimate is combined with the bound 
\begin{align*}
\norm{\pp}_{L^\infty} \norm{\p_1 \pp}_{L^\infty}
\leq \left( \tfrac{\kappa_0}{4} + \eps^{\frac 18}\right) e^{-\frac s2}
\,,
\end{align*}
which follows from \eqref{eq:shalala:5} and \eqref{eq:shalala:7}, and with the estimate
\begin{align*}
\norm{\check \nabla \pp}_{L^\infty} \norm{\hh^2 \pp}_{\dot{H}^m}
\leq C_m e^{-\frac s2} \left( \norm{\pp}_{\dot{H}^m} + \kappa_0 \norm{\hh}_{\dot{H}^m}\right)
\leq C_m \kappa_0 e^{-\frac s2} \lambda^{-\frac m2} E_m
\,,
\end{align*}
which follows from the fact that $\kappa_0\geq 1$, the Moser inequality, \eqref{norm_compare}, \eqref{eq:useful:crap}, \eqref{eq:shalala:5} and \eqref{eq:shalala:7}, to imply that the right side of \eqref{FU2.5} is further estimated as
\begin{align}   
&2 \sum_{\abs{\gamma}=m}
 \lambda^\modckg \int_{ \mathbb{R}^3  }\abs{ F _{U_i,(3)}^{(m)}  \p^\gamma \uu_i  }
 \notag\\
&\leq C_m \kappa_0 \lambda^{-\frac m2} E_m^2  e^{-s}  + 4  \left( \tfrac{\kappa_0}{4} + \eps^{\frac 18}\right)  \left( 2 \eps^{\frac 13} e^{-\frac s2} E_m^2 + \eps  e^{-s} + 3 \kappa_0^{-1} E_m^2\right)
\notag\\
&\leq (3 + \delta) E_m^2 + \eps^{\frac 12} e^{-s}
\, ,
\label{FU3}
\end{align} 
after taking $\eps$ to be sufficiently small, in terms of $\delta, \kappa_0$, and $m$.

The bounds \eqref{FU1}, \eqref{FU2}, and \eqref{FU3} provide the needed estimate for the contribution of the $\mathcal{F}_{\uu_i}^{(m)}$ term in \eqref{eq:F:U:sum} to \eqref{eq:Hk:est:FU}. It remains to bound the contribution from the lower order term $\mathcal{F}_{\uu_i}^{(<m)}$, which we recall is decomposed in three pieces, according to \eqref{eq:F:U:less:k}. Next, we estimate these three contributions.

The difficulty in addressing the $\mathcal{F}_{\uu_i,(1)}^{(<m)}$ term defined in \eqref{eq:F:U:less:k} arises due to the fact that  the bootstrap assumption  for $A$ in \eqref{eq:A_bootstrap} does not include bounds on the full Hessian $\nabla^2 A$. Therefore, we need to split off the $A_\nu$ (i.e. $\uu \cdot \Tcal^\nu$) contributions from the $W$ and $Z$ contributions (i.e. $\uu \cdot \Ncal$) to this term.  Using \eqref{Udef00}   we write the first term in \eqref{eq:F:U:less:k} as 
\begin{align} 
\mathcal{F}_{U_i,(1)}^{(<m)}
&=  \mathcal{I}_1 + \mathcal{I}_2+ \mathcal{I} _3\,, 
\label{eq:basic:shitstorm}
\end{align}
where 
\begin{align*}
\mathcal{I} _1 
&=
 -   \sum_{j=1}^{m-2} \sum_{\abs{\beta}=j, \beta \leq \gamma} {\gamma \choose \beta} 
\p^{\gamma -\beta} g_U \p^\beta \p_1(U \cdot \Ncal \, \Ncal_i )
 \,, \ \ \ \\
 \mathcal{I} _2 
&= - \sum_{j=1}^{m-2} \sum_{\abs{\beta}=j, \beta \leq \gamma} {\gamma \choose \beta} 
\p^{\gamma -\beta} g_U \p^\beta (\p_1 A_\nu \Tcal^\nu_i )
\,, \\
 \mathcal{I} _3
 &=
 -  \sum_{j=1}^{m-2} \sum_{\abs{\beta}=j, \beta \leq \gamma} {\gamma \choose \beta} 
 \p^{ \alpha - \beta } h_U^\nu \p^\beta \p_\nu U_i
 \,.
\end{align*} 
We estimate the contributions of the three terms in \eqref{eq:basic:shitstorm} individually. 
 
First, for the ${\mathcal I}_1$ term in \eqref{eq:basic:shitstorm}, by  Lemma \ref{lem:tailored:interpolation}  with $q=\frac{6(2m-3)}{2m-1}$, we have that
\begin{align} 
2 \sum_{\abs{\gamma}=m}
 \lambda^\modckg \int_{ \mathbb{R}^3  }\abs{ \mathcal{I} _1 \,  \gui  } 
 &
  \les \snorm{D^m g_U }_{L^2}^\acal        \snorm{D^m U}_{L^2}^\bcal 
   \snorm{ D^2 g_U}_{L^q}^{1- \acal } \snorm{ D^2 ( U\cdot \Ncal \Ncal)}_{L^q}^{1- \bcal }  \snorm{D^m U}_{L^2} \,,
   \label{decomp1}
 \end{align} 
where $ \acal $ and $ \bcal $  obey $\acal + \bcal = 1 - \frac{1}{2m-4}$. Note by \eqref{eq:gA} that $g_U$ does not include any $A$ term. Thus, using the bootstrap bounds \eqref{mod-boot}--\eqref{eq:Z_bootstrap}, or alternatively by appealing directly to \eqref{eq:W_decay}, \eqref{e:bounds_on_garbage} and the last bound in \eqref{e:G_ZA_estimates}, and the definition of $\XXX(s)$ in \eqref{eq:support} we 
deduce that 
\begin{align}
  \snorm{ D^2 g_U}_{L^q(\XXX(s))} \les M \snorm{\eta^{-\frac 16}}_{L^q(\XXX(s))}  + M^2 e^{-\frac s2} \abs{\XXX(s)}^{\frac 1q} \les M 
  \label{eq:D2:gU}
\end{align}
since $q \in [\frac{11}{2},6)$ for $m\geq 18$.  Similarly, from the first four bounds in \eqref{eq:US_est} (bounds which do not rely on any $A$ estimates) and from \eqref{e:bounds_on_garbage} (which only uses \eqref{eq:speed:bound} and  \eqref{e:space_time_conv}), we deduce that 
\begin{align}
  \snorm{ D^2((U \cdot \Ncal) \Ncal)}_{L^q(\XXX(s))} \les M e^{-\frac s2} \snorm{\eta^{-\frac 16}}_{L^q(\XXX(s))}  + M e^{-s} \abs{\XXX(s)}^{\frac 1q} \les M  e^{-\frac s2}  \, .
  \label{eq:D2:UdotN}
\end{align}
Moreover, from \eqref{eq:Scottie}, the bounds listed above \eqref{FU1}, the Poincar\'e inequality in the $\check y$ direction, and the fact that the diameter of $\XXX(s)$ in the $e_\mu$ directions is  $4 \eps^{\frac 16} e^{\frac{s}{2}}$ we have that 
\begin{align}
  \snorm{ D^m g_U}_{L^2} \les e^{\frac s2} \norm{U}_{\dot{H}^m} + \eps^{\frac 14} \norm{U}_{\dot{H}^{m-1}} \les e^{\frac s2} \norm{U}_{\dot{H}^m}\, .
  \label{eq:Dk:gU}
\end{align}
By combining \eqref{eq:D2:gU}--\eqref{eq:Dk:gU} we  obtain that the right side of \eqref{decomp1} is bounded from above as
\begin{align*} 
 & \snorm{D^m g_U }_{L^2}^\acal  \snorm{D^m U}_{L^2}^\bcal
   \snorm{ D^2 g_U}_{L^q}^{1- \acal } \snorm{ D^2 (U \cdot \Ncal \Ncal)}_{L^q}^{1- \bcal }  \snorm{D^m U}_{L^2}\\
&\qquad
  \les (e^{\frac{  s}{2}} \snorm{U}_{\dot{H}^m})^\acal        \snorm{U}_{\dot{H}^m}^{ \bcal }
   M^{1- \acal } (M e^{-\frac s2})^{1- \bcal }    \snorm{U}_{\dot{H}^m} \\
 &\qquad
  \les M^{2-\acal - \bcal} e^{\frac{(\acal+\bcal-1) s}{2}} \snorm{U}_{\dot{H}^m}^{1+\acal +\bcal} \,.
\end{align*} 
Recalling from Lemma \ref{lem:tailored:interpolation}  that $1- \acal - \bcal =  \frac{1}{2m-4} \in (0,1)$, the and using the norm equivalence  \eqref{norm_compare}, by 
Young's inequality with a small parameter $ \delta >0$, we have that the left side of \eqref{decomp1} is bounded as
\begin{align} 
2 \sum_{\abs{\gamma}=m}
 \lambda^\modckg \int_{ \mathbb{R}^3  }\abs{ \mathcal{I} _1 \,  \gui  } 
&\le C_m M^{2-\acal - \bcal}
 e^{\frac{(\acal+\bcal-1) s}{2}} \lambda ^\frac{-m(1+\acal +\bcal )}{2}    E_m^{1+\acal +\bcal }     \notag\\
&\le  \delta  E_m^2  + e^{-s} M^{4m-3} \,.
 \label{decomp2}
\end{align} 
In the last inequality we have used that by definition $\lambda = \lambda(m,\delta)$, $\delta \in (0,\frac{1}{32}]$ is a fixed universal constant, and $C_m$ is a constant that only depends on $m$; thus, we may use a power of $M$ (which is taken to be sufficiently large) to absorb all the $m$ and $\delta$ dependent constants.

Next, we estimate the   $\mathcal{I} _2$ term in \eqref{eq:basic:shitstorm}.  First, we note that by \eqref{eq:special1} we have
\begin{align*} 
\norm{ \mathcal{I} _2}_{L^2} 
&\les \sum_{j=1}^{m-2} \norm{D^{m-1-j} Dg_U}_{L^ {\frac{2(m-1)}{m-1-j}} }  \norm{ D^j (\p_1 A_\nu \Tcal^\nu )}_{L^ {\frac{2(m-1)}{j}} } \\
& \les \sum_{j=1}^{m-2} \norm{g_U}_{\dot{H}^{m} }^{\frac{m-1-j}{m-1}}  \norm{D g_U}_{ L^ \infty }^{\frac{j}{m-1}}
\norm{ \p_1 A_\nu \Tcal^\nu }_{ \dot{H}^{m-1} }^{\frac{j}{m-1}} \norm{ \p_1 A_\nu \Tcal^\nu  }_{L^\infty}^{ \frac{m-1-j}{m-1}}  \, .
\end{align*} 
Then, by appealing to \eqref{eq:gA}, \eqref{eq:W_decay}, \eqref{eq:A_bootstrap}, \eqref{e:bounds_on_garbage}, \eqref{e:G_ZA_estimates}, \eqref{norm_compare}, \eqref{eq:Dk:gU}, and \eqref{eq:Moser:inequality}, we deduce  
\begin{align*} 
\norm{ \mathcal{I} _2}_{L^2} 
& \les \sum_{j=1}^{m-2} \left(e^{\frac s2} \norm{U}_{\dot{H}^{m} }\right)^{\frac{m-1-j}{m-1}} \left(\norm{A}_{\dot{H}^m} + M \eps e^{-\frac{m+2}{2}s} \right)^{\frac{j}{m-1}} \left(M e^{-\frac{3s}{2}}\right)^{ \frac{m-1-j}{m-1}} \notag\\
& \les \sum_{j=1}^{m-2} \left( \lambda^{-\frac k2} E_m \right)^{\frac{m-1-j}{m-1}} \left( \lambda^{-\frac k2} E_m + M \eps e^{-\frac{m+2}{2}s} \right)^{\frac{j}{m-1}} \left(M e^{-s}\right)^{ \frac{m-1-j}{m-1}} \notag\\
&\les (M \eps)^{\frac{1}{m-1}} \lambda^{-\frac m2} E_m + M e^{-s}
\end{align*} 
since $\norm{D g_U}_{L^\infty} \les 1$. By taking $\eps$ sufficiently small, in terms of $M$, $\lambda = \lambda(m,\delta)$, $\delta$, and $m$, we obtain from the above estimate that 
\begin{align} 
2 \sum_{\abs{\gamma}=m}
 \lambda^\modckg \int_{ \mathbb{R}^3  }\abs{ \mathcal{I} _2 \,  \gui  } 
 &
  \le \delta E_m^2 +  e^{-s} 
   \label{decomp2.5}
 \end{align} 
 for all $s\geq -\log \eps$.
 
At last, we estimate the ${\mathcal I}_3$ term in \eqref{eq:basic:shitstorm}, which is estimated similarly to the ${\mathcal I}_2$ term as
\begin{align*}
\norm{{\mathcal I}_3}_{L^2}
&\les \sum_{j=1}^{m-2} \norm{h_U}_{\dot{H}^{m} }^{\frac{m-1-j}{m-1}}  \norm{D h_U}_{ L^ \infty }^{\frac{j}{m-1}}
\norm{ \p_\nu U_i}_{ \dot{H}^{m-1} }^{\frac{j}{m-1}} \norm{ \p_\nu U_i}_{L^\infty}^{ \frac{m-1-j}{m-1}}  \, .
\end{align*}
From \eqref{eq:Dennis}, the bounds \eqref{eq:W_decay},  \eqref{eq:Z_bootstrap}, \eqref{eq:A_bootstrap}, \eqref{e:bounds_on_garbage}, and the Moser inequality \eqref{eq:Moser:inequality}, we have  
\begin{align*}
\norm{ h_U}_{\dot H^m} 
\les e^{-\frac{s}{2}} \norm{\Ncal U\cdot \Ncal}_{\dot H^m} + \kappa e^{-\frac s2} \norm{A_\gamma \Tcal^\gamma}_{\dot{H}^m}  
\les M e^{-\frac{s}{2}} \norm{ U}_{\dot H^m} + M \eps e^{-\frac{m+1}{2} s} \,.
\end{align*}
On the other hand, by \eqref{e:h_estimates} we have $\norm{D h_U}_{L^\infty} \les e^{-s}$, while from  \eqref{eq:W_decay}, \eqref{eq:Z_bootstrap}, \eqref{eq:A_bootstrap},  and \eqref{Udef00} we obtain $\norm{\check\nabla  U}_{L^\infty} \les e^{-\frac s2}$. Combining the above three estimates, we deduce that 
\begin{align*}
\norm{{\mathcal I}_3}_{L^2}
&\les \sum_{j=1}^{m-2} \left( M e^{-\frac s2} \norm{U}_{\dot H^m} + e^{-2s} \right)^{\frac{m-1-j}{m-1}}  e^{-\frac{j}{m-1}s}
\norm{ U}_{ \dot{H}^{m} }^{\frac{j}{m-1}} e^{- \frac{m-1-j}{2(m-1)}s} 
\les Me^{-s} \norm{ U}_{ \dot{H}^{m} } + e^{-s}
\end{align*}
from which we deduce
\begin{align} 
2 \sum_{\abs{\gamma}=m}
 \lambda^\modckg \int_{ \mathbb{R}^3  }\abs{ \mathcal{I} _3 \,  \gui  } 
 &
  \le \eps^{\frac 12} E_m^2 +  e^{-s} 
   \label{decomp3}
 \end{align} 
upon taking $M$ to be sufficiently large in terms of $m$, and $\eps$ sufficiently large in terms of $M$.
Combining \eqref{decomp2}, \eqref{decomp2.5}, and \eqref{decomp3}, we have thus shown that
\begin{align} 
2 \sum_{\abs{\gamma}=m}
 \lambda^\modckg \int_{ \mathbb{R}^3  }\abs{ F_{U_i,(1)}^{(<m)}  \,  \gui  } 
 \le (2 \delta  + \eps^ {\frac{1}{2}} ) E_m^2 + M^{4m-2} e^{-s}  \,.
\label{FU5}
\end{align} 

We next turn to the second term in \eqref{eq:F:U:less:k}, which involves only derivatives of $\pp$, $\hh$, and $\Jcal \Ncal$. For the first term (the one with an $e^{\frac s2}$ prefactor) we apply the same bound as in \eqref{decomp1}, while for the second term we use \eqref{eq:special1}, to obtain
\begin{align}
&2 \sum_{\abs{\gamma} = m} \lambda^{\abs{\check \gamma}} \int_{\RR^3} \abs{\mathcal{F} _{ U_i,(2)}^{(<m)} \p^\gamma \uu_i} 
\notag\\
&\qquad \les (e^{\frac s2}  \snorm{D^m (\hh^2 \pp \Jcal \Ncal)}_{L^2})^\acal        \snorm{D^m \pp}_{L^2}^\bcal 
   (e^{\frac s2} \snorm{ D^2 (\hh^2 \pp \Jcal \Ncal)}_{L^q})^{1- \acal } \snorm{ D^2 \pp}_{L^q}^{1- \bcal }  \snorm{D^m U}_{L^2} 
   \notag\\
 &\qquad \qquad + e^{-\frac s2} \sum_{j=1}^{m-2} \norm{\hh^2 \pp}_{\dot{H}^{m} }^{\frac{m-1-j}{m-1}}  \norm{D (\hh^2 \pp)}_{ L^ \infty }^{\frac{j}{m-1}}
\norm{\pp }_{ \dot{H}^{m} }^{\frac{j}{m-1}} \norm{ \check \nabla \pp }_{L^\infty}^{ \frac{m-1-j}{m-1}}    \snorm{D^m U}_{L^2} 
\notag\\
&=: \mathcal{T}_1 + \mathcal{T}_2 \,,
\label{FU6}
\end{align}
with $q = \frac{6(2m-3)}{2m-1}$, and $\acal + \bcal = 1 - \frac{1}{2m-4}$. 
Recalling that $\pp = \sound \hh^{-1} = (\uu \cdot \Ncal - Z) \hh^{-1}$, the definition of $\hh$, our bootstrap assumptions on $Z$ and $K$, exactly as in \eqref{eq:D2:UdotN} we have the estimate 
\begin{align*}
\norm{D^2 \pp}_{L^q(\XXX(s))} 
&\les \norm{D^2(\uu \cdot \Ncal)}_{L^q(\XXX(s))} + \left( \norm{D^2 Z}_{L^\infty} + \norm{D\hh}_{L^\infty} \norm{D(\uu\cdot \Ncal - Z)}_{L^\infty}   \right) \abs{\XXX(s)}^{\frac 1q} 
\notag\\
&\qquad + \norm{\uu\cdot \Ncal - Z}_{L^\infty} \left(\norm{D^2 \hh}_{L^\infty} + \norm{D\hh}_{L^\infty}^2  \right) \abs{\XXX(s)}^{\frac 1q}
\notag\\
&\les M e^{-\frac s2} \,.
\end{align*}
Thus, the Hessian of $\pp$ obeys the same estimate as the Hesssian of $(\uu \cdot \Ncal) \Ncal$ in \eqref{eq:D2:UdotN}. Similarly, by using \eqref{e:bounds_on_garbage}, \eqref{eq:useful:crap}, \eqref{eq:shalala:5}, and \eqref{eq:shalala:7}, as in \eqref{eq:D2:gU} and \eqref{eq:D2:UdotN} we have
\begin{align*}
e^{\frac s2} \norm{D^2 (\hh^2 \pp \Jcal \Ncal)}_{L^q(\XXX(s))} 
 &\les e^{\frac s2} \norm{D^2 (\hh^2 \pp)}_{L^q(\XXX(s))} +  \norm{D (\hh^2 \pp)}_{L^\infty}  \abs{\XXX(s)}^{\frac 1q} \les M \,.
\end{align*}
The above estimate is exactly the same as the Hessian of $g_U$ bound in \eqref{eq:D2:gU}. Clearly we have that $\norm{\pp}_{\dot{H}^m} \les \lambda^{-\frac m2} E_m$, and additionally, from the Moser inequality \eqref{e:bounds_on_garbage}, \eqref{eq:useful:crap}, \eqref{eq:shalala:5}, and \eqref{eq:shalala:7}
we have that 
\begin{align*}
 e^{\frac s2} \norm{ \hh^2 \pp \Jcal \Ncal }_{\dot{H}^m} 
 &\les e^{\frac s2} \left( \kappa_0 \norm{\hh}_{\dot{H}^m} + \norm{\pp}_{\dot{H}^m} \right) \les e^{\frac s2} \lambda^{-\frac m2} E_m
\end{align*}
which is the same as the bound on on the $\dot{H}^m$ norm of $g_U$ obtained in \eqref{eq:Dk:gU}.
In view of these analogies, proceeding in exactly the same way as in \eqref{decomp2}, we obtain that the first term in \eqref{FU6} is estimated as
\begin{align}
\mathcal{T}_1 \leq \delta E_m^2 + e^{-s} M^{4m-3} \,.
\label{eq:quantitative:farts:1} 
\end{align}
For the second term in \eqref{decomp2} we recall that by the Moser inequality, \eqref{eq:useful:crap}, and \eqref{eq:shalala:7} we have $\norm{\hh^2 \pp}_{\dot{H}^m} \les \norm{\pp}_{\dot{H}^m} + \kappa_0 \norm{\hh}_{\dot{H}^m} \les \lambda^{-\frac m2} E_m$, and by also appealing to \eqref{eq:shalala:5} we obtain
\begin{align}
\mathcal{T}_2 
&\les \lambda^{-m} e^{-\frac s2} E_m^2 \sum_{j=1}^{m-2}  \norm{D (\hh^2 \pp)}_{ L^ \infty }^{\frac{j}{m-1}}
  \norm{ \check \nabla \pp }_{L^\infty}^{ \frac{m-1-j}{m-1}}  
\les \lambda^{-m} e^{-s} E_m^2 
\leq \delta E_m^2
\label{eq:quantitative:farts:2} 
\end{align}
after taking $\eps$ to be sufficiently small to absorb the $m$, $\lambda$, and $\delta$-dependent constants. 
By combining \eqref{FU6}, \eqref{eq:quantitative:farts:1}, and \eqref{eq:quantitative:farts:2}, we obtain that
\begin{align}
2 \sum_{\abs{\gamma} = m} \lambda^{\abs{\check \gamma}} \int_{\RR^3} \abs{\mathcal{F} _{ U_i,(2)}^{(<m)} \p^\gamma \uu_i} 
\leq 2 \delta E_m^2 + e^{-s} M^{4m-3}\,.
\label{FU7}
\end{align} 

At last, we consider the third term in \eqref{eq:F:U:less:k}. 
Recall that cf.~\eqref{eq:shalala:5} that $e^{\frac s2} \norm{\p_1 \pp}_{L^\infty} \leq 1$, and that since $\Jcal \Ncal$ is linear in $\check y$, by Poincar\'e inequality in the $\check y$ direction, and the fact that the diameter of $\XXX(s)$ in the $\check e$ directions is  $4 \eps^{\frac 16} e^{\frac{s}{2}}$, we obtain that $\norm{\hh^2 \Jcal\Ncal}_{\dot{H}^m} \les \norm{\hh}_{\dot{H}^m}$. Thus, by appealing to \eqref{e:bounds_on_garbage}, \eqref{norm_compare}, \eqref{eq:shalala:5},  \eqref{eq:special1} and the Poincar\'e inequality in the $\check y$ direction we arrive at 
\begin{align}
&2 \sum_{\abs{\gamma} = m} \lambda^{\abs{\check \gamma}} \int_{\RR^3} \abs{\mathcal{F} _{ U_i,(3)}^{(<m)} \p^\gamma \uu_i} 
\notag\\
&\les E_m \sum_{j=1}^{m-1} \norm{D(\hh^2 \Jcal\Ncal)}_{L^\infty}^{ \frac{j}{m-1}} \norm{\hh^2 \Jcal \Ncal}_{\dot{H}^m}^{\frac{m-j-1}{m-1}} \norm{\pp}_{L^\infty}^{1 - \frac{j}{m-1}} \norm{\pp}_{\dot{H}^{m-1}}^{\frac{j}{m-1}} \notag\\
&\les E_m \sum_{j=1}^{m-1} \left(\eps e^{-\frac s2}\right)^{ \frac{j}{m-1}} \norm{\hh}_{\dot{H}^m}^{\frac{m-j-1}{m-1}} \kappa_0^{1 - \frac{j}{m-1}} \left(\eps^{\frac 16} e^{\frac{s}{2}} \norm{\pp}_{\dot{H}^{m}}\right)^{\frac{j}{m-1}} \notag\\
&\leq \delta E_m^2
\label{FU8}
\end{align} 
upon taking $\eps$ to be sufficiently small, in terms of $\lambda$, $m$, and $\kappa_0$. 

The bounds \eqref{FU5}, \eqref{FU7}, and \eqref{FU8} provide the needed estimate for the contribution of the $\mathcal{F}_{\uu_i}^{(<m)}$ term in \eqref{eq:F:U:sum} to \eqref{eq:Hk:est:FU}, thereby completing the proof of \eqref{eq:Hk:est:FU}.

{\bf Proof of \eqref{eq:Hk:est:FP}.} The proof is extremely similar to that of \eqref{eq:Hk:est:FU}. Comparing the forcing terms in \eqref{eq:F:P:k} with those in \eqref{eq:F:U:k}, and those in \eqref{eq:F:P:less:k}, with those in \eqref{eq:F:U:less:k}, we see that they only differ by exchanging $\uu$ with $\pp$ in several places; in fact, here we have fewer terms to bound. The contribution from $\mathcal{F}_{\pp,(1)}^{(m)}$ is estimated in precisely the same way as the one from $\mathcal{F}_{\uu_i,(1)}^{(m)}$ in \eqref{FU1}. Similarly, the contribution from $\mathcal{F}_{\pp,(2)}^{(m)}$ is estimated in precisely the same way as the one from $\mathcal{F}_{\uu_i,(2)}^{(m)}$ in \eqref{FU2}. Note that there is no third term in the definition of $\mathcal{F}_{\pp}^{(m)}$, and thus we do not need to add a $(3+\delta)$ to our error estimate, as we had to do for the $\uu$ forcing in view of \eqref{FU3}. Next, $\mathcal{F}_{\pp,(1)}^{(<m)}$, $\mathcal{F}_{\pp,(2)}^{(<m)}$, and $\mathcal{F}_{\pp,(3)}^{(<m)}$ are bounded in precisely the same way as $\mathcal{F}_{\uu_i,(1)}^{(<m)}$, $\mathcal{F}_{\uu_i,(2)}^{(<m)}$, and $\mathcal{F}_{\uu_i,(3)}^{(<m)}$ in \eqref{FU5}, \eqref{FU7} and respectively \eqref{FU8}. To avoid redundancy, we omit these details.

{\bf Proof of \eqref{eq:Hk:est:FE}.} Again, the proof is similar to that of \eqref{eq:Hk:est:FU}, except that in \eqref{eq:F:E:k} and \eqref{eq:F:E:less:k} we have much fewer terms. We need to be slightly careful here, as the $\p^\gamma \hh$ evolution is tested with $\kappa_0^2 \p^\gamma \hh$, rather than just $\p^\gamma \hh$, and we need to ensure that our damping bounds are independent of $\kappa_0$! The reason this is achieved is as follows. For the terms which contain a $D^m \hh$, such as the first two terms in \eqref{eq:F:E:k}, there is no issue as each of the two powers of $\kappa_0$ are paired with an $\norm{\hh}_{\dot{H}^m}$. An issue may arise in terms which contain $D^m \uu$, such as the last two terms in \eqref{eq:F:E:k}. The important thing to notice here is that each such term is paired with $\norm{\nabla \hh}_{L^\infty}$. As opposed to $\nabla \pp$, which satisfies $\norm{\nabla \pp}_{L^\infty} \approx \frac 12 e^{-\frac s2}$, by  \eqref{eq:shalala:5} we have that $\norm{\nabla \hh}_{L^\infty} \leq \eps^{\frac 13} e^{-\frac s2}$. This additional factor of $\eps^{\frac 13}$ is able to absorb all the stray powers of $\kappa_0$. A similar argument applies to the terms in \eqref{eq:F:E:less:k}, showing that the resulting bounds are independent of $\kappa_0$.  The contribution from $\mathcal{F}_{\hh}^{(m)}$ is estimated in precisely the same way as the one from $\mathcal{F}_{\uu_i,(1)}^{(m)}$ in \eqref{FU1}, while the contribution of $\mathcal{F}_{\hh}^{(<m)}$, is  bounded in precisely the same way as $\mathcal{F}_{\uu_i,(1)}^{(<m)}$ in \eqref{FU5}. To avoid redundancy, we omit further details.
\end{proof} 

\subsection{The $E_m$ energy estimate}
\label{sec:Ek:energy:estimate}
We now turn to the main energy estimate for the differentiated system \eqref{UPS-L2}. 
\begin{proposition}[$\dot H^m$  estimate for $\uu$, $\pp$, and $\hh$]\label{prop:L2}
For any integer $m$ satisfying
\begin{align}
 m \geq 18 \,,
 \label{eq:k-cond}
\end{align}
with $\delta$ and $\lambda= \lambda(m,\delta)$ as specified in Lemma~\ref{lem:forcing:1}, we have the estimate
\begin{align}
&E_m^2(s)
  \leq  e^{-2(s-s_0)} E_m^2(s_0) + 3 e^{-s} M^{4m-1} \left(1 - e^{-(s-s_0)}\right) 
 \label{eq:L2}
\end{align}
for all $s\geq  s_0 \geq -\log \eps$.
\end{proposition}

\begin{proof}[Proof of Proposition~\ref{prop:L2}]
We fix a   multi-index $\gamma \in {\mathbb N}_0^3$ with $\abs{\gamma} = m$,  and consider the sum of the $L^2$ inner-product of \eqref{UPS-L2-U} with 
$2 \lambda^\modckg    \p^\gamma U^i$ and the $L^2$ inner-product of  \eqref{UPS-L2-P} with $2 \lambda^\modckg  \hh^2 \p^\gamma P$
 and the $L^2$ inner-product of  \eqref{UPS-L2-K} with $2 \kappa_0^2 \lambda^\modckg  \p^\gamma \hh$.   
With the damping function $ \mathcal{D} _\gamma$ defined in \eqref{Dgamma} and the transport velocity
$ \mathcal{V} _U$ defined in \eqref{V_U},   using the fact that $\dot Q$ is skew-symmetric and that $(\p_s + \mathcal{V}_U \cdot \nabla) \hh = 0$,   we find that
\begin{align}
&\frac{d}{ds}  \lambda^\modckg \int_{ \mathbb{R}^3  } \left(   \abs{\p^\gamma \uu}^2    +\hh^2 \abs{ \p^\gamma \pp}^2 + \kappa_0^2 \abs{ \p^\gamma \hh}^2 \right) 
\notag\\
&\quad + \lambda^\modckg \int_{ \mathbb{R}^3  } \left( 2 \mathcal{D} _\gamma  - \div \mathcal{V}_U \right) \,  
\left(  \abs{\p^\gamma U }^2 +\hh^2 \abs{\p^\gamma \pp }^2 + \kappa_0^2 \abs{ \p^\gamma \hh}^2\right) 
\notag\\
&\quad + 8 \gamma_1 \beta_\tau \beta_3 \lambda^\modckg \int_{ \mathbb{R}^3  }
\hh^2 (\p^\gamma \pp)  \Jcal (\Ncal \cdot \p^\gamma \uu) e^{\frac s2} \p_1 \pp
\notag \\
&\quad + 4\beta_\tau \beta_3  \lambda^\modckg \int_{ \mathbb{R}^3  } 
(\p^\gamma \pp) \hh^2 \Jcal (\Ncal \cdot \p^\gamma \uu) e^{\frac s2} \p_1 \pp + (\p^\gamma \pp)^2 \hh^2   \Jcal e^{\frac s2} \p_1 (\uu \cdot \Ncal)  
\notag\\
 & \quad + 4 \beta_\tau \beta_3\lambda^\modckg \int_{ \mathbb{R}^3  }    
 \hh^2 \pp \left( e^{\frac s2} \Jcal (\Ncal \cdot \p^\gamma \uu) \p_1 (\p^\gamma \pp) + e^{\frac s2} \Jcal \p_1 (\Ncal \cdot \p^\gamma \uu) (\p^\gamma \pp) \right)
 \notag\\
& \quad + 4 \beta_\tau \beta_3\lambda^\modckg \int_{ \mathbb{R}^3  }    
 \hh^2 \pp \left(  e^{-\frac s2} (\p^\gamma \uu_\nu) \p_\nu (\p^\gamma \pp) + e^{-\frac s2} (\p^\gamma \pp) \p_\nu (\p^\gamma \uu_\nu) \right)
 \notag \\
& = 2 \lambda^\modckg  \int_{ \mathbb{R}^3  } \left( \mathcal{F} _{ U^i}^{(\gamma)} \,  \p^\gamma U_i  + \hh^2\mathcal{F} _{\pp}^{(\gamma)}\, \p^\gamma \pp
+\kappa_0^2 \mathcal{F} _{\hh}^{(\gamma)}\, \p^\gamma \hh \right)  \,.
\label{energy0}
\end{align}
We note that in the last two integrals on the left-hand side of the identity \eqref{energy0} we may integrate by parts:
\begin{align*} 
&4 \beta_\tau \beta_3\lambda^\modckg \int_{ \mathbb{R}^3  }    
 \hh^2 \pp \left( e^{\frac s2} \Jcal (\Ncal \cdot \p^\gamma \uu) \p_1 (\p^\gamma \pp) + e^{\frac s2} \Jcal \p_1 (\Ncal \cdot \p^\gamma \uu) (\p^\gamma \pp) \right)
 \notag\\
&\qquad + 4 \beta_\tau \beta_3\lambda^\modckg \int_{ \mathbb{R}^3  }    
 \hh^2 \pp \left(  e^{-\frac s2} (\p^\gamma \uu_\nu) \p_\nu (\p^\gamma \pp) + e^{-\frac s2} (\p^\gamma \pp) \p_\nu (\p^\gamma \uu_\nu) \right)
\\
&= - 4 \beta_\tau \beta_3\lambda^\modckg \int_{ \mathbb{R}^3  }    
e^{\frac s2} \p_1 \left(\hh^2 \pp\right)  \Jcal (\Ncal \cdot \p^\gamma \uu)  (\p^\gamma \pp)  
 \notag\\
&\qquad - 4 \beta_\tau \beta_3\lambda^\modckg \int_{ \mathbb{R}^3  }    
e^{-\frac s2} \p_\nu \left( \hh^2 \pp  \right)   (\p^\gamma \uu_\nu)  (\p^\gamma \pp) 
\end{align*} 
where we have used that $\p_1 \Jcal = 0$.
Therefore, upon rearranging,  the energy equality \eqref{energy0} becomes
\begin{align}
&\frac{d}{ds} \lambda^\modckg \int_{ \mathbb{R}^3  }  \left(   \abs{\p^\gamma \uu}^2    +\hh^2 \abs{ \p^\gamma \pp}^2 + \kappa_0^2 \abs{ \p^\gamma \hh}^2 \right) 
\notag\\
& + \lambda^\modckg \int_{ \mathbb{R}^3  } \left( 2 \mathcal{D} _\gamma  - \div \mathcal{V}_U \right) \,  
\left(  \abs{\p^\gamma U }^2 +\hh^2 \abs{\p^\gamma \pp }^2 + \kappa_0^2 \abs{ \p^\gamma \hh}^2\right) 
\notag\\
&\ + 8 \gamma_1 \beta_\tau \beta_3 \lambda^\modckg \int_{ \mathbb{R}^3  }
\hh^2 (\p^\gamma \pp)  \Jcal (\Ncal \cdot \p^\gamma \uu) e^{\frac s2} \p_1 \pp
\notag \\
&\ + 4\beta_\tau \beta_3  \lambda^\modckg \int_{ \mathbb{R}^3  } 
 (\p^\gamma \pp)^2 \hh^2   \Jcal e^{\frac s2} \p_1 (\uu \cdot \Ncal)  
 - 2 (\p^\gamma \pp)  \hh  \Jcal (\Ncal \cdot \p^\gamma \uu) e^{\frac s2} \pp   \p_1 \hh
 -  (\p^\gamma \uu_\nu)  (\p^\gamma \pp) e^{-\frac s2} \p_\nu \left( \hh^2 \pp  \right)  
\notag\\
& = 2 \lambda^\modckg  \int_{ \mathbb{R}^3  } \left( \mathcal{F} _{ U^i}^{(\gamma)} \,  \p^\gamma U_i  + \hh^2\mathcal{F} _{\pp}^{(\gamma)}\, \p^\gamma \pp
+\kappa_0^2 \mathcal{F} _{\hh}^{(\gamma)}\, \p^\gamma \hh \right)  \,.
\label{energy00}
\end{align}
We shall next obtain a lower bound for the second thru fourth integrals on the right side of \eqref{energy00}.

For the second integral, we recall \eqref{Dgamma}, use \eqref{V_U}, and the bounds \eqref{Jp1W},  \eqref{e:G_ZA_estimates}, \eqref{e:h_estimates}, \eqref{eq:Celtics:suck} to obtain the lower bound
\begin{align} 
 2 \mathcal{D} _\gamma - \operatorname{div} \mathcal{V} _U  
&=   \abs{\gamma}  - \tfrac{5}{2}+2 \gamma_1 +  (2 \gamma_1 -1)(\beta_\tau \beta_1 \Jcal \p_1 W + \p_1 G_U ) - \p_\nu h_U^\nu \notag \\
&\ge    \abs{\gamma}  - \tfrac{5}{2}+2 \gamma_1 -  \beta_\tau \beta_1 (2 \gamma_1 -1)_+ - \eps^ {\frac{1}{16}}   \,.   \label{pendejo1}
\end{align}
For the third integral, we note that by the definitions \eqref{eq:UdotN:Sigma}, \eqref{Pnew} and \eqref{Snew}
\begin{align}
2 \hh \nabla \pp  = e^{-\frac s2} \nabla W - \nabla Z - \tfrac{1}{ \upgamma} \sound \nabla K
\label{pendejo3/2}
\end{align}
and thus,  from \eqref{Jp1W}, \eqref{eq:W_decay}, \eqref{eq:Z_bootstrap},  \eqref{eq:US_est}, \eqref{eq:Celtics:suck},  the third integral on the left-hand side of  \eqref{energy00} has an integrand which is bounded as
\begin{align}
\frac{8 \gamma_1 \beta_\tau \beta_3 \hh^2  \sabs{ (\p^\gamma \pp)  \Jcal (\Ncal \cdot \p^\gamma \uu) e^{\frac s2} \p_1 \pp }}{\abs{\p^\gamma U}^2 + \hh^2 \abs{\p^\gamma \pp}^2}
&\leq
 4 \gamma_1 \beta_\tau \beta_3 \Jcal  \hh  e^{\frac s2}   \abs{  \p_1 \pp } 
 \notag\\
 &\leq 
2 \gamma_1 \beta_\tau \beta_3 \Jcal   \sabs{  \p_1 W - e^{\frac s2} \p_1 Z - \upgamma^{-1}\sound  e^{\frac s2}  \p_1 K} 
\notag\\
 &\leq 
2 \gamma_1 \beta_\tau \beta_3  \left( 1 +  2 M^{\frac 12} e^{-s} \right)
\notag\\
&\leq 2 \gamma_1 \beta_\tau \beta_3 + \eps^{\frac 12}\,.
\label{pendejo2}
\end{align}
Lastly, we compute $\p_1 (U \cdot \Ncal)$ from \eqref{eq:UdotN:Sigma}, $\p_\nu \pp$ from \eqref{pendejo3/2}, and by using \eqref{Jp1W}, \eqref{eq:W_decay}, \eqref{eq:Z_bootstrap},  \eqref{eq:S_bootstrap},   \eqref{eq:US_est},  \eqref{eq:Celtics:suck}, \eqref{eq:useful:crap}, and \eqref{eq:shalala:7}, the integrand in the fourth integral on the left-hand side of \eqref{energy00} may be estimated as
\begin{align}
&\frac{4\beta_\tau \beta_3 \sabs{
 (\p^\gamma \pp)^2 \hh^2   \Jcal e^{\frac s2} \p_1 (\uu \cdot \Ncal)  
 - 2 (\p^\gamma \pp)  \hh  \Jcal (\Ncal \cdot \p^\gamma \uu) e^{\frac s2} \pp   \p_1 \hh
 -  (\p^\gamma \uu_\nu)  (\p^\gamma \pp) e^{-\frac s2} \p_\nu \left( \hh^2 \pp  \right)}}
{\abs{\p^\gamma U}^2 + \hh^2 \abs{\p^\gamma \pp}^2}
\notag\\
&\leq 4\beta_\tau \beta_3 \left( 
  \Jcal e^{\frac s2} \sabs{\p_1 (\uu \cdot \Ncal)}
 +    \Jcal   e^{\frac s2} \sabs{\pp   \p_1 \hh}
 + \tfrac 12 \hh^{-1} e^{-\frac s2} \sabs{\p_\nu \left( \hh^2 \pp  \right)}\right)
 \notag\\
 &\leq 4\beta_\tau \beta_3 \left( \tfrac 12 + M e^{-s}\right)
 \notag\\
 &\leq 2 \beta_\tau \beta_3 + \eps^{\frac 12} \,.
 \label{pendejo3}
\end{align}
Combining the bounds \eqref{pendejo1}, \eqref{pendejo2}, and \eqref{pendejo3}, with the energy equality \eqref{energy00}, we arrive at
\begin{align}
&\frac{d}{ds} \int_{ \mathbb{R}^3  } \lambda^\modckg \left(   \abs{\p^\gamma \uu}^2    +\hh^2 \abs{ \p^\gamma \pp}^2 + \kappa_0^2 \abs{ \p^\gamma \hh}^2 \right) 
+{\mathcal D}_{\rm total}   \int_{ \mathbb{R}^3  }\lambda^\modckg 
\left(  \abs{\p^\gamma U }^2 +\hh^2 \abs{\p^\gamma \pp }^2 +\kappa_0^2 \abs{ \p^\gamma \hh}^2\right) 
\notag\\
& \leq 2 \lambda^\modckg  \int_{ \mathbb{R}^3  } \abs{ \mathcal{F} _{ U^i}^{(\gamma)} \,  \p^\gamma U_i  + \hh^2\mathcal{F} _{\pp}^{(\gamma)}\, \p^\gamma \pp
+ \kappa_0^2 \mathcal{F} _{\hh}^{(\gamma)}\, \p^\gamma \hh }  \,,
\label{energy000}
\end{align}
where we have denoted
\begin{align*}
{\mathcal D}_{\rm total} = \abs{\gamma}  - \tfrac{5}{2}+2 \gamma_1 -  \beta_\tau \beta_1 (2 \gamma_1 -1)_+ - \eps^ {\frac{1}{16}}  - 2 \gamma_1 \beta_\tau \beta_3 - 2 \beta_\tau \beta_3 - 2 \eps^{\frac 12} \,.
\end{align*}
The crucial observation here is that because $\beta_1 + \beta_3 =1$ (cf.~\eqref{eq:various:beta}), and appealing to \eqref{eq:beta:tau},  the damping term ${\mathcal D}_{\rm total}$ has the lower bound
\begin{align}
{\mathcal D}_{\rm total} 
\geq \abs{\gamma} - \tfrac 52  + 2 \gamma_1 (1- \beta_\tau) - 2\beta_\tau \beta_3 - \eps^{\frac{1}{16}} - 2 \eps^{\frac 12} 
\geq m - \tfrac 92 
\label{eq:total:dumping}
\end{align}
for $\eps$ taken sufficiently small, in terms of $\alpha$ and $m$.
Upon summing over $\abs{\gamma}=m$, the energy inequality and \eqref{energy000} and the damping lower bound \eqref{eq:total:dumping} thus yield
\begin{align} 
& \frac{d}{ds} E_m^2  + \left(m- \tfrac{9}{2}\right) E_m^2  
 \le 
\sum_{\abs{\gamma}=m}
 2\lambda^\modckg \int_{ \mathbb{R}^3  } \abs{ \mathcal{F} _{ U^i}^{(\gamma)} \,  \p^\gamma U_i  + \hh^2\mathcal{F} _{\pp}^{(\gamma)}\, \p^\gamma \pp
+\kappa_0^2 \mathcal{F} _{\hh}^{(\gamma)}\, \p^\gamma \hh }
  \,. \label{energy1}
\end{align} 
We are left with estimating the right side of \eqref{energy1}, which is the content of Lemma~\ref{lem:forcing:1} above. 
By Lemma \ref{lem:forcing:1}, for $0< \delta \le \tfrac{1}{32} $,
\begin{align*} 
& \tfrac{d}{ds} E_m^2(s) + (m-6) E_m^2(s) 
 \le 
(9+ 21 \delta)  E_m^2  + 3 e^{-s} M^{4m-1}  \,,
\end{align*} 
and hence, by since $m$ was taken sufficiently large in \eqref{eq:k-cond}, we have that  
\begin{align*}
\tfrac{d}{ds} E_m^2 + 2 E_m^2  
&\leq 3 e^{-s} M^{4m-1}  \,,
\end{align*}
and so 
we obtain that 
\begin{align*}
E_m^2(s) \leq  e^{-2(s-s_0)} E_m^2(s_0) +  3 e^{-s} M^{4m-1} \left(1 - e^{-(s-s_0)}\right) \,,
\end{align*}
for all $s\geq s_0 \geq -\log \eps$. This concludes the proof of Proposition~\ref{prop:L2}.
\end{proof}

In conclusion of this section,  we mention that Proposition~\ref{prop:L2} applied with $s_0=-\log \eps$, in conjunction with Lemma~\ref{lem:norm:equivalence}, yields the proof of Proposition \ref{cor:L2}.

\begin{proof}[Proof of Proposition \ref{cor:L2}]
The initial datum assumption \eqref{eq:data:Hk}  together with the first bound in \eqref{eq:ghastly:0} implies that  
$$
E_m^2(-\log \eps) \leq 2 \kappa_0^2 \eps\,.
$$
Thus, from \eqref{eq:L2} the second bound in \eqref{eq:ghastly:0}  we obtain
\begin{align*}
&e^{-s} \norm{W}_{\dot{H}^m}^2 + \norm{Z}_{\dot {H}^m}^2 + \norm{A}_{\dot{H}^m}^2 +  \norm{K}_{\dot {H}^m}^2
 \notag\\
&\qquad 
\leq 4 \lambda^{-m} E_m^2(s) + 4 e^{-2s} 
\notag\\
&\qquad \leq 8 \kappa_0^2 \lambda^{-m}  \eps^{-1} e^{-2s} + 12 \lambda^{-m} e^{-s} M^{4m-1} (1- \eps^{-1} e^{-s}) + 4 e^{-2s}
\notag\\
&\qquad \leq 16 \kappa_0^2 \lambda^{-m}  \eps^{-1} e^{-2s} +  e^{-s} M^{4m} (1- \eps^{-1} e^{-s})  
\end{align*}
by taking $M$ sufficiently slow. The inequalities \eqref{eq:AZ-L2}--\eqref{eq:W-L2} immediately follow.
\end{proof}

\section{Auxiliary lemmas and bounds on forcing functions} 
\label{sec:preliminary}
We begin by recording some useful bounds that will be used repetitively throughout the section.
\begin{lemma}
\label{lem:BBS}
For $y\in\XXX(s)$ and for $m\geq 0$ we have
\begin{subequations} 
\label{f-bounds}
\begin{align}
\label{e:bounds_on_garbage}
\abs{\check \nabla^m f} &+ \abs{\check \nabla^m (\Ncal- \Ncal_0 )}+\abs{\check \nabla^m (\tt^\nu -  \Tcal_0^\nu )} \notag\\
& +\abs{\check \nabla^m (\Jcal-1)} +\abs{\check \nabla^m (\Jcal^{-1}-1)}  
\les \eps M^2 e^{-\frac{m+2}{2}s } \abs{\check y}^2 \les \eps e^{-\frac{m}{2}s } \,,
\\
\abs{\check \nabla^m \dot f}&+\abs{\check \nabla^m \dot \Ncal}  \les M^2 e^{-\frac{m+2}{2}s } \abs{\check y}^2 \les \eps^{\frac 14} e^{-\frac{m}{2}s }\,.
\label{e:bounds_on_garbage_2}
\end{align}
\end{subequations} 
Moreover, we have the following estimates on $V$  
\begin{equation}\label{eq:V:bnd}
\abs{\p^\gamma V} \les \begin{cases}
 M^{\frac 14}  & \mbox{if } \abs{\gamma}=0\\
 M^2 \eps^{\frac 12}  e^{-\frac32 s} & \mbox{if }\abs{ \gamma}=1\mbox{ and } \gamma_1=1\\
 M^2 \eps^{\frac 12}  e^{-\frac s2 } & \mbox{if } \abs{ \gamma}=1\mbox{ and } \gamma_1=0\\
 M^4 \eps^{\frac 32}  e^{-s}  &  \mbox{if } \abs{ \gamma}=2\mbox{ and } \gamma_1=0\\
0&  \mbox{else}
\end{cases}
\end{equation}
for all $y \in \XXX(s)$. 
\end{lemma}
\begin{proof}[Proof of Lemma~\ref{lem:BBS}]
The estimates \eqref{e:bounds_on_garbage} follow directly from the definitions of $f$,  $\Ncal$, $\tt$ and $\Jcal$, together with the bounds on $\phi$ given in \eqref{eq:speed:bound} and the inequality \eqref{e:space_time_conv}. Similarly, \eqref{e:bounds_on_garbage_2} follows by using the $\dot \phi$ estimate in \eqref{eq:acceleration:bound}. To obtain the bound \eqref{eq:V:bnd}, we recall that $V$ is defined in \eqref{def_V}, employ the bounds on $\dot \xi$ and $\dot Q$  given by \eqref{eq:acceleration:bound} and \eqref{eq:dot:Q}, and the fact that $\abs{R-\Id} \leq 1$ which follows from \eqref{eq:speed:bound} and the definition of $R$ in (2.2) of \cite{BuShVi2019b}.
\end{proof}

\subsection{Transport estimates}

\begin{lemma}[Estimates for $G_W$, $G_Z$, $G_{U}$, $h_W$, $h_Z$ and $h_{U}$]
\label{lemma_g} 
For $\eps>0$ sufficiently small, and $y\in \XXX(s)$, we have
\begin{align} 
&\abs{\partial^{\gamma}G_W}  \les 
\begin{cases}
M e^{-\frac s2}+ M^{\frac 12}   \abs{y_1}e^{-s}+\eps^{\frac 13} \abs{ \check y},  &\mbox{if } \abs{\gamma}=0\\
 M^2  \eps^{\frac12}, & \mbox{if } \gamma_1=0\mbox{ and } \abs{\check \gamma}=1\\
M  e^{-\frac s2}, &\mbox{if } \gamma=(1,0,0)\mbox{ or } \abs{\gamma}=2 \\
  M^ {\frac{1}{2}}   e^{-s}, &\mbox{if } \gamma=(2,0,0)
\end{cases}\label{e:G_W_estimates}\,, \\
&\abs{\partial^{\gamma}(G_Z+(1-\beta_2)e^{\frac s2}\kappa_0)}+\abs{\partial^{\gamma}(G_{U}+(1-\beta_1) e^{\frac s2}\kappa_0)} \les 
\begin{cases}
\eps^{\frac 12} e^{\frac s2},  &\mbox{if } \abs{ \gamma}=0\\
 M^2 \eps^{\frac12}, & \mbox{if } \gamma_1=0\mbox{ and } \abs{\check \gamma}=1\\
 M  e^{-\frac s2}, &\mbox{if } \gamma=(1,0,0)\mbox{ or } \abs{\gamma}=2 \\
 M^ {\frac{1}{2}}   e^{-s}, &\mbox{if } \gamma=(2,0,0)
\end{cases}\label{e:G_ZA_estimates}\,, \\
& \abs{\partial^{\gamma}h_W}+\abs{\partial^{\gamma}h_Z}+\abs{\partial^{\gamma}h_{U}} \les 
\begin{cases}
 e^{-\frac s2},  &\mbox{if } \abs{ \gamma}=0\\
e^{-s}, & \mbox{if } \gamma_1=0\mbox{ and } \abs{\check \gamma}=1\\
 e^{-s} \eta^{-\frac 16}, & \mbox{if } \gamma=(1,0,0),\mbox{ or } (\abs{\gamma}=2 \mbox{ and } \abs{\check\gamma}=1,2)\\
e^{-( 2  - \frac{3}{2m-5})s}, & \mbox{if } \gamma=(2,0,0)
\end{cases}\label{e:h_estimates}\,. 
\end{align}
 Furthermore, for $\abs{\gamma} \in \{3, 4\}$ we have the lossy global estimates  
 \begin{align}
\abs{\partial^{\gamma}G_W}&\les  e^{- (\frac12 - \frac{ \abs{\gamma}-1}{2m-7})s}\,,
\label{eq:GW:lossy}\\
\abs{\partial^{\gamma}h_W}&\les   e^{-s}\,,
\label{eq:h:lossy}
 \end{align}
for all $y \in \XXX(s)$.
 \end{lemma}
\begin{proof}[Proof of Lemma~\ref{lemma_g}]
The bounds for the first three cases in \eqref{e:G_W_estimates} and \eqref{e:G_ZA_estimates} are the same as in  Lemma 7.2 in \cite{BuShVi2019b}.
It remains to consider the case $\gamma=(2,0,0)$. By  \eqref{eq:g:def},  we have that
\begin{align*}
\abs{\partial_1^2 G_W}+\abs{\partial_1^2 G_Z}+\abs{\partial_1^2 G_U}\les  e^{\frac s2} \abs{\partial_1^2  Z } \,,
\end{align*}
so that an application of \eqref{eq:Z_bootstrap} provides  the bounds for both \eqref{e:G_W_estimates} and \eqref{e:G_ZA_estimates}.

For the estimates \eqref{e:h_estimates},  the proof of the first three cases is given in Lemma 7.2 in \cite{BuShVi2019b}.
For the case $\gamma=(2,0,0)$, by \eqref{eq:h:def},  we have that
\begin{align*}
\abs{\partial_1^2 h_W}+\abs{\partial_1^2 h_Z}+\abs{\partial_1^2 h_U}\les  e^{-\frac s2} (\abs{\partial_1^2  Z }+\abs{\partial_1^2  A })\les M^{\frac14}e^{-2s}+e^{-( 2  - \frac{3}{2m-5})s}  \,,
\end{align*}
where we have applied  \eqref{eq:Z_bootstrap} and \eqref{eq:A:higher:order} to attain the desired estimate.
\end{proof} 

\subsection{Forcing estimates}
\begin{lemma}[Estimates on $\partial^\gamma F_W$, $\partial^\gamma F_Z$ and $\partial^\gamma F_A$] 
\label{lem:forcing}
For $y \in \XXX(s)$ we have the force bounds 
\begin{align}
\abs{\partial^\gamma F_W}+
e^{\frac s2}\abs{\partial^\gamma F_Z} &\les  
\begin{cases}
 e^{-\frac s 2 },  &\mbox{if } \abs{ \gamma}=0\\
 e^{-s} \eta^{- {\frac{1}{15}}  }
 ,&\mbox{if } \gamma_1=1\mbox{ and } \abs{\check \gamma}=0\\
e^{-{\frac{5}{8}s}   } , & \mbox{if } \gamma_1=0\mbox{ and } \abs{\check \gamma}=1\\
e^{- (1- {\frac{4}{2m-7}} )s} \eta^{- {\frac{1}{15}}  } ,&\mbox{if } \gamma_1 \ge 1  \mbox{ and } \abs{ \gamma}=2 \\
e^{-(\frac{5}{8} - \frac{9}{4(2 m-7) })s} , & \mbox{if } \gamma_1=0\mbox{ and } \abs{\check \gamma}=2
\end{cases} \,,\label{eq:F_WZ_deriv_est}
\\
\abs{\partial^\gamma F_{A \nu}} &\les 
\begin{cases}
 M^{\frac 12}  e^{-s},  &\mbox{if }  \abs{ \gamma}=0\\
  (M^{\frac 12} + M^2 \eta^{-\frac 16})   e^{-s} , & \mbox{if } \gamma_1=0\mbox{ and } \abs{\check \gamma}=1\\
 e^{-\left( 1-\frac{3}{2m-7}\right)s}  \eta^ {-\frac{1}{6}} , & \mbox{if } \gamma_1=0\mbox{ and } \abs{\check \gamma}=2
\end{cases} \,.
\label{eq:F_A_deriv_est}
\end{align}
Moreover, we have the following higher order estimate at $y=0$
\begin{align}
\abs{(\partial^\gamma \tilde F_W)^0}
\les   e^{- (\frac12 - \frac{4}{2m-7})s} \quad\mbox{for}\quad \abs{\gamma}=3 \,, \label{e:FW3}
\end{align}
and the bound on $\tilde{F}_W$ given by
\begin{equation}\label{eq:Ftilde_est}
\abs{\partial^\gamma \tilde F_W}
\les 
M \eps^{\frac 16} 
\begin{cases}
\eta^{-\frac{1}{6}},&\mbox{if}\quad \abs{\gamma}=0 \\
  \eta^ {-\frac25} ,&\mbox{if}\quad \gamma=(1,0,0)\\
 \eta^{-\frac{1}{3}} ,&\mbox{if}\quad \gamma_1=0 \mbox{ and }\abs{\check \gamma}=1 \\
 1 ,&\mbox{if}\quad \abs{\gamma}=4\quad\mbox{and}\quad \abs{y}\leq \ell
\end{cases}
\end{equation}
holds for all $ \abs{y} \leq \LLL $.
\end{lemma}
\begin{proof}[Proof of Lemma~\ref{lem:forcing}]
By the definition \eqref{eq:FW:def} we have
\begin{align*}
\abs{\partial^\gamma F_W}&\les {\abs{\partial^{\gamma}(\sound \Tcal^\nu_\mu \partial_{\mu} A_\nu)}}+e^{-\frac s2} \abs{\partial^{\gamma}(A_\nu \Tcal^\nu_i \dot{\Ncal}_i)}+ e^{-\frac s2} \abs{\partial^{\gamma}(A_\nu \Tcal^\nu_j \Ncal_i)}
\notag \\
&+ e^{-\frac s2} \abs{\partial^{\gamma}\left(\left(V_\mu +\Ncal_\mu  U \cdot \nn +  A_\nu \Tcal^\nu_\mu \right)A_\gamma \Tcal^\gamma_i \Ncal_{i,\mu}\right)}
+ e^{-\frac s2} \abs{\partial^{\gamma}\left(\sound \left( A_\nu \Tcal^\nu_{\mu,\mu} + U\cdot \Ncal \Ncal_{\mu,\mu} \right)\right)}
\notag \\
& +\underbrace{\abs{ e^s \p^\gamma (\Jcal\sound^2   \p_1 K)} + \abs{ \p^\gamma(\Ncal_\mu \sound^2 \p_\mu K)}}_{\mathcal I_{W,\gamma}}
\notag 
\end{align*}
The bounds for the first five terms on the right side follow as in the proof of Lemma 7.3 in \cite{BuShVi2019b},  and we have that
\begin{align*}
\abs{\partial^\gamma F_W}&\les\abs{\mathcal I_{W,\gamma}}+
\begin{cases}
e^{-\frac s 2 },  &\mbox{if } \abs{ \gamma}=0\\
e^{-s} \eta^{- \frac16 + \frac{2 \abs{\gamma} +1}{3(2m-5)}}(y) ,&\mbox{if } \gamma_1 \geq 1 \mbox{ and } \abs{ \gamma}=1,2 \\
 M^2 e^{-s}, & \mbox{if } \gamma_1=0\mbox{ and } \abs{\check \gamma}=1\\
 e^{- (1 -  \frac{3}{2m-7} )s}, & \mbox{if } \gamma_1=0\mbox{ and } \abs{\check \gamma}=2
\end{cases} 
\end{align*}
Invoking  \eqref{gerd},  \eqref{eq:S_bootstrap}, \eqref{eq:K:higher:order}, \eqref{e:bounds_on_garbage} and Lemma \ref{lem:US_est},  we obtain that
\begin{align}
\abs{\mathcal I_{W,\gamma}}& \les\sum_{\beta\leq \gamma,~\beta_1=0}e^{-\frac{\abs{\beta}}{2}s}
\left(e^{s}  \abs{\partial^{\gamma-\beta}\left( \sound ^2 \p_1 K \right)}
+\eps \abs{\partial^{\gamma-\beta}\left( \sound ^2 \check\nabla   K \right)} \right)\\&
 \les\begin{cases}
e^{-\frac s2},  &\mbox{if } \abs{ \gamma}=0\\
e^{-s}(\eta^{-{\frac{1}{15}} }+e^{-{\frac{5}{8}} s}) ,& \mbox{if } \gamma= (1,0,0) \\
e^{- {\frac{5}{8}} s}, & \mbox{if } \gamma_1=0\mbox{ and } \abs{\check \gamma}=1\\
e^{-(1-\frac{4}{2m-7})s}\eta^{-{\frac{1}{15}} }+e^{-({\frac{13}{8}} -\frac{3}{2m-7})s} ,  
    & \mbox{if } \gamma_1=1\mbox{ and } \abs{\check \gamma}=1\\
e^{-(1-\frac{4}{2m-7})s}\eta^{-{\frac{1}{15}} } ,& \mbox{if } \gamma= (2,0,0) \\
e^{-( {\frac{5}{8}} + \frac{1 - 2 \beta}{2 m-7 })s} , & \mbox{if } \gamma_1=0\mbox{ and } \abs{\check \gamma}=2
\end{cases}\label{e:IWest}
\end{align}
Using the same set of estimates we also obtain the lossy bound
\begin{equation}
\abs{\mathcal I_{W,\gamma}}\les e^{ -\frac{s}{2}    } \label{e:FW3A}
\end{equation}
for $\abs{\gamma}=3$,
which we shall need later in order to prove \eqref{e:FW3},
and 
\begin{equation}
 \abs{\mathcal I_{W,\gamma}}\les \eps^{\frac 16  }      \label{e:FW4A}
\end{equation}
for $\abs{\gamma}=4$ and $\abs{y}\leq \ell$,
which we shall need later in order to prove the last case of \eqref{eq:Ftilde_est}. 

Then,  additionally  using \eqref{e:space_time_conv},  we obtain the stated bounds claimed in  \eqref{eq:F_WZ_deriv_est} for $\p^\gamma F_W$.  
Comparing \eqref{eq:FZ:def} and \eqref{eq:FW:def}, we note that the estimates on $\p^\gamma F_Z$ claimed in \eqref{eq:F_WZ_deriv_est} are 
completely analogous to the estimates ones $\p^\gamma F_W$ up to a factor of $e^{-\frac s2}$.

Now we consider the estimates on $F_A$. By definition \eqref{eq:A:def}, we have that
\begin{align*}
\abs{\partial^{\gamma} {F}_{A_\nu}}&\les e^{-\frac s2}  \abs{\partial^{\gamma}( \sound T^\nu_\mu \p_{\mu} \sound)} 
 +  e^{-s}  \abs{ \partial^{\gamma}\left(\left( U\cdot\Ncal \Ncal_i + A_\gamma \Tcal^\gamma_i\right) \dot{\Tcal}^\nu_i \right)} 
 +  e^{-s}  \abs{\partial^{\gamma }\left(\left(  U\cdot \Ncal \Ncal_j + A_\gamma \Tcal^\gamma_j\right)\Tcal^\nu_i \right)}\\
 &\qquad
 + e^{-s} \abs{\partial^{\gamma}\left( \left(V_\mu + U\cdot \Ncal \Ncal_\mu + A_\gamma \Tcal^\gamma_\mu\right) \left(U\cdot \Ncal \Ncal_i 
 +    A_\gamma \Tcal^\gamma_i \right) \Tcal^\nu_{i,\mu} \right)}  
 +\underbrace{\abs{ e^ {-\frac{s}{2}} \p^\gamma (\sound^2 \Tcal^\nu_\mu \p_\mu K)}}_{\mathcal I_{A,\gamma}}
 \,.
  \end{align*}
 Applying the bounds for the first four terms on the right side from Lemma 7.3 of \cite{BuShVi2019b}, we see that
 \begin{align}
 \abs{\partial^\gamma F_{A \nu}} &\les \abs{\mathcal I_{A,\gamma}}+
\begin{cases}
 M^{\frac 12}  e^{-s},  &\mbox{if }  \abs{ \gamma}=0\\
  (M^{\frac 12} + M^2 \eta^{-\frac 16})   e^{-s} , & \mbox{if } \gamma_1=0\mbox{ and } \abs{\check \gamma}=1\\
 e^{-\left( 1-\frac{3}{2m-7}\right)s}  \eta^ {-\frac{1}{6}}, & \mbox{if } \gamma_1=0\mbox{ and } \abs{\check \gamma}=2
\end{cases} \,.
\end{align} 
Applying \eqref{gerd},  \eqref{eq:S_bootstrap}, \eqref{eq:K:higher:order}, and Lemma \ref{lem:US_est},  we find that
\begin{align*}\abs{\mathcal I_{A,\gamma} }\les
\sum_{\beta\leq \gamma,~\beta_1=0}e^{-\frac{\abs{\beta}+1}{2}s}\abs{\partial^{\gamma-\beta}( \sound^2 \check\nabla K)}
\les  
 \begin{cases}
e^{-s},  &\mbox{if } \abs{ \gamma}=0\\
M^2e^{-\frac{3}{2}s}, & \mbox{if } \gamma_1=0\mbox{ and } \abs{\check \gamma}=1\\  
e^{- (\frac32 - \frac{\abs{\gamma}-1}{2m-7})s} , & \mbox{if } \gamma_1=0\mbox{ and } \abs{\check \gamma}=2
\end{cases} \,.
\end{align*}
Thus, combining the above estimates,  we obtain \eqref{eq:F_A_deriv_est}.

Again, using the same argument as in Lemma 7.3 in \cite{BuShVi2019b} for $\abs{\gamma}=3$, and using  \eqref{e:FW3A} yields
\begin{equation}
\sabs{(\partial^\gamma \tilde F_W)^0}
\les  \abs{(\mathcal I_{W,\gamma})^0}+e^{- (\frac12 - \frac{4}{2m-7})s}\les e^{- (\frac12 - \frac{4}{2m-7})s} \,,  \label{chicken-chicken1}
\end{equation}
and also for all $ \abs{y} \leq \LLL $, 
\begin{equation}
\sabs{\partial^\gamma \tilde F_W}
\les   \abs{\mathcal I_{W,\gamma}}+M \eps^{\frac 16} 
\begin{cases}
\eta^{-\frac{1}{6}}(y),&\mbox{if}\quad \abs{\gamma}=0 \\
\eta^{-   \frac 12 + \frac{ 3 }{2m-5}},&\mbox{if}\quad \gamma_1=1 \mbox{ and }\abs{\check \gamma}=0 \\
\eta^{-\frac{1}{3}} ,&\mbox{if}\quad \gamma_1=0 \mbox{ and }\abs{\check \gamma}=1 \\
1 ,&\mbox{if}\quad \abs{\gamma}=4\quad\mbox{and}\quad \abs{y}\leq \ell
\end{cases} \,. \label{chicken-chicken2}
\end{equation}
The estimate \eqref{chicken-chicken1} verifies \eqref{e:FW3}, while 
combining \eqref{chicken-chicken2} with \eqref{e:space_time_conv}, \eqref{e:IWest} and \eqref{e:FW4A} verifies \eqref{eq:Ftilde_est}.
\end{proof}

\begin{corollary}[Estimates on the forcing terms] 
\label{cor:forcing}
Assume that $m\geq 18$. Then, we have
\begin{align}
\abs{F^{(\gamma)}_W} &\les 
\begin{cases}
e^{-\frac s 2 },  &\mbox{if } \abs{ \gamma}=0\\
e^{- {\frac{s}{15}}   } \eta^{- \frac{1}{3} },&\mbox{if } \gamma_1=1 \mbox{ and } \abs{\check \gamma}=0\\
\eps^{\frac {5}{24}} \eta^{-\frac {5}{24}} , & \mbox{if } \gamma_1=0\mbox{ and } \abs{\check \gamma}=1\\
\eta^{-(\frac{29}{60}-\frac{8}{3(2m-7)}})\ppp^{\frac{1}{4}}   ,&\mbox{if }  \gamma_1=2 \mbox{ and } \abs{\check \gamma}=0\\
 M^{\frac13} \eta^{-\frac 13}  ,&\mbox{if } \gamma_1=1 \mbox{ and } \abs{\check \gamma}=1 \\
M^{\frac23} \eta^{-( {\frac{5}{24}} -\frac1{2m-7}) }, & \mbox{if } \gamma_1=0\mbox{ and } \abs{\check \gamma}=2
\end{cases}\label{eq:forcing_W}
\\
\abs{F^{(\gamma)}_Z} &\les \begin{cases}
e^{-s  },  &\mbox{if } \abs{ \gamma}=0\\
M^2e^{-\frac 32s} \eta^{- {\frac{1}{15}} },& \mbox{if } \gamma_1=1 \mbox{ and } \abs{\check \gamma}=0 \\
  e^{- s}, & \mbox{if } \gamma_1=0\mbox{ and } \abs{\check \gamma}=1\\
e^ {-\frac{3s}{2}}(  M^ {\frac{1}{2}}  + e^ { {\frac{4s}{2m-7}}} \eta^ {-{\frac{1}{15}} } )  ,& \mbox{if } \mbox{if } \gamma_1 \ge 1  \mbox{ and } \abs{ \gamma}=2  \\
 e^{-({\frac{9}{8}} - \frac{9}{4(2 m-7) })s}  , & \mbox{if } \gamma_1=0\mbox{ and } \abs{\check \gamma}=2
\end{cases} \label{eq:forcing_Z}
\\
\abs{F^{(\gamma)}_{A\nu}} &\les
\begin{cases}
 M^{\frac 12}  e^{-s},  &\mbox{if }  \abs{ \gamma}=0\\
  (M^{\frac 12} + M^2 \eta^{-\frac 16})   e^{-s} , & \mbox{if } \gamma_1=0\mbox{ and } \abs{\check \gamma}=1\\
 e^{-\left( 1-\frac{3}{2m-7}\right)s}  \eta^ {-\frac{1}{6}} , & \mbox{if } \gamma_1=0\mbox{ and } \abs{\check \gamma}=2
\end{cases}\label{eq:forcing_U}
\\
\abs{F^{(\gamma)}_K} &\les \begin{cases}
M^2e^{-\frac32 s}, & \mbox{if }\gamma_1=0\mbox{ and } \abs{\check \gamma}=1,2\\
  \eps^{\frac18}e^{-\frac32 s} \eta^{-\frac 16}, & \mbox{if } \gamma_1=1 \mbox{ and } \abs{\check \gamma}=0,1\\
\eps^{\frac18}e^{-\frac32 s} \eta^{-\frac 13} \ppp^ {\frac{1}{4}} , & \mbox{if } \gamma_1=2  \mbox{ and }\abs{\gamma}=0
\end{cases}
 \label{eq:forcing_S}
\,.
\end{align}
Moreover, we have the following higher order estimate
\begin{align}
\abs{ \tilde F_W^{(\gamma),0}}&\les  e^{- (\frac12 - \frac{4}{2m-7})s}\quad\mbox{for}\quad \abs{\gamma}=3\label{e:forcing:W3}
\end{align} 
and the following estimates on $\tilde F^{(\gamma)}_W$
\begin{alignat}{2}
\abs{\tilde F^{(\gamma)}_W} &\les  \eps^{\frac 1{11}}
\eta^{- \frac 25}
\quad &&\mbox{for }\gamma=(1,0,0) \mbox{ and }  \abs{y} \leq \LLL \label{eq:Ftilde_d1_est}\\
\abs{\tilde F^{(\gamma)}_W} &\les  \eps^{\frac 1{12}} \eta^{-\frac 13}
\quad &&\mbox{for }\gamma_1=0,  \abs{\check \gamma}=1 \mbox{ and } \abs{y} \leq \LLL\label{eq:Ftilde_dcheck_est}\\
\abs{\tilde F^{(\gamma)}_W} &\les \eps^{\frac18}
+ \eps^{\frac{1}{10}} (\log M)^{\abs{\check \gamma}-1}
\quad &&\mbox{for }\abs{\gamma}=4 \mbox{ and }\abs{y}\leq \ell\label{eq:Ftilde_4th_est}\,.
 \end{alignat}
\end{corollary}
\begin{proof}[Proof of Corollary~\ref{cor:forcing}]
First we establish \eqref{eq:forcing_W}. Note that in this estimate $\abs{\gamma}\leq 2$, and thus by definition \eqref{eq:F:W:def} we have
\begin{align*}
\abs{F_{W}^{(\gamma)}}&\les \abs{\p^\gamma F_W }
+ \!\!\! \underbrace{ \sum_{0\leq \beta < \gamma}  \left( \sabs{\p^{\gamma-\beta}G_W \p_1 \p^\beta W} + \sabs{\p^{\gamma-\beta}h_W^\mu \p_\mu \p^\beta W}  \right)}_{=: \mathcal I_1} + \underbrace{{\bf 1}_{|\gamma|= 2}  \!\!\! \sum_{\substack{ |\beta| = |\gamma|-1 \\ \beta\le\gamma, \beta_1 = \gamma_1}}\sabs{ \p^{\gamma-\beta} (\Jcal W)   \p_1\p^\beta W  } }_{=: \mathcal I_2} \,.
\end{align*}
We will first consider the case $\gamma\neq (2,0,0)$, since the estimates are analogous to the estimates in the previous paper. 
We have from Corollary 7.4 of \cite{BuShVi2019b}, that
\begin{align*}
\abs{\mathcal I_1}
&\les M \eta^{-\frac 13}\left(e^{-\frac s2}+  \eps^{\frac13}  ({\bf 1}_{\abs{\gamma}=2 } + {\bf 1}_{\abs{\gamma}=\abs{\check \gamma}=1})  \right) \quad\mbox{and}\quad
\abs{\mathcal I_2}\les  {\bf 1}_{|\gamma|= 2}M^{\frac{\abs{\check \gamma}}{3}}  \eta^{-\frac 13}\,. 
\end{align*}
Thus combining these estimates with \eqref{eq:F_WZ_deriv_est}, we obtain that 
\begin{align}
\abs{F_{W}^{(\gamma)}}\les M \eta^{-\frac 13}e^{-\frac s2}+ \begin{cases}
 e^{-\frac s 2 },  &\mbox{if } \abs{ \gamma}=0\\
e^{-s} \eta^{- \frac{1}{15} }
 ,&\mbox{if } \gamma =(1,0,0) \\
  e^{-\frac{5}{8} s  } +M \eps^{\frac13}\eta^{-\frac 13}, & \mbox{if } \gamma_1=0\mbox{ and } \abs{\check \gamma}=1\\
 e^{- (1- {\frac{4}{2m-7}} )s} \eta^{- \frac{1}{15} } +(M \eps^{\frac13} +M^{\frac13}) \eta^{-\frac 13},&\mbox{if } \gamma_1 = 1  \mbox{ and } \abs{ \check \gamma}=1 \\
  e^{- (\frac{5}{8}   -  \frac{3}{2m-7} )s}  +(M \eps^{\frac13} +M^{\frac23})\eta^{-\frac 13}, & \mbox{if } \gamma_1=0\mbox{ and } \abs{\check \gamma}=2
\end{cases} \,.
\end{align}
Then applying \eqref{e:space_time_conv} we obtain  \eqref{eq:forcing_W} for all cases except $\gamma=(2,0,0)$.

For the special case $\gamma=(2,0,0)$,  we have from  \eqref{e:space_time_conv},  \eqref{eq:phi:interp} (with $q=2$), \eqref{eq:W_decay}, \eqref{e:G_W_estimates} and \eqref{e:h_estimates} 
\begin{align*}
\abs{\mathcal I_1}&\les M^{\frac12}e^{-s}\eta^{-\frac13}+M^{\frac43}e^{-\frac s2}\eta^{-\frac13}\ppp^{{\frac{1}{4}} } +e^{-( 2  - \frac{3}{2m-5})s}+M^{\frac23 }e^{- s} \eta^{-\frac 12}(y) 
\les M^{\frac43 }e^{-\frac s2}\eta^{-\frac13}\ppp^{{\frac{1}{4}} }\,.
\end{align*}
From \eqref{eq:F_WZ_deriv_est} and \eqref{eq:phi:interp} (with $q={\tfrac{3}{4}}  \frac{7-2m}{11-2m}$), we have that
\[\abs{\p^\gamma F_W }\les e^{- (1- {\frac{4}{2m-7}} )s} \eta^{-\frac{1}{15}  }\les \ppp^{\frac{1}{4}}  \eta^{-(\frac{29}{60}-\frac{8}{3(2m-7)} )}\,.\]
Thus since $\mathcal I_2=0$ for  $\gamma=(2,0,0)$, we obtain \eqref{eq:F_WZ_deriv_est} for this case.

Similarly, for $\abs{\gamma}\leq 2$, from \eqref{eq:F:ZA:def} we have that
\begin{align}
\abs{F_{Z}^{(\gamma)}}&\les \abs{\p^\gamma F_Z }
+ \sum_{0\leq \beta < \gamma}  \left(\sabs{\p^{\gamma-\beta}G_Z \p_1 \p^\beta Z} + \sabs{\p^{\gamma-\beta}h_Z^\mu \p_\mu\p^\beta Z}\right)
 \notag\\
&\qquad 
+  {\bf 1}_{\abs{\gamma}=2} \sabs{\p_1 Z \p^\gamma(\Jcal W)}   +
\sum_{\substack{ |\beta| = |\gamma|-1 \\ \beta\le\gamma, \beta_1 = \gamma_1}}\sabs{ \p^{\gamma-\beta} (\Jcal W)   \p_1\p^\beta Z  } \notag \\
&= : \abs{\partial^{\gamma}F_{Z}} +\mathcal I_1 +  {\bf 1}_{\abs{\gamma}=2  }\sabs{\p_1 Z \p^\gamma(\Jcal W)} +\mathcal I_2 \, .
\label{insanity-arrives}
\end{align}
Utilizing the  bounds obtained in Corollary 7.4 of \cite{BuShVi2019b}, we have that
\begin{alignat*}{2}
 \sabs{\p_1 Z \p^\gamma(\Jcal W)} &\les M^{\frac 12} e^{-\frac 32 s} \left( M \eta^{-\frac 16} {\bf 1}_{\gamma_1=0} + M^{\frac 23} \eta^{-\frac 13} {\bf 1}_{\gamma_1 \geq 1}  + \eps e^{-\frac s2} \right) \quad&&\mbox{for } \abs{\gamma}=2 \,, \\
 \mathcal I_1
&\les
e^{-\frac32 s}\left(M^2  e^{-\frac s2} + M^3 \eps^{\frac 12} {\bf 1}_{ \abs{\check\gamma}\geq 1}  + M \eps^{\frac12} \eta^{-\frac 16} \right) \quad&&\mbox{for } \abs{\gamma}\leq 2 \,,\\
\mathcal I_2&\les \left({\bf 1}_{\abs{\check\gamma}=1} M^{\frac 12}+{\bf 1}_{\abs{\check\gamma}=2 }M\right) e^{-\frac32 s} \quad&&\mbox{for } \abs{\gamma}=2,~\abs{\check\gamma}\geq 1
\,.
\end{alignat*}
Using \eqref{e:space_time_conv}, we have
\begin{alignat*}{2}
 \sabs{\p_1 Z \p^\gamma(\Jcal W)} &\les M^2e^{-\frac32s}\eta^{-\frac16} \quad&&\mbox{for } \abs{\gamma}=2\,,\\
 \mathcal I_1
&\les Me^{-\frac32s}\eta^{-\frac16}+
 \eps^{\frac 12} {\bf 1}_{ \abs{\check\gamma}\geq 1} e^{-\frac32s}
 \quad&&\mbox{for } \abs{\gamma}\leq 2 \,.
 \end{alignat*}
Thus combining the above estimates with \eqref{e:space_time_conv} and \eqref{eq:F_WZ_deriv_est}, we obtain that
\begin{align*}
\abs{F_{Z}^{(\gamma)}}\les M^2e^{-\frac32s}\eta^{-\frac16}+ e^{-\frac s2}\begin{cases}
 e^{-\frac s 2 },  &\mbox{if } \abs{ \gamma}=0\\
e^{-s} \eta^{-\frac{1}{15}  }
 ,&\mbox{if } \gamma =(1,0,0) \\
  e^{-\frac{5}{8} s  } +M^{\frac12}e^{-s} & \mbox{if } \gamma_1=0\mbox{ and } \abs{\check \gamma}=1\\
 e^{- (1- {\frac{4}{2m-7}} )s} \eta^{-\frac{1}{15} }+M^{\frac12}e^{-s} ,&\mbox{if } \gamma_1 \ge 1  \mbox{ and } \abs{ \gamma}=2 \\
  e^{-( \frac{5}{8} + \frac{1 - 2 \beta}{2 m-7 })s} +M\eps^{\frac12}e^{-s}, & \mbox{if } \gamma_1=0\mbox{ and } \abs{\check \gamma}=2
\end{cases} \,.
\end{align*}

The bounds for $\sabs{F_{A}^{(\gamma)}}$ are obtained in the identical fashion as the bounds for (7.20) in \cite{BuShVi2019b}.

To prove the $\sabs{F_{K}^{(\gamma)}}$ estimate for $\abs{\gamma}\leq 2$, from \eqref{eq:F:ZA:def}, we have that
\begin{align}
\abs{F_{K}^{(\gamma)}}&\les \!\!
 \sum_{0\leq \beta < \gamma} \!\! \left(\sabs{\p^{\gamma-\beta}G_U \p_1 \p^\beta K} + \sabs{\p^{\gamma-\beta}h_U^\mu \p_\mu\p^\beta K}\right)
+  {\bf 1}_{\abs{\gamma}=2} \sabs{\p_1 K \p^\gamma(\Jcal W)}   +\!\!\!\!
\sum_{\substack{ |\beta| = |\gamma|-1 \\ \beta\le\gamma, \beta_1 = \gamma_1}} \!\! \!\!  \!\! \sabs{ \p^{\gamma-\beta} (\Jcal W)   \p_1\p^\beta K  } \notag \\
&= : \mathcal I_1 +  {\bf 1}_{\abs{\gamma}=2}  \sabs{\p_1 K \p^\gamma(\Jcal W)} +\mathcal I_2 \, .
\label{insanity-arrives-again}
\end{align}
Let us further split $\mathcal I_1$ as
$$
\mathcal I_1= \underbrace{\sum_{0\leq \beta < \gamma} \!\!\sabs{\p^{\gamma-\beta}G_U \p_1 \p^\beta K}}_{\mathcal I_{1,1}}
+ \underbrace{\sum_{0\leq \beta < \gamma} \!\! \sabs{\p^{\gamma-\beta}h_U^\mu \p_\mu\p^\beta K}}_{\mathcal I_{1,2}} \,.
$$
Estimating $\mathcal I_{1,1}$, using \eqref{eq:S_bootstrap} and \eqref{e:G_ZA_estimates}, we have that
\begin{align*}
\abs{\mathcal I_{1,1}}&\les
M^2
\begin{cases}
  \eps^{\frac12}\abs{\partial_1 K}, & \mbox{if } \gamma_1=0\mbox{ and } \abs{\check \gamma}=1\\
  e^{-\frac s2}\abs{\partial_1 K}+   \eps^{\frac12}\abs{\check\nabla\partial_1 K}, & \mbox{if } \gamma_1=0\mbox{ and } \abs{\check \gamma}=2\\
  e^{-\frac s2}\abs{\partial_1 K},  & \mbox{if } \gamma=(1,0,0) \\
 e^{-\frac s2}(\abs{\partial_1 K}+\abs{\check\nabla\partial_1 K})+   \eps^{\frac12}\abs{\partial_1^2 K}, & \mbox{if } \gamma_1=1  \mbox{ and }\abs{\gamma}=1\\
  e^{-s}\abs{\partial_1 K}+ e^{-\frac s2}\abs{\partial_1^2 K}, & \mbox{if } \gamma_1=2  \mbox{ and }\abs{\gamma}=0
\end{cases}\\
&\les 
\begin{cases}
 e^{-\frac32 s}, & \mbox{if } \gamma_1=0\mbox{ and } \abs{\check \gamma}=1,2\\
  e^{-2s } ,  & \mbox{if }  \gamma_1=1 \mbox{ and } \abs{\check \gamma}=0,1\\
e^{- \frac52 s}, & \mbox{if } \gamma_1=2  \mbox{ and }\abs{\gamma}=0
\end{cases} \,.
\end{align*}
Similarly, estimating  $\mathcal I_{1,2}$, using \eqref{eq:S_bootstrap} and \eqref{e:h_estimates}, we have that
\begin{align*}
\abs{\mathcal I_{1,2}}&\les
\begin{cases}
  e^{-s}\abs{\check \nabla K }, & \mbox{if } \gamma_1=0\mbox{ and } \abs{\check \gamma}=1\\
 e^{-s}(\eta^{-\frac 16}\abs{\check \nabla K}+\abs{\check \nabla^2 K}), & \mbox{if } \gamma_1=0\mbox{ and } \abs{\check \gamma}=2\\
e^{-s} \eta^{-\frac 16} \abs{\check \nabla K},  & \mbox{if } \gamma=(1,0,0) \\
e^{-s}\left(\eta^{-\frac 16}(\abs{\check\nabla K}+\abs{\check \nabla^2  K })+\abs{\check\nabla\partial_1 K}\right)), & \mbox{if } \gamma_1=1  \mbox{ and }\abs{\gamma}=1\\
e^{-s}(e^{-( 1  - \frac{3}{2m-5})s}\abs{\check\nabla K}+\eta^{-\frac 16}\abs{\check \nabla\partial_1 K }), & \mbox{if } \gamma_1=2  \mbox{ and }\abs{\gamma}=0
\end{cases}\\
&\les \eps^{\frac18}
\begin{cases}
 e^{-\frac32 s}, & \mbox{if }\gamma_1=0\mbox{ and } \abs{\check \gamma}=1,2\\
  e^{-\frac32 s} \eta^{-\frac 16}, & \mbox{if } \gamma_1=1 \mbox{ and } \abs{\check \gamma}=0,1\\
e^{-( \frac52  - \frac{3}{2m-5})s}, & \mbox{if } \gamma_1=2  \mbox{ and }\abs{\gamma}=0
\end{cases} \,.
\end{align*}

For $ {\bf 1}_{\abs{\gamma}=2}  \sabs{\p_1 K \p^\gamma(\Jcal W)} $,  using \eqref{eq:W_decay} and \eqref{eq:S_bootstrap} yields
\begin{align*}
 \sabs{\p_1 K \p^\gamma(\Jcal W)} &\les
\begin{cases}
e^{-\frac32 s}, & \mbox{if }\gamma_1=0\mbox{ and } \abs{\check \gamma}=2\\
M^{\frac23}\eps^{\frac14}e^{-\frac32 s}\eta^{-\frac13} , & \mbox{if } \gamma_1=1 \mbox{ and } \abs{\check \gamma}=1\\
M^{\frac{1}{3}} \eps^{\frac14}e^{-\frac32 s} \eta^{-\frac 13} \ppp^ {\frac{1}{4}} , &  \mbox{if }\gamma_1=2\mbox{ and } \abs{\check \gamma}=0
\end{cases} \,.
\end{align*}

Next, for  $\mathcal I_2$,  we have that
\begin{align*}
\abs{\mathcal I_{2}}&\les
\begin{cases}
e^{-\frac32 s}, & \mbox{if }\gamma_1=0\mbox{ and } \abs{\check \gamma}=1\\
e^{-2s} , & \mbox{if } \gamma_1=1 \mbox{ and } \abs{\check \gamma}=1\\
0, & \mbox{otherwise}
\end{cases}
\end{align*}

Thus combining the above estimates,  we attain
\begin{align*}
\sabs{F_{K}^{(\gamma)}}&\les
\begin{cases}
 M^2e^{-\frac32 s}, & \mbox{if }\gamma_1=0\mbox{ and } \abs{\check \gamma}=1,2\\
  \eps^{\frac18}e^{-\frac32 s} \eta^{-\frac 16}, & \mbox{if } \gamma_1=1 \mbox{ and } \abs{\check \gamma}=0,1\\
 \eps^{\frac18}e^{-\frac32 s} \eta^{-\frac 13} \ppp^{\frac{1}{4}} , & \mbox{if } \gamma_1=2  \mbox{ and }\abs{\gamma}=0
\end{cases}\,,
\end{align*}
where we used \eqref{eq:phi:interp}  (with $q=\tfrac{3}{4}  \frac{5-2m}{8-2m}$).

The proof of  the bounds \eqref{e:forcing:W3}--\eqref{eq:Ftilde_4th_est} is exactly the same as the proof of (7.21)--(7.24) in \cite{BuShVi2019b},
 with the caveat that we have changed the exponent of $\eta$ in \eqref{eq:Ftilde_d1_est} which reflects the change in exponent of $\eta$ in the estimate \eqref{eq:Ftilde_est} for $\gamma=(1,0,0)$ relative to the corresponding estimate in our previous paper.
\end{proof}

\section{Closure of $L^\infty$ based bootstrap for $Z$, $A$, and $K$}
\label{sec:Z:A}
Having established bounds on trajectories as well as on the vorticity, we now improve the bootstrap assumptions for 
$\partial^\gamma Z $ and $\partial^\gamma A$ stated in \eqref{eq:Z_bootstrap} and \eqref{eq:A_bootstrap}.  We shall obtain estimates for $\partial^\gamma Z \circ \pzy$ and 
$\partial^\gamma A \circ \pay$ which are weighted by an appropriate exponential factor $e^{\mu s}$.

From \eqref{euler_for_Linfinity:b} we obtain that $e^{\mu s}\p^\gamma Z$ is a solution of
\begin{align*} 
 \p_s (e^{\mu s}\p^\gamma Z) + D_Z^{(\gamma,\mu)} ( e^{\mu s}\p^\gamma Z)
  +  \left( \mathcal{V}_Z \cdot \nabla\right) ( e^{\mu s}\p^\gamma Z) &=e^{\mu s} F^{(\gamma)}_Z  \,,
\end{align*} 
where the damping function is given by
$$
D_Z^{(\gamma,\mu)}:=-\mu + \tfrac{3\gamma_1 + \gamma_2 + \gamma_3}{2} + \beta_2 \beta_\tau  \gamma_1  \Jcal \p_1 W \,.
$$
Upon composing with the flow of ${\mathcal V}_Z$, from Gr\"onwall's inequality  it follows that
\begin{align}
e^{\mu s}\abs{\partial^\gamma Z\circ \pz^{y_0}(s)}
&\leq   \eps^{-\mu} \abs{\partial^\gamma Z(y_0,-\log\eps)} \exp\left(- \int_{-\log\eps}^s  D_{Z}^{(\gamma,\mu)}\circ\pz^{y_0}(s')  \,ds'\right) \notag\\
&\qquad +\int_{-\log\eps}^s e^{\mu s'}\abs{F_Z^{(\gamma)}\circ \pz^{y_0}(s')}\exp\left(- \int_{s'}^s D_{Z}^{(\gamma,\mu)}\circ\pz^{y_0}(s'') \,ds''\right) \,ds' \,.\label{eq:weighted:Z:bnd}
\end{align}
Similarly, from \eqref{euler_for_Linfinity:c} we have that  $e^{\mu s}\p^\gamma A$ and $e^{\mu s}\p^\gamma K$ are solutions of
\begin{align*} 
  \p_s (e^{\mu s}\p^\gamma K) + D_K^{(\gamma,\mu)} ( e^{\mu s}\p^\gamma K)
  +  \left( \mathcal{V}_U \cdot \nabla\right) ( e^{\mu s}\p^\gamma K) &=e^{\mu s} F^{(\gamma)}_K  \,,
\end{align*} 
where 
$$
D_K^{(\gamma,\mu)}:=-\mu + \tfrac{3\gamma_1 + \gamma_2 + \gamma_3}{2} + \beta_1 \beta_\tau  \gamma_1 \Jcal \p_1 W   \,,
$$
and hence, again by Gronwall's inequality, we have that
\begin{align}
e^{\mu s}\abs{\partial^\gamma K\circ \pa^{y_0}(s)}
&\leq   \eps^{-\mu} \abs{\partial^\gamma K(y_0,-\log\eps)} \exp\left(- \int_{-\log\eps}^s  D_K^{(\gamma,\mu)}\circ\pa^{y_0}(s')  \,ds'\right) \notag\\
&\qquad +\int_{-\log\eps}^s e^{\mu s'}\abs{F_K^{(\gamma)}\circ \pa^{y_0}(s')}\exp\left(- \int_{s'}^s D_K^{(\gamma,\mu)}\circ\pa^{y_0}(s'') \,ds''\right) \,ds' \,.\label{eq:weighted:K:bnd}
\end{align}
For each choice of $\gamma \in {\mathbb N}_0^3$ present in \eqref{eq:Z_bootstrap} and \eqref{eq:A_bootstrap}, we shall require that the exponential factor $\mu$ satisfies
\begin{equation}\label{eq:mu_cond}
 \mu  \leq  \tfrac{3\gamma_1 + \gamma_2 + \gamma_3}{2} \,,
\end{equation}
which, in turn, shows that
\begin{equation}\label{eq:DZ_lower_bnd}
D_Z^{(\gamma,\mu)}\leq 2 \beta_2 \gamma_1 \abs{\partial_1 W}\,.
\end{equation}
For the last inequality, we have used the bound $\abs{\beta_\tau \Jcal}\leq 2$, which follows from \eqref{eq:beta:tau} and \eqref{e:bounds_on_garbage}. Combining \eqref{eq:mu_cond}, \eqref{eq:DZ_lower_bnd}, and \eqref{eq:p1W:PhiZ},  for $s\geq s' \ge -\log \eps$,  we find that
\begin{equation}\label{eq:damping_Z}
 \exp\left(- \int_{s'}^s  D_{Z}^{(\gamma,\mu)}\circ\pz^{y_0}(s')  \,ds'\right)\les  \exp\left(  \left( \mu -  \tfrac{3\gamma_1 + \gamma_2 + \gamma_3}{2}\right) (s-s') \right)\les 1\,.
 \end{equation}
Replacing $\beta_2$  with $\beta_1$ in \eqref{eq:DZ_lower_bnd}, we  similarly obtain that  for $s\geq s'\geq -\log\eps$,
\begin{equation}\label{eq:damping_U}
  \exp\left(- \int_{s'}^s  D_{K}^{(\gamma,\mu)}\circ\pa^{y_0}(s') \,ds'\right)
 \les 1\,.
 \end{equation}
Then as a consequence of \eqref{eq:weighted:Z:bnd}, \eqref{eq:mu_cond}, \eqref{eq:damping_Z} and \eqref{eq:damping_U}, 
\begin{align}
e^{\mu s}\abs{\partial^\gamma Z\circ \pz^{y_0}(s)}&\les    \eps^{-\mu} \abs{\partial^\gamma Z(y_0,-\log\eps)}\notag\\
&\qquad+\int_{-\log\eps}^s e^{\mu s'}\abs{F_Z^{(\gamma)}\circ \pz^{y_0}(s')} \exp\left(\left( \mu -  \tfrac{3\gamma_1 + \gamma_2 + \gamma_3}{2}\right) (s-s')\right) \,ds' \label{eq:weighted:Z:bnd2:alt}
\\
e^{\mu s}\abs{\partial^\gamma K\circ \pz^{y_0}(s)}&\les    \eps^{-\mu} \abs{\partial^\gamma K(y_0,-\log\eps)}\notag\\
&\qquad+\int_{-\log\eps}^s e^{\mu s'}\abs{F_S^{(\gamma)}\circ \pz^{y_0}(s')} \exp\left(\left( \mu -  \tfrac{3\gamma_1 + \gamma_2 + \gamma_3}{2}\right) (s-s')\right) \,ds' \label{eq:weighted:K:bnd2:alt}
\end{align}
and
\begin{align}
e^{\mu s}\abs{\partial^\gamma Z\circ \pz^{y_0}(s)}
&\les   \eps^{-\mu} \abs{\partial^\gamma Z(y_0,-\log\eps)}
+\int_{-\log\eps}^s e^{\mu s'}\abs{F_Z^{(\gamma)}\circ \pz^{y_0}(s')} \,ds' \,, \label{eq:weighted:Z:bnd2}
\\
e^{\mu s}\abs{\partial^\gamma K\circ \pu^{y_0}(s)}
&\les   \eps^{-\mu} \abs{\partial^\gamma K(y_0,-\log\eps)}
+\int_{-\log\eps}^s e^{\mu s'}\abs{F_S^{(\gamma)}\circ \pu^{y_0}(s')} \,ds' \,. \label{eq:weighted:K:bnd2} 
\end{align}

\subsection{Estimates on $Z$}

For convenience of notation, in this section we  set $\Phi=\pz^{y_0}$.
We start with the case $\gamma=0$, for which we set $\mu=0$. Then, the first line of \eqref{eq:forcing_Z} combined with \eqref{eq:weighted:Z:bnd2} and our initial datum assumption \eqref{eq:Z_bootstrap:IC} show that
\begin{align*}
\abs{ Z\circ \Phi(s)}
&\les   \abs{ Z(y_0,-\log\eps)}
+\int_{-\log\eps}^s e^{-s'} \,ds' \les \eps \,.
\end{align*}
This improves the bootstrap assumption \eqref{eq:Z_bootstrap} for $\gamma=0$, upon taking $M$ to be sufficiently large to absorb the implicit universal constant in the above inequality.

For the case $\gamma=(1,0,0)$, we set $\mu=\frac 32$ so that \eqref{eq:mu_cond} is verified, and hence from \eqref{eq:Z_bootstrap:IC}, the second case in \eqref{eq:forcing_Z},  and \eqref{eq:weighted:Z:bnd2}, we find that
\begin{align*}
e^{\frac32 s}\abs{\partial_1 Z\circ \Phi(s)}
&\les   \eps^{-\frac32} \abs{\partial_1 Z(y_0,-\log\eps)}
 +\int_{-\log\eps}^s e^{\frac{3}{2}  s'}\abs{F_Z^{(\gamma)}\circ \pz^{y_0}(s')}  \,ds' \\
& \les 1 + M^2\int_{-\log\eps}^s \left(1+\abs{\Phi_1(s')}^{2}\right)^ {-\frac{1}{15} }   \,ds'  \\
& \les 1 + \eps^\frac{1}{30} M^2\int_{-\log\eps}^s  e^{\frac{s}{30}} \left(1+\abs{\Phi_1(s')}^{2}\right)^ {-\frac{1}{15} }   \,ds' 
\, .
\end{align*}
Now, applying 
\eqref{phi-lowerbound_conseq} with $\sigma_1=\frac{1}{30}$ and $\sigma_2 = \frac{2}{15} $, we deduce that by taking $ \eps$ sufficiently small, 
\begin{align} 
Me^{\frac32 s}\abs{\partial_1 Z\circ \Phi(s)} &\les 1\,, \label{zztop1}
\end{align} 
which improves the bootstrap assumption \eqref{eq:Z_bootstrap} for  $M$ taken sufficiently large.

For the case that $\gamma_1 = 1$ and $\abs{  \check \gamma}=1$, we set $\mu= \tfrac{3}{2} $, so that  
$$
\mu - \tfrac{3\gamma_1 + \gamma_2 + \gamma_3}{2} = \tfrac 12 - \gamma_1 \leq - \tfrac 12\,.
$$
We deduce from \eqref{eq:weighted:Z:bnd2:alt}, the fourth case in \eqref{eq:forcing_Z}, the initial datum assumption \eqref{eq:Z_bootstrap:IC}, and Lemma~\ref{lem:phiZ} with $\sigma_1 = { {\frac{5s}{2m-7}}} $, $m \ge 18$,  and $\sigma_2 = \frac{2}{15} $,   that 
\begin{align}
e^{\frac32 s}\abs{\p^\gamma  Z\circ \Phi(s)}
&\les   \eps^{-\frac32} \abs{\p^\gamma  Z(y_0,-\log\eps)}
+\! \int_{-\log\eps}^s \!\!\left( \!\!M^{\frac{1}{2}} \! +\! M\eps^ { {\frac{1}{2m-7}}} e^{ {\frac{5s}{2m-7}}}\!\left(1+\abs{\Phi_1(s')}^{2}\right)^{\!- \frac{1}{15} } \right) e^{-\frac{s-s'}{2}}   \,ds' \notag\\
& \les 1+ M^{\frac{1}{2}} +M\eps^ { {\frac{1}{2m-7}}}   \les  M^ {\frac{1}{2}}  \,.  \label{zztop2}
\end{align}
This improves the bootstrap stated in \eqref{eq:Z_bootstrap} by using the factor $M^{\frac 12}$ to absorb the implicit constant in the above inequality.

We are left to consider $\gamma$ for which $\gamma_1=0$ and $1\leq \abs{\check \gamma} \leq 2$.
For  $\abs{\gamma}=\abs{\check \gamma}=1$, setting $\mu=\frac 12$ (which satisfies
\eqref{eq:mu_cond}) we obtain from 
\eqref{eq:weighted:Z:bnd2}, the forcing bound  \eqref{eq:forcing_Z}, and the initial datum assumption \eqref{eq:Z_bootstrap:IC} that 
\begin{align}
e^{\frac s2 }\abs{\check\nabla Z\circ \Phi(s)}
&\les   \eps^{-\frac12} \abs{\check\nabla Z(y_0,-\log\eps)}
+M^{\frac{1}{2}}  \int_{-\log\eps}^s e^{- s'} \,ds' \les \eps^{\frac 12}\,.  \label{zztop3}
\end{align}
Finally, for $\abs{\gamma}=\abs{\check \gamma}=2$  we set $\mu=1$. As a consequence of \eqref{eq:forcing_Z}, \eqref{eq:Z_bootstrap:IC}, and \eqref{eq:weighted:Z:bnd2}, we obtain
\begin{align}
e^{s }\abs{\check\nabla^2 Z\circ \Phi(s)}
&\les   \eps^{-1} \abs{\check\nabla^2  Z(y_0,-\log\eps)}
+\int_{-\log\eps}^s e^{- ({\frac{1}{8}}  -  \frac{3}{2m-7} )s}   \,ds' \les 1\,,  \label{zztop4}
\end{align}
 Together, the estimates \eqref{zztop1}--\eqref{zztop4} improve the bootstrap bound \eqref{eq:Z_bootstrap} by taking  $M$ sufficiently large.

\subsection{Estimates on $K$}
We shall now set $\Phi=\pu^{y_0}$.
For the case $\gamma=(1,0,0)$, we set $\mu=\frac 32$ so that \eqref{eq:mu_cond} is verified, and hence from \eqref{eq:S_bootstrap:IC}, the second case in \eqref{eq:forcing_S},  and \eqref{eq:weighted:K:bnd2}, we find that
\begin{align*}
e^{\frac32 s}\abs{\partial_1 K\circ \Phi(s)}
&\les   \eps^{-\frac32} \abs{\partial_1 K(y_0,-\log\eps)}
 +\int_{-\log\eps}^s e^{\frac{3}{2}  s'}\abs{F_K^{(\gamma)}(s')}  \,ds' \\
& \les \eps^ {\frac{1}{2}}  +  \eps^ {\frac{1}{8}}  \int_{-\log\eps}^s \left(1+\abs{\Phi_1(s')}^{2}\right)^ {-\frac{1}{6}}   \,ds' \,,
\end{align*}
so that applying
\eqref{phi-lowerbound_conseq} with $\sigma_1=0$ and $\sigma_2 = \frac{1}{3} $, and taking $\eps$ sufficiently small, we deduce that
\begin{align} 
e^{\frac32 s}\abs{\partial_1 K\circ \Phi(s)} &\le \eps^ {\frac{1}{4}}  \,, \label{sstop1}
\end{align} 
which improves the second bootstrap assumption in \eqref{eq:S_bootstrap}.

Next, we study the case that  $\gamma_1=0$ and $1\leq \abs{\check \gamma} \leq 2$.
For  $\abs{\gamma}=\abs{\check \gamma}=1$, setting $\mu=\frac 12$ (which satisfies
\eqref{eq:mu_cond}) we obtain from 
\eqref{eq:weighted:K:bnd2}, the forcing bound  \eqref{eq:forcing_S}, and the initial datum assumption \eqref{eq:S_bootstrap:IC} that 
\begin{align}
e^{\frac s2 }\abs{\check\nabla K\circ \Phi(s)}
&\les   \eps^{-\frac12} \abs{\check\nabla K(y_0,-\log\eps)}
+M^2  \int_{-\log\eps}^s e^{- s'} \,ds' \les \eps^{\frac 12}\,.  \label{sstop3}
\end{align}
For $\abs{\gamma}=\abs{\check \gamma}=2$  we set $\mu=1$. As a consequence of \eqref{eq:forcing_S}, \eqref{eq:S_bootstrap:IC}, and \eqref{eq:weighted:K:bnd2}, we obtain
\begin{align}
e^{s }\abs{\check\nabla^2 K\circ \Phi(s)}
&\les   \eps^{-1} \abs{\check\nabla^2  K(y_0,-\log\eps)}
+M^2 \int_{-\log\eps}^s  e^{-{\frac{s'}{2}} } \,ds' \les \eps^{\frac14}\,,  \label{sstop4}
\end{align}
For $\abs{\gamma_1}=\abs{\check\gamma}=1$  we set $\mu= \frac{13}{8}   $ so that \eqref{eq:mu_cond} is verified.  
From \eqref{eq:forcing_S}, \eqref{eq:S_bootstrap:IC}, and \eqref{eq:weighted:K:bnd2:alt}, we apply 
\eqref{phi-lowerbound_conseq} with $\sigma_1 = {\frac{1}{4}} $ and $\sigma_2= \frac{2}{3} $ to obtain that
\begin{align}
e^{\frac{13}{8}  s}  \abs{\p_1 \check\nabla K\circ \Phi(s)}
&\les   \eps^{- \frac{13}{8} } \abs{\p_1 \check\nabla  K(y_0,-\log\eps)}
+  \eps^ {\frac{1}{8}} \int_{-\log\eps}^s  e^{{\frac{s'}{8}}  }\left(1+\abs{\Phi_1(s')}^{2}\right)^ {-\frac{1}{3}}  \,ds' \notag \\
&\les   \eps^{- \frac{13}{8} } \abs{\p_1 \check\nabla  K(y_0,-\log\eps)}
+  \eps^ {\frac{1}{4}} \int_{-\log\eps}^s  e^{{\frac{s'}{4}}  }\left(1+\abs{\Phi_1(s')}\right)^ {-\frac{2}{3}}  \,ds' \notag \\
&\les \eps^{\frac{3}{8}}   + \eps^  {\frac{1}{4}}   \le \eps^ {\frac{1}{8}} \,.  \label{sstop5}
\end{align}

We next consider  the case that $\gamma = (2,0,0)$.  
From \eqref{euler_for_Linfinity:d}, we have that
\begin{align*} 
\p_s \p_{11}K + (3 + \beta _1 \beta _\tau \Jcal \p_1 W) \p_{11}K + ( \mathcal{V} _U \cdot \nabla )\p_{11}K = F_S^{(2,0,0)} \,,
\end{align*} 
and hence
\begin{align}
  \p_s(e^ {2s} \eta^ \frac{1}{15}  \p_{11}K)  +  D_{K}^{(2,0,0)} ( e^ {2s} \eta^\frac{1}{15}  \p_{11}K)
  + \mathcal{V}_U \cdot \nabla (e^ {2s}\eta^\frac{1}{15}  \p_{11}K) &= e^ {2s}\eta^\frac{1}{15}  F_S^{(2,0,0)}  \notag
 \end{align} 
where
\begin{align}
 D_{K}^{(2,0,0)}=\tfrac{4}{5}  + \beta _1 \beta _\tau \Jcal \p_1 W+ \tfrac{1}{15}  \eta^{-1}-  
\tfrac{2}{15} 
\eta^{-1} \left(y_1 (\beta_1\beta_\tau \Jcal W+G_U ) + 3 h_{U}^{\nu} y_\nu \abs{\check y}^4\right)\,.
 \notag
\end{align}
Composing with $\Phi$, we find that
\begin{align*}
\abs{e^ {2s} \eta^ \frac{1}{15}  \p_{11}K(s)}
&\leq  \abs{ \eps^{-2} \eta^ \frac{1}{15}  \p_{11}K(-\log \eps)} \exp\left(- \int_{s_0}^s   D_{K}^{(2,0,0)}\circ \Phi(s')  \,ds'\right) \notag\\
&\qquad +\int_{s_0}^s\abs{ e^ {2s}\eta^ \frac{1}{15}  F_S^{(2,0,0)} \circ \Phi(s')}\exp\left(- \int_{s'}^s D_{K}^{(2,0,0)}\circ \Phi(s'') \,ds''\right) \,ds' \,.
\end{align*}
Thanks to \eqref{eq:p1W:PhiU} and  \eqref{eq:Deta:upper:bound}, we have that
\begin{align*} 
\exp\left(- \int_{s_0}^s   D_{K}^{(2,0,0)}\circ \Phi(s')  \,ds'\right) \les 1 \,,
\end{align*} 
and thus
using the third case in \eqref{eq:forcing_S}, and the initial datum assumption \eqref{eq:S_bootstrap:IC}, it follows that
\begin{align}
 \eta^\frac{1}{15}  e^{2s}\abs{\p_{11}  K\circ \Phi(s)}
&\les   \eps^ {\frac{1}{4}} 
+  \eps^ {\frac{1}{8}}
\int_{-\log\eps}^s  e^{\frac{s}{2}} \eta(\Phi(s'))^{-\frac{4}{15} } \ppp^{\frac{1}{4}}  (\Phi(s'),s')  \,ds'  \,.
\label{sstop7}
\end{align}
Now by definition of the weight $\ppp$,  we have that 
\begin{align*} 
  e^{\frac{s}{2}}\eta^{-\frac{4}{15} } \ppp^{\frac{1}{4}}   \circ \Phi
& \les  (e^{\frac{s}{2}}\eta^{-\frac{31}{60} } + e^{-\frac{s}{4}}\eta^{-\frac{1}{60}}) \circ \Phi\\
& \les  e^{\frac{s}{2}}\eta^{-\frac{31}{60} } \circ \Phi +\eps^ \frac{1}{60} e^{-{\frac{3}{10}} s}  \\
& \les e^{\frac{s}{2}} (1 + |\Phi| )^{-\frac{31}{30}}  +\eps^ \frac{1}{60}e^{-{\frac{3}{10}} s} 
\end{align*} 
where we used \eqref{e:space_time_conv} for the second inequality.   It follows that
\begin{align*} 
\int_{-\log\eps}^s  e^{\frac{s}{2}} \eta(\Phi(s'))^{-\frac{4}{15} } \ppp^{\frac{1}{4}}  (\Phi(s'),s')  \,ds' 
&\les
\int_{-\log\eps}^s  \!\!\!\! \!\!\!\! 
(e^{\frac{s'}{2}} (1 + |\Phi| )^{-\frac{31}{30}}  +\eps^ \frac{1}{60}e^{-{\frac{3}{10}} s'})   \,ds'  \les 1 \,,
\end{align*} 
where we have used the fact that $\int_{-\log\eps}^s  
\eps^ \frac{1}{60}e^{-{\frac{3}{10}} s'} \,ds'  \les \eps^ {\frac{19}{60}} $ as well as 
 \eqref{phi-lowerbound_conseq} with $\sigma_1= {\sfrac{1}{2}} $ and $\sigma_2 = {\frac{31}{30}}  $.
Hence,  
\begin{align*} 
 \eta^{\alpha} e^{2s}\abs{\p_{11}  K\circ \Phi(s)} \le\eps^ \frac{1}{6} 
\end{align*} 
which improves the fourth bootstrap assumption stated in \eqref{eq:S_bootstrap}.

\subsection{Estimates on $A$}
We can now close the bootstrap bounds \eqref{eq:A_bootstrap} for $\p^\gamma A$.   The bounds for the case that $\gamma_1=0$ and 
$\abs{\check \gamma}=0,1,2$ follow the same argument as given in (10.14) in \cite{BuShVi2019b}, whereas 
the estimate for $\p_1 A$ makes used of estimates for the vorticity.

\begin{lemma}[Relating $A$ and $\Omega$]
\label{lem:remarkable:sheep:structure}
With the self-similar specific vorticity $\Omega$ given by \eqref{svort-trammy}, 
\begin{subequations}
\begin{align}
e^{\frac{3s}{2}} \Jcal \p_1 A_2  
&= (\alpha  e^ {-\frac{K}{2}}  \sound)^{\frac 1\alpha} \Omega \cdot \Tcal^3 + \tfrac 12 \Tcal^2_\mu \left(\p_\mu W + e^{\frac s2} \p_\mu Z\right) - e^{\frac s2} \Ncal_\mu \p_\mu A_2 \notag\\
&\qquad - \tfrac 12 \left( \kappa + e^{-\frac s2} W + Z\right) (\operatorname{curl}_{\tilde x} \Ncal) \cdot \Tcal^3 - A_2 (\operatorname{curl}_{\tilde x} \Tcal^2) \cdot \Tcal^3 \\
e^{\frac{3s}{2}} \Jcal \p_1 A_3  
&= - (\alpha e^ {-\frac{K}{2}}  \sound)^{\frac 1\alpha} \Omega \cdot \Tcal^2 + \tfrac 12 \Tcal^3_\mu \left(\p_\mu W + e^{\frac s2} \p_\mu Z\right) - e^{\frac s2} \Ncal_\mu \p_\mu A_3 \notag\\
&\qquad + \tfrac 12 \left( \kappa + e^{-\frac s2} W + Z\right) (\operatorname{curl}_{\tilde x} \Ncal) \cdot \Tcal^2 - A_3 (\operatorname{curl}_{\tilde x} \Tcal^3) \cdot \Tcal^2 
\,.
\end{align} 
\end{subequations}
\end{lemma}
Propositions~\ref{prop:sound} and~\ref{prop:vorticity},  together with  the estimates \eqref{eq:W_decay}, \eqref{eq:Z_bootstrap}, \eqref{eq:A_bootstrap}, \eqref{e:space_time_conv} and \eqref{e:bounds_on_garbage}, and
Lemma~\ref{lem:remarkable:sheep:structure} show that
\begin{align} 
  e^{\frac{3}{2}s} \abs{ \p_1 A_\nu } \les \kappa_0^{\frac{1}{\alpha}}  \eps^ {\frac{1}{21}}   +  (1 + \eps^{\frac 12} M^{\frac 12} ) + (\kappa_0 + \eps^{\frac 16} + M \eps)  + M \eps \le M^ {\frac{1}{4}}   \,,  \label{A-estimates-2}
\end{align} 
for $M$ taken sufficiently large with respect to $\kappa_0^{\frac{1}{\alpha}}  C_{\kappa_0,\alpha}$.

\begin{proof}[Proof of Lemma~\ref{lem:remarkable:sheep:structure}]
We note that for the
velocity $\mathring u$ and with respect to the orthonormal basis $(\Ncal, \Tcal^2,\Tcal^3)$ we have that 
$$
\operatorname{curl}_{\tilde x} \mru = \left( \p_{\Tcal^3} \mru \cdot \Ncal - \p_{\Ncal} \mru \cdot \Tcal ^3\right)\Tcal^2
- \left( \p_{\Tcal^2} \mru \cdot \Ncal - \p_{\Ncal} \mru \cdot \Tcal ^2\right)\Tcal^3
+  \left( \p_{\Tcal^2} \mru \cdot \Tcal^3 - \p_{\Tcal^3} \mru \cdot \Tcal ^2\right)\Ncal \,.
$$
Now, from the definitions \eqref{eq:tilde:u:def}, \eqref{vort00}, \eqref{usigma-sheep},
\eqref{sv-sheep}, \eqref{s_ansatz}, \eqref{Sigma-trammy}, and  \eqref{svort-trammy}, we have that
\begin{align} 
( \alpha e^ {-\frac{K}{2}}  \sound)^ {\sfrac{1}{\alpha }} (y,s) \Omega(y,s) =  (\alpha e^ {-\frac{\mathring \scal}{2}}  \mathring \sigma( x, t))^ {\sfrac{1}{\alpha }}  \mathring \zeta(x, t) = \tilde \rho( \tilde x, t) \tilde \zeta( \tilde x, t) = \tilde \omega ( \tilde x, t) 
= \operatorname{curl} _{ \tilde x}  \tilde u ( \tilde x, t)  = \operatorname{curl} _{ \tilde x} \mathring u (x, t) \,,
\notag
\end{align} 
In particular, 
\begin{align} 
( \alpha e^ {-\frac{K}{2}}  \sound)^ {\sfrac{1}{\alpha }} (y,s) \Omega(y,s) = \operatorname{curl} _{\tilde x} \mathring u(x,t)
= \operatorname{curl} _{ \tilde x} \left( \mathring u (\tilde x_1 - f(\check{\tilde x}, t), \tilde x_2, \tilde x_3, t) \right)\,. \label{vort-correct1}
\end{align} 
We only establish the formula for $\p_1 A_3$, as the one for $\p_1 A_2$ is obtained identically. 
 To this end, we write
\begin{align*} 
\operatorname{curl}_{\tilde x} \mru \cdot \Tcal^2 & = \Tcal^3_j  \p_{\tilde x_j} \mru(x,t)     \cdot \Ncal -    \Ncal_j \p_{\tilde x_j} \mru(x,t)  \cdot \Tcal ^3  \,.
\end{align*} 
By the chain-rule and the fact that $\Ncal$ is orthogonal to $\Tcal^3$, we have that
\begin{align*} 
 \p_{\tilde x_j} \mru(x,t) \Tcal^3_j & = \p_{ x_1} \mru \Tcal^3_1 - f,_\nu \p_{x_1} \mru \Tcal^3_\nu + \p_{x_\nu} \mru \Tcal^3_\nu   =  \Jcal \Ncal \cdot \Tcal^3 \p_{x_1}\mru  +  \p_{x_\nu} \mru \Tcal^3_\nu   =  \p_{x_\nu} \mru(x,t) \Tcal^3_\nu \,.
\end{align*} 
The important fact to notice here is that no $x_1$ derivatives of $\mru$ remain.
Similarly, 
\begin{align*} 
 \p_{\tilde x_j} \mru(x,t) \Ncal_j & = \p_{ x_1} \mru \Ncal_1 - f,_\nu \p_{x_1} \mru \Ncal_\nu + \p_{x_\nu} \mru \Ncal_\nu  =  \Jcal \Ncal \cdot \Ncal \p_{x_1}\mru  +  \p_{x_\nu} \mru \Ncal_\nu   =  \Jcal \p_{x_1}\mru  + \p_{x_\nu} \mru(x,t) \Ncal_\nu \,.
\end{align*} 
Hence, it follows that
\begin{align} 
&\operatorname{curl}_{\tilde x} \mru \cdot \Tcal^2 \notag\\
& = \Tcal^3_\nu\p_{x_\nu} \mru(x,t)     \cdot \Ncal -   \Jcal \p_{x_1}(\mru\cdot \Tcal ^3) -   \Ncal_\nu\p_{x_\nu} \mru(x,t)  \cdot \Tcal ^3 \notag \\
&=\Tcal^3_\nu \p_{x_\nu} (\mru(x,t) \cdot \Ncal )  -   \Jcal \p_{x_1}a_3   - \Ncal_\nu    \p_{x_\nu}( \mru(x,t) \cdot \Tcal ^3)  - \mru(x,t)\cdot   \p_{x_\nu}  \Ncal\,  \Tcal^3_\nu  +   \mru(x,t)  \cdot \p_{x_\nu} \Tcal ^3\,  \Ncal_\nu \notag\\
&= \tfrac{1}{2} \Tcal^3_\nu \p_{x_\nu} (w+z ) -   \Jcal \p_{x_1}a_3   - \Ncal_\nu    \p_{x_\nu}a_3 
+\left( \tfrac{1}{2} (w+z) \Ncal + a_\nu \Tcal^\nu\right) \cdot (\p_{\Ncal}\Tcal ^3 - \p_{\Tcal^3} \Ncal) \label{vort-correct2}
\end{align}
where we have used \eqref{tildeu-dot-N}, \eqref{tildeu-dot-T}, and \eqref{tildewz}.   The identities \eqref{vort-correct1} and \eqref{vort-correct2} 
and the definition of the self-similar transformation in \eqref{eq:y:s:def} and \eqref{eq:ss:ansatz} yield the desired formula for $\p_1 A_3$.
\end{proof}

\section{Closure of $L^\infty$ based bootstrap for $W$}
\label{sec:W}

The goal of this section is to close the bootstrap assumptions which involve $W$, $\tilde W$ and their derivatives, stated in \eqref{eq:W_decay} and \eqref{eq:bootstrap:Wtilde}--\eqref{eq:bootstrap:Wtilde3:at:0}.

\subsection{Estimates for $\partial^{\gamma}\tilde W(y,s)$ for $\abs{y}\leq \ell$}

The estimates in this section closely mirror those given in Section 11.1 of \cite{BuShVi2019b}, as such will we simply summarize the argument.

\subsubsection{The fourth derivative}

Composing with the flow $\pw^{y_0}(s)$, we have that for $\abs{\gamma}=4$ that
\begin{align*}
\tfrac{d}{ds} \left(\p^\gamma \tilde W \circ \pw^{y_0}\right) + \left( D_{\tilde W}^{(\gamma)}\circ \pw^{y_0} \right) \left(  \p^\gamma \tilde W \circ \pw^{y_0}\right)  =  \tilde F_W^{(\gamma)} \circ \pw^{y_0} 
\, ,
\end{align*}
where 
\begin{align}
 D_{\tilde W}^{(\gamma)} := \tfrac{3\gamma_1 + \gamma_2 + \gamma_3-1}{2} 
+ \beta_\tau \Jcal  \left(\p_1  \bar W  + \gamma_1   \p_1 W \right)
 &\geq \tfrac 13  \,,
 \label{eq:p:gamma:tilde:W:damping}
\end{align}
which is a consequence of \eqref{eq:beta:tau} and \eqref{Jp1W}. Then as a consequence of \eqref{eq:Ftilde_4th_est}, \eqref{eq:p:gamma:tilde:W:damping}, and \eqref{eq:tilde:W:4:derivative} and the Gr\"onwall inequality we have that for all $\abs{y_0} \leq \ell$  and  all $s\geq -\log \eps$ such that $\abs{\pw^{y_0}(s)}\leq \ell$ the following estimate
\begin{align}
\abs{\p^\gamma \tilde W \circ \pw^{y_0}} 
\les \eps^{\frac18}
+ \eps^{\frac{1}{10}} (\log M)^{\abs{\check \gamma}-1}\,.
\label{eq:end:of:november}
\end{align}
Hence the bootstrap assumption \eqref{eq:bootstrap:Wtilde4} closes assuming the $\eps$ is chosen sufficiently small relative to $M$.

\subsubsection{Estimates for $\p^\gamma \tilde W$ with $\abs{\gamma}\leq 3$ and $\abs{y}\leq \ell$}

We first consider the estimate on $(\p^\gamma \tilde W )^0$ for $\abs\gamma =3$. Evaluating \eqref{eq:p:gamma:tilde:W:evo} at $y=0$ and applying \eqref{eq:bootstrap:Wtilde4}, \eqref{eq:bootstrap:Wtilde3:at:0},  \eqref{eq:GW:hW:0},  \eqref{e:forcing:W3}, and \eqref{eq:beta:tau} yields the estimate
\begin{align}
\abs{ \p_s (\p^\gamma \tilde W )^0   }
&\les e^{- (\frac 12- \frac{4}{2m-7})s}  + M (\log M)^{4} \eps^{\frac{1}{10}}  {e^{-s + \frac{4s}{2m-7}}} +  M  \eps^{\frac 14} e^{-s}  \les e^{- (\frac 12- \frac{4}{2m-7})s} \,. \label{cool-k}
\end{align}
Using the initial datum assumption \eqref{eq:tilde:W:3:derivative:0} and integrating in time, we may show
\begin{align}
\abs{\p^\gamma \tilde W(0,s)} \leq \tfrac{1}{10} \eps^{\frac 14}
 \label{eq:Schweinsteiger:2}
\end{align}
for all $\abs{\gamma}\leq 3$, and all $s\geq-\log\eps$, closing the bootstrap bound \eqref{eq:bootstrap:Wtilde3:at:0}.

The bootstraps \eqref{eq:bootstrap:Wtilde:near:0} corresponding to $0 \leq \abs{y} \le \ell$, then follow as a consequence  of constraints \eqref{eq:constraints} which imply
\[
\tilde W(0,s) =  \nabla \tilde W(0,s) = \nabla^2 \tilde W(0,s) = 0 \, ,
\]
together with  the estimates \eqref{eq:bootstrap:Wtilde4},  \eqref{eq:Schweinsteiger:2}, and the fundamental theorem of calculus, integrating from $y=0$.

Note that the bootstraps  \eqref{eq:bootstrap:Wtilde}, \eqref{eq:bootstrap:Wtilde1} and \eqref{eq:bootstrap:Wtilde2}, for the case  $\abs{y}\leq \ell$, follows as a consequence of \eqref{eq:bootstrap:Wtilde:near:0}, assuming $\eps$ is sufficiently small. 

\subsection{A framework for weighted estimates}\label{s:hydroxychloroquine}

Let us briefly recall the framework for weighted estimates introduced in Section 11.2 of \cite{BuShVi2019b}. For brevity will drop some intermediary calculations. Suppose some quantity $\RSZ$, satisfies an evolution equation of the form
\begin{equation}\label{eq:q_eq}
\p_s \RSZ + D_\RSZ \; \RSZ   + \mathcal{V}_W\cdot \nabla \RSZ= F_\RSZ \,.
\end{equation}
Weighting $\RSZ$ by $\eta^{\mu}$,
\[q :=\eta^{\mu}\, \RSZ,\]
then
 $q$ satisfies the evolution equation
\begin{align}
  \p_sq  + \underbrace{\left( D_{\RSZ} - \eta^{-\mu} \mathcal{V}_W \cdot \nabla \eta^\mu \right)}_{=: D_q} q + \mathcal{V}_W \cdot \nabla q &= \underbrace{\eta^{\mu}F_{\RSZ}}_{:=F_q} .\label{eq:tildeq}
\end{align}
where $\mathcal D_{q}$ may be expanded as 
\begin{align}
 D_{q}=D_\RSZ -3\mu+3\mu\eta^{-1}- 2 \mu
 \underbrace{\eta^{-1} \left(y_1 (\beta_\tau \Jcal W+G_W ) + 3 h_{W}^{\nu} y_\nu \abs{\check y}^4\right)\,.}_{=: \mathcal D_\eta}\,.
 \label{eq:Dq:def}
\end{align}
As a consequence of \eqref{eq:W_decay}, \eqref{e:space_time_conv}, \eqref{e:bounds_on_garbage_2}, \eqref{eq:beta:tau}, \eqref{e:G_W_estimates}, and \eqref{e:h_estimates} we have for all $s\geq -\log \eps$
\begin{align}
\abs{\mathcal{D}_\eta}
& \leq 5 \eta^{-\frac 13} +   e^{-\frac s3},
\label{eq:Deta:upper:bound}
\end{align}
assuming $\eps$ to be sufficiently small in order to absorb powers of $M$.

Using the evolution equation \eqref{eq:tildeq}, composing with the trajectories $\pw^{y_0}(s)$ such that $\pw^{y_0}(s_0) = y_0$ for some $s_0\geq - \log \eps$ with $\abs{y_0} \geq \ell$ and applying Gr\"onwall's inequality yields
\begin{align}
\label{eq:q:est}
\abs{q\circ \pw^{y_0}(s)}
&\leq   \abs{q(y_0)} \exp\left(- \int_{s_0}^s  D_{q}\circ\pw^{y_0}(s')  \,ds'\right) \notag\\
&\qquad +\int_{s_0}^s\abs{F_q\circ \pw^{y_0}(s')}\exp\left(- \int_{s'}^s D_{q}\circ\pw^{y_0}(s'') \,ds''\right) \,ds' \,.
\end{align}

For the special case $\ell \leq \abs{y_0} \leq \LLL$, we may may apply \eqref{eq:Deta:upper:bound},   \eqref{eq:escape_from_LA},  and the inequality $2 \eta(y) \geq 1 + \abs{y}^2 $ to conclude 
\begin{equation}
2 \mu \int_{s_0}^s  \abs{\mathcal D_\eta\circ\pw^{y_0}(s')}\,ds'
 \leq 70 \log \tfrac{1}{\ell} \,,
 \label{eq:Deta:est}
\end{equation}
 for all $\abs{\mu} \leq \frac 12$. Consequently, the estimates \eqref{eq:q:est} and \eqref{eq:Deta:est} yield
\begin{align}
\abs{q\circ \pw^{y_0}(s)}
&\leq \ell^{-70}    \abs{q(y_0)} \exp\left( \int_{s_0}^s \left( 3\mu - D_{\RSZ}- 3\mu\eta^{-1}  \right)\circ\pw^{y_0}(s')  ds'\right) \notag\\
& \qquad + \ell^{-70} \int_{s_0}^s\abs{F_q\circ \pw^{y_0}(s')}\exp\left(  \int_{s'}^s \left( 3\mu - D_{\RSZ}- 3\mu\eta^{-1}  \right)\circ\pw^{y_0}(s'')  ds''\right) \,ds' \, .
\label{eq:q:est:new}
\end{align}
We will need to consider two scenarios for the initial trajectory: either $s_0> -\log \eps$ and $\abs{y_0}=0$ or $s_0=-\log\eps$ and $\abs{y_0}\geq \ell$. We note that as long as $\abs{y_0}\geq \ell$, then $\abs{\pw^{y_0}(s)}\geq \ell$ for all $s>s_0$ as a consequence of Lemma~\ref{lem:escape} .

Now consider the case $\abs{y_0} \geq \LLL$. In place of \eqref{eq:Deta:est} for the case $\ell \leq \abs{y_0} \leq \LLL$, we have the stronger estimate
\begin{equation}
2 \mu \int_{s_0}^s  \abs{\mathcal D_\eta\circ\pw^{y_0}(s')}\,ds'
 \leq \eps^{\frac{1}{16}} \,,
 \label{eq:Deta:est:new}
\end{equation}
for $s_0\geq -\log\eps$, and $\abs{\mu} \leq \frac 12$.
 Hence \eqref{eq:q:est} and \eqref{eq:Deta:est:new} yield
\begin{align}
\abs{q\circ \pw^{y_0}(s)}
&\leq e^{\eps^{\frac{1}{16}}}    \abs{q(y_0)} \exp\left( \int_{s_0}^s \left( 3\mu - D_{\RSZ} - 3\mu\eta^{-1}  \right) \circ\pw^{y_0}(s') ds'\right) \notag\\
&\qquad  + e^{\eps^{\frac{1}{16}}} \int_{s_0}^s\abs{F_q\circ \pw^{y_0}(s')}\exp\left(  \int_{s'}^s \left( 3\mu - D_{\RSZ}  - 3\mu\eta^{-1}  \right) \circ \pw^{y_0} (s'')ds''\right) \,ds' \, .
\label{eq:q:est:new:new}
\end{align}

\subsection{Estimates of $\tilde W(y,s)$, $\partial_1\tilde W(y,s)$  and $\check \nabla \tilde W(y,s)$ for $\ell \leq \abs{y}\leq \LLL$}
\label{sec:tilde:W:middle}
The estimates of $\tilde W(y,s)$, $\partial_1\tilde W(y,s)$  and $\check \nabla \tilde W(y,s)$ for $\ell \leq \abs{y}\leq \LLL$ mimic those given in Section 11.3 - 11.4 in \cite{BuShVi2019b}. As such, we prove only an abridged summary of the arguments.

In order to close the bootstrap bound \eqref{eq:bootstrap:Wtilde}  on $\tilde W(y,s)$ for $\abs{y}\geq \ell$, we will use the framework in Section \ref{s:hydroxychloroquine} with $\RSZ = \tilde W$, $\mu = - \frac 16$. With these choices, the weighted quantity  $q=\eta^{-\frac16}\tilde W$, the quantity $3 \mu - D_{\RSZ} - 3\mu \eta^{-1}$ present in \eqref{eq:q:est:new} is $- \beta_\tau \Jcal \p_1 W + \frac 12 \eta^{-1}$ and $F_q=\eta^{-\frac 16} \tilde F_W$. 

Applying \eqref{eq:beta:tau}, \eqref{e:bounds_on_garbage}, \eqref{eq:escape_from_LA} and  \eqref{eq:W_decay}, we have
\begin{align}
\label{eq:Wtilde_damping}
\int_{s_0}^{s} \beta_\tau  \abs{\Jcal \p_1 W} \circ \pw^{y_0}(s') + \tfrac 12 \eta^{-1}\circ \pw^{y_0}(s') \,ds' 
\leq 40 \log \tfrac{1}{\ell}
\end{align}
for all $s\geq s_0 \geq -\log \eps$.  The estimate \eqref{eq:escape_from_LA} and \eqref{eq:Ftilde_est} yield the forcing estimate
\begin{align} 
\int_{s_0}^{s}\abs{\eta^{-\frac16} \tilde F_W}\circ \pw^{y_0}(s')\,ds'
\les 
\eps^{\frac 18} \log \tfrac{1}{\ell} 
\label{eq:Wtilde_forcing}
\end{align}
for all $s\geq s_0 \geq -\log \eps$, and $\ell \in (0,1/10]$. 

Combining the bounds \eqref{eq:Wtilde_damping} and \eqref{eq:Wtilde_forcing} into \eqref{eq:q:est:new},  and using
 the initial data assumption \eqref{eq:tilde:W:zero:derivative} if $s_0=-\log \eps$, or alternatively   \eqref{eq:bootstrap:Wtilde:near:0} if $s_0>-\log \eps$, we obtain
 \begin{align} 
\eta^{-\frac 16}(y) \abs{\tilde W(y,s)} \leq \tfrac{1}{10} \eps^{\frac{1}{11}}
\label{tildeWfinal}
\end{align} 
for all $\ell \leq \abs{y} \leq \LLL$ and all $s\geq -\log \eps$. Where we have employed small powers of $\eps$ to absorb all the $\ell$ and $M$ factors. The above estimate \eqref{tildeWfinal} closes the bootstrap \eqref{eq:bootstrap:Wtilde}. 

We now aim to close the bootstrap bound  \eqref{eq:bootstrap:Wtilde1} on  $\partial_1\tilde W(y,s)$ for $\ell \leq \abs{y}\leq \LLL$. For this case, we set $\RSZ = \p_1 \tilde W$, $\mu = \frac 13$ and hence $q=\eta^{\frac13} \p_1 \tilde W$. By \eqref{eq:p:gamma:tilde:W:evo} with $\gamma = (1,0,0)$, we have $3 \mu - D_{\RSZ}=- \beta_\tau \Jcal (\p_1 W + \p_1 \bar W)$, and $F_q =\eta^{\frac 13} \tilde F_W^{(1,0,0)}$. 

Similar to the estimate \eqref{eq:Wtilde_damping}), we may bound the the contributions to \eqref{eq:q:est:new} due to the damping term $3\mu -D_\RSZ$ by
\begin{align}
\label{eq:Wtilde_damping:d1}
\int_{s_0}^{s} \beta_\tau  \abs{\Jcal (\p_1 W + \p_1 \bar W)} \circ \pw^{y_0}(s')\,ds' 
\leq 80 \log \tfrac{1}{\ell} \,.
\end{align}
The contribution due to the forcing $F_q=\eta^{\frac 13} \tilde F_W^{(1,0,0)}$ is bounded using  \eqref{eq:escape_from_LA}  and \eqref{eq:Ftilde_d1_est} in order to attain
\begin{align} 
\int_{s_0}^{s}\abs{F_q}\circ \pw^{y_0}(s')\,ds'
 \les \eps^{\frac{1}{11}}  \log \tfrac{1}{\ell} \,.
\label{eq:Wtilde_forcing:d1}
\end{align}
Inserting \eqref{eq:Wtilde_damping:d1} and \eqref{eq:Wtilde_forcing:d1} into \eqref{eq:q:est:new}, and using our initial datum assumption \eqref{eq:tilde:W:p1=1} when $s_0 = -\log \eps$, respectively \eqref{eq:bootstrap:Wtilde4} for $s_0 > -\log \eps$, yields
\begin{align} 
\eta^{\frac 16}(y) \abs{\p_1 \tilde W(y,s)} \leq \tfrac{1}{10} \eps^{\frac{1}{12}}
\label{p1tildeWfinal}
\end{align} 
for all $\ell \leq \abs{y} \leq \LLL$ and all $s\geq -\log \eps$, where we again have used small powers of $\eps$ to absorb all the $\ell$ and $M$ factors. The above estimate closes the bootstrap \eqref{eq:bootstrap:Wtilde1}. 

Finally, we aim to close the bootstrap \eqref{eq:bootstrap:Wtilde2} on  $\check \nabla \tilde W(y,s)$ for $\abs{y}\geq \ell$. We set $\RSZ = \check\nabla    W$ and $\mu=0$, so that $q= \check \nabla \tilde W$. From \eqref{eq:p:gamma:tilde:W:evo} with $\gamma \in \{(0,1,0),(0,0,1)\}$, we have  $3 \mu - D_{\RSZ} = - \beta_\tau \Jcal  \p_1 \bar W$ and $F_q = \tilde F_W^{(\gamma)}$. 

The integral of the damping term arising in \eqref{eq:q:est:new} is bounded using \eqref{eq:Wtilde_damping} by $40 \log \ell^{-1}$. The  contribution due to the forcing $F_q$ is bounded using  \eqref{eq:escape_from_LA}  and \eqref{eq:Ftilde_dcheck_est} in order to attain
\begin{align} 
\int_{s_0}^{s}\abs{F_q}\circ \pw^{y_0}(s')\,ds'
 \les \eps^{\frac{1}{12}}  \log \tfrac{1}{\ell} \,.
\label{eq:Wtilde_forcing:dtilde}
\end{align}
Inserting \eqref{eq:Wtilde_damping} and \eqref{eq:Wtilde_forcing:dtilde} into \eqref{eq:q:est:new}, and using our initial datum assumption \eqref{eq:tilde:W:check=1}  and \eqref{eq:bootstrap:Wtilde4}, we arrive at
\begin{align} 
\abs{\check \nabla \tilde W(y,s)} \leq \tfrac{1}{10} \eps^{\frac{1}{13}}
\label{p2tildeWfinal}
\end{align} 
for all $\ell \leq \abs{y} \leq \LLL$ and all $s\geq -\log \eps$, thereby closing the bootstrap bound \eqref{eq:bootstrap:Wtilde2}.   We also note that the bootstrap
bound \eqref{eq:W_decay} for the cases that $\abs{\gamma}=0,1$ and $\ell\leq\abs{y}\leq \mathcal L$ follow as a consequence of \eqref{eq:tildeW_decay} together with the $\bar W$ bound (2.48) in \cite{BuShVi2019b}.

\subsection{Estimate for $\partial^{\gamma} W(y,s)$ with $\abs{\gamma}=2$ for $\abs{y}\geq \ell$}
We now consider the case $\abs{\gamma} = 2$, and establish  the third and  fifth bounds of \eqref{eq:W_decay}. 
Unlike  the bounds given in Section 11.6 of \cite{BuShVi2019b},  the bound for  $\partial_{11}W$ makes use of two weight functions, and requires a
new type of analysis.  
 
As such, we now consider the case that $\gamma_1=2$ and $\abs{\check \gamma}=0$.   We have that
\begin{align*}
&  \p_s (\eta^ {\frac{1}{3}} \p_{11}W)  + \bar{ \mathcal{D} }_W^{(2,0,0)}( \eta^ {\frac{1}{3}} \p_{11}W) + \mathcal{V}_W \cdot \nabla (\eta^ {\frac{1}{3}} \p_{11}W )= \eta^{{\frac{1}{3}} }F_W^{(2,0,0)} \\
& \qquad  \bar {\mathcal{D} }_W^{(2,0,0)}=\tfrac32+ \eta^{-1}- {\tfrac{2}{3}} 
 \underbrace{\eta^{-1} \left(y_1 (\beta_\tau \Jcal W+G_W ) + 3 h_{W}^{\nu} y_\nu \abs{\check y}^4\right)}_{=: \mathcal D_\eta}  \,,
\end{align*}
from which it follows that
\begin{align*}
&  \p_s (\eta^ {\frac{1}{3}} \ppp^ {-{\frac{1}{4}} }  \p_{11}W)  + \mathcal{D} _W^{(2,0,0)}( \eta^ {\frac{1}{3}}\ppp^ {-{\frac{1}{4}} } \p_{11}W) + \mathcal{V}_W \cdot \nabla (\eta^ {\frac{1}{3}} \ppp^ {-{\frac{1}{4}} }\p_{11}W )= \eta^{{\frac{1}{3}} }\ppp^ {-{\frac{1}{4}} }F_W^{(2,0,0)} \,,
\end{align*} 
where
\begin{align*} 
 \mathcal{D} _W^{(2,0,0)} 
 & =\tfrac32+ \eta^{-1}- {\tfrac{2}{3}} D_\eta
 - \tfrac{3}{4} e^{-3s} \ppp ^{-1} \eta +\tfrac{1}{2} \ppp^{-1} y_1(e^{-3s} -  \eta^{-2}) \mathcal{V} _W^1 + \tfrac{3}{2}  \ppp^{-1} \abs{\check y}^4 y_\mu (e^{-3s} -  \eta^{-2})\mathcal{V} _W^\mu   \\
& =\tfrac32+ \eta^{-1}- {\tfrac{2}{3}} D_\eta  -\tfrac{3}{4}  e^{-3s} \ppp ^{-1}
-\tfrac{3}{4} \ppp ^{-1} \tfrac{y_1^2+ \abs{\check y}^6}{\eta^2} 
\\
& \qquad\qquad \qquad \qquad 
+ \tfrac{1}{2}  \underbrace{\ppp^{-1}  \left(e^{-3s} - \eta^{-2} \right)\left( y_1( \beta_\tau \Jcal W +G_W) 
+ 3 \abs{\check y}^4  y_\mu h_W^\mu\right)}_{=:\mathcal{D} _\ppp}
 \,,
\end{align*}
and therefore
\begin{align}
\abs{\eta^ {\frac{1}{3}} \ppp^ {-{\frac{1}{4}} }  \p_{11}W\circ \pw^{y_0}(s)}
&\leq   \abs{\eta^ {\frac{1}{3}} \ppp^ {-{\frac{1}{4}} }  \p_{11}W(y_0)} \exp\left(- \int_{s_0}^s   \mathcal{D} _W^{(2,0,0)} \circ\pw^{y_0}(s')  \,ds'\right) \notag\\
&\qquad +\int_{s_0}^s\abs{\eta^{{\frac{1}{3}} }\ppp^ {-{\frac{1}{4}} }F_W^{(2,0,0)}\circ \pw^{y_0}(s')}\exp\left(- \int_{s'}^s  \mathcal{D} _W^{(2,0,0)} \circ\pw^{y_0}(s'') \,ds''\right) \,ds' \,.   \label{chicken-leg}
\end{align}

Since $\ppp^{-1}\leq \eta$, we then have that
$$
\ppp ^{-1} \tfrac{ y_1^2+ \abs{\check y}^6 }{ \eta^2 } \le 1 \,.
$$
Moreover, using \eqref{e:space_time_conv}, we see that
\begin{align*} 
e^{-3s} \ppp ^{-1} \le e^{-3s} \eta \le 40\eps\,,
\end{align*} 
and thus,  we have that
\begin{align} 
\tfrac32 - \tfrac{3}{4}  \ppp ^{-1} \tfrac{y_1^2+ \abs{\check y}^6}{\eta^2} - e^{-3s} \ppp ^{-1}  \ge 0 \,. \label{chicken-wing}
\end{align} 
Again, since $\ppp^{-1}\leq \eta$, then   \eqref{e:space_time_conv}  yields
\begin{align*} 
\abs{\ppp^{-1}  \left(e^{-3s} - \eta^{-2} \right) } \le \tfrac{4}{3}\eta^{-1} \,,
\end{align*} 
Therefore, we see from the definition \eqref{eq:Dq:def} of $\abs{\mathcal{D} _\eta}$, that  $\abs{\mathcal{D} _\ppp}\leq \frac{4}{3}\abs{\mathcal{D} _\eta}$.
It follows from \eqref{eq:Deta:est} that for all $\abs{y_0} \geq \ell$, 
\begin{align} 
 \int_{s_0}^s  \abs{(\tfrac{2}{3} \mathcal D_\eta + \tfrac{1}{2} \mathcal{D} _\ppp)\circ\pw^{y_0}(s')}\,ds'\leq \tfrac{4}{3}\int_{s_0}^s  \abs{\mathcal D_\eta \circ\pw^{y_0}(s')}\,ds'
  \leq 140 \log \tfrac{1}{\ell} \,.
 \label{eq:Dphi:est}
\end{align} 
By \eqref{chicken-wing} and \eqref{eq:Dphi:est}, we see that \eqref{chicken-leg} is bounded as
\begin{align}
\abs{\eta^ {\frac{1}{3}} \ppp^ {-{\frac{1}{4}} }   \p_{11}W\circ \pw^{y_0}(s)}
&\leq  \ell^{-140} \abs{\eta^ {\frac{1}{3}} \ppp^ {-{\frac{1}{4}} }   \p_{11}W(y_0)} 
+ \ell^{-140}\int_{s_0}^s\abs{\eta^{{\frac{1}{3}} }\ppp^ {-{\frac{1}{4}} } F_W^{(2,0,0)}\circ \pw^{y_0}(s')} ds'\,.  \label{chicken-neck}
\end{align}
With the estimate
 \eqref{eq:forcing_W} for $F_W^{(2,0,0)} $, we obtain that
\begin{align*} 
\abs{\eta^{{\frac{1}{3}} }\ppp^ {-{\frac{1}{4}} } F_W^{(2,0,0)} }
\les \eta^{-\frac{3}{20}+\frac{8}{3(2m-7)}} \les \eta^ {-\frac{1}{10}}   \,.
\end{align*} 
Hence, following \eqref{eq:Deta:est:new}, we see that for $\abs{y_0} \ge \ell $, 
\begin{align*} 
\int_{s_0}^s\abs{\eta^{{\frac{1}{3}} }\ppp^ {-{\frac{1}{4}} } F_W^{(2,0,0)}\circ \pw^{y_0}(s')} ds'\
\leq \int_{s_0}^{\infty} \left( 1+ \ell^2 e^{\frac 25 (s'-s_0)}\right)^{-\frac 1{10}} ds' \les \ell^ {-\frac{1}{5}}  \,.
\end{align*} 
By appealing to our initial datum assumption~\eqref{eq:W:gamma=2:p2}  if $s_0 = -\log \eps$, and to \eqref{eq:bootstrap:Wtilde:near:0} when $s> -\log \eps$,
the bound \eqref{chicken-neck} shows that
\begin{align}\label{chicken-fat}
\abs{\eta^ {\frac{1}{3}} \ppp^ {-{\frac{1}{4}} }   \p_{11}W\circ \pw^{y_0}(s)}\les \ell^{-141}\les M^{\frac14}\,.
\end{align}
By choosing first $M$ sufficiently large, the bootstrap assumption \eqref{eq:W_decay} is then improved by
\eqref{chicken-fat}.

It remains to consider the case $\abs\gamma=2$ and $\abs{\check\gamma}=1,2$. The arguments will mimic those given in Section 11.6 of \cite{BuShVi2019b}, and as such, we provide an abridged version of those arguments. For the case $\abs{\check\gamma}=1$ and $\gamma_1=1$, we set $\mu=\frac13$, whereas, for the case $\abs{\check\gamma}=2$ and $\gamma_1=0$, we set $\mu=\frac16$.  Consequently, the damping term $3 \mu - D_{\RSZ}$ present in \eqref{eq:q:est:new} is given by
\begin{align}
3 \mu - D_{\RSZ} 
= 
\begin{cases}
-\tfrac{1}{2} -  \beta_\tau \Jcal \p_1 W\, ,\qquad &\mbox{for} \qquad \abs{\check\gamma}=1 \; \mbox{  and  } \; \gamma_1 = 1 \, ,\\
-  \beta_\tau \Jcal \p_1 W \, ,\qquad &\mbox{for} \qquad \abs{\gamma}=2 \; \mbox{  and  } \; \gamma_1 = 0 \, .
\end{cases}
\label{eq:Magic:Johnson}
\end{align}

Let us first restrict to the case  $\gamma_1=0$ and $\abs{\check \gamma} =2$. Analogous to \eqref{eq:Wtilde_damping},  we have 
\begin{align}
\int_{s_0}^{s} \beta_\tau  \abs{\Jcal \p_1 W} \circ \pw^{y_0}(s')\,ds' 
\leq 40 \log \tfrac{1}{\ell}
\label{eq:Magic:Johnson:1}
\end{align}
and analogously to \eqref{eq:Wtilde_forcing}, applying \eqref{eq:forcing_W}, we have
\begin{align} 
\int_{s_0}^{s}\abs{\eta^{\frac16} F_W^{(\gamma)}}\circ \pw^{y_0}(s')\,ds'
\leq  
M^{\frac56} \log \tfrac{1}{\ell} .
\label{eq:Magic:Johnson:2}
\end{align}
Substituting the bounds \eqref{eq:Magic:Johnson:1} and \eqref{eq:Magic:Johnson:2} into \eqref{eq:q:est:new}, and utilizing our initial datum assumption \eqref{eq:W:gamma=2}  when $s_0 = - \log \eps$, and to \eqref{eq:bootstrap:Wtilde:near:0} when $s> -\log \eps$,
we deduce
\begin{align*}
\eta^{\frac 17}(y) \abs{\check\nabla^2 W(y,s)} 
&\leq \ell^{-110} \eta^{\frac 16}(y_0) \abs{\check\nabla^2 W(y_0,s_0)} +  M^{\frac 56} \ell^{-110}  \log \tfrac{1}{\ell} \notag\\
&\leq \tfrac{1}{10} M^{\frac 16} 
\end{align*}
where we have assumed that $M$ is sufficiently large, used our choice $\ell = (\log M)^{-5} $ and assumed $\eps$ is sufficiently small relative to $M$. Thus we close the bootstrap \eqref{eq:W_decay} for the case $\gamma_1=0$ and $\abs{\check \gamma} =2$.

We now turn our attention to the case $\abs{\check\gamma}=1$, with $\gamma_1 =1$. Applying \eqref{eq:Magic:Johnson} and \eqref{eq:Magic:Johnson:1}, yields the damping bound
\begin{align}
\exp\left( \int_{s'}^s \left( 3\mu - D_{\RSZ}\circ\pw^{y_0}(s'') \right) ds'' \right)\leq \ell^{-120} e^{\frac{s'-s}{2}} 
\label{eq:Magic:Johnson:3}
\end{align}
for any $s> s' > s_0 \geq -\log \eps$.   Substituting  \eqref{eq:Magic:Johnson:3}, together with the forcing estimate  \eqref{eq:forcing_W} into \eqref{eq:q:est:new}, and appealing to our initial datum  assumption~\eqref{eq:W:gamma=2:p1}  if $s_0 = -\log \eps$, and to \eqref{eq:bootstrap:Wtilde:near:0} when $s> -\log \eps$, we deduce
\begin{align}
\eta^{\frac 13}(y) \abs{\p^\gamma W(y,s)} 
&\leq \tfrac{1}{10} M^{\frac 16}
\label{D2Wfinal}
\end{align}
where we have assumed that $M$ is sufficiently large, used our choice $\ell = (\log M)^{-5} $ and assumed $\eps$ is sufficiently small relative to $M$. Thus we close the bootstrap \eqref{eq:W_decay} for the case $\abs{\check\gamma}=1$, with $\gamma_1 =1$.

\subsection{Estimate of $W(y,s)$, $\p_1 W(y,s)$ and $\check \nabla W(y,s)$ for $ \abs{y}\geq \LLL$}

The estimates of $W(y,s)$, $\p_1 W(y,s)$ and $\check \nabla W(y,s)$ for $ \abs{y}\geq \LLL$ are nearly identical to those given in Section 11.7, 11.8 and 11.9 of \cite{BuShVi2019b}. As such, we prove only an abridged summary of the arguments.

Consider first the estimate on $W(y,s)$. We set $\mu = -\frac 16$ and $\RSZ =W$, so that $q = \eta^{-\frac 16} \tilde W$. We have $3 \mu - D_{\RSZ} - 3\mu \eta^{-1}=\frac 12 \eta^{-1}$, and $F_q=\eta^{-\frac 16} ( F_W - e^{-\frac s2} \beta_\tau \dot \kappa)$. The contribution of the damping in \eqref{eq:q:est:new:new} gives us
\begin{align*}
\int_{s_0}^{s}  \tfrac 12 \eta^{-1}\circ \pw^{y_0}(s') \,ds' 
\leq  \LLL^{-\frac 23} = \eps^{\frac{1}{16}},
\end{align*}
and we have from \eqref{eq:F_WZ_deriv_est} and \eqref{eq:acceleration:bound} the forcing bound \begin{align*}
 \int_{s_0}^s \abs{F_q \circ \pw^{y_0}}(s') ds'\les \eps^{\frac 12}.
\end{align*}
Substituting the above two estimates into \eqref{eq:q:est:new:new}, we obtain
\begin{align*}
\abs{\eta^{-\frac 16}   W\circ \pw^{y_0}(s)} \leq 1 + \eps^{\frac{1}{19}}\,,
\end{align*}
where for the case $s_0> -\log \eps$, we used \eqref{eq:bootstrap:Wtilde} and $\bar W$ bound (2.48) in \cite{BuShVi2019b}, and for the case $s_0 = -\log \eps$, we use the  initial data assumption \eqref{eq:rio:de:caca:1}. Thus we close the bootstrap bound in the first line of \eqref{eq:W_decay}.

For the case $\p_1 W(y,s)$ we set $q = \eta^{\frac 13} \p_1 W$, so that $3 \mu - D_{\RSZ} - 3\mu \eta^{-1} \leq - \beta_\tau \Jcal \p_1 W$ and  $F_q\eta^{\frac 13}  F_W^{(1,0,0)}$.  Applying \eqref{eq:W_decay},  and Lemma~\ref{lem:escape}, yields
\begin{align}
\int_{s_0}^s \left( 3\mu - D_{\RSZ} - 3\mu\eta^{-1}  \right) \circ\pw^{y_0}(s') ds'
\leq \eps^{\frac{1}{16}}\,.
\label{eq:need:a:label}
\end{align}
As a consequence of \eqref{eq:forcing_W} and the fact that  $\abs{y_0} \geq \LLL$, we obtain
$$
\int_{s_0}^s \abs{F_q \circ\pw^{y_0}(s')}  ds'\les
 \eps^ {3\alpha } \,.
$$
Substituting the above two estimates into \eqref{eq:q:est:new:new}, we obtain
\begin{align*}
\abs{\eta^{\frac 13}   \p_1W\circ \pw^{y_0}(s)} \leq   \tfrac 32 \,.
\end{align*}
where for the case $s_0> -\log \eps$, we used \eqref{eq:bootstrap:Wtilde1} and the $\bar W$ bound (2.48) in \cite{BuShVi2019b}, and for the case $s_0 = -\log \eps$, we use the  initial data assumption \eqref{eq:rio:de:caca:2}. Thus we close the bootstrap bound in the second line of \eqref{eq:W_decay}.

Finally, we consider the estimate of $\check \nabla W(y,s)$ for $ \abs{y}\geq \LLL$. We set $\mu = 0$ and $q = \check \nabla W$. The damping term is $3 \mu - D_{\RSZ} - 3\mu \eta^{-1} = - \beta_\tau \Jcal \p_1 W$ and so we may reuse the estimate \eqref{eq:need:a:label}. The forcing term $F_q$ may be bounded directly using the third case in \eqref{eq:forcing_W}, which yields
$$
\int_{s_0}^s \abs{F_q \circ \pw^{y_0}(s')} ds' \leq  \eps^{\frac 18}  \,.
$$
We deduce from \eqref{eq:q:est:new:new} that 
\begin{align*}
\abs{\check \nabla W\circ \pw^{y_0}(s)} \leq   \tfrac56\,.
\end{align*}
where for the case $s_0> -\log \eps$, we used \eqref{eq:bootstrap:Wtilde2} and the $\bar W$ bound (2.48) in \cite{BuShVi2019b}, and for the case $s_0 = -\log \eps$, we use the  initial data assumption \eqref{eq:rio:de:caca:3}. Thus we close the bootstrap bound in the second line of \eqref{eq:W_decay}.

\section{Constraints and evolution of modulation variables}
\label{sec:constraints}

\subsection{Solving for the dynamic modulation parameters} 
\label{sec:implicit:crap}
In Section~\eqref{sec:constraints:initial} we have used the evolution equations for $W$, $\nabla W$ and $\nabla^2 W$ at $y=0$ to derive implicit equations for the time derivatives our modulation parameters. The goal of this subsection is to show that these implicit equations are indeed solvable with the initial conditions \eqref{eq:modulation:IC}.
For this purpose it convenient to introduce the notation 
$$
\PP_{\diamondsuit}({\mathsf{b}}_1,\ldots,{\mathsf{b}}_n  \big|  {\mathsf{c}}_1,\ldots, {\mathsf{c}}_n) \qquad  \text{ and } \qquad
\mathcal{R}_{\diamondsuit}({\mathsf{b}}_1,\ldots,{\mathsf{b}}_n  \big|  {\mathsf{c}}_1,\ldots, {\mathsf{c}}_n)
$$
to denote a linear function in the parameters $ {\mathsf{c}}_1,\ldots, {\mathsf{c}}_n$ with  coefficients  which depend on 
${\mathsf{b}}_1,\ldots,{\mathsf{b}}_n$ through smooth polynomial (for $\PP_\diamondsuit$), respectively,  rational functions (for $\mathcal{R} _\diamondsuit$), 
and on the derivatives of $Z$, $A$, and $K$ evaluated at $y=0$. 
The subscript 
$\diamondsuit$ denotes a label, used to 
distinguish the various functions $\PP_\diamondsuit$ and  $\mathcal{R} _\diamondsuit$.
We note that all of the denominators in $ \mathcal{R} _\diamondsuit$ are
bounded from below by a universal constant.   It is important to note that the notation $\PP_\diamondsuit$ and  $\mathcal{R} _\diamondsuit$ is never used when explicit bounds are required.
Throughout this section, we will use  the bootstrap assumptions in Section~\ref{sec:bootstrap} to establish  uniform bounds   on the  coefficients, which in turn,  yields local well-posedness of the coupled system of  
ODE for the modulation variables.

The definition of $\dot\kappa$ in \eqref{eq:dot:kappa:1} may be written schematically using the notation introduced above as  
\begin{align}
\dot \kappa = \PP_\kappa \left(\kappa, \phi \, \big| \,  \dot{Q}, \tfrac{1}{\beta_\tau} e^{\frac s2} h_W^{,0}, \tfrac{1}{\beta_\tau} e^{\frac s2} G_W^{0}  \right)  \, ,
  \label{eq:dot:kappa:2}
\end{align}
where we have used the explicit formula \eqref{eq:FW:0:a} to determine the dependence of $\PP_\kappa$. 
Once we compute $h_W^{,0}$ and $G_W^{0}$ (cf.~\eqref{eq:GW:def:1}--\eqref{eq:hj:def:1} below) we will return to the formula \eqref{eq:dot:kappa:2}.  We point out at this stage that in \eqref{eq:GW:hW:0} below we will show that both $h_W^{,0}$ and $G_W^{0}$ decay at a rate which is strictly faster than $e^{-\frac s2}$, which shows that their contribution to $\dot \kappa$ will be under control. 

Similarly, the definition of $\dot\tau$ in \eqref{eq:dot:tau:1} may be written schematically  as  
\begin{align}
\dot \tau = \PP_{  \tau} \left(\kappa,\phi  \, \big| \,  e^{-2s} \dot{Q},  \tfrac{1}{\beta_\tau} h_W^{,0} \right)   \, ,
  \label{eq:dot:tau:2}
\end{align}
where we have used the explicit formulae \eqref{eq:GW:0:b} and \eqref{eq:FW:0:b} to determine the dependence of $\PP_{\tau}$.  

The schematic dependence of $\dot Q_{1,\nu}$ is determined from  \eqref{eq:check_derivative_GW}. Using \eqref{eq:GW:0:c} and \eqref{eq:FW:0:c} and placing the leading order term in  $\dot Q$ on one side, we obtain
\begin{align}
\dot Q_{1 \nu} 
&= - e^{-\frac s2} \dot Q_{1\mu} \p_\nu A_\mu^0 +   e^{-s} \dot{Q}_{\mu \zeta} A_\zeta^0   \phi_{\mu \nu}  +  e^{-s} \dot{Q}_{\mu \nu} A_\zeta^0 \phi_{\zeta \mu} - \tfrac{\beta_2}{2 \beta_1} e^{\frac s2} \p_\nu Z^0 
+    e^{-s} A_\mu^0 \dot{\phi}_{\mu \nu} \notag\\
&  + \tfrac{\beta_3}{2 \beta_1} \left( (\kappa - Z^0) \p_{\nu \mu} A_\mu^0 - \p_\nu Z^0 \p_\mu A_\mu^0 \right)
 + \tfrac{\beta_3}{\beta_1} e^{-\frac s2}  Z^0   \p_\nu Z^0   (\phi_{22} + \phi_{33}) + \tfrac{\beta_3}{2 \beta_1}  e^{-s} \left(\kappa - Z^0 \right)  A_\zeta^0 \Tcal^{\zeta,0}_{\mu, \mu\nu}  
\notag\\
&   +  e^{-\frac s2}  \left( ( \p_\nu A_\mu^0 - \tfrac 12 e^{-\frac s2} ( \kappa + Z^0) \phi_{\mu\nu}) A_\gamma^0 \right) \phi_{\gamma \mu}  +  \tfrac{1}{2\beta_1\beta_\tau} h_W^{\mu,0} \p_\nu A_\gamma^0   \phi_{\gamma \mu}   +   \left(\tfrac{1}{2\beta_1\beta_\tau} e^{\frac s2} h_W^{\gamma,0} - A_\gamma^0 \right) \phi_{\gamma \nu}
\notag\\
& - \tfrac 14  \beta_4 (\kappa - Z^0) \left( (\kappa - Z^0) ( e^s \p_{1\nu}K^0   - e^{-\frac s2} \p_\mu K \phi_{\mu\nu}) - 2 \p_\nu Z^0 e^s \p_1 K \right)   \, ,
\label{eq:dot:Q:1}
\end{align}
which may be written schematically as
\begin{align}
\dot{Q}_{1\nu} = \PP_{Q,\nu} \left(\kappa, \phi  \, \big| \,  \tfrac{1}{\beta_\tau} e^{\frac s2} h_W^{,0}, e^{-s} \dot \phi, e^{-s} \dot Q \right)
\label{eq:dot:Q:2} \, . 
\end{align}
Note that once $\dot Q_{1\nu}$ is known, we can determine $\dot{n}_2$ and $\dot{n}_3$ by recalling from \cite[Equations (A.4)--(A.5)]{BuShVi2019b} that
\begin{align}
\left[
\begin{matrix} 
1 + \tfrac{n_2^2}{n_1(1+n_1)} & \tfrac{n_2 n_3}{n_1(1+n_1)} \\
\tfrac{n_2 n_3}{n_1(1+n_1)} & 1 + \tfrac{n_3^2}{n_1(1+n_1)}   \\
\end{matrix}\right]  
\left[
\begin{matrix} 
\dot{n}_2  \\
\dot{n}_3 \\
\end{matrix}\right]  
=
\left(\Id + \tfrac{\check n \otimes \check n}{n_1(1+n_1)} \right) \dot{\check n}
= 
\left[
\begin{matrix} 
\dot{Q}_{12} \\
\dot{Q}_{13} \\
\end{matrix}\right]  \,,
\label{eq:dot:n:def}
\end{align}
where $n_1 = \sqrt{1- n_2^2-n_3^2}$. Since the vector $\check n$ is small (see~\eqref{eq:speed:bound} below), and the matrix on the left side is an $\OO(|\check n|^2)$ 
perturbation of the identity matrix, we obtain from \eqref{eq:dot:n:def} a definition of $\dot n$, as desired.

Next, we determine the dependence of $h_W^{\mu,0}$ and $G_W^{0}$. Inspecting \eqref{eq:GW:0:d}--\eqref{eq:GW:0:e} and \eqref{eq:FW:0:d}--\eqref{eq:FW:0:e} and inserting them into \eqref{eq:hj:def:1}, we obtain the   dependence
$$
\tfrac{1}{\beta_\tau}  h_W^{\mu,0}  = e^{-\frac s2} \mathcal{R} _{h,\mu}\left(\kappa, \phi  \, \big| \,  e^{-s} \dot Q, e^{-2s} \dot \phi \right) - \tfrac{1}{\beta_{\tau}} h_W^{\gamma,0} ({\mathcal H}^0)^{-1}_{\mu i}  \phi_{\zeta \gamma} \p_{1i} A_\zeta^0 \, .
$$
Note that although $h_W^{,0}$ appears on both sides of the above, in view of \eqref{eq:A:higher:order} the dependence on the right side is paired with a factor less than $e^{-s} \leq \eps$, and the functions $\phi_{\zeta\gamma}$ are themselves expected to be $\leq \eps$ for all $s\geq - \log \eps$ (cf.~\eqref{eq:speed:bound} below). This allows us to solve for $h_W^{\mu,0}$ and schematically write 
\begin{align}
\tfrac{1}{\beta_\tau}  h_W^{\mu,0}  = e^{-\frac s2} \mathcal{R} _{h,\mu}\left(\kappa, \phi  \, \big| \,  e^{-s} \dot Q, e^{-2s} \dot \phi \right) \,.
\label{eq:hj:def:2}
\end{align}
Returning to \eqref{eq:GW:def:1},
inspecting \eqref{eq:GW:0:d}--\eqref{eq:GW:0:e} and \eqref{eq:FW:0:d}--\eqref{eq:FW:0:e}, and using \eqref{eq:hj:def:2} we also obtain the dependence
\begin{align}
\tfrac{1}{\beta_\tau} G_W^0 =  e^{-\frac s2} \mathcal{R} _{h,\mu}\left(\kappa, \phi  \, \big| \,  e^{-s} \dot Q, e^{-2s} \dot \phi \right) \label{eq:GW:def:2}  \, .
\end{align}
 
 Next, we determine the dependence of $\dot \xi_j$. From \eqref{eq:GW:def:1}--\eqref{eq:hj:def:1}, \eqref{eq:hj:0:a}, \eqref{eq:GW:0:a},    and the fact that $R R^T = \Id$ we deduce that 
\begin{align}
 \dot{\xi}_j 
 = R_{ji}  ( R^T \dot \xi)_i 
 = R_{j1} \left( \tfrac{1}{2\beta_1} (\kappa + \beta_2 Z^0) - \tfrac{1}{2\beta_1\beta_\tau} e^{-\frac s2} G_W^0 \right) + R_{j\mu} \left(A_\mu^0 - \tfrac{1}{2\beta_1\beta_\tau} e^{\frac s2} h_W^{\mu,0} \right) 
\label{eq:dot:xi:def:1}
\end{align}
for $j \in \{1,2,3\}$. Using \eqref{eq:hj:def:2} and \eqref{eq:GW:def:2}, we may then schematically write
\begin{align}
 \dot{\xi}_j  =\mathcal{R} _{\xi,j} \left(\kappa, \phi  \, \big| \,  e^{-s} \dot Q,  e^{-2s} \dot \phi \right) \, .
\label{eq:dot:xi:def:2} 
\end{align}

Lastly, note that $\dot \phi_{\nu\gamma}$ is determined  in terms of $e^{\frac s2} \p_{\nu\gamma} G_W^0$ (which we rewrite in terms of $G_W^0$, $h_W^{\mu,0}$ and $\p_{\nu\gamma}F_W^0$ via \eqref{eq:vanilla:candle}) through the first term on the right side of~\eqref{eq:GW:0:f}
\begin{align}
\dot{\phi}_{\gamma \nu} &=  -  \tfrac{1}{ \beta_\tau}  e^{\frac s2}  \left(G_W^0 \p_{1\nu\gamma} W^0 + h_W^{\mu,0} \p_{\mu\nu\gamma} W^0 - \p_{\nu\gamma} F_W^0 \right)+  \beta_2  e^s \p_{\gamma \nu} Z^0 -   2\beta_1 (\dot{Q}_{\zeta  \gamma} \phi_{\zeta \nu} +  \dot{Q}_{\zeta \nu} \phi_{\zeta \gamma}) \notag\\
&\quad + \left(\tfrac{1}{\beta_\tau} e^{-\frac s2} G_W^0 -  \kappa - \beta_2 Z^0 \right)  \Ncal_{1,\gamma\nu}^0   + \Jcal_{,\gamma\nu}^0 \tfrac{1}{\beta_\tau} e^{-\frac s2} G_W^0  \,,
\label{eq:dot:phi:def:1}
\end{align}
and \eqref{eq:GW:def:1} is used to determine $G_W^0$. In light of \eqref{eq:FW:0:f}, \eqref{eq:GW:def:2} and of \eqref{eq:dot:phi:def:1}, we may schematically write
\begin{align*}
\dot{\phi}_{\gamma \nu} =\mathcal{R} _{\phi,\gamma\nu} \left(\kappa, \phi  \, \big| \,  e^{-s}\dot{Q}, e^{-s} \dot \phi  \right)   - \dot Q_{\zeta\gamma} \phi_{\zeta\nu} -  \dot Q_{\zeta\nu} \phi_{\zeta\gamma}\,,
\end{align*}
which may be then combined with \eqref{eq:dot:Q:2} and \eqref{eq:hj:def:2}
to yield
\begin{align}
\dot{\phi}_{\gamma \nu} = \mathcal{R} _{\phi,\gamma\nu} \left(\kappa, \phi  \, \big| \,  e^{-s}\dot{Q}, e^{-s} \dot \phi  \right)  \,,
\label{eq:dot:phi:def:2}
\end{align}
thus spelling out the dependences of $\dot \phi$ on the other dynamic variables.

The  equations \eqref{eq:dot:kappa:2}, \eqref{eq:dot:tau:2}, \eqref{eq:dot:Q:2}, \eqref{eq:dot:xi:def:2}, and \eqref{eq:dot:phi:def:2} only {\em implicitly} define  $\dot \kappa, \dot \tau, \dot Q_{1\nu}, \dot \xi_j$, and $\dot \phi_{\gamma\nu}$. We may however spell out this implicit dependence and arrive at an autonomous system of ODEs for all $10$ of our modulation parameters, cf.~\eqref{eq:dot:phi:dot:n}--\eqref{eq:dot:kappa:dot:tau} below.

By combining \eqref{eq:dot:Q:2} and \eqref{eq:hj:def:2} with \eqref{eq:dot:n:def}, and recalling \eqref{eq:dot:phi:def:2} we obtain that 
\[
\dot{\phi}_{\gamma \nu} = \mathcal{R} _{\phi,\gamma\nu} \left(\kappa, \phi, \check n  \, \big| \,  e^{-s}\dot{\check n}, e^{-s} \dot \phi  \right) 
\quad \mbox{and} \quad 
\dot{n}_\nu = \mathcal{R} _{n,\nu} \left(\kappa, \phi, \check n  \, \big| \,  e^{-s}\dot{\check n}, e^{-s} \dot \phi  \right) \,.
\]
Therefore, since $e^{-s} \leq \eps$, and the functions $\PP_{\phi,\gamma\nu}$ and $\PP_{n,\nu}$ are linear in $e^{-s} \dot{\check n}$ and $e^{-s} \dot \phi$, then as long as $\kappa$, $\phi$, and $\check n$ remain bounded,  and $\eps$ is taken to be sufficiently small (in particular, for short time after $t= -\log \eps$), we may analytically solve for $\dot \phi$ and $\dot n$ as rational functions (with bounded denominators) of $\kappa, \phi$, and $\check n$, with coefficients which only depend on the derivatives of $Z, A, K$ at $y=0$. We write this schematically as 
\begin{align}
\dot{\phi}_{\gamma \nu} = \EE_{\phi,\gamma\nu} \left(\kappa, \phi, \check n  \right) 
\quad \mbox{and} \quad 
\dot{n}_\nu = \EE_{n,\nu} \left(\kappa, \phi, \check n    \right) \,.
\label{eq:dot:phi:dot:n}
\end{align}
Here the $\EE_{\phi,\gamma\nu}(\kappa,\phi,\check n)$ and $\EE_{n,\nu}(\kappa,\phi,\check n)$ are suitable smooth functions of their arguments, as described above.
With \eqref{eq:dot:phi:dot:n} in hand, we return to \eqref{eq:dot:kappa:2} and \eqref{eq:dot:tau:2}, which are to be combined with \eqref{eq:hj:def:2}, and with \eqref{eq:dot:xi:def:2} to obtain that
\begin{align}
\dot{\kappa} = \EE_{\kappa} \left(\kappa, \phi, \check n  \right) 
\,, \qquad
\dot{\tau} = \EE_{\tau} \left(\kappa, \phi, \check n    \right) \,
\quad \mbox{and} \quad 
\dot{\xi}_j = \EE_{\xi,j} \left(\kappa, \phi, \check n    \right)\,.
\label{eq:dot:kappa:dot:tau}
\end{align}
for suitable smooth functions $\EE_{\kappa}, \EE_{\tau},$ and $\EE_{\xi,j}$ of $(\kappa,\phi,\check n)$,  with coefficients  which depend on the derivatives of $Z, A$, and $K$ at $y=0$. 

\begin{remark}[Local solvability]
\label{rem:local:constraints}
The system of ten nonlinear ODEs described in \eqref{eq:dot:phi:dot:n} and \eqref{eq:dot:kappa:dot:tau} are used to determine the time evolutions of our $10$ dynamic modulation variables. The local in time solvability of this system is ensured by the fact that $\EE_{\phi,\gamma\nu}, \EE_{n,\nu}, \EE_{\kappa},  \EE_{\tau}, \EE_{\xi,j}$ are rational functions of $\kappa, \phi, n_2$, and $n_3$, with coefficients that only depend on $\p^\gamma Z^0, \p^\gamma A^0$ and $\p^\gamma K^0$ with $\abs{\gamma} \leq 3$, and moreover that these functions are smooth in the neighborhood of the initial values given by \eqref{eq:modulation:IC}; hence, unique $C^1$ solutions exist for a sufficiently small time.
We emphasize that these functions are explicit.
\end{remark}

\subsection{Closure of bootstrap estimates for the dynamic variables}
\label{sec:dynamic:closure}
Once one traces back the identities in Sections~\ref{sec:implicit:crap} and Appendix~\ref{sec:explicit:crap} we may close  the bootstrap assumptions for the modulation parameters, \eqref{mod-boot}.

The starting point is to obtain bounds for $G_W^0$ and $h_W^{\mu,0}$, by appealing to \eqref{eq:GW:def:1}--\eqref{eq:hj:def:1}. The matrix ${\mathcal H}^0$ defined in \eqref{eq:d1W:Hessian} can be rewritten as
\begin{align*}
{\mathcal H}^0(s) = (\p_1\nabla^2 W)^0(s)  =  (\p_1\nabla^2 \bar W)^0 + (\p_1\nabla^2 \tilde W)^0(s) = {\rm diag}(6,2,2) + (\p_1\nabla^2 \tilde W)^0(s).
\end{align*}
From the bootstrap assumption \eqref{eq:bootstrap:Wtilde3:at:0}  we have that 
$\abs{(\p_1\nabla^2 \tilde W)^0(s)} \leq \eps^{\frac 14}$ for all $s\geq - \log \eps$, and thus 
\begin{align}
\abs{({\mathcal H}^0)^{-1}(s)} \leq 1
\label{eq:inverse:Hessian}
\end{align}
for all $s\geq -\log \eps$. Next, we estimate $\p_1 \nabla F_W^0$. Using \eqref{eq:FW:0:d}, \eqref{eq:FW:0:e}, the bootstrap assumptions  \eqref{eq:speed:bound}--\eqref{eq:beta:tau}, the bounds \eqref{eq:Z_bootstrap}--\eqref{eq:K:higher:order}, and the fact that $\sabs{\Tcal^{\zeta,0}_{\mu,\mu\nu} }\leq \abs{\phi}^2$, after a computation we arrive at
\begin{align}
\abs{\p_1 \nabla F_W^0 }
&\les M \eps^{\frac 12} e^{-s} + M^2 e^{- \frac{3}{2} (1- \frac{4}{2m-5})s} + \sabs{h_W^{\cdot,0}} M^3 \eps e^{- \frac{3}{2} (1- \frac{4}{2m-5})s} +  M e^{-(1-{\frac{5}{2m-7}} )s}  \notag\\
&\les \eps^2 \sabs{h_W^{\cdot,0}}  +  M e^{-(1-{\frac{5}{2m-7}} )s}  \, .
\label{eq:F_WZ_deriv_est:replace:1}
\end{align}
Moreover,  from \eqref{eq:GW:0:d}, \eqref{eq:GW:0:e}, \eqref{eq:speed:bound}, \eqref{eq:acceleration:bound},  the first line in \eqref{eq:Z_bootstrap}, and the previously established bound \eqref{eq:F_WZ_deriv_est:replace:1} we establish that
\begin{align}
\abs{\p_{1}\nabla G_W^{0}} + \abs{\p_{1} \nabla F_W^{0}} 
&\les e^{\frac s2} \abs{\p_1 \nabla Z^0} +  M^4 \eps^{\frac 32} e^{-\frac{3s}{2}}  + \eps^2  \sabs{h_W^{\cdot,0}}+ M e^{-(1-{\frac{5}{2m-7}} )s}   \notag\\
& \les   \eps^2  \sabs{h_W^{\cdot,0}} +  M e^{-(1-{\frac{5}{2m-7}} )s} \,.
\label{eq:rainy:Halloween}
\end{align}
The bounds  \eqref{eq:inverse:Hessian} and \eqref{eq:rainy:Halloween}, are then inserted  into \eqref{eq:GW:def:1}--\eqref{eq:hj:def:1}. After absorbing the $\eps^2  \sabs{h_W^{\cdot,0}}$ term into the left side, we obtain  to estimate  
\begin{align}
 \abs{G_W^{0}(s)} + \abs{h_W^{\mu,0}(s)} \les M e^{-(1-{\frac{5}{2m-7}} )s}  \,.
\label{eq:GW:hW:0} 
\end{align}
The bound \eqref{eq:GW:hW:0} plays a crucial role in the following subsections. We note that for $m\geq 18$ we have $1-{\frac{5}{2m-7}} > \frac 45$, and hence so the bound \eqref{eq:GW:hW:0}  implies that $ \sabs{G_W^{0}(s)} + \sabs{h_W^{\mu,0}(s)} \les M e^{-\frac{4s}{5}}$. 

\subsubsection{The $\dot \tau$ estimate}
From \eqref{eq:dot:tau:1}, the definition of $\p_1 G_W^0$ in \eqref{eq:GW:0:b},   the definition of $\p_1 F_W^0$ in \eqref{eq:FW:0:b} , the bootstrap estimates \eqref{eq:speed:bound}--\eqref{eq:beta:tau}, \eqref{eq:Z_bootstrap}--\eqref{eq:S_bootstrap}, and the previously established bound \eqref{eq:GW:hW:0}, we  obtain that
\begin{align}
\abs{\dot \tau} &\les \abs{\p_1 G_W^0} + \abs{\p_1 F_W^0} \notag\\
&\les e^{\frac s2} \abs{\p_1 Z^0}  +  e^{-\frac s2} \abs{\check \nabla A^0} + M \abs{\check \nabla \p_1 A^0} + M^2 \eps^{\frac 12} e^{-\frac s2}\abs{\p_1 A^0} + M^2 \eps  e^{-2s} \abs{A^0} + M^3 \eps e^{-s} \notag\\
&\qquad +  M e^s |\p_{11}K^0| + M e^{\frac s2} |\p_1S^0| \notag\\ 
&\les M^{\frac 12} e^{-s} + M \eps^{\frac 12} e^{-s} + M e^{- \frac 32 (1 - \frac{2}{2m-5})s} + M^3 \eps e^{s} +  M \eps^{\frac 18} e^{-s}   \notag\\
&\leq \tfrac{M}{4} e^{-s}\,, \label{tau-final}
\end{align}
where we have used a power of $M$ to absorb the implicit constant in the first inequality above. This improves the bootstrap bound for $\dot \tau$ in \eqref{eq:acceleration:bound} by a factor of $4$. Integrating in time from $-\eps$ to $T_*$, where $\abs{T_*} \leq \eps$, we also improve the $\tau$ bound in \eqref{eq:speed:bound} by a factor of $2$, thereby closing the $\tau$ bootstrap.

\subsubsection{The $\dot \kappa$ estimate} 
From \eqref{eq:dot:kappa:1}--\eqref{eq:beta:tau}, the bound \eqref{eq:GW:hW:0}, the definition of $F_W^0$ in \eqref{eq:FW:0:a},   the estimates \eqref{eq:Z_bootstrap}--\eqref{eq:S_bootstrap}, and the fact that $\frac{5}{2m-7} < \frac 15$, we deduce that 
\begin{align*}
\abs{\dot \kappa} 
&\les  e^{\frac s2} \abs{G_W^0} + e^{\frac s2} \abs{F_W^0}  
\notag\\
&\les M  e^{-\frac s2 + \frac{5s}{2m-7}}  + (\kappa_0 + M\eps)  M \eps^{\frac 12} e^{-\frac s2} + M^3 \eps^{\frac 32} e^{-\frac s2} + M^4 \eps^2 e^{-\frac s2}  + e^{-\frac s2} (\kappa_0^2 + M^2 \eps^2) M^2 \eps \notag\\
&\qquad +  (\kappa_0+M\eps) \eps^{\frac 14} e^{-\frac s2} 
\notag\\
&\leq \tfrac{1}{2}  e^{-\frac{3s}{10}} \,.
\end{align*}
Here we have used a small ($m$-dependent) power of $\eps$ to absorb the implicit constant in the second  estimate above, thereby improving the $\dot \kappa$ bootstrap bound in \eqref{eq:acceleration:bound} by a factor of $2$. Integrating in time, we furthermore deduce that 
\begin{equation}\label{kappa-kappa0}
\abs{\kappa(t) - \kappa_0} \leq  \eps^{\frac{13}{10}} 
\end{equation} 
since $\abs{T_*} \leq \eps$. Upon taking $\eps$ to be sufficiently small in terms of $\kappa_0$, we improve the $\kappa$ bound in \eqref{eq:speed:bound}.

\subsubsection{The $\dot \xi$ estimate}
In order to bound the $\dot \xi$ vector, we appeal to \eqref{eq:dot:xi:def:1}, to \eqref{eq:GW:hW:0}, to the $\abs{\gamma}=0$ cases in \eqref{eq:Z_bootstrap} and \eqref{eq:A_bootstrap},  to the bound $\abs{R-\Id} \leq \eps$, and to the $\abs{\check n}$ estimate in \eqref{eq:speed:bound}, to deduce that 
\begin{align}
\sabs{\dot{\xi}_j} 
\les \kappa_0 + \abs{Z^0}  +  e^{-\frac s2} \abs{G_W^0} +  \abs{A_\mu^0} + e^{\frac s2} \sabs{h_W^{\mu,0}} 
\les \kappa_0 + M \eps + M e^{-\frac s2 + \frac{5s}{2m-7}} \les \kappa_0 \,, \label{xi-final}
\end{align}
upon taking $\eps$  sufficiently small in terms of $M$ and $\kappa_0$. The bootstrap estimate for $\dot \xi$ in \eqref{eq:acceleration:bound} is then improved by taking $M$ sufficiently large, in terms of $\kappa$, while the bound on $\xi$ in \eqref{eq:speed:bound} follows by integration in time.

\subsubsection{The $\dot{\phi}$ estimate}
Using \eqref{eq:dot:phi:def:1}, the fact that  $\abs{\Ncal_{1,\mu\nu}^0} + \abs{\Jcal_{,\mu\nu}^0} \les |\phi|^2$, the bootstrap assumptions \eqref{eq:speed:bound}, \eqref{eq:acceleration:bound},  \eqref{eq:bootstrap:Wtilde3:at:0}, the bounds \eqref{eq:dot:Q}, and the previously established estimate \eqref{eq:GW:hW:0}, we obtain  
\begin{align}
\sabs{\dot{\phi}_{\gamma \nu}}  
&\les   e^{\frac s2}  \left( M  e^{-s(1 - \frac{5}{2m-7})  } + \abs{\p_{\nu\gamma} F_W^0} \right)+ e^s \abs{\p_{\gamma \nu} Z^0}  \notag\\
&\qquad  + M^4 \eps^{\frac 32} + \left(  M e^{-\frac{3s}{2}+\frac{5s}{2m-7}}  + \kappa_0 + \abs{Z^0}\right)  M^4 \eps^2  + M^5 \eps^2  e^{-\frac{3s}{2}+\frac{5s}{2m-7}}   \,. 
\label{eq:going:slightly:mad}
\end{align}
Using the definition of $\check \nabla^2 F_W^0$ in \eqref{eq:FW:0:f}, appealing to the bootstrap assumptions (and their consequences) from Section~\ref{sec:bootstrap}, the previously established estimate \eqref{eq:GW:hW:0}, and the fact that $\sabs{\Tcal^{\zeta,0}_{\mu, \gamma\nu}} +\sabs{\Ncal_{1,\mu\nu}^0} + \sabs{\Jcal_{,\mu\nu}^0} + \sabs{\Ncal^0_{\zeta,\mu\nu\gamma}} \les |\phi|^2$, after a lengthy computation one may show that 
\[
 \abs{\p_{\nu\gamma} F_W^0} \les e^{-\frac s2} \,,
\]
which shows that the term $e^{\frac s2} \abs{\p_{\nu\gamma} F_W^0}$ in \eqref{eq:going:slightly:mad} is subdominant when compared to $e^s \abs{\p_{\gamma \nu} Z^0} \les M$ present in \eqref{eq:going:slightly:mad}. In establishing the above estimate it was crucial that $e^s\abs{\p_{1\gamma\nu}K^0} \les e^{-\frac s2} $, which from \eqref{eq:K:higher:order} since $m\geq 18$. 
Combining the above two estimates with the  $Z$ bounds in \eqref{eq:Z_bootstrap}, 
we derive
\begin{align}
 \sabs{\dot{\phi}_{\gamma \nu}}  \les   e^{\frac s2}  \left(M  e^{-\frac{4s}{5}} + e^{-\frac s2} \right)+ M + M^4 \eps^{\frac 32} + \left(M e^{-s} + \kappa_0 + \eps M \right)  M^4 \eps^2  + M^5 \eps^2  e^{-s}  \les M \,.  \label{phi-final}
\end{align}
Taking $M$ sufficiently large to absorb the implicit constant, we deduce $|\dot \phi| \leq \frac 14 M^2$, which improves the $\dot \phi$ bootstrap in \eqref{eq:acceleration:bound} by a factor of $4$. Integrating in time on $[-\eps,T_*)$, an interval of length $\leq 2 \eps$,  and using that $\abs{\phi(-\log \eps)}\leq \eps$  we improve the $\phi$ bootstrap in \eqref{eq:speed:bound} by a factor of $2$.

\subsubsection{The $\dot{n}$ estimate}
First we obtain estimates on $|\dot Q_{1\nu}|$, by appealing to the identity \eqref{eq:dot:Q:1}. 
Using the bootstrap assumptions \eqref{eq:speed:bound}, \eqref{eq:acceleration:bound}, \eqref{eq:Z_bootstrap}--\eqref{eq:S_bootstrap}, the estimates \eqref{eq:dot:Q} and \eqref{eq:GW:hW:0}, and the fact that $\sabs{\Tcal^{\zeta,0}_{\mu, \mu\nu}} \les |\phi|^2$, we obtain 
\begin{align}
\sabs{\dot Q_{1 \nu} }
&\les  M^2 \eps^{\frac 12} e^{-\frac s2}  \abs{\p_\nu A_\mu^0} +   M^4 \eps^{\frac 32} e^{-s}  \abs{A^0} + e^{\frac s2} \abs{\check \nabla Z^0} 
+ M^2   e^{-s} \abs{A^0}  
 \notag\\
& \quad +   \left(M \abs{\check \nabla^2A^0} +\abs{\check \nabla Z^0} \abs{\check \nabla A^0} \right)
 +  M^2 \eps e^{-\frac s2}  \abs{Z^0}  \abs{\check \nabla Z^0}   + M^5 \eps^2 e^{-s}  \abs{A^0} \notag\\
& \quad  +  e^{-\frac s2}  \left( ( \abs{\check \nabla A^0} + M^3 \eps e^{-\frac s2} ) \abs{A^0} \right) M^2 \eps  +  M^3 \eps e^{-s} \abs{\check \nabla A^0}
+ M^2 \eps  \left(M e^{-\frac s2} + \abs{A^0} \right) 
\notag\\
&\quad +   (\kappa_0 + M \eps) \left( (\kappa_0 + M\eps) (e^s \abs{\p_1 \check \nabla K^0} + M^2 \eps e^{-\frac s2} \abs{\check\nabla K^0}) - 2 \abs{\check \nabla Z^0} e^s \abs{\p_1 K^0}\right) 
\notag\\
&\les  M \eps^{\frac 12} 
\label{eq:dot:Q:bootstrap:close}
\,,
\end{align}
upon taking $\eps$ sufficiently small, in terms of $M$.
Moreover, using the bootstrap assumption $\abs{\check n} \leq M \eps^{\frac 32}$, we deduce that the matrix on the left side of \eqref{eq:dot:n:def} is within $\eps$ of the identity matrix, and thus so is its inverse. We deduce from \eqref{eq:dot:n:def} and \eqref{eq:dot:Q:bootstrap:close} that 
\begin{align} 
\abs{\dot{\check n}} \leq \tfrac{M^2 \eps^{\frac 12}}{4}.   \label{n-final}
\end{align} 
upon taking $M$ to be sufficiently large to absorb the implicit constant. The closure of the $\check n$ bootstrap is then achieved by integrating in time on $[-\eps, T_*)$.

\section{Conclusion of the proof: Theorems~\ref{thm:main:S-S} and~\ref{thm:open:set:IC}}\label{sec:conclusion-of-proof}

We first note that the system \eqref{euler-ss} for $(W,Z,A,K)$, with initial data $(W_0, Z_0, Z_0,K_0)$ chosen to satisfy the conditions of the theorem, is locally well-posed. To see this, we note that the transformations from \eqref{eq:Euler2} to \eqref{euler-ss} are smooth for sufficiently short time, and that \eqref{eq:Euler2} is locally well-posed in the Sobolev space $H^k$, for $k\geq 3$. Here we have implicitly used that the system of ten nonlinear ODEs  
\eqref{eq:dot:phi:dot:n} and \eqref{eq:dot:kappa:dot:tau} which specify the modulation functions have local-in-time existence and uniqueness as discussed in 
Remark \ref{rem:local:constraints}. Moreover, solutions to \eqref{eq:Euler2} satisfy the following continuation principle (see, for example,  \cite{Ma1984}): Suppose $(u,\sigma,\scal)\in C([-\eps,T), H^k)$ is a solution to $\eqref{eq:Euler2}$ satisfying the uniform bound 
$\norm{u( \cdot , t)}_{C^1}+\norm{\sigma( \cdot , t)}_{C^1}+\norm{\scal( \cdot , t)}_{C^1} \le K < \infty $, 
then if in addition $\sigma$ is uniformly bounded from below on the interval $[-\eps,T)$, there there exists $T_1>T$ such that $(u,\sigma,\scal)$ extends  to 
a unique solution of \eqref{eq:Euler2} on $[0,T_1)$. Consequently, the solution $(W,Z,A,K)$ in self-similar variables may be continued so long as   $(W,Z,A,K)$ remain uniformly bounded in $H^k$, the modulation functions remain bounded, and the density remains bounded from below.

In Sections \ref{sec:Lagrangian}--\ref{sec:constraints}, we close the bootstrap  assumptions on $W$, $Z$, $A$, $K$ and  on the modulation functions.
By
 Proposition \ref{prop:sound},  the density remains uniformly strictly positive and  bounded.
Thus, as a consequence of the continuation principle stated above, we obtain a global in self-similar 
time solution  $(W,Z,A, K) \in C([-\log\eps,+ \infty ); H^m) \cap C^1([-\log\eps,+ \infty ); H^{m-1})$ to  \eqref{euler-ss} for $m \ge 18$. This solution satisfies the bounds stated in Sections~\ref{subsec:support}--\ref{sec:spiny-anteater0}. The asymptotic stability of $W(y,s)$ follows from:
\begin{theorem}[Convergence to stationary solution]\label{thm-2week-delay}
There exists a $10$-dimensional symmetric $3$-tensor $\mathcal{A}$ such that, with $\bar W_ \mathcal{A} $ defined in Appendix~\ref{sec:delay}, we have that the solution $W(\cdot,s)$ of \eqref{eq:euler:ss:a} satisfies
$$
\lim_{s\to \infty } W(y,s)  = \bar W_ \mathcal{A} (y) 
$$
for any fixed $y \in \RR^3$.
\end{theorem} 
\noindent 
We note that
the proof of Theorem \ref{thm-2week-delay} is the same as  the proof of Theorem 13.4 in \cite{BuShVi2019b} once we include the contributions
of the entropy function $K$, which can be estimated using \eqref{eq:S_bootstrap}.
The limiting profile $\bar W_ \mathcal{A} $ satisfies the conditions stated in Theorem 
\ref{thm:main:S-S} due to Proposition \ref{prop-stationary-burgers}.

The remaining  conclusions of Theorem \ref{thm:main:S-S} follow from the  statements given in Sections~\ref{sec:spiny-anteater1} and \ref{sec:spiny-anteater2} (for the time and location of the singularity, and the regularity of the solution at this time), Proposition~\ref{cor:L2} (for the vanishing of derivatives of $A$, $Z$, and $K$ as $s\to \infty$),  Proposition \ref{prop:vorticity} (for the vorticity upper bounds), and Theorem \ref{thm:vorticity:creation} (for the vorticity creation estimates).

The proof of Theorem~\ref{thm:open:set:IC} is the same as the proof of Theorem~3.2 in~\cite{BuShVi2019b}. The addition of entropy does not necessitate modifications to that proof as the assumptions on the initial entropy in Theorem~\ref{thm:main:S-S} (see~\eqref{eq:S_bootstrap:IC} and \eqref{eq:data:Hk}) are stable with respect to small perturbations.

\appendix

\section{Appendix}\label{sec:toolshed}

\subsection{A family of self-similar solutions to the 3D Burgers equation}
\label{sec:delay}

\begin{proposition}[Stationary solutions for self-similar 3D Burgers]\label{prop-stationary-burgers}
Let $ \mathcal{A} $ be  a symmetric $3$-tensor   such that $ \mathcal{A}_{1jk} = \mathcal{M} _{jk}$ with $ \mathcal{M} $  a positive definite
symmetric matrix.   Then, there exists a $C^ \infty $ solution $\bar W_ \mathcal{A} $ to 
\begin{align}
- \tfrac 12 \bar W_{\mathcal A} + \left( \tfrac{3y_1}{2} + \bar W_{\mathcal A} \right) \p_1 \bar W_{\mathcal A} + \tfrac{\check y}{2} \cdot\check \nabla \bar W_{\mathcal A} = 0 \,,
\label{eq:EQUATION}
\end{align}
which has the following properties:
\begin{itemize} 
\item  $\bar W_{\mathcal A}(0) = 0$, $\p_1 \bar W_{\mathcal A}(0) = -1$, $\p_2 \bar W_{\mathcal A}(0) = 0$, 
\item $\p^\alpha \bar W_{\mathcal A}(0) = 0$ for  $|\alpha|$  even, 
\item  $\p^\alpha \bar W_{\mathcal A}(0) = \mathcal{A}_ \alpha  $  for $|\alpha| = 3$.   
 \end{itemize} 
\end{proposition} 
See Appendix A.1 in \cite{BuShVi2019b} for the proof of Proposition \ref{prop-stationary-burgers}.

\subsection{Interpolation}
\label{sec:interpolation}
The following is taken from \cite[Appendix A.3]{BuShVi2019b}.  We include the inequalities here for convenience to the reader. 

\begin{lemma}[Gagliardo-Nirenberg-Sobolev]\label{lem:GN} 
Let $u: \mathbb{R}  ^d \to \mathbb{R}  $.  Fix $1\le q , r \le \infty $ and $j,m \in \mathbb{N}  $,   and $ \tfrac{j}{m}  \le 
\alpha \le 1$.  Then,  if
$$
\tfrac{1}{p}  = \tfrac{j}{d} + \alpha \left( \tfrac{1}{r} - \tfrac{m}{d} \right)  + \tfrac{1- \alpha }{q} \,,
$$
then
\begin{align} 
\| D^ju\|_{L^p} \le C \| D^m u\|_{L^r}^ \alpha \|u\|^{1 - \alpha }_{L^q} \,. 
\label{eq:special0}
\end{align} 
\end{lemma} 
We shall make use of \eqref{eq:special0} for the case that $p= \tfrac{2m}{j} $, $r=2$, $q= \infty $, which yields
\begin{align} 
\norm{D^j \varphi}_{L^{\frac{2m}{j}}} \les  \norm{\varphi}_{\dot{H}^m}^{\frac{j}{m}}\norm{\varphi}_{L^\infty}^{1 - \frac{j}{m}} \,,
\label{eq:special1}
\end{align} 
whenever $\varphi \in H^m(\RR^3)$ has compact support. The above estimate and the Leibniz rule  classically imply the Moser inequality
\begin{align}
\norm{ \phi \,  \varphi}_{\dot{H}^m}  \les  \norm{\phi}_{L^\infty} \norm{\varphi}_{\dot H^{m}} + 
 \norm{\phi}_{\dot H^{m}} \norm{\varphi}_{L^\infty}\,.
 \label{eq:Moser:inequality}
\end{align}
for all $\phi, \varphi \in H^m( \mathbb{R}^3  ) $ with compact support.
At various stages in the proof we also appeal to the following special case  of \eqref{eq:special0}
\begin{align} 
\snorm{\varphi}_{ \dot H^{m-2}} &\les \snorm{ \varphi}_{\dot H^{m-1} }^\frac{2m-7}{2m-5}  \snorm{  \varphi}_{L^\infty }^\frac{2}{2m-5}  \,, \label{eq:special3} 
\end{align} 
for $\varphi \in H^{m-1}( \mathbb{R}^3  ) $ with compact support.
Lastly, in Section~\ref{sec:energy} we make use of:
\begin{lemma}
\label{lem:tailored:interpolation}
Let $m\geq 4$ and $0 \le l \le m-3$.   Then for $\acal + \bcal = 1 - \frac{1}{2m-4} \in (0,1)$,  and  $q=\frac{6(2m-3)}{2m-1}$,
\begin{align} 
\snorm{D^{2+l} \phi \, D^{m-1-l} \varphi}_{L^2}  \les \snorm{D^m \phi}_{L^2}^\acal \snorm{D^m \varphi}_{L^2}^\bcal 
\snorm{D^2 \phi}_{L^q}^{1-\acal}  \snorm{D^2 \varphi}_{L^q}^{1-\bcal}  \,. \label{cor1}
\end{align} 
\end{lemma}
See \cite{BuShVi2019b} for the proof.

\subsection{The functions $G_W, F_W$ and their derivatives at $y=0$}
\label{sec:explicit:crap}

Using \eqref{def_f}, 
the definition of $G_W$ in \eqref{eq:gW}, and the constraints in \eqref{eq:constraints}, we deduce that\footnote{Here we have used the identities: $\Tcal_{\mu,\nu}^{\gamma,0} = 0$, $\Ncal_{\mu,\nu \gamma}^0 = 0$,  and $\Tcal_{1,\nu \gamma}^{\zeta,0} = 0$, $ \Ncal_{1,\nu}^0 =  0$, and $ \Ncal_{\mu,\nu}^0 =  -\phi_{\mu\nu}$, $\Ncal_{\zeta,\mu\nu}^0 = 0$.}
\begin{subequations}
\label{Birds_are_rad}
\begin{align}
\tfrac{1}{\beta_\tau}\p_1 G_W^0 &=  \beta_2 e^{\frac s2} \p_1 Z^0 \label{eq:GW:0:b}\\
\tfrac{1}{\beta_\tau}\p_\nu G_W^0 &=    
\beta_2  e^{\frac s2} \p_\nu Z^0  + 2  \beta_1  ( \dot Q_{1 \nu} + A_\gamma^0 \phi_{\gamma\nu})   -  e^{\frac s2} \tfrac{1}{\beta_\tau} h_W^{\gamma,0}   \phi_{\gamma \nu}  \label{eq:GW:0:c}\\
\tfrac{1}{\beta_\tau}\p_{11}G_W^0 &=  \beta_2  e^{\frac s2} \p_{11} Z^0\label{eq:GW:0:d} \\
\tfrac{1}{\beta_\tau}\p_{1 \nu} G_W^0 &=    \beta_2 e^{\frac s2}\p_{1\nu} Z^0   - 2  \beta_1 e^{-\frac{3s}{2}} \dot{Q}_{\gamma 1} \phi_{\gamma \nu}   \label{eq:GW:0:e} \\
\tfrac{1}{\beta_\tau}\p_{\gamma \nu} G_W^0 &=    e^{- \frac s2} \left(-  \dot{\phi}_{\gamma \nu} 
+ \beta_2 e^s \p_{\gamma \nu} Z^0  -2\beta_1(  \dot{Q}_{\zeta  \gamma} \phi_{\zeta \nu} +  \dot{Q}_{\zeta \nu} \phi_{\zeta \gamma} +  R_{j1} \dot{\xi_j} \Ncal_{1,\gamma\nu}^0 )  + e^{-\frac s2} \tfrac{G_W^0}{ \beta_\tau} \Jcal_{,\gamma\nu}^0   \right) 
\, .\label{eq:GW:0:f}
\end{align} 
\end{subequations}

Appealing to \eqref{def_f} and \eqref{eq:FW:def} which is equivalent to 
\begin{align*}
\tfrac{1}{\beta_\tau} {F}_{ W}&= - 2 \beta_3  \sound \Tcal^\nu_\mu \partial_{\mu} A_\nu
 +  2 \beta_1   e^{-\frac s2}   A_\nu \Tcal^\nu_i \dot{\Ncal}_i
 +  2 \beta_1  e^{-\frac s2} \dot Q_{ij} A_\nu \Tcal^\nu_j \Ncal_i \notag \\
 & \quad
 +\left(\tfrac{1}{\beta_\tau} h_W^\mu - \beta_3 e^{- \frac s2} \Ncal_\mu \left( \kappa + e^{-\frac s2} W - \tfrac{\beta_1+\beta_2}{\beta_3} Z\right)\right) A_\gamma \Tcal^\gamma_i \Ncal_{i,\mu}
 - 2 \beta_3  e^{-\frac s2} \sound \left( A_\nu \Tcal^\nu_{\mu,\mu} + U\cdot \Ncal \Ncal_{\mu,\mu} \right) \notag \\
 & \quad + \beta _4 \sound^2( \Jcal e^s \p_1 K  + \Ncal_\mu \p_\mu K) \,,
\end{align*}
we may derive the following explicit expressions\footnote{Here we have used the identities: $\Ncal_{\mu,\mu}^0 = - \phi_{22} - \phi_{33}$, $\Tcal^{\nu,0}_{\mu,\mu} = 0$, $\dot \Ncal_i^0 = 0$, $\dot \Ncal_{1,\nu}^0 =0$, $\dot \Ncal_{\mu,\nu}^0 = - \dot{\phi}_{\mu\nu}$, $\Tcal^{\gamma,0}_{1,\nu} = \phi_{\gamma\nu}$, $\Tcal^{\gamma,0}_{i,\nu} \Ncal_{i,\mu}^0 = 0$, $\Tcal_i^{\gamma,0} \Ncal_{i,\mu \nu}^0 = 0$, $\Ncal_{\mu,\mu\nu}^0 = 0$, $\dot{\Ncal}_{\zeta,\nu\gamma} = 0$,  and $\Jcal_{,\nu \gamma}^0 = \phi_{2\nu} \phi_{2\gamma} + \phi_{3\nu} \phi_{3\gamma}$.}   
\begin{subequations}
\label{Bears_are_not_rad}
\begin{align}
\tfrac{1}{\beta_\tau} F_W^0 &= -\beta_3 \left(\kappa - Z^0\right) \p_\mu A_\mu^0  + 2\beta_1 e^{-\frac s2} \dot Q_{1\mu} A_\mu^0 - \tfrac{1}{\beta_\tau} h_W^{\mu,0} A_\zeta^0 \phi_{\zeta \mu}  \notag\\
&\quad + \tfrac 12 \beta_3  e^{-\frac s2} (\kappa - Z^0) (\kappa + Z^0) (\phi_{22}+ \phi_{33})
+ \tfrac 14  \beta_4   (\kappa - Z^0)^2  e^s \p_1 K^0 
 \label{eq:FW:0:a}  \\
\tfrac{1}{\beta_\tau} \p_1 F_W^0 &= \beta_3 \left(e^{-\frac s2} + \p_1 Z^0\right) \p_\mu A_\mu^0 - \beta_3  \left(\kappa - Z^0\right) \p_{1\mu} A_\mu^0   + 2\beta_1  e^{-\frac s2} \dot Q_{1\mu} \p_1 A_\mu^0 \notag\\
&\quad -   \left(\tfrac{1}{\beta_\tau} h_W^{\mu,0}  \p_1 A_\zeta^0
+ 2 \beta_1  e^{-\frac s2} ( \p_1 A_\mu^0 + e^{-\frac{3s}{2}} \dot{Q}_{\mu 1})  A_\zeta^0 \right) \phi_{\zeta \mu}  \notag\\
&\quad  - \tfrac 12 \beta_3  e^{-s} 
\left( (1 + e^{\frac s2} \p_1 Z^0) (\kappa + Z^0)  + (\kappa - Z^0) (1 - e^{\frac s2}\p_1 Z^0)  \right) (\phi_{22}+ \phi_{33}) 
\notag\\
&\quad+ \tfrac 14 \beta_4(\kappa - Z^0)  \left( (\kappa - Z^0) e^s \p_{11}K^0 - 2 (e^{-\frac s2} + \p_1 Z^0)   e^s  \p_1 K^0  \right)
 \label{eq:FW:0:b} \\
\tfrac{1}{\beta_\tau} \p_\nu F_W^0&= - \beta_3 ( (\kappa - Z^0) \p_{\nu \mu} A_\mu^0 - \p_\nu Z^0 \p_\mu A_\mu^0 )
- 2 \beta_1  e^{-s} A_\mu^0 \dot{\phi}_{\mu \nu} + 2 \beta_1  e^{-\frac s2} \dot Q_{1\mu} \p_\nu A_\mu^0 
\notag\\
&\quad - 2 \beta_1 e^{-s} \dot{Q}_{\mu \zeta} A_\zeta^0   \phi_{\mu \nu}  - \beta_3 e^{-\frac s2}  Z^0   \p_\nu Z^0   (\phi_{22} + \phi_{33}) - \beta_3  e^{-s} \left(\kappa - Z^0 \right)  A_\zeta^0 \Tcal^{\zeta,0}_{\mu, \mu\nu}  
\notag\\
&\quad  - 2 \beta_1 e^{-\frac s2}  \left( (e^{-\frac s2} \dot{Q}_{\mu \nu} + \p_\nu A_\mu^0 - \tfrac 12 e^{-\frac s2} ( \kappa + Z^0) \phi_{\mu\nu}) A_\gamma^0 \right) \phi_{\gamma \mu} 
- \tfrac{1}{\beta_\tau} h_W^{\mu,0} \p_\nu A_\gamma^0  \phi_{\gamma \mu}
\notag\\
&\quad + \tfrac 14  \beta_4 (\kappa - Z^0) \left( (\kappa - Z^0) ( e^s \p_{1\nu}K^0   - e^{-\frac s2} \p_\mu K \phi_{\mu\nu}) - 2 \p_\nu Z^0 e^s \p_1 K \right)
\label{eq:FW:0:c}\\
\tfrac{1}{\beta_\tau} \p_{11} F_W^0 &= \beta_3  \left(e^{-\frac s2} + \p_1 Z^0\right) \p_\mu A_\mu^0 - \beta_3 \left(\kappa - Z^0\right) \p_{1\mu} A_\mu^0   + 2\beta_1  e^{-\frac s2} \dot Q_{1\mu} \p_{11} A_\mu^0 \notag\\
&\quad -  \left(2 \beta_1   e^{-\frac s2} + \tfrac{1}{\beta_\tau} h_W^{\mu,0} \right) \p_{11} A_\zeta^0 \phi_{\zeta \mu}  - 4 \beta_1   e^{-\frac s2}   ( \p_1 A_\mu^0 + e^{-\frac{3s}{2}} \dot{Q}_{\mu 1})  \p_1 A_\zeta^0  \phi_{\zeta \mu} \notag\\
&\quad  - \beta_3 e^{- \frac{s}{2}}  \left( Z^0 \p_{11} Z^0  - e^{-s} (1 - e^{s} (\p_1 Z^0)^2)  \right) (\phi_{22}+ \phi_{33})  + \tfrac 12  \beta_4 (e^{-\frac s2} + \p_1 Z^0)^2 \p_1 K^0 e^s
\notag\\
&\quad +  \beta_4 (\kappa - Z^0) \left(\tfrac 14 (\kappa - Z^0) \tempred{e^s \p_{111}K^0}  - (e^{-\frac s2} + \p_1 Z^0) e^s  \p_{11}K^0 - \tfrac 12 \p_{11}Z^0 e^s \p_{1}K  \right)
 \label{eq:FW:0:d}\\
\tfrac{1}{\beta_\tau} \p_{1\nu} F_W^0 &= - \beta_3  \left( (\kappa - Z^0) \p_{1 \nu \mu} A_\mu^0 - \p_{1\nu} Z^0 \p_\mu A_\mu^0 - \p_\nu Z^0 \p_{1\mu} A_\mu^0 - (e^{-\frac s2} + \p_1 Z^0) \p_{\nu \mu} A_\mu^0 \right) \notag\\
&\quad - 2 \beta_1  e^{-s} \p_1 A_\mu^0 \dot{\phi}_{\mu \nu} + 2 \beta_1  e^{-\frac s2} \dot Q_{1\mu} \p_{1\nu} A_\mu^0  - 2 \beta_1  e^{-s} \dot{Q}_{\mu \zeta} \p_1 A_\zeta^0 \phi_{\mu \nu} 
\notag\\
&\quad   - \beta_3 e^{-\frac s2}  ( \p_1 Z^0   \p_\nu Z^0  + Z^0 \p_{1\nu} Z^0) (\phi_{22} + \phi_{33}) \notag\\
&\quad - \beta_3  e^{-s} \left( (\kappa - Z^0) \p_1 A_\zeta^0  - (e^{-\frac s2} + \p_1 Z^0) A_\zeta^0 \right) \Tcal^{\zeta,0}_{\mu, \mu\nu}  
\notag\\
&\quad  - 2 \beta_1  e^{-\frac s2} \left( (e^{-\frac s2} \dot{Q}_{\mu \nu} + \p_\nu A_\mu^0) \p_1 A_\gamma^0 + (e^{-\frac{3s}{2}} \dot{Q}_{\mu 1} + \p_1 A_\mu^0) \p_\nu A_\gamma^0 + A_\mu^0  \p_{1 \nu} A_\gamma^0 \right) \phi_{\gamma \mu} \notag\\
&\quad  - \tfrac{1}{\beta_\tau} h_W^{\mu,0} \p_{1 \nu} A_\gamma^0  \phi_{\gamma \mu} +   \beta_1  e^{-s} \left(   ( \kappa + Z^0) \p_1 A_\gamma^0 - (e^{-\frac s2} - \p_1 Z^0) A_\gamma^0   \right) \phi_{\mu\nu}    \phi_{\gamma \mu}  
\notag\\
&\quad - \tfrac 12 \beta_4  (\kappa -Z^0) \left( \p_{1\nu}Z^0 e^s \p_1 K^0 + \p_\nu Z^0 e^s \p_{11} K^0 + (e^{-\frac s2} + \p_1 Z^0) (e^s \p_{1\nu} K^0 - \phi_{\mu \nu} e^{-\frac s2} \p_\mu K^0) \right)
\notag\\
&\quad + \tfrac 12  \beta_4 (e^{-\frac s2} + \p_1 Z^0) \p_\nu Z^0 \p_1 K^0 e^s 
+ \tfrac 14 \beta_3\beta_4 (\kappa - Z^0)^2 \left(\tempred{e^s \p_{11\nu} K^0} - \phi_{\mu \nu} e^{-\frac s2} \p_{1\mu}K^0 \right) 
\label{eq:FW:0:e}\\
\tfrac{1}{\beta_\tau} \p_{\gamma\nu} F_W^0 
&=  - 2 \beta_3   ( \p_{\nu \gamma} (K \p_\mu A_\mu))^0 - \beta_3  e^{-s} (\kappa - Z^0) \p_{\mu} A_\zeta^0 \Tcal^{\zeta,0}_{\mu,\nu\gamma} \notag\\
&\quad - 2 \beta_1  e^{-s} \p_\nu A_\mu^0 \dot{\phi}_{\mu \gamma}  - 2 \beta_1  e^{-s} \p_\gamma A_\mu^0 \dot{\phi}_{\mu \nu}  -  \beta_3 e^{-\frac s2}  \p_\gamma Z^0 \p_\nu Z^0   (\phi_{22} + \phi_{33}) \notag\\
&\quad + 2 \beta_1  e^{-\frac s2} \dot{Q}_{1\mu} \p_{\gamma \nu} A_\mu^0 - 2 \beta_1 e^{-s} \dot{Q}_{\zeta \mu} \p_\nu A_\mu^0   \phi_{\zeta\gamma}- 2 \beta_1  e^{-s} \dot{Q}_{\zeta \mu } \p_\gamma A_\mu^0   \phi_{\zeta \nu}  \notag\\
&\quad +2 \beta_1   e^{-\frac{3s}{2}} A_\mu^0 \left( \dot{Q}_{1\zeta} (\phi_{\nu \mu} \phi_{\zeta\gamma} + \phi_{\mu\gamma} \phi_{\zeta\nu} + \phi_{\nu \gamma} \phi_{\mu \zeta} +  \Tcal^{\mu,0}_{\zeta,\nu\gamma}) + \dot{Q}_{1\mu} \Ncal_{1,\nu\gamma}^0 \right) \notag\\
&\quad - \beta_3 e^{-s} \left( (\kappa - Z^0)  \p_\nu A_\zeta^0 - \p_\nu Z^0 A_\zeta^0 \right)\Tcal^{\zeta,0}_{\mu, \mu\gamma}  - \tfrac 12 \beta_3   e^{-\frac{3s}{2}} (\kappa - Z^0) (\kappa + Z^0) \Ncal^0_{\mu,\mu\nu\gamma}
\notag\\
&\quad  - 2 \beta_1 e^{-\frac s2} \left(  e^{-\frac s2} \dot{Q}_{\mu \nu} \p_\gamma A_\zeta^0 + e^{-\frac s2} \dot{Q}_{\mu \gamma} \p_\nu A_\zeta^0 + \p_{\nu\mu} A_\mu^0 A_\zeta^0 + \p_\mu A_\mu^0 \p_\nu A_\nu^0 + \p_\nu A_\mu^0 \p_\mu A_\nu^0\right) \phi_{\zeta \mu} \notag\\
&\quad + 2 \beta_1 e^{-s} \left( \p_\nu ( (U\cdot \Ncal) A_\zeta)^0 \phi_{\mu \gamma} \phi_{\zeta\mu} + \p_\gamma ( (U\cdot \Ncal) A_\zeta)^0 \phi_{\mu \nu} \phi_{\zeta\mu}\right)   - 2\beta_1 e^{-\frac{3s}{2}} A_\iota^0 A_\zeta^0 \Tcal^{\zeta,0}_{\mu, \nu \gamma} \phi_{\iota\mu}
\notag\\
&\quad - \tfrac{1}{\beta_\tau} h_{W}^{\mu,0} \p_{\nu\gamma} A_\zeta^0  \phi_{\zeta \mu} +  e^{-s} \tfrac{1}{\beta_\tau} h_W^{\mu,0} A_\iota^0 \left( \phi_{\iota\nu} \Ncal^0_{1,\mu\gamma} + \phi_{\iota \gamma} \Ncal^0_{1,\mu\nu} + \Ncal^0_{\alpha,\mu\nu\gamma} \right)
\notag\\
&\quad  + \tfrac 14  \beta_4 (\kappa-Z^0)^2 \left(\tempred{e^s \p_{1\gamma\nu}K^0} + (\phi_{2\nu} \phi_{2\gamma} + \phi_{3\nu}\phi_{3\gamma}) \p_1 K^0 - e^{-\frac s2}(\phi_{\mu \nu} \p_{\mu \gamma}K^0 + \phi_{\mu\gamma} \p_{\mu\nu}K^0)\right) 
\notag\\
&\quad  - \tfrac 12  \beta_4 (\kappa-Z^0) \left( \p_{\gamma}Z^0 (e^s \p_{1\nu} K^0 - \phi_{\mu\nu} e^{-\frac s2} \p_\mu K^0) + \p_\nu Z^0 (e^s \p_{1\gamma} K^0 - \phi_{\mu\gamma} e^{-\frac s2} \p_\mu K^0)\right)  
\notag\\
&\quad  +\tfrac 12  \beta_4 \left(\p_\gamma Z^0 \p_\nu Z^0- (\kappa-Z^0) \p_{\gamma \nu}Z^0\right) e^s \p_1S^0  
 \label{eq:FW:0:f}
\end{align} 
\end{subequations}

\subsection*{Acknowledgments} 
T.B.\ was supported by the NSF grant DMS-1900149 and a Simons Foundation Mathematical and Physical Sciences Collaborative Grant.   S.S.\ was supported by the Department of Energy Advanced Simulation and Computing (ASC) Program. V.V.\ was supported by the NSF grant DMS-1911413.



\begin{bibdiv}
\begin{biblist}



\bib{Al1999b}{article}{
      author={Alinhac, S.},
       title={Blowup of small data solutions for a class of quasilinear wave
  equations in two space dimensions. {II}},
        date={1999},
        ISSN={0001-5962},
     journal={Acta Math.},
      volume={182},
      number={1},
       pages={1\ndash 23},
         url={https://doi.org/10.1007/BF02392822},
      review={\MR{1687180}},
}

\bib{Al1999a}{article}{
      author={Alinhac, S.},
       title={Blowup of small data solutions for a quasilinear wave equation in
  two space dimensions},
        date={1999},
        ISSN={0003-486X},
     journal={Ann. of Math. (2)},
      volume={149},
      number={1},
       pages={97\ndash 127},
         url={https://doi.org/10.2307/121020},
      review={\MR{1680539}},
}

\bib{BuShVi2019a}{article}{
      author={Buckmaster, T.},
      author={Shkoller, S.},
      author={Vicol, V.},
      title={Formation of shocks for {2D} isentropic compressible {E}uler},
      year={2019},
      eprint={arXiv:1907.03784},
      archivePrefix={arXiv},
      primaryClass={math.AP}
}

\bib{BuShVi2019b}{article}{
      author={Buckmaster, T.},
      author={Shkoller, S.},
      author={Vicol, V.},
       title={Formation of point shocks for {3D}  compressible {E}uler},
      year={2019},
      eprint={arXiv:1912.04429},
      archivePrefix={arXiv},
      primaryClass={math.AP}
}

\bib{CaErHoLa1993}{article}{
   author={Caflisch, R.E.},
   author={Ercolani, N.},
   author={Hou, T.Y.},
   author={Landis, Y.},
   title={Multi-valued solutions and branch point singularities for
   nonlinear hyperbolic or elliptic systems},
   journal={Comm. Pure Appl. Math.},
   volume={46},
   date={1993},
   number={4},
   pages={453--499},
   issn={0010-3640},
   review={\MR{1211738}},
   doi={10.1002/cpa.3160460402},
}


\bib{CaSmWa1996}{article}{
      author={Cassel, K.W.},
      author={Smith, F.T.},
      author={Walker, J.D.A.},
       title={The onset of instability in unsteady boundary-layer separation},
        date={1996},
     journal={Journal of Fluid Mechanics},
      volume={315},
       pages={223\ndash 256},
}



\bib{Ch2007}{book}{
      author={Christodoulou, D.},
       title={The formation of shocks in 3-dimensional fluids},
      series={EMS Monographs in Mathematics},
   publisher={European Mathematical Society (EMS), Z\"{u}rich},
        date={2007},
        ISBN={978-3-03719-031-9},
         url={https://doi.org/10.4171/031},
      review={\MR{2284927}},
}

\bib{Ch2019}{book}{
      author={Christodoulou, D.},
       title={The shock development problem},
      series={EMS Monographs in Mathematics},
   publisher={European Mathematical Society (EMS), Z\"{u}rich},
        date={2019},
        ISBN={978-3-03719-192-7},
         url={https://doi.org/10.4171/192},
      review={\MR{3890062}},
}



\bib{ChMi2014}{book}{
      author={Christodoulou, D.},
      author={Miao, S.},
       title={Compressible flow and {E}uler's equations},
      series={Surveys of Modern Mathematics},
   publisher={International Press, Somerville, MA; Higher Education Press,
  Beijing},
        date={2014},
      volume={9},
        ISBN={978-1-57146-297-8},
      review={\MR{3288725}},
}



\bib{CoGhMa2018}{article}{
      author={Collot, C.},
      author={Ghoul, T.-E.},
      author={Masmoudi, N.},
       title={Singularity formation for {B}urgers equation with transverse
  viscosity},
     year={2018},
      eprint={arXiv:1803.07826},
      archivePrefix={arXiv},
      primaryClass={math.AP}
}

\bib{Da2010}{book}{
      author={Dafermos, C.~M.},
       title={Hyperbolic conservation laws in continuum physics},
     edition={Third},
      series={Grundlehren der Mathematischen Wissenschaften [Fundamental
  Principles of Mathematical Sciences]},
   publisher={Springer-Verlag, Berlin},
        date={2010},
      volume={325},
        ISBN={978-3-642-04047-4},
         url={https://doi.org/10.1007/978-3-642-04048-1},
      review={\MR{2574377}},
}




\bib{Euler1757}{article}{
	Author = {Euler, L.},
	Date-Added = {2020-05-14 16:22:06 -0400},
	Date-Modified = {2020-05-14 16:22:06 -0400},
	Journal = {Acad\'emie Royale des Sciences et des Belles Lettres de Berlin, M\'emoires},
	Pages = {274--315},
	Title = {{P}rincipes g\'en\'eraux du mouvement des fluides},
	Volume = {11},
	Year = {1757}}

\bib{John1974}{article}{
      author={John, F.},
       title={Formation of singularities in one-dimensional nonlinear wave
  propagation},
        date={1974},
        ISSN={0010-3640},
     journal={Comm. Pure Appl. Math.},
      volume={27},
       pages={377\ndash 405},
         url={https://doi.org/10.1002/cpa.3160270307},
      review={\MR{0369934}},
}



\bib{Lax1964}{article}{
      author={Lax, P.~D.},
       title={Development of singularities of solutions of nonlinear hyperbolic
  partial differential equations},
        date={1964},
        ISSN={0022-2488},
     journal={J. Mathematical Phys.},
      volume={5},
       pages={611\ndash 613},
         url={https://doi.org/10.1063/1.1704154},
      review={\MR{0165243}},
}

\bib{Li1979}{article}{
      author={Liu, T.~P.},
       title={Development of singularities in the nonlinear waves for
  quasilinear hyperbolic partial differential equations},
        date={1979},
        ISSN={0022-0396},
     journal={J. Differential Equations},
      volume={33},
      number={1},
       pages={92\ndash 111},
         url={https://mathscinet.ams.org/mathscinet-getitem?mr=540819},
      review={\MR{540819}},
}

\bib{LuSp2018}{article}{
      author={Luk, J.},
      author={Speck, J.},
       title={Shock formation in solutions to the 2{D} compressible {E}uler
  equations in the presence of non-zero vorticity},
        date={2018},
        ISSN={0020-9910},
     journal={Invent. Math.},
      volume={214},
      number={1},
       pages={1\ndash 169},
         url={https://doi.org/10.1007/s00222-018-0799-8},
      review={\MR{3858399}},
}

\bib{Ma1984}{book}{
      author={Majda, A.},
       title={Compressible fluid flow and systems of conservation laws in
  several space variables},
      series={Applied Mathematical Sciences},
   publisher={Springer-Verlag, New York},
        date={1984},
      volume={53},
        ISBN={0-387-96037-6},
         url={https://doi.org/10.1007/978-1-4612-1116-7},
      review={\MR{748308}},
}


\bib{Merle96}{article}{
      author={Merle, F.},
       title={Asymptotics for {$L^2$} minimal blow-up solutions of critical
  nonlinear {S}chr\"{o}dinger equation},
        date={1996},
        ISSN={0294-1449},
     journal={Ann. Inst. H. Poincar\'{e} Anal. Non Lin\'{e}aire},
      volume={13},
      number={5},
       pages={553\ndash 565},
         url={https://mathscinet.ams.org/mathscinet-getitem?mr=1409662},
      review={\MR{1409662}},
}

\bib{MeRa05}{article}{
      author={Merle, F.},
      author={Raphael, P.},
       title={The blow-up dynamic and upper bound on the blow-up rate for
  critical nonlinear {S}chr\"{o}dinger equation},
        date={2005},
        ISSN={0003-486X},
     journal={Ann. of Math. (2)},
      volume={161},
      number={1},
       pages={157\ndash 222},
         url={https://mathscinet.ams.org/mathscinet-getitem?mr=2150386},
      review={\MR{2150386}},
}

\bib{MeRaRoSz2019a}{article}{
      author={Merle, F.},
      author={Raphael, P.},
      author={Rodnianski, I.},
      author={Szeftel, J.},
      title={On the implosion of a three dimensional compressible fluid},
      year={2019},
      eprint={arXiv:1912.11009},
      archivePrefix={arXiv},
    primaryClass={math.AP}
}

\bib{MeRaRoSz2019b}{article}{
      author={Merle, F.},
      author={Raphael, P.},
      author={Rodnianski, I.},
      author={Szeftel, J.},
      title={On smooth self similar solutions to the compressible Euler equations},
     year={2019},
     eprint={arXiv:1912.10998},
     archivePrefix={arXiv},
     primaryClass={math.AP}
}

\bib{MeRaSz2018}{article}{    
    author = {Merle, F.} 
    author = {Rapha\"el, P.}
    author = {Szeftel, J.},
    title = {On Strongly Anisotropic Type I Blowup},
    journal = {International Mathematics Research Notices},
    volume = {2020},
    number = {2},
    pages = {541-606},
    year = {2018},
    issn = {1073-7928},
    doi = {10.1093/imrn/rny012},
    url = {https://doi.org/10.1093/imrn/rny012},
}


\bib{MeZa97}{article}{
      author={Merle, F.},
      author={Zaag, H.},
       title={Stability of the blow-up profile for equations of the type
  {$u_t=\Delta u+|u|^{p-1}u$}},
        date={1997},
        ISSN={0012-7094},
     journal={Duke Math. J.},
      volume={86},
      number={1},
       pages={143\ndash 195},
         url={https://mathscinet.ams.org/mathscinet-getitem?mr=1427848},
      review={\MR{1427848}},
}

\bib{Mi2018}{article}{
   author={Miao, S.},
   title={On the formation of shock for quasilinear wave equations with weak
   intensity pulse},
   journal={Ann. PDE},
   volume={4},
   date={2018},
   number={1},
   pages={Paper No. 10, 140},
   issn={2524-5317},
   review={\MR{3780823}},
   doi={10.1007/s40818-018-0046-z},
}

\bib{MiYu2017}{article}{
      author={Miao, S.},
      author={Yu, P.},
       title={On the formation of shocks for quasilinear wave equations},
        date={2017},
        ISSN={0020-9910},
     journal={Invent. Math.},
      volume={207},
      number={2},
       pages={697\ndash 831},
         url={https://doi.org/10.1007/s00222-016-0676-2},
      review={\MR{3595936}},
}


\bib{Ri1860}{article}{
      author={Riemann, B.},
       title={{\"{U}}ber die {F}ortpflanzung ebener {L}uftwellen von endlicher
  {S}chwingungsweite},
        date={1860},
     journal={Abhandlungen der K\"oniglichen Gesellschaft der Wissenschaften in
  G\"ottingen},
      volume={8},
       pages={43\ndash 66},
         url={http://eudml.org/doc/135717},
}

\bib{Si1985}{article}{
      author={Sideris, T.~C.},
       title={Formation of singularities in three-dimensional compressible
  fluids},
        date={1985},
        ISSN={0010-3616},
     journal={Comm. Math. Phys.},
      volume={101},
      number={4},
       pages={475\ndash 485},
         url={http://projecteuclid.org/euclid.cmp/1104114244},
      review={\MR{815196}},
}

\bib{Sp2016}{book}{
   author={Speck, J.},
   title={Shock formation in small-data solutions to 3D quasilinear wave
   equations},
   series={Mathematical Surveys and Monographs},
   volume={214},
   publisher={American Mathematical Society, Providence, RI},
   date={2016},
   pages={xxiii+515},
   isbn={978-1-4704-2857-0},
   review={\MR{3561670}},
}

\bib{Sp2018}{article}{
   author={Speck, J.},
   title={Shock formation for $2D$ quasilinear wave systems featuring
   multiple speeds: blowup for the fastest wave, with non-trivial
   interactions up to the singularity},
   journal={Ann. PDE},
   volume={4},
   date={2018},
   number={1},
   pages={Paper No. 6, 131},
   issn={2524-5317},
   review={\MR{3740634}},
   doi={10.1007/s40818-017-0042-8},
}

\end{biblist}
\end{bibdiv}
\end{document}